%% file: main_imbalance.tex
\documentclass[11pt]{article}

\input{macros}
\usepackage{mathrsfs}
\usepackage{setspace} 
\allowdisplaybreaks
\usepackage{comment}
\usepackage[normalem]{ulem}
\usepackage{mlmodern}
\DeclareSymbolFont{largesymbols}{OMX}{cmex}{m}{n}

\usepackage{etoc} 

\newcommand{\Err}{\mathrm{Err}}
\providecommand{\keywords}[1]
{
  {
  \small	
  \textbf{\textit{Keywords---}} #1
  }
}

\begin{document}

\title{A statistical theory of overfitting for imbalanced classification}
 










\author{Jingyang Lyu, Kangjie Zhou, and Yiqiao Zhong}
\author{Jingyang Lyu\thanks{Department of Statistics, University of Wisconsin--Madison, Madison, WI 53706, USA. Emails:
\texttt{jlyu55@wisc.edu}, \texttt{yiqiao.zhong@wisc.edu}} \and Kangjie Zhou\thanks{Department of Statistics, Columbia University, New York, NY 10027, USA. Email: \texttt{kz2326@columbia.edu}}
\and Yiqiao Zhong\footnotemark[1]}
\date{\today}

\pagenumbering{arabic}
\maketitle

\begin{abstract}
Classification with imbalanced data is a common challenge in data analysis, where certain classes (minority classes) account for a small fraction of the training data compared with other classes (majority classes). Classical statistical theory based on large-sample asymptotics and finite-sample corrections is often ineffective for high-dimensional data, leaving many overfitting phenomena in empirical machine learning unexplained.

In this paper, we develop a statistical theory for high-dimensional imbalanced classification by investigating support vector machines and logistic regression. We find that dimensionality induces truncation or skewing effects on the logit distribution, which we characterize via a variational problem under high-dimensional asymptotics. In particular, for linearly separable data generated from a two-component Gaussian mixture model, the logits from each class follow a normal distribution $\normal(0,1)$ on the testing set, but asymptotically follow a rectified normal distribution $\max\{\kappa, \normal(0,1)\}$ on the training set---which is a pervasive phenomenon we verified on tabular data, image data, and text data. This phenomenon explains why the minority class is more severely affected by overfitting. Further, we show that margin rebalancing, which incorporates class sizes into the loss function, is crucial for mitigating the accuracy drop for the minority class. Our theory also provides insights into the effects of overfitting on calibration and other uncertain quantification measures.






\end{abstract}

\keywords{
    Imbalanced classification, overfitting, margin, logistic regression, support vector machine, overparametrization, calibration
}

\etocdepthtag.toc{mtchapter}
\etocsettagdepth{mtchapter}{subsection}
\etocsettagdepth{mtappendix}{section}
\tableofcontents

\newpage
\input{src/intro}

\input{src/related}

\input{src/setup}
\input{src/main_logit_distribution}
\input{src/main_rebalance_margin}
\input{src/main_calibration}
\input{src/discussions}

\section*{Acknowledgments}

Y.Z.~is supported by NSF-DMS grant 2412052 and by the Office of the Vice Chancellor for Research and Graduate Education at the UW Madison with funding from the Wisconsin Alumni Research Foundation. K.Z.~is supported by the Founder's Postdoctoral Fellowship in Statistics at Columbia University. J.L.~is grateful for the feedback from Zhexuan Liu, Zexuan Sun, Zhuoyan Xu, Congwei Yang, Zhihao Zhao, and audience from the Institute for Foundations of Data Science (IFDS) Ideas Forum.

\newpage

\appendix

\etocdepthtag.toc{mtappendix}
\etocsettagdepth{mtchapter}{none}
\etocsettagdepth{mtappendix}{subsubsection}
\tableofcontents


\input{src/append_experiments}

\input{src/append_prelim}
\input{src/append_overparam_prf}

\input{src/append_underparam_prf}
\input{src/append_margin_reb}
\input{src/append_high_imb_prf}

\input{src/append_calibration}
\input{src/append_tech_lem}

\newpage
\bibliographystyle{unsrt}
\bibliography{refs}
\end{document}

%% file: macros.tex
\usepackage{amsmath,amssymb,amsfonts,amsthm,bbm}
\usepackage{setspace}
\usepackage{graphicx}
\usepackage{enumitem}
\usepackage[pagebackref,letterpaper=true,colorlinks=true,pdfpagemode=none,citecolor=OliveGreen,linkcolor=BrickRed,urlcolor=BrickRed,pdfstartview=FitH,linktocpage=true]{hyperref}

\usepackage[dvipsnames]{xcolor}
\usepackage[capitalise]{cleveref}

\crefname{figure}{Figure}{Figures}
\Crefname{figure}{Figure}{Figures}

\usepackage{array}
\usepackage{mathtools}      
\usepackage{multirow}

\newif\ifStat
\newif\ifNote
\newif\ifComments
\newif\ifThmitalic

\Stattrue
\Notefalse
\Commentstrue
\Thmitalictrue

\ifComments
\else
\fi

\ifStat
\renewcommand{\hat}{\widehat}
\renewcommand{\tilde}{\widetilde}
\renewcommand{\bar}{\overline}
\renewcommand{\top}{{\sf T}}
\else

\fi
\newcommand{\iidsim}{\overset{\text{i.i.d.}}{\sim}}

\usepackage[top=1in, bottom=1in, left=1in, right=1in]{geometry}

\makeatletter
\newcommand*{\rom}[1]{\expandafter\@slowromancap\romannumeral #1@}
\makeatother

\newcommand{\<}{\langle}
\renewcommand{\>}{\rangle}
\newcommand{\parallelsum}{\mathbin{\!/\mkern-5mu/\!}}

\newcommand{\spann}{\mathrm{span}}

\newcommand{\sign}{\mathrm{sign}}

\newcommand{\var}{\mathrm{Var}}

\DeclareMathOperator*{\argmin}{arg\,min}
\DeclareMathOperator*{\argmax}{arg\,max}

\newcommand{\iid}{\mathrm{i.i.d.}}
\newcommand{\beq}{\begin{equation}}
\newcommand{\eeq}{\end{equation}}

\newcommand{\Law}{\mathtt{Law}}

\newcommand{\indep}{\perp\!\!\!\perp}
\DeclareMathOperator*{\maximize}{\mathrm{maximize}}
\DeclareMathOperator*{\minimize}{\mathrm{minimize}}
\newcommand{\ind}{\mathbbm{1}}
\newcommand{\conp}{\xrightarrow{\mathrm{p}}}
\newcommand{\cond}{\xrightarrow{\mathrm{d}}}

\newcommand{\conL}[1]{\xrightarrow{\mathcal{L}^{#1}}}
\newcommand{\conw}{\xrightarrow{w}}
\newcommand{\prox}{\mathsf{prox}}
\newcommand{\envelope}{\mathsf{e}}

\newcommand{\abs}[1]{\left|{#1}\right|}
\newcommand{\norm}[1]{\left\|{#1}\right\|}

\newcommand{\R}{\mathbb{R}}

\newcommand{\E}{\mathbb{E}}

\newcommand{\veps}{\varepsilon}

\newcommand{\Q}{\mathbb{Q}}

\renewcommand{\P}{\mathbb{P}}
\renewcommand{\S}{\mathbb{S}}

\renewcommand{\d}{\textup{d}}

\newcommand{\ba}{\mathrm{\bf a}}
\newcommand{\be}{\mathrm{\bf e}}

\newcommand{\bs}{\mathrm{\bf s}}

\newcommand{\bu}{\mathrm{\bf u}}
\newcommand{\bv}{\mathrm{\bf v}}
\newcommand{\bw}{\mathrm{\bf w}}
\newcommand{\bx}{\mathrm{\bf x}}
\newcommand{\by}{\mathrm{\bf y}}

\newcommand{\bA}{\mathrm{\bf A}}
\newcommand{\bB}{\mathrm{\bf B}}

\newcommand{\bW}{\mathrm{\bf W}}

\newcommand{\bI}{\mathrm{\bf I}}

\newcommand{\bP}{\mathrm{\bf P}}

\newcommand{\bQ}{\mathrm{\bf Q}}

\def\vmu{{\boldsymbol{\mu}}}
\def\vtheta{{\boldsymbol{\theta}}}
\def\vbeta{{\boldsymbol{\beta}}}
\def\va{{\boldsymbol{a}}}

\def\vf{{\boldsymbol{f}}}
\def\vg{{\boldsymbol{g}}}
\def\vh{{\boldsymbol{h}}}

\def\vu{{\boldsymbol{u}}}
\def\vw{{\boldsymbol{w}}}
\def\vx{{\boldsymbol{x}}}

\newcommand{\xx}{\text{\boldmath $x$}}
\newcommand{\GG}{\text{\boldmath $G$}}

\newcommand{\XX}{\text{\boldmath $X$}}
\newcommand{\yy}{\text{\boldmath $y$}}
\newcommand{\zz}{\text{\boldmath $z$}}

\newcommand{\ZZ}{\text{\boldmath $Z$}}

\newcommand{\uu}{\text{\boldmath $u$}}

\newcommand{\hh}{\text{\boldmath $h$}}

\newcommand{\bm}{\text{\boldmath $m$}}

\newcommand{\bbeta}{\text{\boldmath $\beta$}}

\newcommand{\btheta}{\text{\boldmath $\theta$}}

\newcommand{\bone}{\mathrm{\bf 1}}
\newcommand{\bzero}{\mathrm{\bf 0}}

\newcommand{\bmu}{\text{\boldmath $\mu$}}

\newcommand{\blambda}{\text{\boldmath $\lambda$}}

\newcommand{\bTheta}{\text{\boldmath $\Theta$}}

\newcommand{\bXi}{\text{\boldmath $\Xi$}}

\newcommand{\cD}{\mathcal{D}}

\newcommand{\cL}{\mathcal{L}}

\newcommand{\cM}{\mathcal{M}}

\def\normal{{\sf N}}
\def\id{{\boldsymbol I}}

\def\subGind{{\sf{subG}_{\indep}}}

\DeclareMathOperator*{\pliminf}{p-lim\, inf}
\DeclareMathOperator*{\plim}{p-lim}

\newcommand{\wt}{\widetilde}

\newcommand{\TODO}[1]{{\color{red}{[#1]}}}

\newcommand{\ljyedit}[2]{%
  \ifmmode
    \text{\sout{\ensuremath{#1}}} \color{blue!60}{\text{#2}}%
  \else
    \sout{#1} {\color{blue!60}#2}%
  \fi
}

\ifNote

\theoremstyle{plain}
\newtheorem{thm}{Theorem}
\newtheorem{claim}{Claim}
\newtheorem{defn}{Definition}
\newtheorem{lem}[thm]{Lemma}
\newtheorem{cor}[thm]{Corollary}
\newtheorem{prop}[thm]{Proposition}
\newtheorem{ass}{Assumption}
\newtheorem{fact}{Fact}
\theoremstyle{definition}
\newtheorem{exm}{Example}
\newtheorem{rem}{Remark}
\newtheorem{conj}{Conjecture}

\else\ifThmitalic

\newtheorem{thm}{Theorem}[section]
\newtheorem{claim}[thm]{Claim}
\newtheorem{lem}[thm]{Lemma}
\newtheorem{cor}[thm]{Corollary}
\newtheorem{prop}[thm]{Proposition}

\theoremstyle{definition}
\newtheorem{defn}{Definition}[section]

\newtheorem{rem}{Remark}[section]
\newtheorem{conj}{Conjecture}[section]

\else

\fi
\fi

%% file: src/intro.tex
\section{Introduction}\label{sec:intro}

Classification tasks are ubiquitous in statistics and machine learning. In many practical applications, training data are often imbalanced, meaning that some classes (minority classes) contain substantially fewer samples than others. In binary classification, particularly, we observe training data $\{(\vx_i, y_i)\}_{i = 1}^n \stackrel{\mathrm{i.i.d.}}{\sim} P_{\xx,y}$ with features $\xx_i \in \R^d$ and binary labels $y_i \in \{\pm 1\}$. Denote by $P_{\vx}$ (resp. $P_y$) the marginal distribution of $\vx$ (resp. $y$), and the expected fractions of the two classes by
\begin{equation*}
\pi_+ := \P(y_i = +1), \qquad \pi_- :=  \P(y_i = -1).
\end{equation*}
We say that the data set is imbalanced if $\pi_+ < \pi_-$\footnote{Without loss of generality, we assume that the minority class is assigned the label $+1$.}. 
Classification problems with class imbalance are common in applications where the minority class represents rare diseases, rare events, anomalies, or underrepresented groups \cite{litjens2017survey, tschandl2018ham10000, king2001logistic, kubat1998machine, ngai2011application, chandola2009anomaly, weiss2003learning, buolamwini2018gender}.

In this paper, we focus on linear classification, where the classifier takes the form of 
$\vx \mapsto 2 \mathbbm{1}\{ f(\vx) > 0\} - 1$ with $f(\vx) = \langle \vx, \vbeta \rangle + \beta_0$. This simple form is widely used in statistics and machine learning: (i) For tabular data, linear classification algorithms such as logistic regression and the support vector machine (SVM) are commonly applied directly to the training data; (ii) For image and text data, the last classification layer of a deep neural network (DNN) usually takes this form (also known as softmax regression), where $\vx_i$ represents the feature vector extracted by the previous layers of the network. In downstream analysis, the base network is often frozen, and only the linear classification layer is retrained.

\paragraph{Challenges of high dimensions.}  For low-dimensional problems where the dimension $d$ is fixed or satisfies $d \ll n$, prediction performance and estimation accuracy are well understood for standard classification methods. Generally, we expect estimation consistency, where the estimated parameter vector $\hat \vbeta$ is close to the target $\vbeta$, and a small generalization gap where the training error is close to the test error. However, for high-dimensional problems where $d$ is typically comparable to $n$, the classical theory depicts an inaccurate picture, motivating recent efforts to refine the asymptotic characterization of learning behavior under high-dimensional regimes.

\begin{table}[h!]
\centering
\begin{tabular}{c|c|c}
   & \textbf{Low dimensions}        & \textbf{High dimensions}   \\ 
   \hline
   \rule[-1ex]{0pt}{4ex}
Parameter estimation                                           & $\left\langle \frac{\hat \vbeta}{\| \hat \vbeta\|}, \frac{\vbeta}{\| \vbeta \|} \right\rangle \approx 1$ & $\left\langle \frac{\hat \vbeta}{\| \hat \vbeta\|}, \frac{\vbeta}{\| \vbeta \|} \right\rangle < 1$ \\
Generalization       & Train error $\approx$ Test error                       & Train error $<$ Test error   \\
Distribution of logits & 1D projection of $P_{\boldsymbol{x}}$   & Skewed/distorted 1D projection of $P_{\boldsymbol{x}}$
\end{tabular}
\caption{Qualitative comparison between low/high dimensions for binary classification, where a linear classifier $\hat{y} (\xx) = 2 \mathbbm{1}\{ \hat f(\vx) > 0\} - 1$ with $\hat f(\vx) = \langle \vx, \hat \vbeta \rangle + \hat \beta_0$ is trained on $\{(\vx_i, y_i)\}_{i=1}^n \stackrel{\iid}{\sim} P_{\vx, y}$. Here, the logits $\{\hat f(\boldsymbol{x}_i)\}_{i = 1}^n $ are obtained by evaluating $\hat{f}$ on the training set.}\label{tab:1}
\end{table}

Regarding parameter estimation, a recent line of work \cite{dobriban2018high, Pragya_highdim_logistic, sur2019logistic, candes2020logistic, montanari2023generalizationerrormaxmarginlinear} has studied the asymptotic properties of logistic regression under the proportional regime $n / d \to \delta$, as $n, d \to \infty$ for some constant $\delta$. Qualitatively speaking, as the dimension $d$ increases (or $\delta$ decreases), the estimation error of the maximum likelihood estimator (MLE) $\hat \vbeta$ continues to grow until $\delta$ reaches a critical threshold, below which the MLE no longer exists \cite{Pragya_highdim_logistic, sur2019logistic, candes2020logistic}. In addition, the classical likelihood test requires modifications to remain valid in high dimensions \cite{Pragya_highdim_logistic, sur2019logistic, candes2020logistic}. 

Regarding generalization, high dimensionality usually leads to a gap between the training and test errors. To further investigate this phenomenon, recent research has focused on understanding the interplay between memorization and generalization, inspired by the \textit{double descent} phenomenon \cite{belkin2019reconciling}. In particular, in overparametrized models where the number of parameters far exceeds the sample size, it has been shown that algorithms such as gradient descent have an implicit regularization effect, leading to benign overfitting \cite{bartlett2020benign}.

Finally, the distribution of logits is highly valuable for feature visualization and interpretation, yet relatively little theory has been developed in this area. One practical example is \textit{linear probing}, a common approach for interpreting the hidden states (or activations) in empirical deep learning. This method involves training a simple linear classifier on top of the hidden states and visualizing their projections \cite{kornblith2019better, he2020momentum, kumar2022fine}. Another example is \textit{projection pursuit} (PP), where data are visualized via low-dimensional projections for exploratory data analysis. A few recent papers have investigated theoretical properties of PP in high dimensions \cite{Bickel_high_dim_PP, montanari2022overparametrizedlineardimensionalityreductions}.

\paragraph{Challenges of data imbalance.} In classification problems, data imbalance poses significant challenges, particularly for the minority class. First, classical asymptotic theory can become practically unreliable under severe data imbalance, compromising the accuracy of maximum likelihood estimation. Second, imbalanced classification is particularly susceptible to label shifts, where the proportion of minority labels (denoted as $\pi_+$ in binary classification) differs in a new test dataset. Third, misclassifying data points from the minority class typically incurs a much higher cost, a factor that is not accounted for in standard (unweighted) logistic regression. For a comprehensive overview, see \cite{he2009learning}. To address these issues, various techniques have been proposed and employed in empirical studies, such as adjusting decision boundaries, reweighting loss functions, subsampling the majority classes, and oversampling the minority classes \cite{king2001logistic, chawla2002smote}, among others.

Overfitting in high dimensions further compounds the challenges of imbalanced classification. Empirical studies have shown that deep learning models, particularly those with large capacity, often suffer from a disproportionate accuracy drop due to overfitting. This is because such models tend to memorize data points from minority classes rather than generalize to them, as these points constitute only a small fraction of the training data \cite{sagawa2020investigation}. While several remedies have been proposed in the literature \cite{huang2016learning, khan2019striking, liu2019large, cao2019learning}, they appear to be ad hoc fixes and fail to provide guidance on hyperparameter selection or feature interpretation.

We identify two major limitations in the existing literature. First, it remains unclear why overfitting is generally more severe in minority classes, despite being consistently observed in empirical studies. Second, there is a lack of comprehensive analysis of the impacts of various key factors---such as dimensionality, imbalance ratios, and signal strength---on the performance metrics such as test accuracy and uncertainty quantification.

\paragraph{Our goal.} The aim of this paper is to develop a statistical theory to address the aforementioned limitations. In particular, we seek to answer the following key questions:

\begin{enumerate}
    \item[\texttt{Q1}.] Can we mathematically characterize overfitting in high-dimensional imbalanced classification?
    \item[\texttt{Q2}.] What are the adverse effects of overfitting, particularly on the minority class?
    \item[\texttt{Q3}.]
    What are the consequences of overfitting for uncertainty quantification, such as calibration?
\end{enumerate}
Building on theoretical tools from high-dimensional statistics, our  analysis focuses on a stylized model. Suppose the i.i.d.~training data $\{(\vx_i, y_i)\}_{i = 1}^n$ are generated from a two-component Gaussian mixture model (2-GMM):
\begin{equation}\label{model}
    \P(y_i = +1) = \pi, \quad \P(y_i = -1) = 1 - \pi, \quad \xx_i \,|\, y_i \sim \normal(y_i \bmu, \bI_d),
\end{equation}
where $\bmu \in \R^d$ is an unknown signal vector. Under this model, the Bayes-optimal classifier has the form $y^* (\xx) = 2\ind\{\langle \vx, \vbeta \rangle + \beta_0 > 0\} - 1$,
where $\vbeta \parallelsum \vmu$. We analyze the behavior of two standard approaches for binary classification: (a slightly generalized version of) logistic regression and support vector machines (SVMs). Denoting by $\ell: \R \to \R$ a strictly convex decreasing function, including the logistic function $\log(1 + e^{-x})$ as a special case, we solve
\begin{subequations}
\begin{align}
\begin{array}{lcl}
    \text{(logistic regression)} 
    & 
    \mathmakebox[\widthof{$ \displaystyle\maximize \limits_{\bbeta \in \R^d, \, \beta_0, \kappa \in \R} $}][c]{
    \minimize \limits_{\vbeta \in \R^d, \beta_0 \in \R}
    } 
    & 
    \mathmakebox[\widthof{$ \displaystyle y_i ( \< \xx_i, \bbeta \> + \beta_0 ) \ge \kappa,
	\quad \forall\, i \in [n], $}][l]{
    \displaystyle \frac{1}{n} \sum_{i=1}^n \ell \bigl( y_i(\langle \vx_i, \vbeta \rangle + \beta_0) \bigr),
    }
\end{array}
\label{eq:logistic}
\\
\begin{array}{lcl}
    \text{(SVM)}  & 
    \maximize \limits_{\bbeta \in \R^d, \, \beta_0, \kappa \in \R} & \kappa,  \\
    & \text{subject to}\vphantom{\displaystyle\max_\kappa} & y_i ( \< \xx_i, \bbeta \> + \beta_0 ) \ge \kappa,
	\quad \forall\, i \in [n],    
    \\
    & & \norm{\bbeta}_2  \le  1.
\end{array}
\label{eq:SVM-0}
\end{align}
\end{subequations}
Both optimization problems are convex and yield solutions $\hat \vbeta, \hat \beta_0$, which are used to predict class labels for a test data point $\vx$ based on $\hat f(\vx) = \langle \vx, \hat \vbeta \rangle + \hat \beta_0$. Namely, the predicted binary label of a test data point $\vx$ is $\hat{y} (\xx) = 2\mathbbm{1}\{\langle \vx, \hat{\vbeta} \rangle + \hat{\beta}_0 > 0\} - 1$.

We will analyze both classifiers with a focus on the SVM for the following reason. In modern machine learning, it is common for the labeled data $\{(\vx_i, y_i)\}_{i = 1}^n$ to be linearly separable due to high dimensionality.
When data are linearly separable, the hard-margin SVM coincides with the max-margin classifier. It is known that the gradient descent iterates of logistic regression converge in direction to the max-margin solution \cite{Soudry_implicit_bias, ji2019riskparameterconvergencelogistic}, which is known as a form of inductive bias \cite{neyshabur2015searchrealinductivebias}. See Section~\ref{sec:background} for the background. In this sense, the two classifiers are closely related.

The code for our experiments can be found in the GitHub repository: 
\begin{center}
\url{https://github.com/jlyu55/Imbalanced_Classification}
\end{center}

\subsection{Characterizing overfitting via empirical logit distribution}
\label{subsec:ELD}

Understanding why test accuracy drops more for the minority class requires a more refined characterization of overfitting. To this end, we study the empirical distribution of the logits $\hat f(\vx_i) = \langle \vx_i, \hat\vbeta \rangle + \hat \beta_0$, $i \in [n]$ on the training set. 

Let $(\xx_1, y_1), \ldots, (\xx_n, y_n) \iidsim P_{\xx, y}$ be the training data, and $(\vx_\mathrm{test}, y_\mathrm{test})$ be an independent test data point. Denote $\mathcal{I}_+ := \{ i \in [n]: y_i = +1 \}$ and $\mathcal{I}_- := \{ i \in [n]: y_i = -1 \}$ as the index sets of training data points from the minority class and the majority class, respectively. Let $n_+ := \abs{\mathcal{I}_+}$ and $n_- := \abs{\mathcal{I}_-}$ be the sample sizes of the two classes. Consider a binary classifier $\hat{y}: \R^d \to \{ \pm 1 \}$ based on $\hat f: \R^d \to \R$ such that we predict $\hat y = + 1$ if $\hat f(\vx) > 0$ and predict $\hat y=-1$ otherwise. Throughout this paper, we ignore the one-class degenerate case and implicitly assume both $n_+, n_- \ge 1$, which occurs with high probability.

\begin{defn}[Logit and margin]
Let $(\xx, y) \in \R^d \times \{ \pm 1 \}$ be a data point. For a binary classifier of the form $\hat y(\xx) = 2 \mathbbm{1}\{ \hat f(\vx) > 0\} - 1$, we define:
\begin{itemize}
    \item The \emph{logit} of $\xx$ is $\hat f(\xx)$.
    \item The \emph{logit margin} of $\xx$ is $y \hat f(\xx)$.
    \item The \emph{margin} of the classifier $\hat f$ (on the training data) is $\hat\kappa_n = \min_{i \in [n]} y_i \hat f(\vx_i)$.
\end{itemize}
\end{defn}
The following definitions highlight the logit distribution on both training and test data. 
\begin{defn}[ELD and TLD]
\label{def:ELD_TLD}
\begin{figure}[t]
    \centering
    \includegraphics[width=0.75\textwidth]{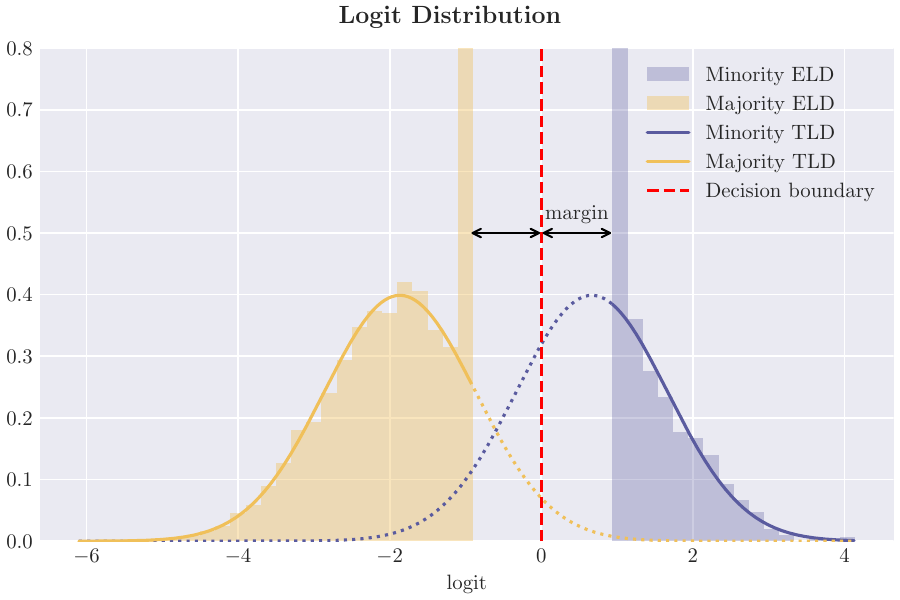}
    \caption{
    \textbf{Empirical logit distribution (ELD) and testing logit distribution (TLD)}. We train a max-margin classifier (namely SVM) $\hat f$ on synthetic data from a 2-component Gaussian mixture model. Colors indicate labels $y_i$ and $x$-axis indicates logits $\hat f(\vx_i)$. \textbf{ELD for both classes:} rectified Gaussian distribution (histogram). \textbf{TLD for both classes:} Gaussian distribution (curve).  \textbf{Overfitting effect:} The density areas below the dotted curves are overlapping in TLD, thus leading to positive test error; however they are ``pushed'' to respective margin boundaries in ELD, thus leading to linear separability and zero training errors.
    }
    \label{fig:GMM_main}
\end{figure}

\begin{enumerate}
\item \textbf{\emph{Empirical logit distribution (ELD)}}, or \emph{training logit distribution}, is defined as the empirical distribution of label-logit pairs based on training data, 
that is,
\begin{equation}\label{eq:ELD}
\hat\nu_n = \frac{1}{n} \sum_{i=1}^n \delta_{(y_i, \hat f(\vx_i))},
\end{equation}
where $\delta_{\va}$ denotes the delta measure supported at point $\va$.

\emph{Minority ELD} and \emph{majority ELD} are defined respectively as the empirical distribution of logits based on training data from minority class and majority class, i.e.,
\begin{equation*}
    \frac{1}{n_+} \sum_{i \in \mathcal{I}_+} \delta_{\hat f(\vx_i)}
    \qquad \text{and}
    \qquad
    \frac{1}{n_-} \sum_{i \in \mathcal{I}_-} \delta_{\hat f(\vx_i)}.
\end{equation*}
Note that these ELDs are all random probability measures.

\item \textbf{\emph{Testing logit distribution (TLD)}}, is defined as the distribution of the label-logit pair for a test data point ($\Law$ means the distribution of random variables/vectors), that is, 
\begin{equation*}
    \hat\nu^\mathrm{test}_n = \Law \, \bigl (y_\mathrm{test},  \hat f(\vx_\mathrm{test}) \bigr).
\end{equation*}

\emph{Minority TLD} and \emph{majority TLD} are defined respectively as the distribution of the logit for a test data point from minority class and majority class, i.e.,
\begin{equation*}
    \Law \, \bigl ( \hat f(\vx_\mathrm{test}) \,|\, y_\mathrm{test} = +1 \bigr)
    \qquad \text{and}
    \qquad
    \Law \, \bigl ( \hat f(\vx_\mathrm{test}) \,|\, y_\mathrm{test} = -1 \bigr).
\end{equation*}
Note that the randomness in TLDs is taken over both the classifier $\hat f$ and the test point $(\vx_\mathrm{test}, y_\mathrm{test})$, so they are deterministic probability measures.
\end{enumerate}
\end{defn}
When $\hat\kappa_n > 0$, the training set is linearly separable, and the training accuracy of $\hat f$ is $100\%$. While linear separability is common in high dimensions, the test accuracy is usually not perfect, which is based on the predicted label $\hat y(\vx_\mathrm{test}) = 2\mathbbm{1} \{\hat f(\vx_\mathrm{test}) > 0\} - 1$ for a test point $\vx_\mathrm{test}$. The discrepancy between train/test accuracies is known as overfitting.

For a classifier $\hat f$, the ELD and TLD are more informative compared with  train/test accuracies. For this reason, we analyze overfitting via a study of ELD and TLD.

\paragraph{Empirical phenomenon.}

\begin{figure}[t]
    \centering
    \includegraphics[width=0.32\textwidth]{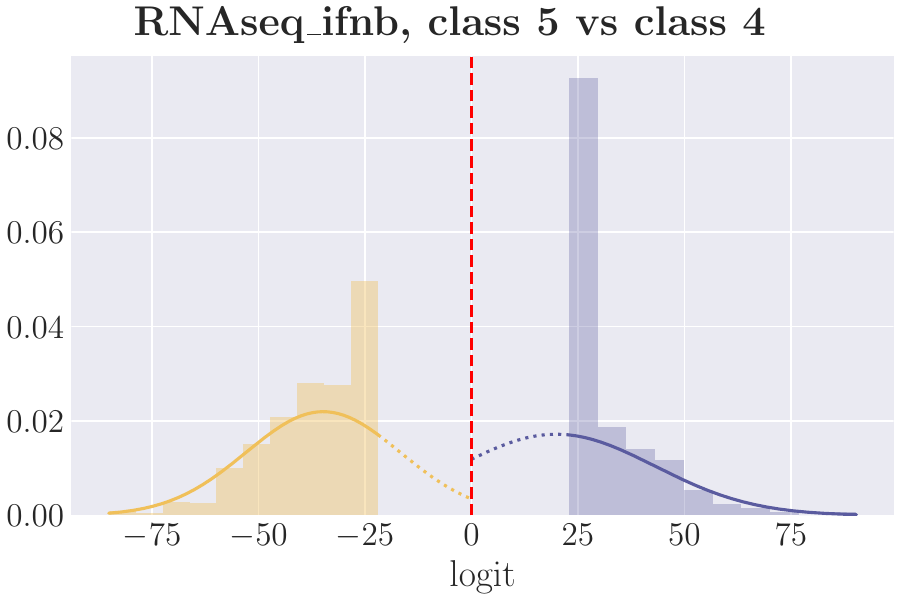}
    \includegraphics[width=0.32\textwidth]{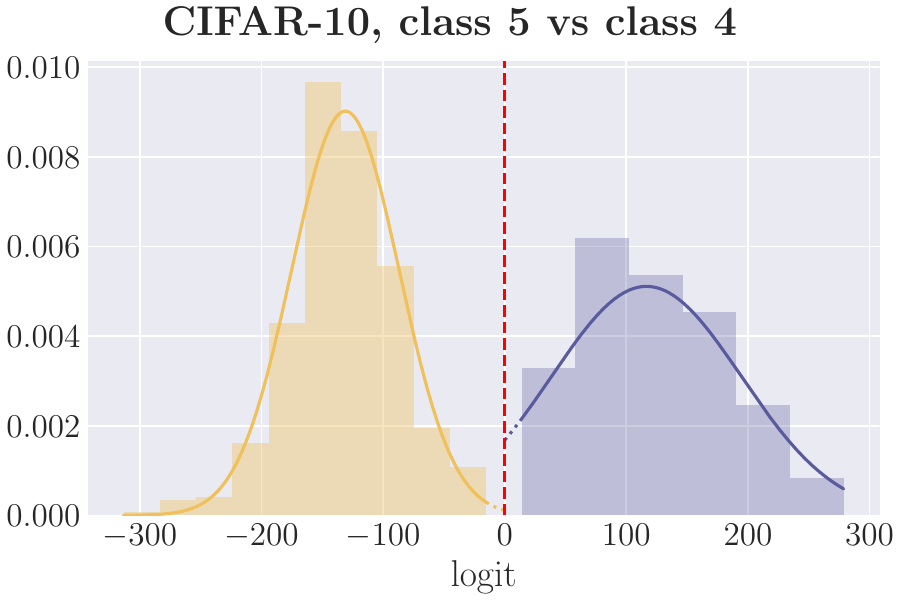}
    \includegraphics[width=0.32\textwidth]{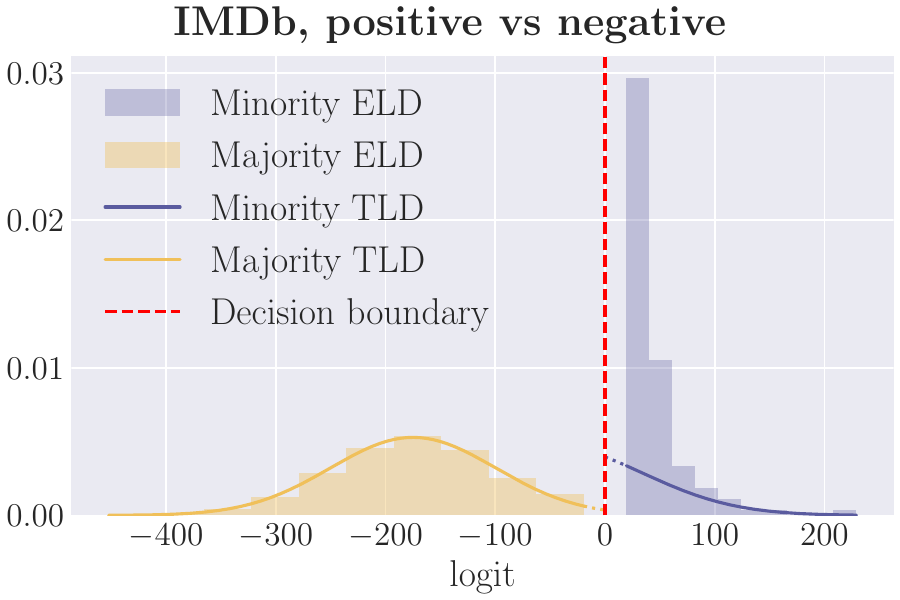}
    \caption{
        \textbf{ELD and TLD of logistic regression classifier (the last fully-connected layer) for
        real data}. \textbf{Left:} IFNB single-cell RNA-seq dataset (tabular data). \textbf{Middle:} CIFAR-10 dataset preprocessed by pretrained ResNet-18 model for feature extraction (image data). \textbf{Right:} IMDb movie review dataset preprocessed by BERT base model (110M) for feature extraction (text data).
        }
    \label{fig:GMM_real}
\end{figure}

First, 
we show a simple yet representative simulated example to illustrate the phenomenon; see \cref{fig:GMM_main}.
We generate training data according to 2-GMM in \cref{model}, with training sample size $n = 10,000$, feature dimension $d = 4,000$, signal strength $\norm{\bmu}_2 = 1.75$, and imbalance ratio $\pi = 0.15$. Such regime guarantees linear separability with high probability. We train an SVM classifier \cref{eq:SVM-0} on this dataset and visualize the ELD and TLD for the minority class ($y_i = +1$) and the majority class ($y_i = -1$) respectively. The TLD for each class follows a Gaussian distribution, while the ELD for each class is the same Gaussian curve but \textit{truncated} at the margin boundary---such discrepancy characterizes the effect of overfitting in imbalanced classification. As an adverse implication, we also observe that more than half of the density is truncated in minority ELD, thus causing severe degrading accuracy compared with the majority ELD. 
Formally, such distribution of ELD is called \textit{rectified Gaussian distribution} \cite{rectifiedGaussian}.

\begin{defn}[Rectified Gaussian distribution]
    For a Gaussian random variable $Z \sim \normal(\mu, \sigma^2)$ and a given threshold $\kappa \in \R$, the distribution of $\max\{ Z, \kappa \}$ or $\min\{ Z, \kappa \}$ is called the \emph{rectified Gaussian distribution}.
\end{defn}

For real-data examples, we consider fine-tuning experiments based on pretrained DNNs, as shown in Figure~\ref{fig:GMM_real}. Given a pretrained neural network, we freeze all parameters except for the last classification layer and fine-tune it on imbalanced labeled data. This fine-tuning approach essentially uses the pretrained neural network as a feature extractor, which is a standard practice in downstream analysis \cite{zhou2016learning, sharif2014cnn, radford2021learning, howard2018universal}.  We consider representative datasets for three data modalities. We train an SVM for each dataset.
\begin{enumerate}
    \item \textit{Tabular data}. We use a single-cell RNA-seq dataset of peripheral blood mononuclear cells treated with interferon-$\beta$ (IFNB) \cite{ifnb}, which has dimension $d = 2,000$. We randomly choose two classes, for example, class 5 (CD4 Naive T cells) and class 4 (CD4 Memory T cells), and subsample an imbalanced training set. Class 4 is the minority class, with imbalance ratio $\pi = 0.2$ and total sample size $n = 953$. 
    \item \textit{Image data}. We use CIFAR-10 image dataset \cite{KrizhevskyCIFAR102009}. The pretrained ResNet-18 model \cite{resnet, ResNet18_CIFAR10} is applied to the test set to extract the features of dimension $d = 512$. We randomly choose two classes, for example, class 5 (dog) and class 4 (deer), and subsample a imbalanced training set. Class 4 is the minority class, with imbalance ratio $\pi = 0.1$ and total sample size $n = 555$. 
    \item \textit{Text data}. We use IMDb movie review dataset \cite{IMDB} to perform binary sentiment classification. The BERT base model (110M) \cite{BERT} is applied to extract the features of dimension $d = 768$. We subsample an imbalanced training set. Negative reviews belong to the minority class, with imbalance ratio $\pi = 0.02$ and total sample size $n = 6,377$. 
\end{enumerate}

The empirical experiments reveal a pervasive structure in the ELD: for linearly separable data, the ELD can be fitted by two rectified Gaussian distributions, and such distributional truncation solely explains overfitting in high dimensions. Moreover, the minority class suffers more from this truncation effect as its test accuracy is worse. As our theoretical insights below reveal, this is because both classes share a common ``overfitting budget'' so that the minority margin boundary is disproportionally shifted.

\paragraph{Theoretical foundation.} Now we present a variational characterization of the ELD. We highlight a summary of our result for the separable case here and defers the non-separable case to Section~\ref{sec:logit}.
Consider the asymptotic regime $n/d \to \delta$ where $\delta \in (0,\infty)$ is called the limiting aspect ratio. Recall that $(\hat\vbeta, \hat\beta_0, \hat\kappa)$ are the trained parameters in \cref{eq:SVM-0}, where $\hat\kappa$ is the margin of classifier $\hat{y} (\xx) = 2 \mathbbm{1}\{ \hat f(\vx) > 0\} - 1$ with $\hat f(\vx) = \langle \vx, \hat \vbeta \rangle + \hat \beta_0$. Denote 
\begin{equation}\label{eq:rho_hat}
\hat \rho := 
\biggl\< \frac{\hat \vbeta}{\| \hat \vbeta \|}, \frac{\vmu}{\| \vmu \|} \biggr\>\,.
\end{equation}
One may expect $(\hat\rho, \hat\beta_0, \hat\kappa)$ to have some limits as $n,d \to \infty$. We define their asymptotics as follows.

\begin{defn}
    Let $(\rho^*, \beta_0^*, \kappa^*)$ be an optimal solution to the following variational problem:
			\begin{equation}
                \label{eq:asymp}
				\begin{array}{cl}
					\underset{ \rho \in [-1, 1], \beta_0 \in \R, \kappa > 0 , \xi \in \cL^2  }{ \mathrm{maximize} } & \kappa, \\
					\underset{ \phantom{\smash{\bm\beta \in \R^d, \beta_0 \in \R, \kappa \in \R} } }{\text{subject to}} &  \rho \| \bmu\| + G + Y \beta_0 + \sqrt{1 - \rho^2} \xi \ge \kappa,  
					\qquad \E[\xi^2]  \le  1/\delta,
					\end{array}
			\end{equation}
    where $\cL^2$ is the space of all square integrable random variables in the probability space $(\Omega, \mathcal{F}, \P)$, and $(Y, G) \sim P_y \times \normal(0, 1)$. Here $\xi$ is an unknown random variable (function) to be optimized.
\end{defn}
On a test point $(\vx_\mathrm{test}, y_\mathrm{test}) \sim P_{\vx,y}$, we consider the minority error and majority error
\begin{equation}\label{eq:Err_n}
    \Err_+ := \P\left( \hat f(\xx_\mathrm{test}) \le 0 \,\big|\, y_\mathrm{test} = +1 \right),
    \qquad
    \Err_- := \P\left( \hat f(\xx_\mathrm{test}) > 0 \,\big|\, y_\mathrm{test} = -1 \right).
\end{equation}
Note that the probabilities are taken over the training data, so these errors are nonrandom. The precise asymptotics of SVM and its ELD/TLD are summarized in the following theorem.

\begin{thm}[Separable data, informal version of \cref{thm:SVM_main}] \label{thm:SVM}
		Consider 2-GMM with asymptotics $n/d \to \delta \in (0,\infty)$ as $n, d \to \infty$. There is a critical threshold $\delta_c = \delta_c(\pi, \|\bmu\|)$, such that when $\delta < \delta_c$, the following holds as $n, d \to \infty$:
		\begin{enumerate}[label=(\alph*)]
			\item \textbf{Phase transition.} 
                \begin{equation*}
                    \P\left\{ \text{training set is linearly separable} \right\} \to 1.
                \end{equation*}
                
			\item \textbf{Parameter convergence.} 
                \begin{equation*}
                    (\hat \rho, \hat\beta_0, \hat\kappa) \xrightarrow{\mathrm{p}} (\rho^*, \beta_0^*, \kappa^*),
                \end{equation*}
                where $(\rho^*, \beta_0^*, \kappa^*)$ is the unique solution to \cref{eq:asymp}.

                \item 
                \textbf{Asymptotic errors.} The limits of minority an majority errors are
                \begin{equation*}
                    \Err_{+}  \to  \Phi \left(- \rho^* \norm{\bmu}_2 - \beta_0^* \right),
                    \qquad
                    \Err_{-}  \to  \Phi \left(- \rho^* \norm{\bmu}_2  + \beta_0^* \right),
                \end{equation*}
                where $\Phi$ denotes the cumulative distribution function of standard Gaussian.
			\item \label{thm:SVM_c}
            \textbf{ELD convergence.} The empirical (training) logit distribution $\hat{\nu}_n$ has limit $\nu_*$ in the sense that
                \begin{equation*}
                    W_2(\hat{\nu}_n , \nu_* ) \conp 0, \qquad \text{where} ~ 
                    \nu_* := \Law \,\bigl( Y,  Y \max\{\kappa^*, \rho^* \| \bmu \| + G + Y \beta_0^* \} \bigr).
                \end{equation*}
                \textbf{TLD convergence.} The testing logit distribution $\hat{\nu}_n^\mathrm{test}$ has limit $\nu^\mathrm{test}_*$ in the sense that
                \begin{equation*}
                    \hat{\nu}_n^\mathrm{test} \conw \nu^\mathrm{test}_*, \qquad \text{where} ~ 
                    \nu^\mathrm{test}_* := \Law \,\bigl( Y, Y (\rho^* \| \bmu \| + G + Y \beta_0^*) \bigr).
                \end{equation*}
		\end{enumerate}
\end{thm}

We make a few comments.
\begin{itemize}
\item \textit{Overfitting effect.} The random variable $\xi$ represents the distortion of ELD due to high dimensions. With a smaller aspect ratio $\delta$, there is more flexibility in finding $\xi$ that satisfies the constraint $\E[\xi^2] \le 1/\delta$ to maximize the margin, 
thereby distorts the TLD to produce the truncation effect.
In fact, in order to maximize $\kappa$, the first inequality constraint must be tight, which yields the explicit formula for any given $\rho, \beta_0, \kappa$:
\begin{align}
\sqrt{1 - \rho^2} \, \xi  & = \max \{ \kappa, \rho \| \vmu\| + G + Y \beta_0 \} - ( \rho \| \vmu\| + G + Y \beta_0 ) \notag \\
&= (\kappa - \rho \| \vmu\| - G - Y \beta_0)_+, \qquad \text{where}~a_+ := \max\{0,a\}. \label{def:opt-map}
\end{align}
Thus, we can view $\xi$ as a map that pushes the overlapping probability masses in the TLD to the margin boundaries in the ELD. It is the cause for the discrepancy between the ELD and the TLD.

\item \textit{More truncation for minority class.} Due to imbalance, transporting the probability mass in the minority ELD as \cref{def:opt-map} incurs less ``cost'' to the overall ``budget'' $\E[\xi^2] \le 1/\delta$.
Formally, according to \cref{thm:SVM}, the limiting TLD for each class is
\begin{equation*}
    \text{minority:}~ \normal \left( \rho^* \|\vmu\| + \beta_0^*, 1\right),
    \qquad
    \text{majority:}~ \normal \left( -\rho^* \|\vmu\| + \beta_0^*, 1\right).
\end{equation*}
It can be shown that $\rho^* \ge 0$ and $\beta_0^* < 0$. So the minority TLD is closer to the decision boundary.
As the result, the minority class suffers more from the truncation effect (overfitting) than the majority class.

\item \textit{Optimal transport perspective.} As a simple consequence, we show in \cref{sec:logit_SVM} that $\mathtt{T}^*(x) = \max\{\kappa^*, x \}$ gives the optimal transport map that maps the limit distribution $\nu_*$ to $\nu_*^\mathrm{test}$ and minimizes the $W_2$ distance between them.

\item \textit{Non-separable case.} When $\delta > \delta_c$, the training dataset is not separable with high probability, and SVM can no longer be viewed as the limit of the logistic regression. We analyze the logistic regression classifier and characterize its corresponding ELD. Instead of truncation, overfitting emerges as nonlinear shrinkage governed by the proximal operator (gradient of the Moreau envelope). As $\delta \in (\delta_c, \infty)$ decreases, the nonlinear shrinkage transitions from an identity map (no overfitting) to the truncation map $\mathtt{T}^*(x) = \max\{\kappa^*, x \}$ (severe overfitting). See \cref{thm:logistic_main} for its explicit expression and \cref{append_subsec_prox} for its function plot.
\end{itemize}

Here, we provide some intuitions about the ELD and TLD of the max-margin classifier.
On a test point $(\vx_\mathrm{test}, y_\mathrm{test}) \sim P_{\vx,y}$, it is easy to see that $y_\mathrm{test}\hat f(\vx_\mathrm{test})$ is approximately distributed as a mixture of two Gaussians:
\begin{align*}
    y_\mathrm{test} (\langle \vx_\mathrm{test} , \hat \vbeta \rangle + \hat \beta_0)  
    &= y_\mathrm{test} \left\langle  y_\mathrm{test}\vmu + \normal(\bzero, \bI_d),  \hat\vbeta \right\rangle + y_\mathrm{test} \hat \beta_0 \\
    &= \hat\rho \,  \| \vmu\| +  \left\langle  \normal(\bzero, \bI_d),  \hat\vbeta \right\rangle  + y_\mathrm{test} \hat \beta_0 \\
    &\approx \rho^* \| \vmu\| + G + Y \beta_0^*, \qquad \qquad  \text{where}~(Y, G) \sim P_y \times \normal(0,1).
\end{align*}
Essentially, $\hat f (\vx)$ is a linear projection of the 2-GMM input points to one dimension with a smaller separation between the two classes. However, on a training point $(\vx_i, y_i)$, there is a distortion effect on the distribution due to the margin constraints---the following holds for a ``typical'' training point: 
\begin{equation*}
    y_i (\langle \vx_i, \hat \vbeta \rangle + \hat \beta_0) \approx \max\{ \kappa^*, \rho^* \| \vmu\| + G + Y \beta_0^* \} , \qquad   \text{where}~(Y, G) \sim P_y \times \normal(0,1).
\end{equation*}
This is because the max-margin classifier achieves overfitting in high dimensions, where a substantial portion of points (a.k.a. support vectors) fall on the margins $\hat\kappa \approx \kappa^*$.

\subsection{Rebalancing margin is crucial}
\label{subsec:rebal}

Rebalancing the margin is a common practice for remedying severe overfitting for the minority class. In binary classification, we choose a hyperparameter $\tau > 0$ and consider the margin-rebalanced SVM \cref{eq:SVM-m-reb} which shifts the decision boundary as shown in Figure~\ref{fig:SVM_cartoon}. For the logistic loss in \cref{eq:logistic}, we can similarly incorporate $\tau$ into the objective function. Margin rebalancing is widely used in machine learning \cite{li2002perceptron, li2005using, cao2019learning, karakoulas1998optimizing, wu2003class}, but the impact of $\tau$ on test accuracy is not fully explored. 

\begin{figure}[h]
\noindent
\begin{minipage}[h]{0.6\textwidth}
    \includegraphics[width=1\textwidth]{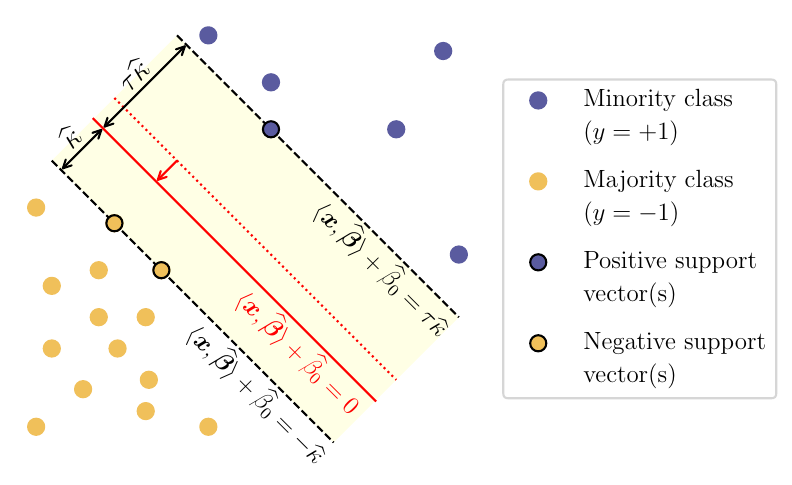}
\end{minipage}
\hfill
\begin{minipage}[h]{0.4\textwidth}
    \begin{equation}\label{eq:SVM-m-reb}
    \begin{array}{l}
	\underset{ \vbeta \in \R^d, \beta_0 \in \R, \kappa \in \R }{ \mathrm{maximize} }  \quad \kappa, \\
	\underset{ \phantom{\smash{\vbeta \in \R^d, \beta_0 \in \R, \kappa \in \R} } }{\text{subject to}} 
        \vphantom{\dfrac12} \\
			\ \ y_i ( \langle \xx_i, \vbeta \rangle + \beta_0 ) \ge \tau \kappa,
			\quad \forall\, i \in \mathcal{I}_+, \\ 
			\ \ y_i ( \langle \xx_i, \vbeta \rangle + \beta_0 ) \ge \phantom{\tau} \kappa,
                \vphantom{\dfrac12}
			\quad \forall\, i \in \mathcal{I}_-, \\
			\ \ \|\vbeta\|_2  \le  1.  
			\vphantom{\dfrac12}
    \end{array}
    \end{equation}
\end{minipage}
    \caption{
    \textbf{Schematic illustration of margin-rebalanced SVM}. The dotted line is the decision boundary for the original SVM, and the solid line is the decision boundary for margin-rebalanced SVM.
    }
    \label{fig:SVM_cartoon}
\end{figure}

We will conduct analysis under two regimes: \textbf{(i) proportional regime} where $n,d \to \infty$ and $n/d \to \delta$ with $\delta \in (0,\infty)$, and \textbf{(ii) high imbalance regime} in the following sense:
\begin{equation}\label{setup-high-imbalance}
\pi \propto d^{-a}, \qquad \norm{\vmu}^2 \propto d^b, \qquad n \propto d^{c+1}.
\end{equation}

For problems with data imbalance, often correctly classifying the minority data points is equally important as correctly classifying the majority data points. 
For this purpose, we introduce \textit{the balanced error}: 
\begin{equation}\label{eq:balanced_err}
\Err_\mathrm{b} = \frac{1}{2} \Err_+ + \frac{1}{2} \Err_-\, .
\end{equation}

\paragraph{Empirical phenomenon.} For the  proportional regime, we generate imbalanced 2-GMM based on \cref{setup-high-imbalance} with sample size $n = 100$ and dimension $d = 200$, under different settings of $\norm{\bmu}_2$ and $\pi \in (0, \frac12]$. We train an SVM \cref{eq:SVM-m-reb} with margin rebalancing (set $\tau$ to certain optimal value) and without (set $\tau = 1$) respectively, for each configuration. We calculate the minority error $\Err_+$, majority error $\Err_-$, and balanced error $\Err_\mathrm{b}$ on an independent test set, averaged over 100 replications, and plot these errors against different values of $\pi$ in Figure~\ref{fig:Err_pi}. The smooth curves represent the asymptotic test errors, i.e., the limits of $\Err_+$, $\Err_-$, $\Err_\mathrm{b}$ as $n, d \to \infty$ according to \cref{thm:SVM}. Experimental  details are deferred to \cref{append_sec:exp}. 




For the naive SVM where $\tau = 1$, as $\pi$ decreases to $0$, the minority error $\Err_+$ increases to $1$, the majority error $\Err_-$ decreases to $0$, and the balanced error tends to the trivial $\frac12$. The opposite trends of minority and majority errors show that overfitting hurts minority class more than majority class. In contrast, under optimal $\tau$, we are able to even out the minority and majority errors at the same level. As a result, margin rebalancing is advantageous for reducing the balanced error.

\begin{figure}[t]
    \centering
    \includegraphics[width=0.32\textwidth]{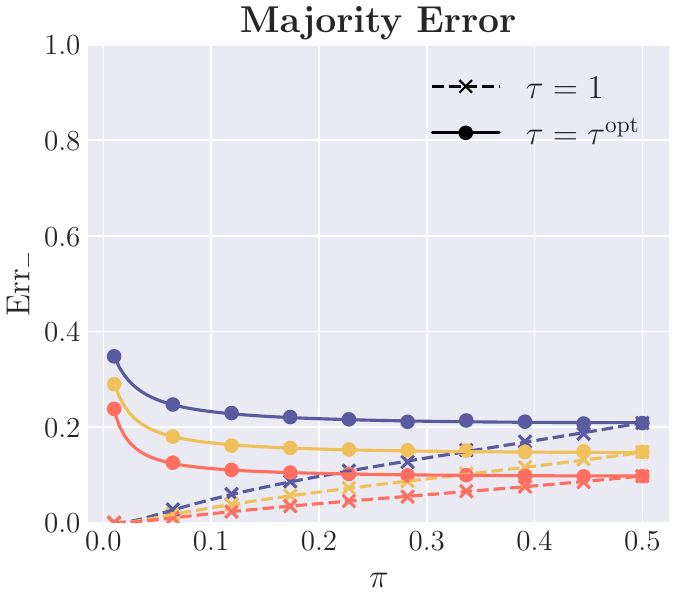}
    \includegraphics[width=0.32\textwidth]{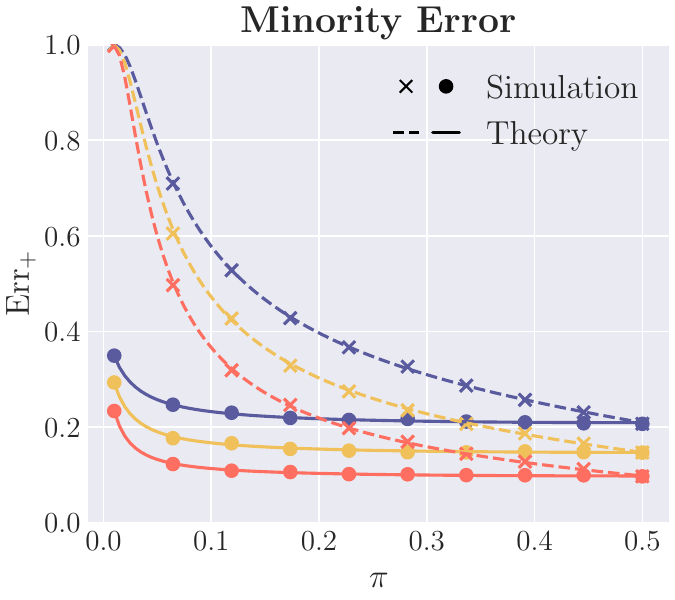}
    \includegraphics[width=0.32\textwidth]{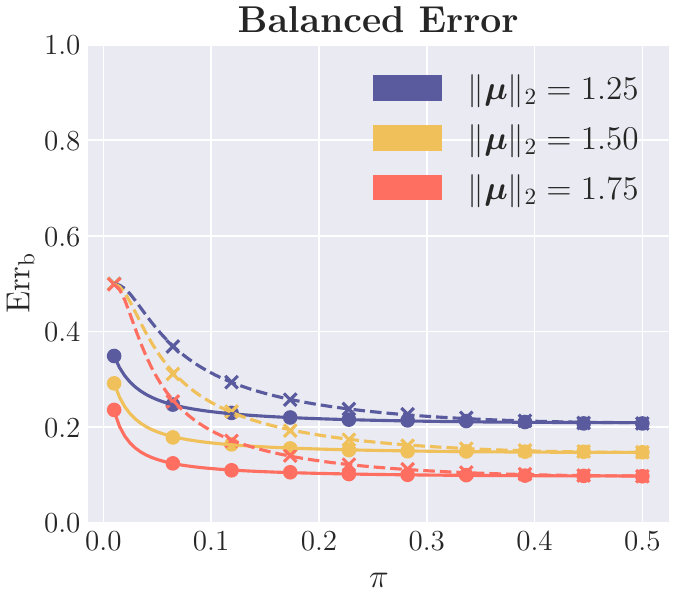}
    \caption{
    \textbf{Impact of imbalance on test errors}. We show test errors from 2-GMM simulations with margin rebalancing (solid curves) and without (dashed curves) at three levels of signal strength $\| \vmu\|_2$ under varying imbalanced ratios $\pi$.
    }
    \label{fig:Err_pi}
\end{figure}

\begin{figure}[t]
    \centering
    \includegraphics[width=0.32\textwidth]{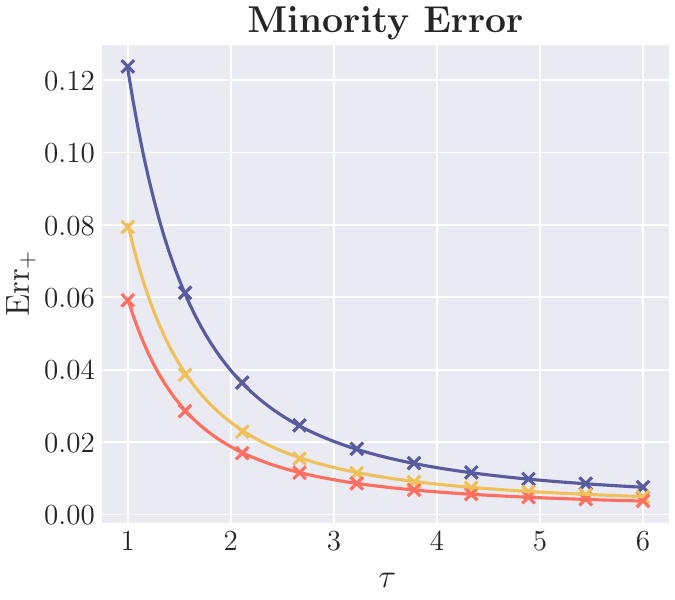}
    \includegraphics[width=0.32\textwidth]{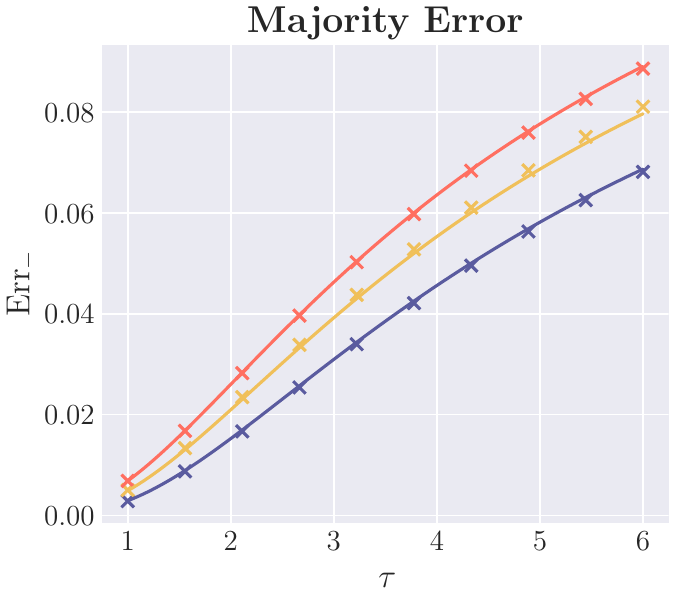}
    \includegraphics[width=0.32\textwidth]{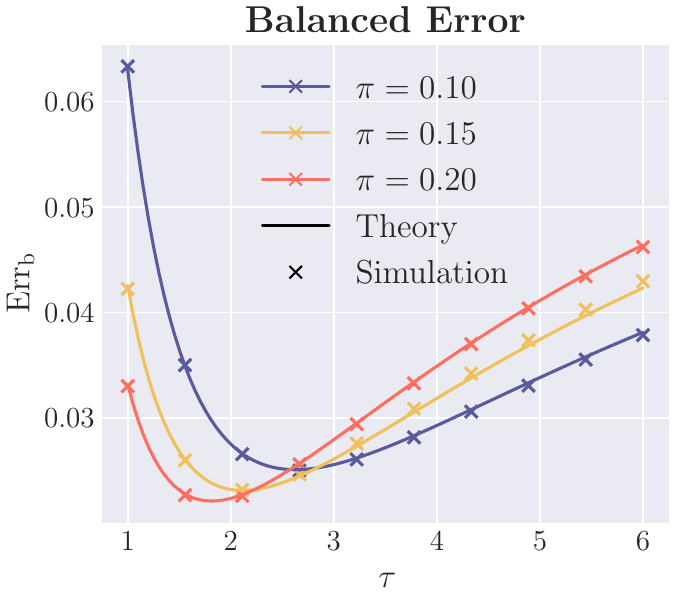}
    \caption{
    \textbf{Effects of margin rebalancing on test errors}. We show test errors from 2-GMM simulations at three different imbalance ratios under varying $\tau$.
    }
    \label{fig:Err_tau}
\end{figure}

We also plot test errors against different values of $\tau$ in \cref{fig:Err_tau} under the same simulation setting. The minority and majority errors have monotone but opposite trends in $\tau$, since increasing $\tau$ essentially moves the decision boundary from the side of minority class to the majority class. Such trade-off between the two classes results in a U-shaped curve for the balanced error. This indicates that we can find a unique optimal $\tau=\tau^{\mathrm{opt}}$ which minimizes $\Err_\mathrm{b}$, and $\tau^{\mathrm{opt}}$ is larger as $\pi$ becomes smaller.

For the high imbalance regime, we generate imbalanced 2-GMM based on \cref{setup-high-imbalance} with a sufficiently large dimension $d=2000$. We choose $\tau = \tau_d = d^r$ for different values of $r \ge 0$. We fix $b = 0.3$, $c = 0.1$ and vary $a, r$, and then we train a margin-rebalanced SVM \cref{eq:SVM-m-reb} for each configuration. Figure~\ref{fig:High_imb_heat} shows that there are three phases in terms of the majority/minority errors. In particular, the margin rebalancing is crucial for one phase with moderate signal strength.

\begin{figure}[t]
    \centering
    \includegraphics[height=0.47\textwidth]{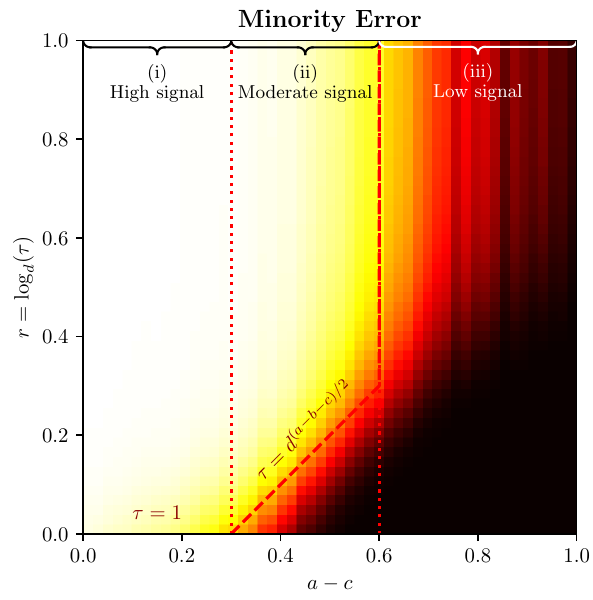}
    \includegraphics[height=0.47\textwidth]{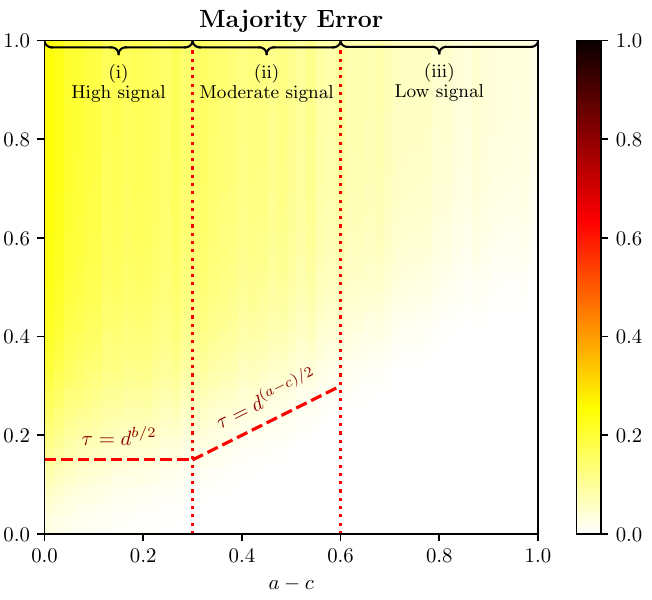}
    \caption{
        \textbf{Phase transition in high imbalance regime}. Minority/majority errors  under different settings of parameters $(a, b, c)$ and $\tau = d^r$. \textbf{Left:} minority accuracy is (i) high for any $\tau$ under high signal, (ii) high for $\tau \gg d^{(a - b - c)/2}$ under moderate signal, but (iii) low for any $\tau$ under low signal. 
        \textbf{Right:} majority accuracy is close to 1 under high and moderate signal as long as $\tau$ is not too large.   
    }
    \label{fig:High_imb_heat}
\end{figure}

\paragraph{Theoretical foundation.} For the proportional regime, denote $\Err_+^*$, $\Err_-^*$, $\Err_\mathrm{b}^*$ as the limits of $\Err_+$, $\Err_-$, $\Err_\mathrm{b}$ as $n \to \infty$, respectively, then we have the following result.
\begin{prop}[Optimal $\tau$ in proportional regime, informal version of \cref{prop:tau_optimal}]\label{prop:tau_opt}
    Consider 2-GMM with asymptotics $n/d \to \delta \in (0,\infty)$ as $n, d \to \infty$.
    Define $\tau^\mathrm{opt}$
    as the optimal margin ratio which minimizes the asymptotic balanced error
\begin{equation*}
    \tau^\mathrm{opt} :=  \argmin_{\tau} \Err_\mathrm{b}^* = \argmin_{\tau} \big\{ \Phi(- \rho^*\norm{\bmu} - \beta_0^* )
    + \Phi(- \rho^*\norm{\bmu} + \beta_0^* ) \big\}.
\end{equation*}
When $\tau = \tau^\mathrm{opt}>0$, we have $\beta_0^* = 0$, $\Err_+^* = \Err_-^* = \Err_\mathrm{b}^*$. 
\end{prop}
A critical observation is that, changing $\tau$ only has an effect on $\hat \beta_0$ but not $\hat \vbeta$. Thus, changing $\tau$ is effectively shifting the decision boundaries between the two classes. Our analysis reveals that the optimal $\tau$ has a complicated dependence on $\pi, \norm{\vmu}, \delta$. When the problem is not close to degenerate, roughly speaking $\tau^\mathrm{opt} \asymp \sqrt{1/\pi}$; see Section~\ref{sec:rebalacing} for details. 


Under $\tau = \tau^\mathrm{opt}$, our theoretical result (see \cref{prop:Err_monotone} for a formal statement) shows monotone trends of the errors: the limiting minority/majority/balanced error is a decreasing function of $\pi \in (0, \frac12)$ (imbalance ratio), $\norm{\bmu}$ (signal strength), and $\delta$ (aspect ratio); see summary in \cref{tab:monotone}.

For the high imbalance regime, margin rebalancing is necessary to achieve a small balanced error when the ``signal strength'' is moderate, which matches our empirical observations in \cref{fig:High_imb_heat}.

\begin{thm}[High imbalance] \label{thm:high-imbalance}
Consider 2-GMM with asymptotics \cref{setup-high-imbalance} as $d \to \infty$. Suppose that $a - c < 1$.
\begin{enumerate}
\item \label{thm:high-imb_high}
\textbf{High signal} (no need for margin rebalancing): $a - c< b$. If we choose $1 \le \tau_d \ll d^{b/2}$, then 
\begin{equation*}
\Err_+ = o(1), \qquad \Err_- = o(1).
\end{equation*}
\item \label{thm:high-imb_moderate}
\textbf{Moderate signal} (margin rebalancing is crucial): $b < a - c < 2b$. If we choose $d^{a-b-c} \ll \tau_d \ll d^{(a-c)/2}$, then
\begin{equation*}
\Err_+ = o(1), \qquad \Err_- = o(1).
\end{equation*}
However, if we naively choose $\tau_d \asymp 1$, then 
\begin{equation*}
\Err_+ = 1 - o(1), \qquad \Err_- = o(1).
\end{equation*}
\item \label{thm:high-imb_low}
\textbf{Low signal} (no better than random guess): $a - c > 2b$. For any $\tau_d$, we have
\begin{equation*}
\Err_\mathrm{b} \ge \frac12 -  o(1).
\end{equation*}
\end{enumerate}
\end{thm}

\subsection{Consequences for confidence estimation and calibration}
\label{subsec:conf_calib}

In the deep learning literature, the \emph{confidence} of a classifier often refers as the probability of the correctness of a prediction, i.e., the probability that the predicted label matches the true label. Formally, we define the confidence of the max-margin classifier as
    \begin{equation*}
        \hat p(\xx) := \sigma\bigl( \hat f(\xx) \bigr) = 
        \sigma\bigl( \< \xx, \hat\vbeta \> + \hat\beta_0 \bigr), \qquad \text{where}~\sigma(t) = \frac{1}{1 + e^{-t}}.
    \end{equation*}
We would like the confidence $\hat p$ to be a good approximation of the true conditional probability (namely Bayes-optimal probability), that is $\hat p(\xx) \approx p^*(\xx) = \P( y = 1 \,|\, \xx )$.
However, finding the Bayes-optimal probability is usually intractable in deep learning models, and estimating this conditional probability $p^*(\vx)$ requires nonparametric estimation in high dimensions, which is difficult. 
The notion of \emph{calibration} is therefore widely used in literature, which measures the faithfulness of prediction probabilities \cite{murphy1967verification, dawid1982well, gupta2020distribution, guo2017calibration}.
Formally, the confidence $\hat p$ is (approximately) \emph{calibrated} if
\begin{equation}\label{eq:calibrated}
    \hat p(\xx) \approx \hat p_0(\xx) := \P \bigl( y = 1 \,|\,  \hat p(\xx) \bigr).
\end{equation}
This notion requires that the predicted probability by $\hat p$ for any $\vx$ matches the actual probability.
We list several popular miscalibration metrics below \cite{kumar2019verified, kuleshov2015calibrated, vaicenavicius2019evaluating}:
\begin{itemize}
    \item \textbf{Calibration error.}
    \begin{equation}\label{eq:CalErr}
    \begin{aligned}
        \mathrm{CalErr}(\hat p) & :=    \E\left[ \Bigl(  
        \hat p(\xx) - \P \bigl( y = 1 \,|\,  \hat p(\xx) \bigr)
        \Bigr)^2 \right].
    \end{aligned}
    \end{equation}

    \item \textbf{Mean squared error (MSE).}
    \begin{equation}\label{eq:MSE}
        \mathrm{MSE}(\hat p)  := \E\left[ \bigl( \mathbbm{1}\{ y = 1 \} - \hat p(\xx) \bigr)^2 \right]\,.
    \end{equation}
    \item \textbf{Confidence estimation error.}

    \begin{equation}\label{eq:ConfErr}
        \mathrm{ConfErr}(\hat p) :=  \E\left[ \bigl( \hat p(\xx) - p^*(\xx) \bigr)^2 \right]\,.
    \end{equation}
\end{itemize}

\paragraph{Empirical phenomenon.}

We consider the same 2-GMM simulation experiment as in Figure~\ref{fig:Err_pi} (the proportional regime).  After margin rebalancing, we calculate the above three miscalibration metrics Eqs.~\eqref{eq:CalErr}---\eqref{eq:ConfErr} on an independent test set, average over 100 replications, and plot these errors against different values of $\pi$ in \cref{fig:Calibration}. The smooth curves represent the asymptotic errors, i.e., the limits of $\mathrm{CalErr}(\hat p)$, $\mathrm{MSE}(\hat p)$, $\mathrm{ConfErr}(\hat p)$ as $n, d \to \infty$ according to \cref{thm:SVM}. Experimental details are deferred to \cref{append_sec:exp}. Notably, all these errors increase as imbalance becomes more severe (namely $\pi$ being smaller).

\begin{figure}[t]
    \centering
    \includegraphics[width=0.32\textwidth]{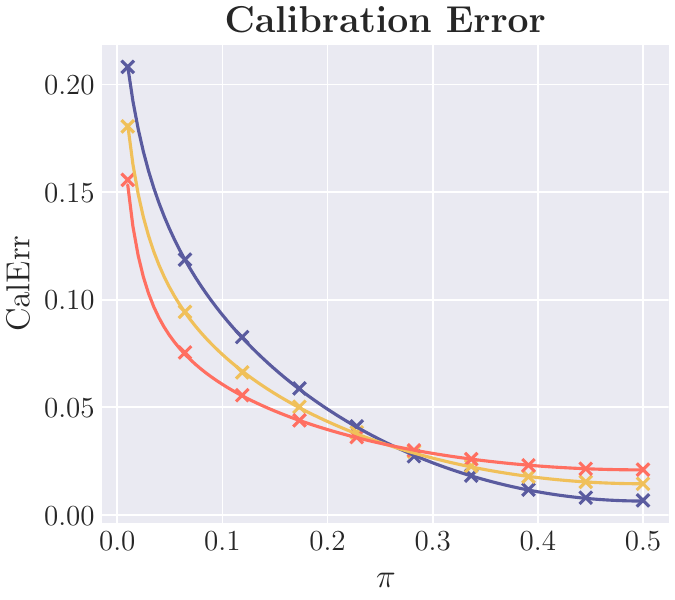}
    \includegraphics[width=0.32\textwidth]{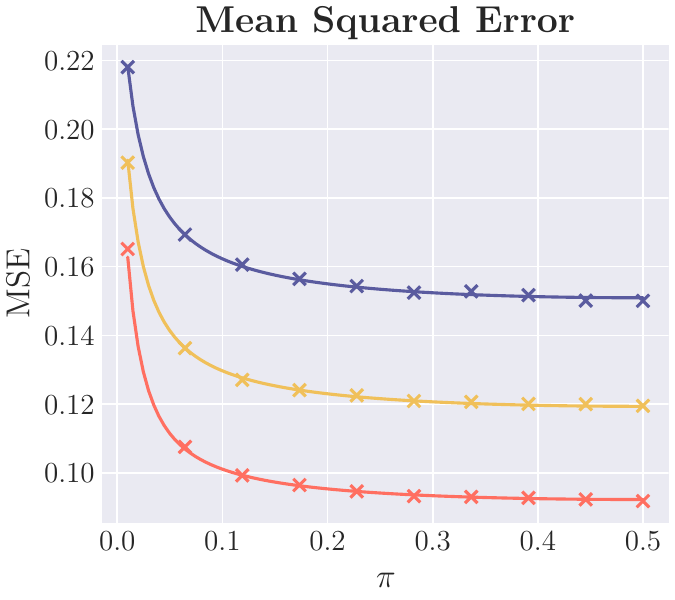}
    \includegraphics[width=0.32\textwidth]{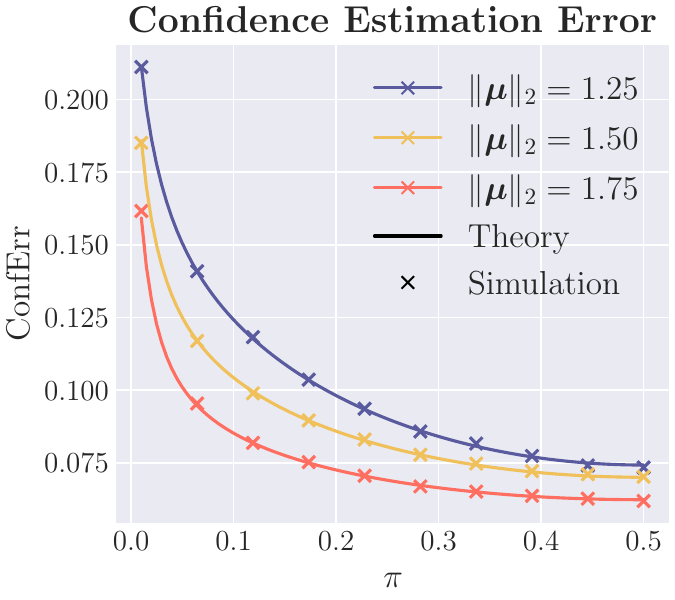}
    \caption{
    \textbf{Impact of imbalance on uncertainty quantification}. We plot miscalibration metrics \cref{eq:CalErr}--\eqref{eq:ConfErr} for 2-GMM simulations with optimal margin rebalancing $(\tau = \tau^\mathrm{opt}).$ We find that high imbalance (namely small $\pi$) exacerbates miscalibration.
    }
    \label{fig:Calibration}
\end{figure}

We also plot confidence reliability diagrams for the 2-GMM simulations. Widely used for diagnosis of classifiers \cite{guo2017calibration}, the reliability diagrams plot the value of $\mathbb{P}(y=1 \,|\, \hat p(\vx) = p)$ as a function of $p$. In \cref{fig:reliability_GMM}, we calculate $\mathbb{P}(y=1 \,|\, \hat p(\vx) = p)$ for varying $p \in \{0.05, 0.15, \ldots, 0.95\}$ using an empirical estimate based on an independent test set. The histograms show these empirical conditional probabilities after binning. The dashed diagonal line represents perfect calibration (i.e., when \cref{eq:calibrated} strictly equals), and deviation from this line means miscalibration of the classifier. 
    In our simulations, we fix $\norm{\bmu}=1$, $n=1000$, $d=500$ and choose a range for different $\pi$, under $\tau = \tau^\mathrm{opt}$. Again, we observe miscalibration getting worse when data becomes increasingly imbalanced (i.e., as $\pi$ decreases).

\begin{figure}[h]
    \centering
    \includegraphics[width=1\textwidth]{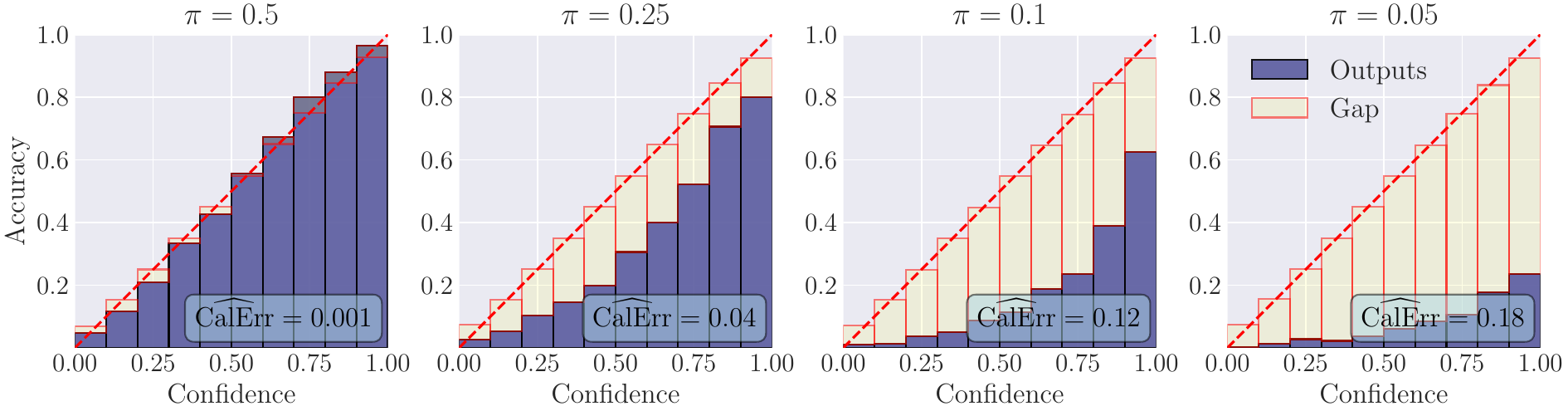}
    \caption{
    \textbf{Reliability diagrams: imbalance worsens calibration}. In our 2-GMM simulations, we train SVMs and obtain confidence $\hat p(\vx)$. For each $p$ ($x$-axis), we calculate $\mathbb{P}(y=1 \,|\, \hat p(\vx) = p)$ ($y$-axis) based on an independent test set. We find that as imbalance increases (smaller $\pi$), the classifier becomes more miscalibrated as the predicted probabilities are more inflated.
    }
    \label{fig:reliability_GMM}
\end{figure}

\paragraph{Theoretical foundation.} 

We provide theoretical results to partially explain the monotone trends. All three miscalibration metrics have limits as $n, d \to \infty$, $n/d \to \delta$, which we denote by $\mathrm{CalErr}^*$, $\mathrm{MSE}^*$, $\mathrm{ConfErr}^*$. For example, we prove that $\mathrm{MSE}^*$ is a monotone decreasing function of model parameters $\pi$, $\norm{\bmu}_2$, $\delta$, and $\mathrm{CalErr}^*$ is monotone decreasing in $\delta$. \cref{tab:monotone} summarizes the monotone behavior of test errors and miscalibration metrics that we 
establish
in this paper.

\begin{table}[h!]
\begin{equation*}
\renewcommand{\arraystretch}{1.2}
    \begin{array}{r|c|ccc}
    \hline
           & \mathrm{Err}_{+}^*, \mathrm{Err}_{-}^*, \mathrm{Err}_{\mathrm{b}}^* & \mathrm{CalErr}^*  & \mathrm{MSE}^*  &  \mathrm{ConfErr}^* \\
    \hline
      \text{imbalance ratio}~\pi \uparrow
          &  \downarrow ~ \text{(Prop.~\ref{prop:Err_monotone})}    
          &  
          &  \downarrow ~ \text{(Prop.~\ref{prop:conf})}    
          &  \downarrow ~ \text{(Claim~\ref{claim:conf})}   \\
      \text{signal strength}~\norm{\bmu}_2 \uparrow
          &  \downarrow ~ \text{(Prop.~\ref{prop:Err_monotone})}    
          &  \downarrow ~ \text{(Claim~\ref{claim:conf})}
          &  \downarrow  ~ \text{(Prop.~\ref{prop:conf})}   
          &       \\
      \text{aspect ratio}~n/d \to \delta \uparrow
          &  \downarrow ~ \text{(Prop.~\ref{prop:Err_monotone})}    
          &  \downarrow ~ \text{(Claim~\ref{claim:conf})}
          &  \downarrow ~ \text{(Prop.~\ref{prop:conf})}    
          &  \downarrow  ~ \text{(Prop.~\ref{prop:conf})}  \\
    \hline
    \end{array}
\end{equation*}
\vspace{-5mm}
\caption{Monotonicity of test errors and miscalibration metrics on model parameters. }
\label{tab:monotone}
\end{table}

\subsection{Extensions}

Our theory focuses on two-class classification problems with an isotropic covariance matrix. For the cases of multiple classes and non-isotropic covariance matrices, we believe that similar characterization of ELDs exist. In Section~\ref{sec:discuss}, we provide theoretical conjectures and empirical support for the general cases.

%% file: src/related.tex
\section{Related work}

\paragraph{Imbalanced classification.} In the classical literature on logistic regression with class imbalance, it is known that large-sample asymptotics is inaccurate with small sample sizes and thus bias correction formulas are derived \cite{schaefer1983bias, anderson1979logistic, mccullagh1983generalized}. Under label shift (also known as choice-based sampling), intercept correction and upweighting are developed to address imbalanced classes \cite{xie1989logit, king2001logistic}. For kernel methods and tree-based methods, additional techniques such as oversampling, undersampling, and synthetic data generation are used to mitigate limited minority samples \cite{he2009learning, chawla2002smote, he2008adasyn}. However, these methods are typically ineffective for separable data \cite{cao2019learning}, nor do they aim to address overfitting in high dimensions.


\paragraph{Margin-based methods.} Margin plays an important role in classification methods such as SVM. For imbalanced classification, promoting unequal margins is proposed for the perceptron algorithm \cite{li2002perceptron} and SVM \cite{li2005using}, and more recently for training deep neural networks \cite{huang2016learning, khan2019striking, liu2019large, cao2019learning}. For theoretical analysis, 
many earlier works \cite{bartlett1996valid, bartlett1998boosting, koltchinskii2002empirical, bartlett2002rademacher, bartlett2017spectrally} obtained margin-dependent generalization bounds for a variety of classifiers and algorithms,
showing a large margin is beneficial to good generalization.
In particular, the theoretical result \cite{koltchinskii2002empirical}
is used by \cite{cao2019learning} to motivate the margin-rebalancing loss function.
However, the margin-dependent bounds are agnostic to data distributions and may be excessively conservative. 

\paragraph{High-dimensional asymptotics.} Classical asymptotic analysis for fixed dimensions or low dimensions are known to be inaccurate \cite{el2013robust, donoho2016high}. A line of recent work studies estimation and inference for high-dimensional classification problems, e.g., \cite{dobriban2018high, Pragya_highdim_logistic, sur2019logistic, candes2020logistic, montanari2023generalizationerrormaxmarginlinear, kini2021label, deng2022model, mignacco2020role, salehi2019impact, montanari2024tractability}, which refine Table~\ref{tab:1} in various ways. One core technique is Gordon's theorem, which reduces a random-matrix-based minimax problem to a simpler form. Our proofs are mainly based on this technique. Most relevant to this paper is \cite{montanari2022overparametrizedlineardimensionalityreductions}, where the authors studied projection pursuit of high-dimensions data and characterized attainable asymptotic low-dimensional distributions. However, these papers do not offer a clear characterization of the impact of class imbalance on overfitting, nor do they study the high imbalance regime or calibration.

%% file: src/setup.tex
\section{Preliminaries}
\label{sec:preliminary}

\subsection{SVM and linear separability} \label{sec:background}

Consider our 2-GMM in \cref{model}. Denote $\XX = (\vx_1,\ldots,\vx_n)^\top \in \R^{n\times d}$ and $\yy = (y_1,\ldots,y_n)^\top \in \R^n$. Recall the general margin-rebalanced SVM in \cref{eq:SVM-m-reb}. For $\tau > 0$, it is convenient to write this SVM formulation into
\begin{equation}
	\label{eq:SVM}
    \begin{array}{rl}
    \maximize\limits_{\bbeta \in \R^d, \, \beta_0 \in \R} & \min\limits_{i \in [n]} \tilde{y}_i ( \langle \xx_i, \bbeta \rangle + \beta_0 ), \\
    \text{subject to} & \norm{\bbeta}_2 \le 1
    \end{array}
\end{equation}
by introducing the transformed labels
\begin{equation}\label{eq:trans-labels}
    \tilde{y}_i = \begin{cases} \tau^{-1} , & \ \text{if} \ y_i = + 1, \\
    -1, & \ \text{if} \ y_i = -1. \end{cases}
\end{equation}

According to the following deterministic result, the solution to margin-rebalanced SVM \cref{eq:SVM-m-reb} is a simple post-hoc adjustment of the solution to the original SVM \cref{eq:SVM-0}.

\begin{prop}\label{prop:SVM_tau_relation}
	\begin{enumerate}[label=(\alph*)]
		\item 
            When data is linearly separable, \cref{eq:SVM-m-reb} has a unique solution.
            \item 
            Let $(\hat\vbeta(\tau), \hat\beta_0(\tau), \hat\kappa(\tau))$ be an optimal solution to \cref{eq:SVM-m-reb} under hyperparameter $\tau$. Then
        \begin{equation}\label{eq:margin-balance}
		    \hat\vbeta(\tau) = \hat\vbeta(1),
        \qquad
        \hat\beta_{0}(\tau) = \hat\beta_{0}(1) + \frac{\tau - 1}{\tau + 1} \hat\kappa(1),
        \qquad
        \hat\kappa(\tau) = \frac{2}{\tau + 1} \hat\kappa(1).
		\end{equation}
	\end{enumerate}
\end{prop}

\begin{rem}
	As shown in \cref{fig:SVM_cartoon}, there is a clear geometric interpretation of $\hat\vbeta, \hat\beta_0, \hat\kappa$ and $\tau$ in the max-margin classifier.
	\begin{itemize}
		\item $\hat\vbeta(\tau)$ determines the support vectors and the ``direction'' of decision boundary, which does not depend on $\tau$. Notably, margin rebalancing does not change $\hat\vbeta$.
		\item $\hat\beta_0(\tau)$ balances the positive/negative margins via \cref{eq:margin-balance}, where $\tau$ determines the amount of the shift.
		\item $\hat\kappa(\tau) \propto (\tau + 1)^{-1}$ in a fixed dataset.
	\end{itemize}
\end{rem}

\subsection{Connections between logistic regression and SVM}
\label{subsec:LR_vs_SVM}

The two classifiers in Eqs.~\eqref{eq:logistic} and \eqref{eq:SVM-0} are strongly connected in high dimensions: the SVM can be viewed as the limit of logistic regression when the data are linearly separable. 


We first introduce the background of inductive bias from \cite{rosset2003margin, rosset2004boosting} and \cite{Soudry_implicit_bias, ji2019riskparameterconvergencelogistic}. 
In logistic regression, we minimize the empirical loss \cref{eq:logistic}
where $\ell(t) = \log (1+e^{-t})$ is the logistic loss. Since the loss is strictly convex, if it admits a finite minimizer, then the minimizer must be unique. However, when the data are linearly separable, there is no finite minimizer and the objective value goes to $0$ for certain $\vbeta$ with $\|\vbeta\|_2 \to \infty$. To obtain a unique solution, we may add a regularizer:
\begin{equation}\label{eq:logistic-l2}
        \bigl( \hat\vbeta_\lambda, \hat\beta_{0,\lambda} \bigr) := \argmin_{\vbeta \in \R^d, \beta_0 \in \R}
        \left\{
        \frac{1}{n} \sum_{i=1}^n \ell \bigl( y_i(\langle \vx_i, \vbeta \rangle + \beta_0) \bigr) +  \lambda \| \vbeta\|_2^2 \right\}.
    \end{equation}
Let $(\hat\vbeta, \hat\beta_{0})$
be the max-margin solution to SVM \cref{eq:SVM-0}. Then it has been shown by \cite{rosset2003margin, rosset2004boosting} that without the presence of intercept
\begin{equation}\label{eq:explicit_bias}
    \lim_{\lambda \to 0^+} \frac{\hat \vbeta_\lambda}{\| \hat \vbeta_\lambda \|_2} = \hat \vbeta.
\end{equation}
From this view, logistic regression with a vanishing ridge regularizer is equivalent to max-margin classifier in the separable regime. By modifying the proof in \cite{rosset2003margin}, we can generalize their conclusion with $\beta_0$ included.
\begin{prop}\label{prop:explicit_bias}
    Let $(\hat\vbeta_\lambda, \hat\beta_{0,\lambda})$ be the minimizer of the regularized objective function in \cref{eq:logistic-l2}, where $\ell: \R \to \R_{\ge 0}$ is any convex, non-decreasing, rapidly varying loss function in the sense that
    \begin{equation*}
        \lim_{t \to \infty} \frac{\ell(\varepsilon t)}{\ell(t)} = \infty, \qquad \forall\, \varepsilon \in (0, 1).
    \end{equation*}
    Assume the data is linearly separable. 
    Then the convergence in \cref{eq:explicit_bias} holds. Moreover, we have $\lim_{\lambda \to 0^+} \hat\beta_{0,\lambda}/\| \hat \vbeta_\lambda \|_2 = \hat\beta_{0}$.
\end{prop}

Another approach of establishing the connection does not require adding an explicit regularizer. 
For convenience, let 
$\cL(\vbeta) = \frac1n\sum_{i=1}^n \ell( y_i \langle \vx_i, \vbeta \rangle )$ and
consider the gradient descent iterates $\vbeta^{(t+1)} = \vbeta^{(t)} - \eta \nabla \cL(\vbeta^{(t)})$ where $t=1,2,\ldots$ and 
$\vbeta^{(t)}$ is the parameter vector at iteration $t$. It it shown by \cite{Soudry_implicit_bias} that under a sufficiently small step size $\eta$, 
\begin{equation*}
    \lim\limits_{t\to\infty}\frac{\vbeta^{(t)}}{\| \vbeta^{(t)} \|_2} = \hat \vbeta,
\end{equation*}
where $\hat\vbeta := \min_{\norm{\vbeta}_2 \le 1} \min_{i \in [n]} y_i\< \xx_i, \vbeta \>$. This is often referred to as the implicit bias.


\subsection{Notations}

We typically use italic letters to denote scalars and random variables (e.g., $a, b, c, G, Y, \ldots \in \R$), boldface (italic) lowercase letters to denote (random) vectors (e.g., $\ba, \bs, \xx, \yy,  \ldots \in \R^d$), and boldface (italic) uppercase letters to denote (random) matrices (e.g. $\bA, \bP, \XX, \GG,  \ldots \in \R^{d_1 \times d_2}$). For any positive integer $n$, let $[n] = \{1, 2, \ldots, n\}$. For a scalar $a$, let $a_+ = \max\{ a, 0 \}$ and $a_+ = \max\{ -a, 0 \}$. For vectors $\bu$, $\bv$ of the same length, let $\< \bu, \bv \> = \bu^\top \bv$ denote their standard inner product, and write $\bu \perp \bv$ if they are orthogonal ($\< \bu, \bv \> = 0$). The corresponding Euclidean norm is $\norm{\bu} = \norm{\bu}_2 = \< \bu, \bu \>^{1/2}$. For a matrix $\bA$, let $\norm{\bA}_{\mathrm{op}}$ denote its operator norm and $\norm{\bA}_\mathrm{F}$ its Frobenius norm. We use $\phi$ and $\Phi$ to denote the cumulative distribution function (CDF) and probability density function (PDF) of standard normal distribution. Let $\Law(X)$ denote the distribution of random variable (or vector) $X$. We write $X \indep Y$ if $X$ and $Y$ are independent random variables.

We use $O(\cdot)$ and $o(\cdot)$ for the standard big-$O$ and small-$o$ notations. For real sequences $(a_n)_{n \ge 1}$, $(b_n)_{n \ge 1}$, we write $a_n \lesssim b_n$ or  $b_n \gtrsim a_n$ if $a_n = O(b_n)$, and $a_n \asymp b_n$ if $a_n \lesssim b_n$ and $a_n \gtrsim b_n$. We also write $a_n \ll b_n$ or $b_n \gg a_n$ if $a_n = o(b_n)$. We write $a_n \propto b_n$ if $a_n = c b_n, \forall\, n \ge 1$ for some constant $c > 0$. Let $\cond$, $\conp$, $\conL{p}$ denote stochastic convergence in distribution, in probability, in $\cL^p$, respectively, and let $\conw$ denote weak convergence of measures. We also use $O_{\P}(\cdot)$ and $o_{\P}(\cdot)$ for the standard big-$O$ and small-$o$ in probability notations. Denote $\wt O_{\P}(\cdot)$ as a variant of $O_{\P}(\cdot)$ which hides polylogarithmic factors. 

Given two probability measures $P$, $Q$ on $\R^d$, their second Wasserstein ($W_2$) distance is defined as
\begin{equation*}
    W_2(P, Q) := \left( \inf_{\gamma \in \Gamma(P, Q) }
    \int \norm{\bx - \by}_2^2 \gamma(\d\bx \times \d\by)
    \right)^{1/2},
\end{equation*}
where the infimum is taken over the set of couplings $\Gamma(P, Q)$ of distributions $P$ and $Q$. For any $x \in \R$ and $\lambda > 0$, the Moreau envelope of a continuous convex function $\ell: \R \to \R_{\ge 0}$ is defined as
\begin{equation*}
    \envelope_\ell(x; \lambda) = \envelope_{\lambda\ell}(x)
    := \min_{t \in \R} \left\{  \ell(t) +  \frac1{2\lambda} (t - x)^2 \right\},
\end{equation*}
and the proximal operator of $\ell$ is defined as
\begin{equation*}
    \prox_{\ell}(x; \lambda) =
    \prox_{\lambda \ell}(x) := \argmin_{t \in \R} \left\{ \ell(t) +  \frac1{2\lambda} (t - x)^2 \right\}.
\end{equation*}






%% file: src/main_logit_distribution.tex
\section{Precise asymptotics of empirical logit distribution} \label{sec:logit}

In this section, we present our main results on the asymptotics of empirical logit distribution introduced in \cref{subsec:ELD}. Recall that data $\{(\xx_i, y_i)\}_{i = 1}^n$ are i.i.d.~generated from a 2-GMM \cref{model}, i.e., $\xx_i \,|\, y_i \sim \normal(y_i \bmu, \bI_d)$, with label distribution $P_y: \P(y_i = +1) = 1 - \P(y_i = -1) = \pi \in (0, \frac12]$. We consider proportional asymptotics where $n,d \to \infty$ and $n/d \to \delta$ with $\delta \in (0,\infty)$. Based on relations between $\bmu, \pi, \delta$, we will consider linearly separable data (fitted by SVM) and non-separable data (fitted by logistic regression) separately.

We define the following functions $\delta^*: \R \to \R_{\ge 0}$ and $H_\kappa: [-1, 1] \times \R \to \R_{\ge 0}$ that are related to the critical threshold of data separability:
\begin{equation}\label{eq:sep_functions}
     \delta^*(\kappa) := \max_{ \rho \in [-1, 1] , \beta_0 \in \R }  H_\kappa(\rho, \beta_0),
     \qquad
     H_\kappa(\rho, \beta_0) := \frac{1 - \rho^2}{\E\left[ \bigl(  s(Y) \kappa - \rho \norm{\bmu}_2 + G - \beta_0 Y \bigr)_+^2 \right]},
\end{equation}
where $(Y, G) \sim P_y \times \normal(0,1)$ and 
\begin{equation}\label{eq:s_fun}
     s(y) := \begin{cases} \ \tau , & \ \text{if} \ y = + 1, \\
        \ 1, & \ \text{if} \ y = -1. \end{cases}
\end{equation}
We will show in \cref{thm:SVM_main} that the relationship between $\delta$ and $\delta^*(0)$ determines separability, where $\delta^*(0)$ does not depend on $\tau$ by definition.

We summarize the asymptotics of logit distribution for both separable and non-separable case in \cref{tab:ELD}, which is the main contribution of our theoretical results (\cref{thm:SVM_main} and \ref{thm:logistic_main}).
\begin{table}[h!]
\begin{equation*}
\renewcommand{\arraystretch}{1.2}
    \begin{array}{rll}
    \hline
           &  \textbf{limiting ELD} ~ (\hat\nu_*)  &  \textbf{cause for overfitting} ~ (\xi^*)   \\
    \hline
      \text{separable data}
          &  \Law\left( Y, \,  Y \max \{ \kappa^*, \mathtt{LOGITS} \} \right)
          &  R^* \sqrt{1 - \rho^{*2}} \xi^* = \left( \kappa^* - \mathtt{LOGITS} \right)_+  \\
      \text{non-separable data}
          &  \Law\left( Y, \,  Y \, \prox_{ \lambda^* \ell}( \mathtt{LOGITS} ) \right)
          &   R^* \sqrt{1 - \rho^{*2}} \xi^*  =  
            - \lambda^* \nabla \envelope_{ \lambda^* \ell}( \mathtt{LOGITS} ) \\
    \hline
          \textbf{limiting TLD} ~ (\hat\nu_*^\mathrm{test})
          &  \Law\left( Y, \,  Y \cdot \mathtt{LOGITS} \right)     &   \\
    \hline
    \multicolumn{3}{c}{\mathtt{LOGITS} := \rho^*\norm{\vmu}_2 R^* + R^*G + \beta_0^* Y 
    \quad \text{($R^* := 1$ in separable case)}
    }
    \end{array}
\end{equation*}
\vspace{-5mm}
\caption{Comparison of logit distributions on separable and non-separable data ($\tau = 1$).}
\label{tab:ELD}
\end{table}

\subsection{Separable data} \label{sec:logit_SVM}

For linearly separable data, recall the margin-rebalanced SVM in \cref{eq:SVM-m-reb} and \eqref{eq:SVM}. The following theorem summarizes the precise asymptotics of SVM under arbitrary $\tau$, including the limits of parameters, margin, and logit distribution. The proofs are deferred to the appendices.

Recall that data $\{(\xx_i, y_i)\}_{i = 1}^n$ are generated from 2-GMM with fixed parameters $\bmu \in \R^d$, $\pi \in (0, \frac12)$. Let $(\hat \vbeta_n, \hat \beta_{0, n})$ be an optimal solution to the margin-rebalanced SVM \cref{eq:SVM}, and let $\hat\kappa_n$ be the maximum margin as per \cref{def:max-margin}. Recall the cosine angle $\hat \rho_n:= \hat \rho$ between $\vmu$ and $\hat\vbeta_n$ defined in \cref{eq:rho_hat}. Let $\delta^*(\kappa)$ be defined as per \cref{eq:sep_functions}, and $\rho^*, \beta_0^*, \kappa^*, \xi^*$ be a solution to the variational problem
\begin{equation}\label{eq:SVM_variation}
    \begin{aligned}
        \begin{array}{cl}
            \underset{ \rho \in [-1, 1], \beta_0 \in \R, \kappa \in \R, \xi \in \cL^2  }{ \mathrm{maximize} } & \kappa, \\
            \underset{ \phantom{\smash{\bm\beta \in \R^d, \beta_0 \in \R, \kappa \in \R} } }{\text{subject to}} &  
            \rho \norm{\bmu}_2 + G + Y \beta_0 + \sqrt{1 - \rho^2} \xi \ge s(Y) \kappa,  
            \qquad \E[\xi^2]  \le  1/\delta .
        \end{array}
    \end{aligned}
    \end{equation}
    where $\cL^2$ is the space of all square integrable random variables in $(\Omega, \mathcal{F}, \P)$, and $(Y, G) \sim P_y \times \normal(0,1)$. 
    We define
    \begin{equation*}
    \begin{aligned}
        \nu_* & := \Law \,\bigl( Y,  Y \max\{ s(Y)\kappa^*, \rho^* \| \bmu \| + G + Y \beta_0^* \} \bigr),  \\
        \nu^\mathrm{test}_* & := \Law \,\bigl( Y, Y (\rho^* \| \bmu \| + G + Y \beta_0^*) \bigr).
    \end{aligned}
    \end{equation*}
which we will prove to be the limiting ELD and TLD respectively.
\begin{thm}[Separable data] \label{thm:SVM_main}
    Assume $n, d \to \infty$ with $n/d \to \delta \in (0, \infty)$. Fix $\tau \in (0, \infty)$. 
    \begin{enumerate}[label=(\alph*)]
        \item \label{thm:SVM_main_trans}
        \textbf{(Phase transition)} With probability tending to one, the data is linearly separable if $\delta < \delta^*(0)$ and is not linearly separable if $\delta > \delta^*(0)$.

        \item \label{thm:SVM_main_var} 
        \textbf{(Variational problem)} In the separable regime $\delta < \delta^*(0)$, $(\rho^*, \beta_0^*, \kappa^*, \xi^*)$ is the unique solution to \cref{eq:SVM_variation} with $\rho^* \in (0, 1)$ (not depend on $\tau$), $\kappa^* > 0$, and the random variable $\xi^*$ satisfies (a.s.)
        \begin{equation}\label{eq:SVM_main_xi_star}
            \sqrt{1 - \rho^{*2}} \xi^* = \bigl( s(Y) \kappa^* - \rho^*\norm{\vmu}_2 - G - Y \beta_0^*) \bigr)_+.
        \end{equation}
        Moreover, $(\rho^*, \beta_0^*, \kappa^*)$ is also the unique solution to
        \begin{equation}\label{eq:SVM_asymp_simple}
        \begin{array}{rl}
        \maximize\limits_{\rho \in [-1, 1], \beta_0 \in \R, \kappa \in \R} & \kappa, \\
        \text{\emph{subject to}} & H_\kappa(\rho, \beta_0) \ge \delta
        \end{array}
        \end{equation}
        and $\kappa^* = \sup\left\{ \kappa \in \R: \delta^*(\kappa) \ge \delta \right\}$.
        
        \item \label{thm:SVM_main_mar} 
        \textbf{(Margin convergence)} In the separable regime $\delta < \delta^*(0)$,
        \begin{equation*}
            \hat\kappa_n \conL{2} \kappa^*.
        \end{equation*}
        In the non-separable regime $\delta > \delta^*(0)$ we have negative margin, i.e., with probability tending to one, for some $\overline{\kappa} > 0$,
        \begin{equation*}
            \max_{ \substack{ \norm{\bbeta}_2 = 1 \\ \beta_0 \in \R } } \min_{i \in [n]} \, \wt y_i ( \< \xx_i, \bbeta \> + \beta_0 )  \le  -\overline{\kappa}.
        \end{equation*}

        \item \label{thm:SVM_main_param} 
        \textbf{(Parameter convergence)} In the separable regime $\delta < \delta^*(0)$,
        \begin{equation*}
            \hat\rho_n \conp \rho^*,
            \qquad
            \hat\beta_{0,n} \conp \beta_0^*.
        \end{equation*}

        \item \label{thm:SVM_main_err}
        \textbf{(Asymptotic errors)} Recall the minority and majority test prediction errors , $\Err_{+,n}$ and $\Err_{-,n}$ respectively, of the max-margin classifier defined in \cref{eq:Err_n} (writing subscript $n$ for clarity). 
        Then in the separable regime $\delta < \delta^*(0)$,
        \begin{equation*}
            \Err_{+,n}  \to  \Phi \left(- \rho^* \norm{\bmu}_2 - \beta_0^* \right),
            \qquad
            \Err_{-,n}  \to  \Phi \left(- \rho^* \norm{\bmu}_2  + \beta_0^* \right).
        \end{equation*}
        \item \label{thm:SVM_main_logit}
        \textbf{(ELD/TLD convergence)} 
        Recall the ELD $\hat\nu_n$ and TLD $\hat\nu_n^\mathrm{test}$ defined as per \cref{def:ELD_TLD}, where $\hat f(\xx) = \< \xx, \hat\vbeta_n \> + \hat\beta_{0, n}$.
        Then in the separable regime $\delta < \delta^*(0)$ we have logit convergence for both training and test data, i.e.,
        \begin{equation*}
            W_2\bigl(\hat{\nu}_n , \nu_* \bigr) \conp 0,
            \qquad
            \hat{\nu}_n^\mathrm{test} \conw \nu^\mathrm{test}_*.
        \end{equation*}
    \end{enumerate}
\end{thm}

\begin{rem}
    By taking $\tau = 1$, the ELD convergence $W_2(\hat{\nu}_n , \nu_* ) \conp 0$ in \cref{thm:SVM}\ref{thm:SVM_c} is a consequence of \cref{thm:SVM_main}\ref{thm:SVM_main_logit}, and the TLD convergence $\hat{\nu}_n^\mathrm{test} \conw \nu^\mathrm{test}_*$ is a corollary of \cref{thm:SVM_main}\ref{thm:SVM_main_param}.
\end{rem}

As discussed in \cref{subsec:ELD}, random variable $\xi^*$ and the nonlinear transformation $\mathtt{T}^*(x) = \max\{x, \kappa^*\}$ therein characterize the effect of overfitting on logits. The following result provides an optimal transport perspective of this overfitting effect. For ease of description, we reformulate $\nu_*$ and $\nu^\mathrm{test}_*$ in terms of the following one-dimensional measures
\begin{equation*}
    \cL_*  := \Law \,\bigl(  \max\{ \kappa^*, \rho^* \| \bmu \| + G + Y \beta_0^* \} \bigr),  
        \qquad 
    \cL^\mathrm{test}_* := \Law \,\bigl(  \rho^* \| \bmu \| + G + Y \beta_0^* \bigr).
\end{equation*}

\begin{prop}[Optimal transport map]\label{prop:opt_transport}
    $\mathtt{T}^*(x) = \max\{\kappa^*, x\}$ is the unique optimal transport map from $\cL_*^\mathrm{test}$ to $\cL_*$ under the cost function $c(x, y) = h(x - y)$ for any strictly convex $h: \R^2 \to \R_{\ge 0}$. That is, 
    \[ 
    \mathtt{T}^* = \argmin_{\mathtt{T}: \R \to \R} \left\{ \int_{\R} c \bigl( x, \mathtt{T}(x) \bigr)  \d \cL_*^\mathrm{test}(x)
    \,\middle|\,
    \mathtt{T}_\sharp \cL_*^\mathrm{test} = \cL_*
    \right\},
    \]
    where $\mathtt{T}_\sharp$ is the pushforward operator.
\end{prop}

\subsection{Non-separable data}
\label{sec:logit_logistic}

For non-separable data, SVM yields a trivial solution $\vbeta=\boldsymbol{0}, \beta_0 =0$. 
A typical approach to fitting a classifier is to solve regression problem \cref{eq:logistic}. Similar to the margin-rebalanced SVM \cref{eq:SVM}, we can also incorporate $\tau$ into the objective function by substituting $y_i$ for $\wt y_i = y_i/s(y_i)$, that is,
\begin{equation}\label{eq:logistic_reg}
    \min \limits_{\vbeta \in \R^d, \beta_0 \in \R} \qquad 
    \frac1n \sum_{i=1}^n \ell \bigl( \wt y_i(\langle \vx_i, \vbeta \rangle + \beta_0) \bigr),
\end{equation}
where $\ell: \R \to \R_{\ge 0}$ is the loss function. We consider a more general form than logistic regression. We say that $\ell$ is \emph{pseudo-Lipschitz} if there exists a constant $L > 0$ such that, for all $x, y \in \R$,
\begin{equation*}
    \abs{ \ell(x) - \ell(y) } \le L \left( 1 + \abs{ x } + \abs{ y } \right) \abs{ x - y }.
\end{equation*}
This condition is satisfied, for instance, by the widely used logistic loss $\ell(t) = \log(1 + e^{-t})$. As the counterpart of \cref{thm:SVM_main} in the non-separable regime, the following theorem summarizes the precise asymptotics of regression \cref{eq:logistic_reg}, including the limits of parameters and logit distribution.

We consider the same 2-GMM setting as \cref{sec:logit_SVM}. For any non-increasing, strictly convex, pseudo-Lipschitz, twice differentiable function $\ell: \R \to \R_{\ge 0}$, let $(\hat \vbeta_n, \hat \beta_{0, n})$ be the optimal solution to regression \cref{eq:logistic_reg}. Recall $\hat \rho_n:= \hat \rho$ defined in \cref{eq:rho_hat} and $\delta^*(\kappa)$ defined in \cref{eq:sep_functions}. Let $\rho^*, R^*, \beta_0^*, \xi^*$ be a solution to the variational problem
    \begin{equation}\label{eq:logistic_variation}
    \begin{aligned}
        \begin{array}{cl}
            \underset{ \rho \in [-1, 1], R \ge 0, \beta_0 \in \R,\xi \in \cL^2 }{ \mathrm{minimize} }
            &
        \E \left[ \ell \biggl( \dfrac{ \rho\norm{\vmu}_2 R + RG + \beta_0 Y + R\sqrt{1 - \rho^2} \xi }{s(Y)} \biggr) \right], \\
            \text{subject to} & \vphantom{\dfrac11} \E \left[ \xi^2 \right]  \le  1/\delta .
        \end{array}
    \end{aligned}
    \end{equation}
    where $(Y, G) \sim P_y \times \normal(0,1)$. 
    We define
    \begin{equation*}
    \begin{aligned}
        \nu_* & := 
        \Law \!\left( 
        Y, Ys(Y) \, \prox_{\frac{\lambda^* \ell}{s(Y)}}\biggl( \frac{\rho^* \norm{\vmu}_2 R^* + R^* G + \beta_0^* Y}{s(Y)} \biggr)
        \right).
        \\
        \nu^\mathrm{test}_* & := \Law \,\bigl( Y, Y ( R^*\rho^* \| \bmu \| + R^*G + Y \beta_0^*) \bigr),
    \end{aligned}
    \end{equation*}
    aiming to show they are the limiting ELD and TLD respectively.

\begin{thm}[Non-separable data] \label{thm:logistic_main}
    Consider the same 2-GMM and proportional settings $n/d \to \delta$ as in \cref{thm:SVM_main}. 
    \begin{enumerate}[label=(\alph*)]
        \item \label{thm:logistic_main(a)}
        \textbf{(Variational problem)} In the non-separable regime $\delta > \delta^*(0)$, $(\rho^*, R^*, \beta_0^*, \xi^*)$ is the unique solution to \cref{eq:logistic_variation} with $\rho^* \in (0, 1)$, $R^* \in (0, \infty)$, and the random variable $\xi^*$ satisfies (a.s.)
        \begin{equation}\label{eq:logistic_xi_star}
            R^* \sqrt{1 - \rho^{*2}} \xi^*  =  -\lambda^* 
            \ell' \biggl( \prox_{\frac{\lambda^* \ell}{s(Y)}}\Bigl( \frac{\rho^* \norm{\vmu}_2 R^* + R^* G + \beta_0^* Y}{s(Y)} \Bigr) \biggr) ,
        \end{equation}
        where $\lambda^* \in (0, \infty)$ is the unique constant such that $\E[ \xi^2] = 1/\delta$.
        Moreover, $(\rho^*, R^*, \beta_0^*, \lambda^*)$ is also the unique solution to the following system of equations
        \begin{align*}
            - \frac{\tau R \rho}{2\pi \lambda \delta \norm{\vmu}_2}
            & = 
            \E\left[ \ell'\biggl( \prox_{\frac{\lambda\ell}{\tau}}\Bigl( \frac{\rho\norm{\vmu}_2 R + RG + \beta_0}{\tau} \Bigr) \biggr) \right] ,
            \\
            - \frac{R \rho}{2(1 - \pi) \lambda \delta \norm{\vmu}_2}
            & = 
            \E \left[ \ell'\bigl( \prox_{ \lambda \ell}( \rho\norm{\vmu}_2 R + RG - \beta_0 ) \bigr) \right] ,
            \\
            \frac{1}{\lambda \delta }
            & = 
            \E \left[ \dfrac{1}{s(Y)} \cdot \frac{
            \ell'' \biggl( \prox_{\frac{\lambda \ell}{s(Y)}}\Bigl( \dfrac{\rho\norm{\vmu}_2 R + RG + \beta_0 Y}{s(Y)} \Bigr) \biggr)
            }{s(Y) + 
            \lambda \ell'' \biggl( \prox_{\frac{\lambda \ell}{s(Y)}}\Bigl( \dfrac{\rho\norm{\vmu}_2 R + RG + \beta_0 Y}{s(Y)} \Bigr) \biggr)
            } \right] ,
            \\
            \frac{R^2 (1 - \rho^2)}{\lambda^2 \delta}
            & = 
            \E \left[ \Biggl(  
            \frac{1}{s(Y)} \cdot \ell' \biggl( \prox_{\frac{\lambda \ell}{s(Y)}}\Bigl( \frac{\rho\norm{\vmu}_2 R + RG + \beta_0 Y}{s(Y)} \Bigr) \biggr)
            \Biggr)^2 \right].
        \end{align*}

        \item \label{thm:logistic_main(b)}
        \textbf{(Parameter convergence)} In the non-separable regime $\delta > \delta^*(0)$, as $n \to \infty$,
        \begin{equation*}
            \| \hat\vbeta_n \|_2 \conp R^*,
            \qquad
            \hat\rho_n \conp \rho^*,
            \qquad
            \hat\beta_{0,n} \conp \beta_0^*.
        \end{equation*}

        \item \label{thm:logistic_main(c)}
        \textbf{(Asymptotic errors)} Recall the prediction errors defined as per \cref{eq:Err_n}. Then in the non-separable regime $\delta > \delta^*(0)$, as $n \to \infty$,
        \begin{equation*}
            \Err_{+,n}  \to  \Phi \left(- \rho^* \norm{\bmu}_2 - \frac{\beta_0^*}{R^*} \right),
            \qquad
            \Err_{-,n}  \to  \Phi \left(- \rho^* \norm{\bmu}_2  + \frac{\beta_0^*}{R^*} \right).
        \end{equation*}

        \item \label{thm:logistic_main(d)}
        \textbf{(ELD/TLD convergence)} 
        Recall the ELD $\hat\nu_n$ and TLD $\hat\nu_n^\mathrm{test}$ defined as per \cref{def:ELD_TLD}, where $\hat f(\xx) = \< \xx, \hat\vbeta_n \> + \hat\beta_{0, n}$.
        Then in the non-separable regime $\delta > \delta^*(0)$ we have logit convergence for both training and test data, i.e., as $n \to \infty$,
        \begin{equation*}
            W_2\bigl(\hat{\nu}_n , \nu_* \bigr) \conp 0,
            \qquad
            \hat{\nu}_n^\mathrm{test} \conw \nu^\mathrm{test}_*.
        \end{equation*}
    \end{enumerate}
\end{thm}

\begin{rem}
    Compared to the separable regime, the random variable $\xi$ in the non-separable regime \cref{eq:logistic_variation} can also be interpreted as the cause for overfitting, but its distortion effect on ELD is not truncation. When $\tau = 1$, by \cref{eq:logistic_variation}, \eqref{eq:logistic_xi_star}, the following holds for a ``typical'' training point:
\begin{equation*}
\begin{aligned}
    y_i (\langle \vx_i, \hat \vbeta_n \rangle + \hat \beta_{0,n}) 
    & \approx 
    \rho^*\norm{\vmu}_2 R^* + R^*G + \beta_0^* Y + R^*\sqrt{1 - \rho^{*2}} \xi^* \\
    & = \rho^*\norm{\vmu}_2 R^* + R^*G + \beta_0^* Y 
    - \lambda^* 
            \ell' \bigl( \prox_{ \lambda^* \ell}( \rho^*\norm{\vmu}_2 R^* + R^*G + \beta_0^* Y ) \bigr) \\
    & = \prox_{ \lambda^* \ell}( \rho^*\norm{\vmu}_2 R^* + R^*G + \beta_0^* Y ),
\end{aligned}
\end{equation*}
where the equalities come from \cref{lem:prox}. Hence, the ELD in the non-separable regime is the TLD under nonlinear shrinkage due to the proximal operator of loss function $\ell$.
\end{rem}

%% file: src/main_rebalance_margin.tex
\section{Analysis of margin rebalancing for separable data} \label{sec:rebalacing}

In this section, we show how margin rebalancing improves the test accuracies on imbalanced dataset by choosing the hyperparameter $\tau$ in \cref{fig:SVM_cartoon} appropriately.

\subsection{Proportional regime}
\label{subsec:rebal_prop}
Consider the same 2-GMM and proportional settings in \cref{sec:logit_SVM} on linearly separable dataset ($\delta < \delta^*(0)$). According to \cref{thm:SVM_main}\ref{thm:SVM_main_err}, the asymptotic minority and majority test errors are
\begin{equation}\label{eq:asymp_Err}
    \Err_{+}^*  :=  \Phi \left(- \rho^* \norm{\bmu}_2 - \beta_0^* \right),
            \qquad
    \Err_{-}^*  :=  \Phi \left(- \rho^* \norm{\bmu}_2  + \beta_0^* \right).
\end{equation}
For the purpose of imbalanced classification, we define the \textit{asymptotic balanced error} as
\begin{equation*}
\Err_\mathrm{b}^* := \frac{1}{2} \Err_{+}^* + \frac{1}{2} \Err_{-}^*.
\end{equation*}

\paragraph{Monotonicity analysis.} 
We first provide some monotone results for test errors, which support our empirical observations in \cref{subsec:rebal}. 







\begin{prop}\label{prop:Err-_mono}
    $\Err_{+}^*$ is a decreasing function of $\pi \in (0, \frac12)$, $\norm{\bmu}_2$, and $\delta$ when $\tau = 1$.
\end{prop}

However, the majority error $\Err_{-}^*$ and balanced error $\Err_\mathrm{b}^*$ are not necessarily monotone under arbitrary $\tau$. Thus, we will focus on the monotonicity of these test errors when $\tau$ is chosen to be optimal.


According to \cref{fig:SVM_cartoon} and \ref{fig:Err_tau}, by taking $\tau > 1$, we can improve the minority accuracy at the cost of harming majority accuracy. The opposite effects of $\tau$ on $\Err^*_+$ and $\Err^*_-$ are summarized in the following result.

\begin{prop}\label{prop:tau_mono}
    $\Err^*_+$ is decreasing in $\tau \in (0, \infty)$, and $\Err^*_-$ is increasing in $\tau \in (0, \infty)$.
\end{prop}

\paragraph{Choosing the optimal $\tau$.} A natural idea for margin rebalancing is to choose $\tau$ such that the balanced error $\Err_\mathrm{b}^*$ is minimized.

\begin{prop}[Optimal $\tau$] \label{prop:tau_optimal}
    Let $\tau^\mathrm{opt}$ be the optimal margin ratio $\tau$ defined in \cref{prop:tau_opt}. Denote $g_1 (x) := \E \left[ (G + x)_+ \right]$ where $G \sim \normal(0, 1)$. Then $\tau^\mathrm{opt}$ has the 
 explicit expression 
\begin{equation}\label{eq:tau_opt}
        \tau^\mathrm{opt} =  \dfrac{g_1^{-1} \left( \dfrac{\rho^*}{2 \pi \norm{\bmu}_2 \delta} \right) + \rho^* \norm{\bmu}_2}{g_1^{-1} \left( \dfrac{\rho^*}{2 (1 - \pi) \norm{\bmu}_2 \delta} \right) + \rho^* \norm{\bmu}_2}.
\end{equation}
\end{prop}
\begin{rem}\label{rem:tau_pi}
    The optimal choice of $\tau$ has a complicated dependence on $\pi$. However, we note that the numerator scales as $\tau^\mathrm{opt} \asymp \sqrt{1/\pi}$ for small $\pi$ and fixed $\norm{\bmu}_2$ and $\delta$, which is consistent with the choice of $\tau$ in importance tempering \cite{lu2022importance}. In \cite{cao2019learning}, $\tau$ is suggested to scale with $\pi^{-1/4}$, however it was proved in \cite{kini2021label} that their algorithm won't converge to the solution with the desired $\tau$.

    It is worth noticing that in the near-degenerate cases where $\pi$ is very small or $\norm{\bmu}_2$, $\delta$ are very large, then $\rho^*$ is close to $0$ and the denominator can be negative,
    leading to $\tau^\mathrm{opt} < 0$. While our theory (\cref{prop:Err_monotone}, \cref{prop:conf}) is still valid when we allow potential negative $\tau$, it is rarely used in practice. See \cref{subsec:tau_optimal} for a further discussion. 
    The near-degenerate cases (small $\pi$, large $\norm{\bmu}_2$ or $\delta$) are better addressed under the high imbalance regime, as we analyze in the next subsection. 
\end{rem}
The minority/majority/balanced errors all equal $\Phi(- \rho^* \norm{\bmu}_2 )$ when $\tau = \tau^\mathrm{opt}$. 
We can also obtain the monotonicity of test errors after margin rebalancing.
\begin{prop}\label{prop:Err_monotone}
    When $\tau = \tau^{\rm opt} > 0$, all the test errors $\Err^*_+$, $\Err^*_-$, $\Err^*_\mathrm{b}$ are decreasing functions of $\pi \in (0, 1/2)$ (imbalance ratio), $\delta$ (aspect ratio), and $\norm{\bmu}_2$ (signal strength).
\end{prop}

\subsection{High imbalance regime}\label{subsec:high-imb}

Different from the proportional regime considered in \cref{sec:logit} and \ref{subsec:rebal_prop}, here we focus on a high-imbalanced scenario where $\pi$ is small, $\norm{\bmu}_2$ is large, and $n$ grows much faster than $d$. In this regime, we can even extend the feature distribution beyond Gaussian, and generalize the 2-GMM settings.

\begin{defn}[High imbalance]
We say a dataset $\{(\xx_i, y_i)\}_{i = 1}^n$ is i.i.d.~generated from a \emph{two-component sub-gaussian mixture model (2-subGMM)} if for any $i \in [n]$,
\begin{enumerate}
    \item[i.] Label distribution: $\P(y_i = +1) = 1 - \P(y_i = -1) = \pi$, 
    \item[ii.] Feature distribution: $\xx_i = y_i \bmu + \zz_i$, where $\zz_i$ has independent coordinates with uniformly bounded sub-gaussian norms. Namely, each coordinate $z_{ij}$ of $\zz_i$ satisfies $\E[z_{ij}] = 0$, $\var(z_{ij}) = 1$, and $\norm{z_{ij}}_{\psi_2} := \inf \{ K > 0: \E[\exp(X^2/K^2)] \le 2 \} \le C$ for all $j \in [d]$, where $C$ is an absolute constant.
\end{enumerate}
For any constants $a,b,c >0$, we say a 2-subGMM is \emph{$(a,b,c)$-imbalanced} if \cref{setup-high-imbalance} holds.
\end{defn}
\begin{rem}
    Parameters $a$, $\frac{b}{2}$, and $c$ each specifies the degenerate rate of imbalance ratio $\pi$, and the growth rate of signal strength $\| \bmu \|_2$, aspect ratio $n/d$. We usually require $a < c + 1$ to make sure the minority class sample size $n_+ := \pi n = d^{c - a + 1} \to \infty$ does not degenerate.
\end{rem}

Our goal is to study the performance of margin-rebalanced SVM \cref{eq:SVM-m-reb} in this high imbalance regime, asymptotically as $d \to \infty$. Therefore, we allow $\tau = \tau_d$ depends on dimension $d$ and care about what order of $\tau_d$ would make the test errors vanish. We summarize our findings in the following theorem, which is consistent with the empirical observations in \cref{fig:High_imb_heat} and extends \cref{thm:high-imbalance} to the case of imbalanced 2-subGMM.

\begin{thm}[Phase transition in high imbalance regime]\label{thm:main_high-imbal}
Consider the high imbalance regime where the training data is i.i.d.~generated from an $(a,b,c)$-imbalanced 2-subGMM. Suppose that $a - c < 1$. A margin-rebalanced SVM is trained, with test errors calculated according to \cref{eq:Err_n}. Then as $d \to \infty$, 
the conclusions of the three phases in \cref{thm:high-imbalance} still hold.
\end{thm}

%% file: src/main_calibration.tex
\section{Consequences for confidence estimation and calibration} \label{sec:calibration}


Recall the definition of confidence of the max-margin classifier $\hat p(\xx) := \sigma ( \hat f(\xx) ) = 
        \sigma( \< \xx, \hat\vbeta \> + \hat\beta_0 )$ in \cref{subsec:conf_calib}.
Note that $\hat p(\xx)$ and $1 - \hat p(\xx)$ are the predicted probabilities of $\xx$ for the minority class ($y = +1$) and the majority class ($y = -1$) respectively. 

It is worth noticing that the confidence is sensitive to scales, i.e., $\sigma(t) \not= \sigma(ct)$ if $c \not= 1$, despite the fact that rescaling yields the same label prediction and thus does not affect accuracy. While small models tend to be calibrated, especially when parameter estimation is consistent, larger models such as DNNs are known to suffer from poor calibration \cite{guo2017calibration}. A simple theoretical explanation is that in a DNN, the last layer (usually a logistic regression) $\xx \mapsto \sigma\bigl( \< \xx, \hat\bbeta \> + \hat\beta_0 \bigr)$ trained by gradient descent on separable features often results in a very large $\| \hat\vbeta\|_2$ (as mentioned in \cref{subsec:LR_vs_SVM}), thereby inflating the predicted probabilities. Here we focus on the common form of SVM \eqref{eq:SVM-0} where normalization $\| \hat\vbeta \|_2 = 1$ is applied.

Some probabilities regarding the confidence are as follows.

\begin{enumerate}
    \item \textbf{Max-margin confidence.}
    The confidence of the max-margin classifier is
    \begin{equation*}
        \hat p(\xx) := \sigma\bigl( \hat f(\xx) \bigr) = 
        \sigma\bigl( \< \xx, \hat\vbeta \> + \hat\beta_0 \bigr).
    \end{equation*}
    
    \item \textbf{Bayes optimal probability.}
    The true conditional probability is
    \begin{equation*}
        \begin{aligned}
            p^*(\xx) := \P( y = 1 \,|\, \xx ).
        \end{aligned}
    \end{equation*}
    
    \item \textbf{True posterior probability.} The probability conditioning on max-margin confidence is
    \begin{equation*}
        \begin{aligned}
            \hat p_0(\xx) := \P \bigl( y = 1 \,|\, \hat p(\xx) \bigr).
        \end{aligned}
    \end{equation*}
\end{enumerate}
Note that $p^*(\xx)$ is the confidence of the Bayes classifier $y^*(\xx) := 2\ind\{  \< \xx, 2\bmu \> + \log\frac{\pi}{1 - \pi} > 0 \} - 1$.

Recall the definition of calibration \cref{eq:calibrated} and some miscalibration metrics \cref{eq:CalErr}---\eqref{eq:ConfErr} introduced in \cref{subsec:conf_calib}. We offer some further explanations for them.
\begin{itemize}
    \item \textbf{Calibration error:} The $\cL^2$ distance between confidence and posteriori, which is the most commonly used metric.
    \begin{equation*}
    \mathrm{CalErr}(\hat p) :=  \E\left[ \bigl( \hat p(\xx) - \hat p_0(\xx) \bigr)^2 \right]
    \end{equation*}

    \item \textbf{Mean squared error (MSE):} Also known as the Brier score, subject to a calibration budget \cite{brier1950verification, gneiting2007probabilistic}.
    \begin{equation*}
        \mathrm{MSE}(\hat p)  := \E\left[ \bigl( \mathbbm{1}\{ y = 1 \} - \hat p(\xx) \bigr)^2 \right]
    \end{equation*}
    It can be shown that MSE has the following decomposition
    \begin{equation*}
        \mathrm{MSE}(\hat p) 
        \ = \ \underbrace{ \var\, \bigl[\mathbbm{1}\{ y = 1 \} \bigr] }_{\text{irreducible}}  
        + 
        \underbrace{ \vphantom{\bigl[} \mathrm{CalErr}(\hat p) }_{\text{lack of calibration}}
        - 
        \underbrace{\var\, \bigl[ \hat p_0(\xx) \bigr]}_{\text{sharpness/resolution}}
        .
    \end{equation*}
    Calibration error itself does not guarantee a useful predictor. Sharpness, also known as resolution \cite{murphy1973new, kuleshov2015calibrated}, is another desired property which measures the variance in the response $y$ explained by the probabilistic prediction $\hat p(\xx)$. Hence, a small MSE suggests a classifier to be calibrated with high sharpness. 
    
    Note that $\var[\mathbbm{1}\{ y = 1 \}] = \pi(1 - \pi)$ is an intrinsic quantity unrelated to $\hat f$. When study the effect of $\pi$ on model calibration, we may discard the irreducible variance term and define a modified MSE as
    \begin{equation*}
        \mathrm{mMSE}(\hat p) := \mathrm{CalErr}(\hat p) - \var \bigl[ \hat p_0(\xx) \bigr].
    \end{equation*}
    \item \textbf{Confidence estimation error:} The $\cL^2$ distance between confidence and Bayes optimum.
    \begin{equation*}
        \mathrm{ConfErr}(\hat p) :=  \E\left[ \bigl( \hat p(\xx) - p^*(\xx) \bigr)^2 \right]
        .
    \end{equation*}
    It has the following relation with MSE:
    \begin{equation}
    \label{eq:MSE_vs_ConfErr}
        \mathrm{MSE}(\hat p) 
        =
        \E \left[ p^*(\xx) \bigl( 1 - p^*(\xx) \bigr) \right]
        + \mathrm{ConfErr}(\hat p),
    \end{equation}
    where the first term is intrinsic, which only depends on $\pi$ and $\norm{\bmu}_2$. 
    
\end{itemize}

The asymptotics of these metrics, and some monotone effect of model parameters $\pi \in (0, 1/2)$, $\norm{\bmu}_2$, $\delta$ on them, are summarized in the following proposition.

\begin{prop}[Confidence estimation and calibration]\label{prop:conf}
    Consider 2-GMM and the proportional settings in \cref{sec:logit_SVM} on linearly separable dataset ($\delta < \delta^*(0)$).
    \begin{enumerate}[label=(\alph*)]
        \item \label{prop:conf_asymp}
        Let $(\rho^*, \beta_0^*)$ be defined as per \cref{thm:SVM_main}, and $(Y, G) \sim P_y \times \normal(0,1)$. Denote
        \begin{equation*}
            \begin{aligned}
                \mathrm{MSE}^*  
                & :=  \E \left[ \sigma \bigl( -\rho^* \norm{\bmu}_2 - \beta_0^* Y + G \bigr)^2 \right],
                \qquad
                \mathrm{mMSE}^* = \mathrm{MSE}^* - \pi(1 - \pi),
                \\
                \mathrm{CalErr}^*
                & :=  \E\left[ \left( \sigma\Bigl( 2\rho^* \norm{\bmu}_2 (\rho^* \norm{\bmu}_2 Y + G) + \log\frac{\pi}{1-\pi} \Bigr) - \sigma\bigl( \rho^* \norm{\bmu}_2 Y + G + \beta_0^* \bigr) \right)^2 \right], 
                \\
                V_{y|\xx}^*
                & :=  \E\left[ \sigma\Bigl( -2\norm{\bmu}_2( \norm{\bmu}_2 + G ) - \log\frac{\pi}{1-\pi} Y \Bigr)^2 \right],
                \qquad
                \mathrm{ConfErr}^* = \mathrm{MSE}^* - V_{y|\xx}^*.
            \end{aligned}
        \end{equation*}
        then
        \begin{equation*}
        \begin{aligned}
            \lim_{n \to \infty} \mathrm{MSE}(\hat p) = \mathrm{MSE}^*,
            \qquad
            & \lim_{n \to \infty} \mathrm{CalErr}(\hat p) = \mathrm{CalErr}^*, 
            \\
            \lim_{n \to \infty} \mathrm{mMSE}(\hat p) = \mathrm{mMSE}^*,
            \qquad
            & \lim_{n \to \infty} \mathrm{ConfErr}(\hat p) = \mathrm{ConfErr}^*.
        \end{aligned}
        \end{equation*}
        
        \item \label{prop:conf_mono}
        When $\tau = \tau^\mathrm{opt} > 0$, 
        \begin{itemize}
            \item $\mathrm{MSE}^*$ and $\mathrm{mMSE}^*$ are decreasing functions of $\pi \in (0, \frac12)$, $\norm{\bmu}_2$, $\delta$.
            \item $\mathrm{ConfErr}^*$ is decreasing in $\delta$.
            
        \end{itemize}
    \end{enumerate}
\end{prop}
In addition, there are some monotone relationships that can be verified numerically. We summarized them in the following claim.
\begin{claim}\label{claim:conf}
    Consider the same settings as \cref{prop:conf}. When $\tau = \tau^\mathrm{opt} > 0$, we have
    \begin{itemize}
            \item $\mathrm{CalErr}^*$ is decreasing in $\norm{\bmu}_2$ and $\delta$, for any $\pi \le \overline{\pi}$ fixed, where $\overline{\pi} \approx 0.25$ is some constant.
            \item $\mathrm{ConfErr}^*$ is decreasing in $\pi \in (0, \frac12)$.
        \end{itemize}
\end{claim}

%% file: src/discussions.tex
\section{Discussions}\label{sec:discuss}


Below we present two possible extensions of our main results stated in previous sections.

\paragraph{Multiclass classification.} Our theory does not cover classification with $K > 2$ classes. However, we believe that for multiclass classification, the limiting empirical distribution of the logits is a similar ``rectified Gaussian" in $\R^K$. In what follows, we present some empirical evidences and an informal conjecture on the logits distribution.

In the $K$-class case, we observe features $\xx_i \in \R^d$ and labels $y_i \in [K] \sim P_y$, where the expected fractions of each class is $\pi_k := \P(y_i = k)$, $k \in [K]$. Given $y_i = k$, the conditional distribution of $\xx_i$ is $\normal (\vmu_k, \bI_d)$, where $\{ \vmu_k \}_{k \in [K]}$ are the class means. Let $\hat\vf(\xx) = \hat\bW \xx + \hat \vw_0$ be the logits of multinomial logistic regression for $\{ (\xx_i, y_i) \}_{i=1}^{n}$, where $\hat\bW \in \R^{K \times d}$, $\hat \bw_0 \in \R^K$ are weights. The prediction is given by $\hat y(\xx) := \argmax_{k \in [K]} \hat f_k(\xx)$, where $\hat f_k(\xx)$ is the logit of $\xx$ for label $k$, i.e., the $k$-th component of $\hat\vf(\xx)$. Similar to the binary case discussed in \cref{subsec:LR_vs_SVM}, we expect similar connections between multiclass SVM and multinomial logistic regression exist. 


We first conduct some numerical experiments to demonstrate that the truncation effect is likely generalizable to $K > 2$. In \cref{fig:multiclass}, we present the density heatmaps of joint logits $\bigl(\hat f_{1}(\xx_i), \hat f_{k}(\xx_i)\bigr)$ in both simulations and real-data analysis, where all input features $\xx_i$ are from class $1$. More specifically, in our simulation, we consider 3-component GMM with $\pi_1 = 0.5$, $\pi_2 = 0.3$, $\pi_3 = 0.2$, $n=50,000$, $d=6,000$, and class centers are randomly generated in $\R^d$ from the $\cL^2$-sphere (at the origin with radius $4$). 
For real data, we consider CIFAR-10 image dataset preprocessed by the pretrained ResNet-18. We undersample to obtain an imbalanced dataset with sample size 500, 223, 100 for each class 1, 2, 3. 
Then, we train a multinomial logistic regression (with ridge regularization parameter $\lambda = 10^{-8}$) after prewhitening the features. Similar to the linear separable regime, We find that training accuracy is $100\%$ with high probability. For both experiments, we plot the joint logits with $k = 2, 3$ for class 1. 

Notably, we observe similar truncation phenomena for $3$-class classification on both synthetic and real data, where the Gaussian density is visibly truncated by two hyperplanes.
For general $K \ge 3$, we conjecture that the empirical joint distribution of the logits $\{ \hat\vf(\xx_i) \}_{i=1}^n$ is asymptotically a multivariate Gaussian projected to a convex polytope in $\R^K$, where specific parameters of this limiting distribution depends on certain variational problem analogous to \Cref{eq:asymp}.

\begin{figure}[h]
    \centering
    \includegraphics[width=0.49\textwidth]{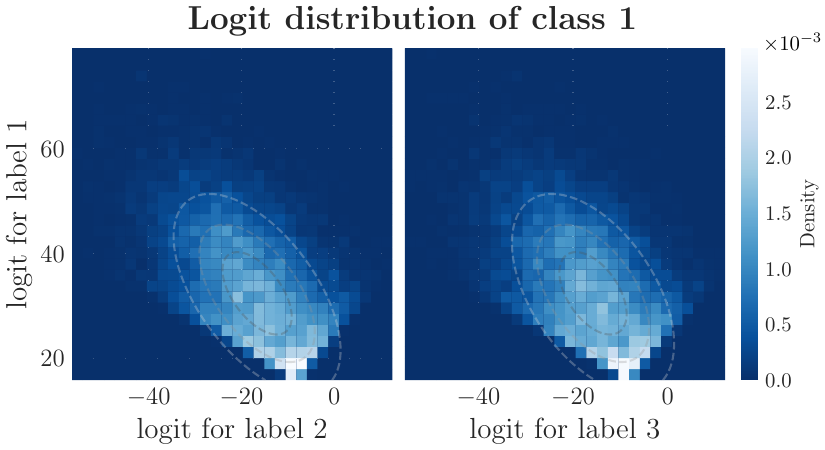}
    \includegraphics[width=0.49\textwidth]{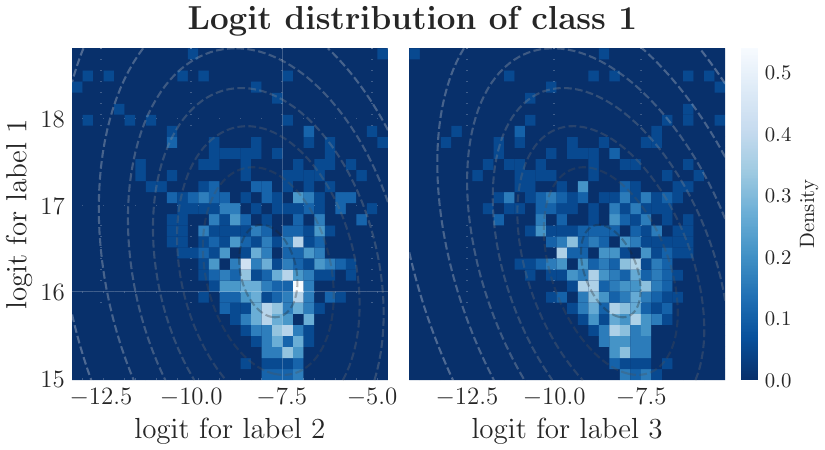}
    \caption{
    \textbf{Joint empirical logit distributions of multinomial logistic regression.} The heatmaps display empirical joint logits $\bigl(\hat f_{1}(\xx_i), \hat f_{k}(\xx_i)\bigr)$ for features $\xx_i$ from class 1, where $k=2,3$. Overlaid Gaussian density contours (dashed curves) depict testing logit distributions. 
    \textbf{Left:} 3-GMM simulation. \textbf{Right:} CIFAR-10 image features  preprocessed by pretrained ResNet-18.
    }
    \label{fig:multiclass}
\end{figure}

\paragraph{Non-isotropic covariance.} We also provide a characterization of the empirical logits distribution under a more general covariance assumption than \eqref{model}. In particular, we assume that the data $\{ (\xx_i, y_i) \}_{i = 1}^{n}$ are i.i.d. satisfying
\begin{equation}\label{spiked_model}
    \P(y_i = +1) = \pi, \quad \P(y_i = -1) = 1 - \pi, \quad \xx_i \,|\, y_i \sim \normal(y_i \bmu, \boldsymbol{\Sigma}),
\end{equation}
where $\boldsymbol{\Sigma} = q^2 \boldsymbol{V} \boldsymbol{V}^\top + \id_d$, and the spike $\boldsymbol{V}$ is a $d \times J$ orthogonal matrix. For this model, we have the following analogous result on the empirical logits distribution of SVM:
\begin{conj}
    Assume that $n/d \to \delta$,  $J/d \to \psi_1$ as $n, d, J \to \infty$. Denote $\psi_2 = 1 - \psi_1$, and further assume that
    \begin{equation*}
        \lim_{n \to \infty} \norm{\boldsymbol{V}^\top \bmu} = c_1, \quad \lim_{n \to \infty} \sqrt{ \norm{\boldsymbol{\mu}}^2 - \norm{\boldsymbol{V}^\top \bmu}^2 } = c_2.
    \end{equation*}
    Let $(\hat{\boldsymbol{\beta}}, \hat{\beta}_0)$ be the max-margin solution to \cref{eq:SVM-0}. Then, as $n \to \infty$, 
    \begin{equation*}
        W_2 \left( \frac{1}{n} \sum_{i=1}^{n} \delta_{( y_i, \<\xx_i, \hat{\boldsymbol{\beta}}\> + \hat{\beta}_0 )}, \, \Law \left( Y, Y \rho_1^* c_1 + Y \rho_2^* c_2 + \sqrt{1 + q^2} r_1^* G_1 + r_2^* G_2 + \beta_0^* \right) \right) \conp \, 0.
    \end{equation*}
    In the above display, $Y \sim P_y$ is independent of $(G_1, G_2) \sim \normal(0, 1)^{\otimes 2}$, and $(\rho_1^*, \rho_2^*, r_1^*, r_2^*, \beta_0^*)$ solves the following convex optimization problem:
    \begin{equation}\label{spike_covariance_opt}
    \begin{split}
        \maximize_{\rho_1, \rho_2, r_1, r_2, \beta_0, \kappa} \quad & \kappa, \\
        \mathrm{subject \, to} \quad & \E \left[ \left( \kappa - \rho_1 c_1 - \rho_2 c_2 - \sqrt{1 + q^2} r_1 G_1 - r_2 G_2 - \beta_0 Y \right)_+^2 \right] \\
        & \le \frac{1}{\delta} \left( \sqrt{1 + q^2} \sqrt{r_1^2 - \rho_1^2} \sqrt{\psi_1} + \sqrt{r_2^2 - \rho_2^2} \sqrt{\psi_2} \right)^2, \\
        & r_1^2 + r_2^2 = 1, \, \rho_1^2 \le r_1^2, \, \rho_2^2 \le r_2^2.
    \end{split}
    \end{equation}
\end{conj}

\paragraph{Future work.} In deep learning, the features are learned by optimizing the loss over all weights in a neural network, and data imbalance impacts on feature learning in a complex way as observed in \cite{cao2019learning}. Also, models tend to erroneously find spurious features if data imbalance is severe \cite{sagawa2020investigation}. It would be interesting to analyze overfitting and propose remedies for these scenarios.

%% file: src/append_experiments.tex
\section{Experiment details}
\label{append_sec:exp}

\subsection{Experiment setup and details}

We present the details of our experiments, including the computational configurations, information about the datasets, and the pretrained neural networks used in our study.

\paragraph{Optimization.} We used the functions \texttt{linear\_model.LogisticRegression} and \texttt{svm.SVC} from Python module \texttt{sklearn} to solve logistic regression \cref{eq:logistic} and SVM \cref{eq:SVM-0} (more precisely, \cref{eq:SVM-1} parametrization with $\tau = 1$). For logistic regression, we used the limited-memory BFGS (L-BFGS) solver, with maximum number of iterations $10^6$.\footnote{If logistic regression is far from converging after the maximum number of iterations is reached, we would add a small explicit regularizer as in \cref{eq:logistic-l2}. In practice, the parameter $\mathtt{C} = \lambda^{-1}$ in \texttt{linear\_model.LogisticRegression} can be chosen as $10^6 \sim 10^8$.} 
For SVM, we set the default value of cost parameter $\mathtt{C} = 1$.\footnote{Note that there is no hard-margin solver available in \texttt{sklearn} and \texttt{svm.SVC} is a soft-margin version. One may set $\mathtt{C}$ large enough, but usually a larger $\mathtt{C}$ will lead to longer running time. To handle this issue, (for separable data) we run \texttt{svm.SVC} with $\mathtt{C}$ increases from 1, until the training error attains zero.} Tolerance for both are set to be $10^{-8}$.

As discussed in \cref{subsec:LR_vs_SVM}, logistic regression and SVM are ``equivalent'' on separable dataset. Indeed, theoretically and empirically, there are advantages and disadvantages to both algorithms, summarized in \cref{tab:LR_vs_SVM}. In particular, SVM is preferred for theoretical analysis and precise 2-GMM simulation, while logistic regression is preferred for large scale real data analysis.
\begin{table}[h!]
\centering
\begin{tabular}{rll}
    \hline
                 & \textbf{Pros} & \textbf{Cons} \\
    \hline
    \multirow{2}{*}{Logistic regression \eqref{eq:logistic}} 
    & robust to near-separability   & infinite-norm solution \\
    & computationally efficient       & slow convergence \\
    \hline
    \multirow{2}{*}{SVM \eqref{eq:SVM-0}}    
    & well-defined solution     & sensitive to outliers \\
    & support vectors available & quadratic programming \\
    \hline
\end{tabular}
\caption{Comparison of empirical behaviors of logistic regression and SVM on separable data.}
\label{tab:LR_vs_SVM}
\end{table}


\paragraph{Datasets.} We provide the details of real data used in our study, including the source, size, and the preprocessing applied.
\begin{itemize}
    \item \textbf{IFNB} \cite{ifnb}: 
    single-cell RNA-seq dataset of peripheral blood mononuclear cells treated with interferon-$\beta$, which has $n=7,451$ cells, $d=2,000$ genes, and $K=13$ categories for cells. The original dataset is available from R package \texttt{SeuratData} (\url{https://github.com/satijalab/seurat-data}, version 0.2.2.9001) under the name \texttt{ifnb}. The data were preprocessed, normalized, and scaled by following the standard procedures by R package \texttt{Seurat} using functions \texttt{CreateSeuratObject}, \texttt{NormalizeData} and \texttt{ScaleData}.
    
    \item \textbf{CIFAR-10} \cite{KrizhevskyCIFAR102009}: the original dataset consists of 60,000 color images of size $32 \times 32$ in $K=10$ classes, with 6,000 images per class. There are 50,000 training images and 10,000 test images. It is available at \url{https://www.cs.toronto.edu/~kriz/cifar.html}. We followed the simple data augmentation in \cite{resnet, cao2019learning} for the training images: 4 pixels are padded on each side, and a $32 \times 32$ crop is randomly sampled from the padded image or its horizontal flip. Normalization is applied for both training and test images.
    
    \item \textbf{IMDb} \cite{IMDB}: the dataset consists of 50,000 movie reviews for binary sentiment classification ($K=2$), with the positive and negative reviews evenly distributed. There are 25,000 training texts and 25,000 test texts. The data can be found at \url{https://huggingface.co/datasets/stanfordnlp/imdb}. The maximum length in number of tokens for inputs was set as 512.
\end{itemize}

\paragraph{Pretrained models.} We downloaded and used pretrained models from Huggingface.
\begin{itemize}
    \item \textbf{ResNet-18} \cite{resnet}: 18-layer, 512-dim, 11.2M parameters, convolutional neural network (CNN), pretrained on CIFAR-10 training set (50,000 images). The pretrained model is downloaded from \url{https://huggingface.co/edadaltocg/resnet18_cifar10}. Notice that for extracting features, we manually removed the last fully-connected layer.
    \item \textbf{BERT} \cite{BERT}: 12-layer, 12-head, 768-dim, 110M parameters, encoder-only transformer, masked prediction, with absolute positional encoding at the input layer, pretrained on BooksCorpus (800M words) and English Wikipedia (2,500M words). The pretrained model is downloaded from \url{https://huggingface.co/google-bert/bert-base-uncased}.

    We also used a fine-tuned version of BERT (same structure as above) on IMDb dataset, which can be found at \url{https://huggingface.co/fabriceyhc/bert-base-uncased-imdb}.
\end{itemize}

\paragraph{Data splitting} For GMM simulated data and IFNB single-cell data, we split the whole dataset into training and test sets in equal proportions. For CIFAR-10 image data and IMDb movie review data, notice that we used the ResNet-18 and BERT model which are pretrained/fine-tuned on the training set of CIFAR-10 and IMDb, respectively. To avoid reusing the data when training the last fully-connected layer (i.e., logistic regression), we split the test set of CIFAR-10 and IMDb into a ``training subset'' and a ``test subset'' in equal proportions. We used this ``training subset'' for logistic regression training and ``test subset'' for evaluation.



\subsection{GMM simulation}

Figures~\ref{fig:GMM_main} and \ref{fig:Err_pi}---\ref{fig:reliability_GMM} are all generated from 2-GMM simulations. By rotational invariance, we may take $\bmu = (\mu, 0, \ldots, 0)^\top \in \R^d$ for some $\mu > 0$. Both minority and majority test errors are calculated on an independent balanced test set, to ensure the accuracy of estimating $\Err_+$.

\subsubsection{Rebalancing margin}

For \textbf{(ii) high imbalance regime}, we provide a simulation study by generating data from a 2-GMM model. More precisely, given $a, b, c > 0$, let
\begin{equation}\label{eq:abc_mod}
    \pi = C_\pi d^{-a},  \quad  \|\bmu\|_2^2 = C_{\mu} d^b,  \quad  n = C_n d^{c+1},
\end{equation}
for some fixed constant $C_\pi = 1, C_\mu = 0.75, C_n = 1$, where $\bmu = (\mu, 0, \ldots, 0)^\top \in \R^d$ and $\mu = \sqrt{C_{\mu} d^b}$. In the experiment, we fix $b = 0.3$, $c = 0.1$, and $d =2 000$ large enough to ensure data separability, while we change the value of $a$. For each tuple $(a, b, c)$, we compute the parameters $\pi, \bmu, n$ as per \cref{eq:abc_mod}, and generate training sets and test sets according to 2-GMM \cref{model}. 

\subsubsection{Calibration}

The confidence reliability diagram \cref{fig:reliability_GMM} is created by partitioning $(0, 1]$ into $M$ interval bins $I_m := (\frac{m-1}{M}, \frac{m}{M}]$, $m \in [M]$, and calculating the average accuracy of each bin. Let $\hat p(\xx_i)$ be the confidence of the $i$-th test point ($i \in [n]$), and denote $\mathcal{B}_m := \{ i \in [n]: \hat p(\xx_i) \in I_m \}$ be the set of indices whose confidence falls into each bin. Then by our definition of confidence and the symmetry of binary classification, the accuracy and confidence of $\mathcal{B}_m$ can be estimated by
\begin{equation*}
    \hat{\mathrm{acc}}(\mathcal{B}_m) = \frac{1}{\abs{\mathcal{B}_m}} \sum_{i \in \mathcal{B}_m} \ind\{ y_i = 1 \},
    \qquad
    \hat{\mathrm{conf}}(\mathcal{B}_m) = \frac{1}{\abs{\mathcal{B}_m}} \sum_{i \in \mathcal{B}_m} \hat p(\xx_i).
\end{equation*}
We can also obtain a binning-based estimator of calibration error \cref{eq:CalErr} by using above quantities:
\begin{equation*}
    \hat{\mathrm{CalErr}} := \sum_{m=1}^M \frac{\abs{\mathcal{B}_m}}{n} \left( \hat{\mathrm{acc}}(\mathcal{B}_m) - \hat{\mathrm{conf}}(\mathcal{B}_m) \right)^2.
\end{equation*}
This is a variant of the prominent estimator called expected calibration error (ECE) \cite{guo2017calibration}.

The confidence reliability diagrams for additional 2-GMM simulations and IMDb movie review dataset are shown in Figures~\ref{fig:reliab_diag_mu=2}---\ref{fig:reliab_imdb}. These plots confirm a similar trend: miscalibration is getting worse when data becomes increasingly imbalanced (i.e., as $\pi$ decreases).

\begin{figure}[t]
    \centering
    \includegraphics[width=1\textwidth]{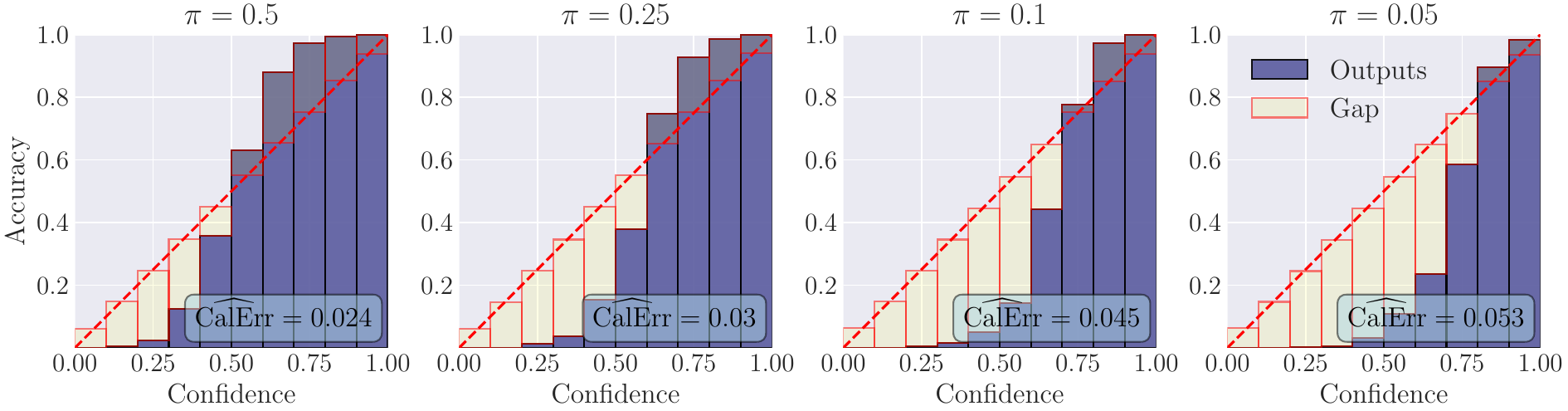}
    \caption{
    Reliability diagram for 2-GMM simulation ($\norm{\bmu}_2=2$, $n=1,\! 000$, $d=100$)
    }\label{fig:reliab_diag_mu=2}
\end{figure}
\begin{figure}[t]
    \centering
    \includegraphics[width=1\textwidth]{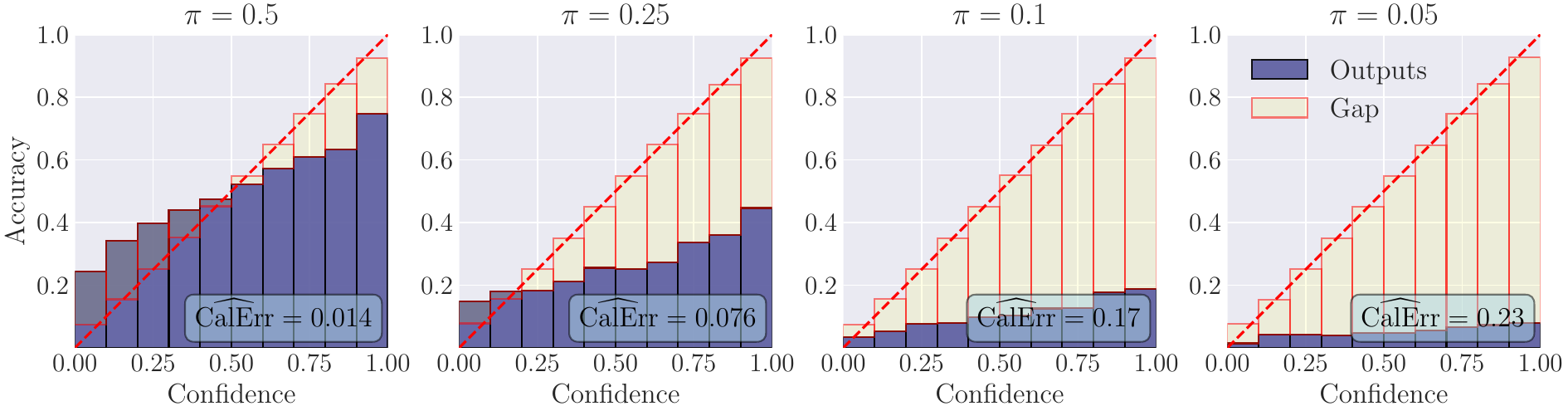}
    \caption{
    Reliability diagram for 2-GMM simulation ($\norm{\bmu}_2=0.5$, $n=1,\! 000$, $d=500$)
    }\label{fig:reliab_diag_mu=0.5}
\end{figure}
\begin{figure}[t]
    \centering
\includegraphics[width=1\textwidth]{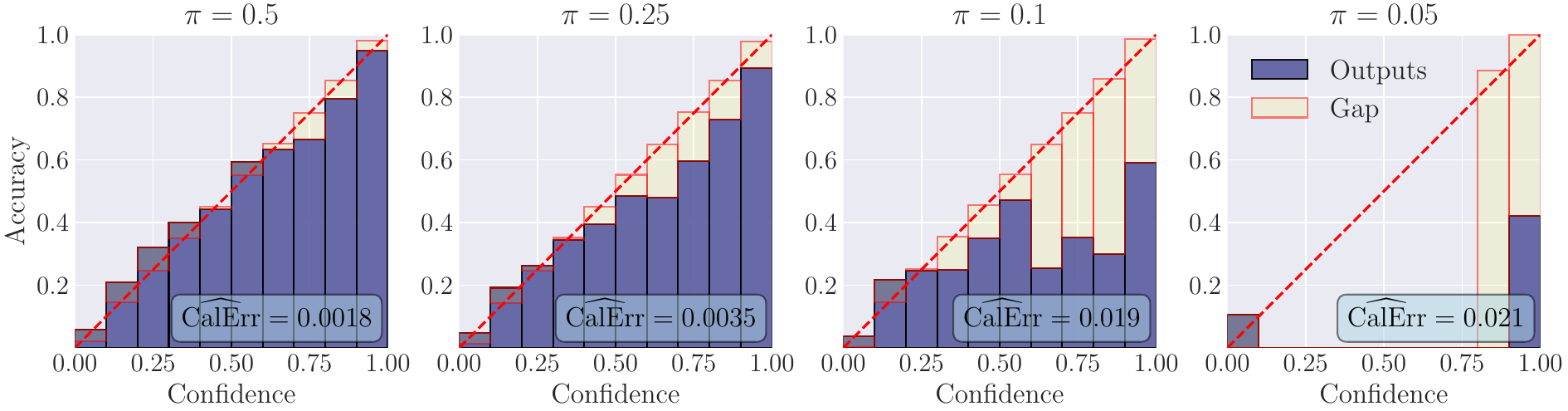}
    \caption{
    Reliability diagram for IMDb dataset preprocessed by BERT base model (110M)
    }\label{fig:reliab_imdb}
\end{figure}


\subsection{Function plot for the proximal operator $\prox_{\lambda \ell}(x)$}
\label{append_subsec_prox}
Recall that 
\begin{equation*}
    \prox_{\lambda \ell}(x) = \argmin_{t \in \R} \left\{ \ell(t) +  \frac1{2\lambda} (t - x)^2 \right\}.
\end{equation*}
We provide the plot for function $x \mapsto \prox_{\lambda \ell}(x)$, which is the specific form of overfitting effect in logistic regression \cref{eq:logistic} on non-separable data (i.e., $\delta > \delta^*(0)$). The plot is shown in \cref{fig:prox}, where $\ell(t) = \log(1 + e^{-t})$ is the logistic loss, and we choose $\lambda = 1, 5, 100$, and 10,000 for visualization. When $\lambda$ is close to zero, the function $x \mapsto \prox_{\lambda \ell}(x)$ is close to the identity map, which is because $\lim_{\lambda \to 0^+}\prox_{\lambda\ell}(x) = x$ by \cref{lem:prox}\ref{lem:prox(b)}. When $\lambda$ is large, the proximal operator (up to scaling) looks like a smooth approximation of the truncation map $x \mapsto \max\{ \kappa, x\}$ for some $\kappa > 0$. Intuitively, $\prox_{\lambda\ell}(x)$ behaves like minimizing $\ell$ when $\lambda$ is large. Therefore, a large $x$ yields $\prox_{\lambda\ell}(x) \approx x$ since $\ell(x) \approx 0$, and a small $x$ would be ``pushed'' to some $\kappa > 0$, since the logistic loss $\ell(x)$ locally is a smoothing of the hinge-type loss $x \mapsto (a - b x)_+$ for some $a, b > 0$.

According to our proof in \cref{append_sec:nonsep}, the limiting value of $\lambda$ as $n, d \to \infty$, $n/d \to \delta$ (denoted by $\lambda^* (\delta)$) is a decreasing function of the asymptotic aspect ratio $\delta$. Then \cref{fig:prox} graphically illustrates the effect of high-dimensionality on overfitting. When $n/d \to \delta$ is large, then $\lambda^*(\delta)$ is small and ELD $\approx$ TLD, and overfitting is negligible. In particular, this is the case for the classical setting where $d$ is fixed and $\delta = \infty$. When $\delta$ is moderate, the ELD is somewhat shrunken compared to TLD. When $\delta \downarrow \delta^*(0)$, approaching the interpolation threshold, then $\lambda^*(\delta)$ is very large, and the ELD is almost a rectified Gaussian and far away from the TLD.

\begin{figure}[h]
    \centering
    \includegraphics[width=1\textwidth]{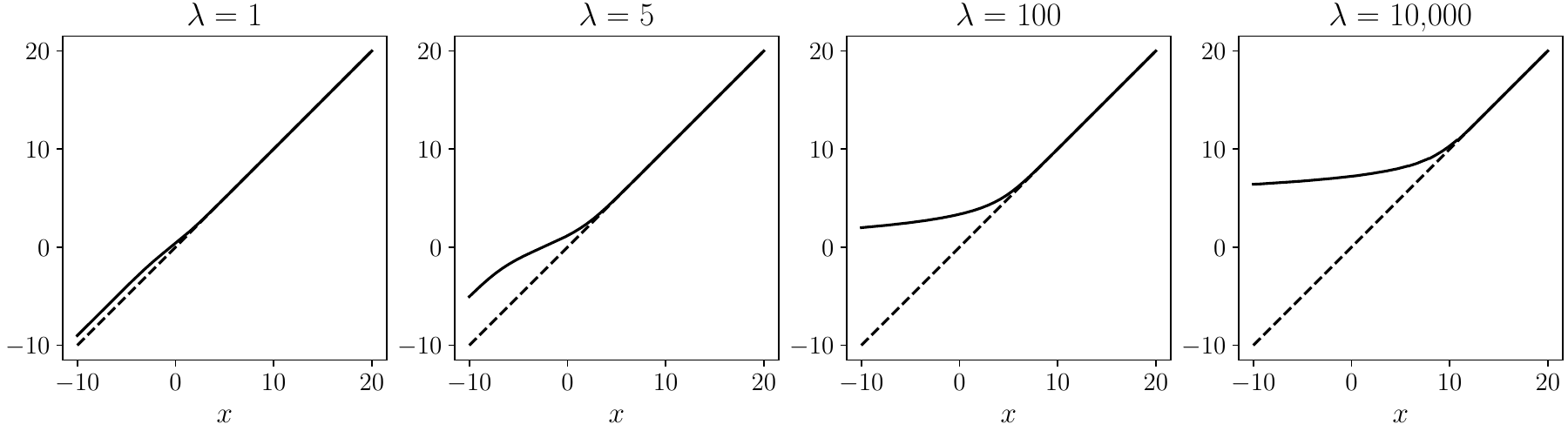}
    \caption{
    \textbf{Function plot for $x \mapsto \prox_{\lambda \ell}(x)$ under different $\lambda$.} The solid curve represents the function $y = \prox_{\lambda \ell}(x)$ and the dashed line represents the identity map $y = x$.
    }
    \label{fig:prox}
\end{figure}

%% file: src/append_prelim.tex
\section{Preliminaries: Proofs for \Cref{sec:preliminary}}
\label{append_sec:prelim}
    


We first introduce some technical adjustments and terms that are used in our proofs.
\paragraph{Well-definedness of margin.}
We define the \emph{margin} of linear classifier $x \mapsto 2\ind\{ f(\xx) > 0 \} -1$ with $f(\xx) = \< \xx, \vbeta \> + \beta_0$ as 
\begin{equation}\label{eq:margin}
	\kappa = \kappa(\vbeta, \beta_0) := \min_{i \in [n]} \tilde{y}_i ( \< \xx_i, \bbeta \> + \beta_0 ),
\end{equation}
which is the objective of margin-rebalanced SVM \cref{eq:SVM-m-reb} and \eqref{eq:SVM}. Note there is a minor caveat about the one-class degenerate case, which is ignored in the main text for simplicity. When $n_+ = 0$ or $n$ (this happens with nonzero probability for any fixed $n$), we have $\kappa(\hat\vbeta, \hat\beta_0) = \infty$. It motivates us to redefine the maximum margin properly.
\begin{defn}\label{def:max-margin}
	The \emph{well-defined maximum margin} is
	\begin{equation}\label{eq:kappa_hat_def}
		\hat\kappa  := 
		\begin{cases} 
			\ \kappa(\hat\vbeta, \hat\beta_0) = \displaystyle\min_{i \in [n]} \tilde{y}_i ( \< \xx_i, \hat\bbeta \> + \hat\beta_0 ) , 
			& \ \text{if} \ 1 \le n_+ \le n - 1, \\
			\ 0    , 
			& \ \text{if} \ n_+ = 0 \text{ or } n. 
		\end{cases}
	\end{equation}
\end{defn}
Therefore, $\hat\kappa$ above is a proper random variable and $\hat\kappa \ge 0$ always holds\footnote{For degenerate case ($n_+ = 0$ or $n$), the dataset is considered as linearly separable.}. Further, $\hat{\kappa} = \kappa(\hat\vbeta, \hat\beta_0)$ with high probability as $n \to \infty$. We will apply similar adjustments to the definition of ELD in \cref{append_sec:sep} and elsewhere, whenever required for the proof.

\paragraph{Support vectors.} Consider the non-degenerate case ($1 \le n_+ \le n - 1$). To study the properties of optimal solution $(\hat\vbeta, \hat\beta_0, \hat\kappa)$ from a non-asymptotic perspective, we inherit the concept of \emph{support vectors} from SVM. Define the \emph{support vector of a linear classifier} $2\ind\{ \< \xx, \vbeta \> + \beta_0 > 0 \} -1$ as the vector(s) $\xx_i$ which attain(s) the smallest (rebalanced) logit margin $\wt y_i(\< \xx_i, \vbeta \> + \beta_0)$ from each class. Namely,
\begin{equation}\label{eq:SV_def}
	\begin{aligned}
		\mathcal{SV}_+ = \mathcal{SV}_+(\vbeta) & :=  \argmin_{i: y_i = +1} \wt y_i(\< \xx_i, \vbeta \> + \beta_0)
	= \argmin_{i: y_i = +1} + \< \xx_i, \vbeta \>, \\
		\mathcal{SV}_- = \mathcal{SV}_-(\vbeta) & :=  \argmin_{i: y_i = -1} \wt y_i(\< \xx_i, \vbeta \> + \beta_0)
	= \argmin_{i: y_i = -1} - \< \xx_i, \vbeta \>, \\
	\end{aligned}
\end{equation}
where $\mathcal{SV}_+, \mathcal{SV}_-$ are sets of (the indices of) \emph{positive} and \emph{negative support vectors}. A key observation from \cref{eq:SV_def} is that support vectors only depend on the data and parameter $\vbeta$, not $\beta_0$ or $\tau$.\footnote{Hence, we can view $\mathcal{SV}_\pm(\vbeta)$ as a mapping from $\R^d$ to the power set of $\{i: y_i = \pm 1\}$.} Let $\mathsf{sv}_+(\vbeta)$ and $\mathsf{sv}_-(\vbeta)$ be any element in $\mathcal{SV}_+(\vbeta)$ and $\mathcal{SV}_-(\vbeta)$, i.e.,
\begin{equation*}
	\mathsf{sv}_+(\vbeta) \in \mathcal{SV}_+(\vbeta),
	\qquad
	\mathsf{sv}_-(\vbeta) \in \mathcal{SV}_-(\vbeta),
\end{equation*} 
which are (the indices of) arbitrary positive and negative support vectors (only depends on $(\XX, \yy)$ and $\vbeta$). In particular, $\mathsf{sv}_+(\hat\vbeta) \in \mathcal{SV}_+(\hat\vbeta)$, $\mathsf{sv}_-(\hat\vbeta) \in \mathcal{SV}_-(\hat\vbeta)$ are support vectors of the max-margin classifier $2 \ind\{\< \xx, \hat\vbeta \> + \hat\beta_0 > 0\} - 1$, which aligns with the definition of support vectors in SVM.

\subsection{Proof of \cref{prop:SVM_tau_relation}}

The lemma below summarizes some important properties of the max-margin solution \cref{eq:SVM-m-reb} characterized by support vectors, which is a stronger statement than \cref{prop:SVM_tau_relation}.

\begin{lem}\label{lem:indep_tau}
	For non-degenerate case, let $(\hat\vbeta, \hat\beta_0, \hat\kappa)$ be an optimal solution to \cref{eq:SVM-m-reb}. Then
	\begin{enumerate}[label=(\alph*)]
		\item \label{lem:indep_tau(a)}
            $\hat\vbeta$ does NOT depend on $\tau$, and
		\begin{equation}\label{eq:kappa_hat}
			(\tau + 1) \hat\kappa = 
			\max_{\vbeta \in \S^{d-1}} \< \xx_{\mathsf{sv}_+(\vbeta)} - \xx_{\mathsf{sv}_-(\vbeta)}, \vbeta \>
			= 
			\< \xx_{\mathsf{sv}_+(\hat\vbeta)} - \xx_{\mathsf{sv}_-(\hat\vbeta)}, \hat\vbeta \>.
		\end{equation}
		\item \label{lem:indep_tau(b)}
            $\hat\beta_0$ depends on $\tau$ by
		\begin{equation}\label{eq:beta0_hat}
			\hat\beta_0 = -\frac{\tau \< \xx_{\mathsf{sv}_-(\hat\vbeta)}, \hat\vbeta \> + \< \xx_{\mathsf{sv}_+(\hat\vbeta)}, \hat\vbeta \>}{\tau + 1}.
		\end{equation}
		\item \label{lem:indep_tau(c)}
            If the data are linearly separable, then $(\hat\vbeta, \hat\beta_0)$ must be unique.
	\end{enumerate}
\end{lem}

\begin{proof}
	For any feasible solution $(\vbeta, \beta_0)$ of \cref{eq:SVM}, we denote the \emph{positive} and \emph{negative margin} of the classifier $\xx \mapsto 2\ind\{\< \xx, \vbeta \> + \beta_0 > 0 \}  -1$ as
	\begin{equation}\label{eq:margin_pm}
		\begin{aligned}
			\kappa_{+}(\vbeta, \beta_0)
			& := \min_{i: y_i = +1}\wt y_i(\< \xx_i, \vbeta \> + \beta_0)
			= \tau^{-1}(\< \xx_{\mathsf{sv}_+(\vbeta)}, \vbeta \> + \beta_0),  
			\\
			\kappa_{-}(\vbeta, \beta_0)
			& := \min_{i: y_i = -1}\wt y_i(\< \xx_i, \vbeta \> + \beta_0)
			= \mathmakebox[\widthof{$\tau^{-1}$}][r]{-\,} (\< \xx_{\mathsf{sv}_-(\vbeta)}, \vbeta \> + \beta_0).
		\end{aligned}
	\end{equation}
    According to \cref{eq:SVM}, we have
	\begin{equation*}
		\hat\vbeta = \argmax_{\vbeta \in \S^{d-1}} \kappa(\vbeta, \check\beta_0(\vbeta)) = \argmax_{\vbeta \in \S^{d-1}} \min_{i \in [n]} \tilde{y}_i ( \langle \xx_i, \bbeta \rangle + \check\beta_0(\vbeta) ),
	\end{equation*}
	where
	\begin{equation}
		\label{eq:beta0_optim}
		\begin{aligned}
			\check\beta_0(\vbeta) 
		: \! & = \argmax_{\beta_0 \in \R} \kappa(\vbeta, \beta_0)
		= \argmax_{\beta_0 \in \R} \min\limits_{i \in [n]} \tilde{y}_i ( \langle \xx_i, \bbeta \rangle + \beta_0 ) \\
		& = \argmax_{\beta_0 \in \R} \Bigl\{ \min_{i: y_i = +1}\wt y_i(\< \xx_i, \vbeta \> + \beta_0), \min_{i: y_i = -1}\wt y_i(\< \xx_i, \vbeta \> + \beta_0) \Bigr\}
		\\
		& = \argmax_{\beta_0 \in \R} \, \min\left\{ 
			\kappa_{+}(\vbeta, \beta_0),
			\kappa_{-}(\vbeta, \beta_0)
		 \right\}.
		\end{aligned}
	\end{equation}
	Here, $\check\beta_0(\vbeta)$ can be viewed as the optimal intercept for a linear classifier with slope given by $\vbeta$. 
	
	\vspace{0.5\baselineskip}
	\noindent
	\textbf{\ref{lem:indep_tau(b)}:}
	As defined in \cref{eq:margin_pm}, note $\min\{ 
		\kappa_{+}(\vbeta, \beta_0),
		\kappa_{-}(\vbeta, \beta_0) \}$ 
	is a piecewise linear concave function of $\beta_0$. Therefore, $\check\beta_0(\vbeta)$ must satisfy the \emph{margin-balancing} condition\footnote{
		As we have seen, the margin-balancing condition holds regardless of the sign of margin. It holds even if the data is not linearly separable.
	}, i.e., 
	 \begin{equation}\label{eq:margin-bal}
		\kappa_{+}(\vbeta, \check\beta_0(\vbeta)) = \kappa_{-}(\vbeta, \check\beta_0(\vbeta))
		= \kappa(\vbeta, \check\beta_0(\vbeta)).
	 \end{equation}
	In particular, recall that $\check\beta_0(\hat\vbeta) = \hat\beta_0$, then $\kappa_{+}(\hat\vbeta, \hat\beta_0) = \kappa_{-}(\hat\vbeta, \hat\beta_0)$. Substitute this back to \cref{eq:margin_pm} deduce
	\begin{equation*}
		\tau^{-1}(\< \xx_{\mathsf{sv}_+(\hat\vbeta)}, \hat\vbeta \> + \hat\beta_0)  
			= -(\< \xx_{\mathsf{sv}_-(\hat\vbeta)}, \hat\vbeta \> + \hat\beta_0),
	\end{equation*}
	which uniquely solves the expression for $\hat\beta_0$ in \cref{eq:beta0_hat}. This concludes the proof of part \ref{lem:indep_tau(b)}.

	\vspace{0.5\baselineskip}
	\noindent
	\textbf{\ref{lem:indep_tau(a)}:}
	Next, we show that $\hat\vbeta$ does not depend on $\tau$. According to \cref{eq:margin-bal} and \cref{eq:margin_pm},
		\begin{align*}
			\hat\vbeta & = \argmax_{\vbeta \in \S^{d-1}} \kappa(\vbeta, \check\beta_0(\vbeta))  \\
		& =
		\argmax_{\vbeta \in \S^{d-1}} \frac{\tau \kappa_{+}(\vbeta, \check\beta_0(\vbeta)) +
		\kappa_{-}(\vbeta, \check\beta_0(\vbeta))}{\tau + 1} \\
		& = 
		\argmax_{\vbeta \in \S^{d-1}} \frac{\< \xx_{\mathsf{sv}_+(\vbeta)}, 
		\vbeta \> - \< \xx_{\mathsf{sv}_-(\vbeta)}, \vbeta \>}{\tau+1}
		= \argmax_{\vbeta \in \S^{d-1}} \< \xx_{\mathsf{sv}_+(\vbeta)} - \xx_{\mathsf{sv}_+(\vbeta)}, \vbeta \>,
		\end{align*}
	where $\< \xx_{\mathsf{sv}_+(\vbeta)} - \xx_{\mathsf{sv}_+(\vbeta)}, \vbeta \>$ only depends on $\vbeta$ and $(\XX, \yy)$ by definition. Hence, it deduces
	\begin{equation*}
		\begin{aligned}
			\hat\kappa & = \kappa(\hat\vbeta, \check\beta_0(\hat\vbeta)) 
		=
		\frac{\tau \kappa_{+}(\hat\vbeta, \check\beta_0(\hat\vbeta)) +
		\kappa_{-}(\hat\vbeta, \check\beta_0(\hat\vbeta))}{\tau + 1}
		=
		\frac{\< \xx_{\mathsf{sv}_+(\hat\vbeta)}, 
		\hat\vbeta \> - \< \xx_{\mathsf{sv}_-(\hat\vbeta)}, \hat\vbeta \>}{\tau+1}.
		\end{aligned}
	\end{equation*}
        This concludes the proof of part \ref{lem:indep_tau(a)}.

	\vspace{0.5\baselineskip}
	\noindent
	\textbf{\ref{lem:indep_tau(c)}:}
	Since \cref{eq:SVM-1} is a convex optimization problem with objective function $\norm{\bw}_2^2$, which is strictly convex in $\bw$, by equivalence between \cref{eq:SVM-m-reb}, \eqref{eq:SVM} and \eqref{eq:SVM-1}, we know that $\hat\bw$ and $\hat\vbeta = \hat\bw/\|\hat\bw\|_2$ must be unique. And by \ref{lem:indep_tau(a)}, $\hat\beta_0$ is also unique. This concludes the proof of part \ref{lem:indep_tau(c)}.
\end{proof}

Notice that \cref{lem:indep_tau} will also be used in the proof of \cref{lem:theta_hat_z} for the high imbalance regime. Below, we show that \cref{prop:SVM_tau_relation} is a direct consequence of \cref{lem:indep_tau}.

\begin{proof}[\textbf{Proof of \cref{prop:SVM_tau_relation}}]
We only show the relation on $\hat\kappa(\tau)$ and $\hat\beta_0(\tau)$, while the other results are simply restatements of \cref{lem:indep_tau}. According to \cref{eq:kappa_hat}, for any $\tau$, $(\tau + 1) \hat\kappa(\tau)$ equals a quantity which does not depend on $\tau$. Plugging in $\tau = 1$, we get $(\tau + 1) \hat\kappa(\tau) = 2 \hat\kappa(1)$.

Combining \cref{eq:kappa_hat} and \eqref{eq:beta0_hat}, we can solve
\begin{equation*}
    \< \xx_{\mathsf{sv}_+(\hat\vbeta)}, \hat\vbeta \> = \tau \hat\kappa(\tau) - \hat\beta_0(\tau), 
    \qquad
    \< \xx_{\mathsf{sv}_-(\hat\vbeta)}, \hat\vbeta \> = - \hat\kappa(\tau) - \hat\beta_0(\tau).
\end{equation*}
Notice the above holds for any $\tau > 0$. Taking $\tau = 1$ and substituting it into \cref{eq:beta0_hat}, we get
\begin{equation*}
\hat\beta_0(\tau) = -\frac{\tau \bigl(- \hat\kappa(1) - \hat\beta_0(1)\bigr) + \bigl(\hat\kappa(1) - \hat\beta_0(1) \bigr)}{\tau + 1} = \hat\beta_{0}(1) + \frac{\tau - 1}{\tau + 1} \hat\kappa(1).    
\end{equation*}
This completes the proof.
\end{proof}

\subsection{Proof of \cref{prop:explicit_bias}}
\begin{proof}[\textbf{Proof of \cref{prop:explicit_bias}}]
    Our argument follows the proof of \cite[Theorem 2.1]{Soudry_implicit_bias}. Assume that $\vbeta^*$ is a limit point of $\hat \vbeta_\lambda/\| \hat \vbeta_\lambda \|_2$ as $\lambda \to 0^+$, with $\| \vbeta^* \|_2 = 1$. The existence of $\vbeta^*$ is guaranteed by boundedness. Let $\beta_0^* := \limsup_{\lambda \to 0^+} \hat\beta_{0, \lambda}/\| \hat \vbeta_\lambda \|_2$. Now, suppose the max-margin classifier given by  $(\hat\vbeta, \hat\beta_0)$ (with $\| \hat\vbeta \|_2 = 1$) has a larger margin than $(\vbeta^*, \beta_0^*)$, that is,
    \begin{equation*}
        \kappa(\vbeta^*, \beta_0^*) = \min_{i \in [n]} y_i ( \< \xx_i, \vbeta^* \> + \beta_0^* )
        <
        \kappa(\hat\vbeta, \hat\beta_0) = \min_{i \in [n]} y_i ( \< \xx_i, \hat\vbeta \> + \hat\beta_0 ).
    \end{equation*}
    By continuity of $\kappa(\vbeta, \beta_0)$, there exists some open neighborhood of $(\vbeta^*, \beta_0^*)$:
    \begin{equation*}
        \mathcal{N}_{\vbeta^*, \beta_0^*} := \left\{
        \vbeta \in \R^d, \beta_0 \in R: \|\vbeta\|_2 = 1, \| \vbeta - \vbeta^* \|_2^2 + |\beta_0 - \beta_0^*|^2 < \delta^2 
        \right\}
    \end{equation*}
    and an $\varepsilon > 0$, such that
    \begin{equation*}
        \kappa(\vbeta, \beta_0) = \min_{i \in [n]} y_i ( \< \xx_i, \vbeta \> + \beta_0 ) < \kappa(\hat\vbeta, \hat\beta_0) - \varepsilon, \qquad \forall\, (\vbeta, \beta_0) \in \mathcal{N}_{\vbeta^*, \beta_0^*}.
    \end{equation*}
    Since $\ell$ is rapidly varying, now by \cite[Lemma 2.3]{Soudry_implicit_bias} we know that there exists some constant $T > 0$ (depends on $\kappa(\hat\vbeta, \hat\beta_0)$ and $\varepsilon$), such that
    \begin{equation*}
        \sum_{i=1}^n \ell\bigl(y_i ( \< \xx_i, t \hat\vbeta \> + t \hat\beta_0 ) \bigr)
        <
        \sum_{i=1}^n \ell\bigl(y_i ( \< \xx_i, t \vbeta \> + t \beta_0 ) \bigr)
        , \qquad \forall\, t > T , \  (\vbeta, \beta_0) \in \mathcal{N}_{\vbeta^*, \beta_0^*},
    \end{equation*}
    which implies $(t \hat\vbeta, t \hat\beta_0)$ has a smaller loss \cref{eq:logistic-l2} than $(t \vbeta, t \beta_0)$. This indicates that $\vbeta^*$ cannot be a limit point of $\hat \vbeta_\lambda/\| \hat \vbeta_\lambda \|_2$, which is a contradiction. Hence we must have $\kappa(\vbeta^*, \beta_0^*) = \kappa(\hat\vbeta, \hat\beta_0)$. Replacing $\limsup$ by $\liminf$ in the definition of $\beta_0^*$ gives the same conclusion. Then we complete the proof by noticing the max-margin solution is unique on separable data by \cref{lem:indep_tau}\ref{lem:indep_tau(c)}.
\end{proof}

%% file: src/append_overparam_prf.tex
\section{Logit distribution for separable data: Proofs for \cref{sec:logit_SVM}}
\label{append_sec:sep}

\subsection{Proof of \cref{thm:SVM_main}}
\label{append_subsec:sep}

Recall that the margin-rebalanced SVM can be rewritten as
\begin{equation}
    \label{eq:over_max-margin}
    \begin{array}{rl}
    \maximize\limits_{\bbeta \in \R^d, \, \beta_0, \kappa \in \R} & \kappa, \\
    \text{subject to} & \tilde{y}_i ( \< \xx_i, \bbeta \> + \beta_0 ) \ge \kappa,
	\quad \forall\, i \in [n], \\
	&  \norm{\bbeta}_2 \le 1.
    \end{array}
\end{equation}
Let $(\hat \vbeta_n, \hat \beta_{0, n})$ be an optimal solution and $\hat\kappa_n = \ind_{1 \le n_+ \le n - 1} \kappa(\hat \vbeta_n, \hat \beta_{0, n})$ be the well-defined maximum margin as per \cref{def:max-margin}. Our goal is to derive exact asymptotics for $(\hat \vbeta_n, \hat \beta_{0, n}, \hat\kappa_n)$. Similar to the development in \cite{montanari2023generalizationerrormaxmarginlinear}, for any positive margin $\kappa > 0$, we define the event
\begin{align*}
        \mathcal{E}_{n, \kappa} & = \bigl\{  \kappa(\hat \vbeta_n, \hat \beta_{0, n})  \ge \kappa \bigr\}  \\
        & =  \left\{ \text{$\exists\, \vbeta \in \R^d$, $\norm{\bbeta}_2 \le 1$, $\beta_0 \in \R$, such that $\tilde{y}_i \big( \langle \xx_i, \bbeta \rangle + \beta_0 \big) \ge \kappa$ for all $i \in [n]$} \right\} \\
        & = \left\{ \text{$\exists\, \vbeta \in \R^d$, $\norm{\bbeta}_2 \le 1$, $\beta_0 \in \R$, such that $\norm{ \left( \kappa \bs_\yy - \yy \odot \XX \vbeta  - \beta_0 \yy \right)_+ }_2 = 0$} \right\},
\end{align*}
where $\bs_\yy = (s(y_1), \dots, s(y_n))^\top$ and $s$ is the function defined in \cref{eq:s_fun}. Therefore, the data $(\XX, \yy)$ is linearly separable if and only if $\mathcal{E}_{n, \kappa}$ holds for some $\kappa > 0$. We would like to determine for which sets of parameters $(\pi, \vmu, \delta, \tau)$ we have
$\P(\mathcal{E}_{n, \kappa}) \to 1$ and for which instead $\P(\mathcal{E}_{n, \kappa}) \to 0$ as $n,d \to \infty$. To this end, we also define
\begin{equation}
    \label{eq:xi_n_kappa}
    \begin{aligned}
        \xi_{n, \kappa} & := \min_{ \substack{ \norm{\vbeta}_2 \le 1 \\ \beta_0 \in \R} } \frac{1}{\sqrt{d}} \norm{ \left( \kappa \bs_\yy - \yy \odot \XX \vbeta  - \beta_0 \yy \right)_+ }_2 \\
        & \phantom{:}\overset{\mathmakebox[0pt][c]{\text{(i)}}}{=} \min_{ \substack{ \norm{\vbeta}_2 \le 1 \\ \beta_0 \in \R} } \max_{ \substack{ \norm{\blambda}_2 \le 1 \\ \blambda \odot \yy \ge \bzero} } \frac{1}{\sqrt{d}} \blambda^\top \left( \kappa \bs_\yy \odot \yy - \XX \vbeta  - \beta_0 \bone \right),
    \end{aligned}
\end{equation}
where (i) is a consequence of Lagrange duality (dual norm) $\norm{(\ba)_+}_2 = \max_{\norm{\blambda}_2 \le 1, \blambda \ge \bzero} \blambda^\top \ba$. Then we established the following equivalence
\begin{equation*}
    \left\{ \xi_{n, \kappa} = 0 \right\} \Longleftrightarrow \mathcal{E}_{n ,\kappa}
    \qquad
    \left\{ \xi_{n, \kappa} > 0 \right\} \Longleftrightarrow \mathcal{E}^c_{n ,\kappa}.
\end{equation*}
Keep in mind that we are only concerned with the sign (positivity) of $\xi_{n, \kappa}$, not its magnitude. As a consequence, we have
\begin{equation*}
    \hat\kappa_n = \ind_{1 \le n_+ \le n - 1}\cdot \sup\{ \kappa \in \R:  \xi_{n,\kappa} = 0 \}.
\end{equation*}
Let $\mathcal{D}_n := \{ n_+ = 0 \text{ or } n \}$ be the event of degeneration for any datasets of size $n$. Clearly $\P(\mathcal{D}_n) = \pi^n + (1 - \pi)^n \to 0$ as $n \to \infty$. Technically, the empirical logit distribution (ELD) in \cref{eq:ELD} is not well-defined on $\mathcal{D}_n$. Similar as \cref{def:max-margin}, we can also redefine it as follows:
\begin{equation}\label{eq:over_ELD_well}
    \hat \nu_{n} := \frac1n \sum_{i=1}^n \delta_{(y_i, \< \xx_i, \hat\vbeta \> + \hat\beta_{0} ) \cdot \ind\{1 \le n_+ \le n - 1\} }.
\end{equation}

We provide an outline for the main parts of the proofs of \cref{thm:SVM_main}\ref{thm:SVM_main_trans}---\ref{thm:SVM_main_mar}, which involves several steps of transforming and simplifying the random variable $\xi_{n, \kappa}$.
\begin{equation*}
\begin{aligned}
    \xi_{n, \kappa}
    \, \xRightarrow[\text{\cref{lem:over_beta0}}]{\textbf{Step 1}} \, 
    \xi'_{n, \kappa, B}
    \, \xRightarrow[\text{\cref{lem:over_CGMT}}]{\textbf{Step 2}} \, 
    \xi'^{(1)}_{n, \kappa, B}
    \, \xRightarrow[\text{\cref{lem:over_ULLN}}]{\textbf{Step 3}} \, 
    \bar\xi'^{(2)}_{\kappa, B}
    \, \Rightarrow \,
    \bar\xi_{\kappa}^{(2)}
    \\
    \xRightarrow[\text{\cref{lem:over_sign}}]{\textbf{Step 4}} \, 
    \bar\xi_{\kappa}^{(3)}
    \, \Rightarrow \,
    \wt\xi_{\kappa}^{(3)}, F_\kappa(\rho, \beta_0)
    \, \xRightarrow[\text{\cref{lem:over_phase_trans}}]{\textbf{Step 5}} \, 
    \delta^*(\kappa), H_\kappa(\rho, \beta_0).
\end{aligned}
\left.
\vphantom{\begin{matrix} \dfrac12 \\ \dfrac12 \end{matrix}}
\right\} \text{\scriptsize{\cref{lem:over_mar_conp}}}
\end{equation*}

\paragraph{Step 1: Boundedness of the intercept (from $\xi_{n, \kappa}$ to $\xi'_{n, \kappa, B}$)} 
According to the definition of $\xi_{n, \kappa}$, parameters $\vbeta$ and $\blambda$ are optimized in compact sets, but $\beta_0$ is not. Such non-compactness might cause technical difficulties in the following steps, for example, when applying Gordon's Gaussian comparison inequality and establishing uniform convergence. However, it turns out that $\beta_0$ is asymptotically bounded on the event $\mathcal{E}_{n ,\kappa}$. More precisely, we define
\begin{equation}
    \label{eq:xi'_n_kappa_B}
    \xi'_{n, \kappa, B} := \min_{ \substack{ \norm{\vbeta}_2 \le 1 \\ \abs{\beta_0} \le B } } \max_{ \substack{ \norm{\blambda}_2 \le 1 \\ \blambda \odot \yy \ge 0} } \frac{1}{\sqrt{d}} \blambda^\top \left( \kappa \bs_\yy \odot \yy - \XX \vbeta  - \beta_0 \bone \right),
\end{equation}
where $B = B(\tau, \kappa, \pi, \norm{\bmu}_2, \delta)$ is a sufficiently large constant. Then we can show that $\xi_{n, \kappa}$ and $\xi'_{n, \kappa, B}$ have the same sign with high probability, which enables us to work with $\xi'_{n, \kappa, B}$ instead of $\xi_{n, \kappa}$.

\begin{lem}[Boundedness of $\beta_0$] 
\label{lem:over_beta0}    
There exists some constant $B \in (0, \infty)$ (depends on $\tau, \kappa, \pi, \norm{\bmu}_2, \delta$) such that
    \begin{equation*}
        \lim_{n \to \infty} \abs{ \P\bigl( \xi_{n, \kappa} = 0 \bigr) - \P\bigl(\xi'_{n, \kappa, B} = 0 \bigr) } = 0.
    \end{equation*}
\end{lem}
\noindent
See \cref{subsubsec:over_beta0} for the proof.

\paragraph{Step 2: Reduction via Gaussian comparison (from $\xi'_{n, \kappa, B}$ to $\xi'^{(1)}_{n, \kappa, B}$)} 
According to the expression of $\xi'_{n, \kappa, B}$, it is not hard to see the objective function (of $(\vbeta, \blambda)$) is a bilinear form of the Gaussian random matrix $\XX$. To simplify the bilinear term and make the calculation easier, we will use the convex Gaussian minimax theorem (CGMT, see \Cref{lem:CGMT}), i.e., Gordon's comparison inequality \cite{gordon1985some, thrampoulidis2015regularized}. To do so, we introduce another quantity:
\begin{equation}\label{eq:xi1_n_kappa_B}
    \xi_{n, \kappa, B}'^{(1)} := \min_{ \substack{ \rho^2 + \norm{\vtheta}_2^2 \le 1 \\ \abs{\beta_0} \le B } } \max_{ \substack{ \norm{\blambda}_2 \le 1 \\ \blambda \odot \yy \ge 0} } \frac{1}{\sqrt{d}}  \left(
    \norm{\blambda}_2 \vg^\top \btheta + \norm{\btheta}_2 \vh^\top \blambda + \blambda^\top \bigl( 
        \kappa \bs_\yy \odot \yy - \rho\norm{\bmu}_2 \yy + \rho \vu - \beta_0 \bone
     \bigr)
     \right),
\end{equation}
where $\rho \in \R$, $\vtheta \in \R^{d-1}$ are parameters, $\vg \sim \normal(\bzero, \bI_{d-1})$, $\hh \sim \normal(\bzero, \bI_{n})$, $\uu \sim \normal(\bzero, \bI_{n})$ are independent Gaussian vectors. The following lemma connects $\xi'_{n, \kappa, B}$ with $\xi_{n, \kappa, B}'^{(1)}$.

\begin{lem}[Reduction via CGMT] 
    \label{lem:over_CGMT}    
For any $v \in \R$ and $t \ge 0$,
    \begin{equation*}
        \P\Big( \big|\xi'_{n, \kappa, B} - v  \big| \ge t \Big) 
        \le 
        2 \P\Big( \big| \xi'^{(1)}_{n, \kappa, B} - v \big| \ge t \Big).
    \end{equation*}
\end{lem}
\noindent
See \cref{subsubsec:over_CGMT} for the proof.

\paragraph{Step 3: Dimension reduction (from $\xi'^{(1)}_{n, \kappa, B}$ to $\bar\xi'^{(2)}_{\kappa, B}$)} It turns out that $\xi'^{(1)}_{n, \kappa, B}$ can be further simplified for analytical purposes. We define a new (deterministic) quantity
\begin{equation*}
    \bar\xi'^{(2)}_{\kappa, B} :=  \min_{ \substack{ \rho^2 + r^2 \le 1, r \ge 0 \\  \abs{\beta_0} \le B } }
    -r + \sqrt{\delta} \left( \E\left[ \bigl(  s(Y) \kappa - \rho \norm{\bmu}_2 + \rho G_1 + rG_2 - \beta_0 Y \bigr)_+^2 \right] \right)^{1/2},
\end{equation*}
which is a constrained minimization over only three variables $\rho$, $r$, and $\beta_0$, with random variables $(Y, G_1, G_2) \sim P_y \times \normal(0, 1) \times \normal(0, 1)$. The two quantities of interest can be related via the uniform law of large numbers (ULLN) as shown in the following lemma.

\begin{lem}[ULLN]
\label{lem:over_ULLN}    
As $n,d \to \infty$, we have
    \begin{equation*}
        \xi'^{(1)}_{n, \kappa, B}  \conp \left( \bar\xi'^{(2)}_{\kappa, B} \right)_+.
    \end{equation*}
\end{lem}
\noindent
See \cref{subsubsec:over_ULLN} for the proof.

\paragraph{Step 4: Investigation of the positivity (from $\bar\xi'^{(2)}_{\kappa, B}$ to $\bar\xi^{(3)}_{\kappa}$)}
To further simplify the problem, we define the following quantities that are closely related to $\bar\xi'^{(2)}_{\kappa, B}$:
\begin{align}
        \bar\xi_{\kappa}^{(2)} & :=  \min_{ \substack{ \rho^2 + r^2 \le 1, r \ge 0 \\  \beta_0 \in  \R } }
    -r + \sqrt{\delta} \left( \E\left[ \bigl(  s(Y) \kappa - \rho \norm{\bmu}_2 + \rho G_1 + rG_2 - \beta_0 Y \bigr)_+^2 \right] \right)^{1/2} 
    \notag
    \\
        \bar\xi_{\kappa}^{(3)} & :=  \min_{ \substack{ \rho \in [-1, 1] \\  \beta_0 \in  \R } }
        -\sqrt{1 - \rho^2} + \sqrt{\delta} \left( \E\left[ \bigl(  s(Y) \kappa - \rho \norm{\bmu}_2 + G - \beta_0 Y \bigr)_+^2 \right] \right)^{1/2}.
        \label{eq:xi3_kappa}
\end{align}
Firstly, we argue that $\bar\xi'^{(2)}_{\kappa, B} = \bar\xi_{\kappa}^{(2)}$ for constant $B$ large enough, by noticing the optimal (unique) $\beta_0$ in $\bar\xi_{\kappa}^{(2)}$ is always bounded by some constant (depends on $\tau, \kappa, \pi, \norm{\bmu}_2, \delta$). Secondly, notice $\bar\xi_{\kappa}^{(3)}$ can be viewed as fixing $r = \sqrt{1 - \rho^2}$ in the optimization of $\bar\xi_{\kappa}^{(2)}$, and $G := \rho G_1 + \sqrt{1 - \rho^2} G_2 \sim \normal(0, 1)$. The following lemma shows that the sign won't change from $\bar\xi_{\kappa}^{(2)}$ to $\bar\xi_{\kappa}^{(3)}$.

\begin{lem}[Sign invariance] 
\label{lem:over_sign}    
For any $\kappa > 0$, the following result holds:
    \begin{enumerate}[label=(\alph*)]
        \item $\sign(\bar\xi_{\kappa}^{(2)}) = \sign(\bar\xi_{\kappa}^{(3)})$.
        \item If $\bar\xi_{\kappa}^{(2)} \le 0$, then $\bar\xi_{\kappa}^{(2)} = \bar\xi_{\kappa}^{(3)}$.
    \end{enumerate}
\end{lem}
\noindent
See \cref{subsubsec:over_positive} for the proof.

\paragraph{Step 5: Phase transition and margin convergence}
Note the function $\delta^*: \R \to \R_{\ge 0}$ defined in \cref{eq:sep_functions} is closely related to $\bar\xi_{\kappa}^{(3)}$. Let $\kappa^* := \sup\left\{ \kappa \in \R: \delta^*(\kappa) \ge \delta \right\}$. By combining the results from previous steps, we have the following relation.
\begin{lem}[Phase transition] 
\label{lem:over_phase_trans}
For any $\kappa > 0$, we have
\begin{equation*}
    \begin{aligned}
        \lim_{n \rightarrow \infty} \P\left( \xi_{n, \kappa} = 0 \right) = 1, \qquad & \text{if $\delta \le \delta^*(\kappa)$ (i.e., $\kappa \le \kappa^*$)}, \\
        \lim_{n \rightarrow \infty} \P\left( \xi_{n, \kappa} > 0 \right) = 1, \qquad & \text{if $\delta > \delta^*(\kappa)$ (i.e., $\kappa > \kappa^*$)}.
    \end{aligned}
\end{equation*}
In particular,
\begin{equation*}
    \begin{aligned}
        \lim_{n \rightarrow \infty} \P\left\{ \text{$(\XX, \yy)$ is linearly separable} \right\} = 1, \qquad & \text{if $\delta < \delta^*(0)$}, \\
        \lim_{n \rightarrow \infty} \P\left\{ \text{$(\XX, \yy)$ is not linearly separable} \right\} = 0, \qquad & \text{if $\delta > \delta^*(0)$}.
    \end{aligned}
\end{equation*}
\end{lem}
As a consequence, we can also derive the convergence of margin in probability. Notice that the following result is weaker than $\cL^2$ convergence \cref{thm:SVM_main}\ref{thm:SVM_main_mar}. However, we need this preliminary result for the subsequent proof of ELD convergence in \cref{lem:over_logit_conv}.
\begin{lem}[Margin convergence, in probability]
\label{lem:over_mar_conp}
If $\delta < \delta^*(0)$, we have $\hat\kappa_n \conp \kappa^*$.
\end{lem}
\noindent
See \cref{subsubsec:over_phase} for the proof.

\subsubsection{Step 1 --- Boundedness of the intercept: Proof of \cref{lem:over_beta0}}
\label{subsubsec:over_beta0}
\begin{proof}[\textbf{Proof of \cref{lem:over_beta0}}]
Recall that
\begin{equation*}
    \xi_{n, \kappa} = \min_{ \substack{ \norm{\vbeta}_2 \le 1 \\ \beta_0 \in \R} } \frac{1}{\sqrt{d}} \norm{ \left( \kappa \bs_\yy - \yy \odot \XX \vbeta  - \beta_0 \yy \right)_+ }_2.
\end{equation*}
Let $(\wt\vbeta_n, \wt\beta_{0, n})$ be a minimizer of the function above\footnote{
    In general $(\wt\vbeta_n, \wt\beta_{0, n})$ may not be unique and may not be equal to $(\hat\vbeta_n, \hat\beta_{0, n})$.
}. On the event $\mathcal{D}_n^c \cap \mathcal{E}_{n ,\kappa}$ ($\xi_{n, \kappa} = 0$), we have 
\begin{equation*}
        \bigl\| \bigl( \kappa \bs_\yy - \yy \odot \XX \wt\vbeta_n  -  \wt\beta_{0, n} \yy \bigr)_+  \bigr\|_2 = 0, 
        \quad
        \Longrightarrow
        \quad
        \begin{cases} 
            \ \tau \kappa - \< \xx_i, \wt\vbeta_n \> - \wt\beta_{0, n} \le 0, & \ \text{if} \ y_i = + 1, \\
            \ \mathmakebox[\widthof{$\tau\kappa$}][r]{\kappa} + \< \xx_i, \wt\vbeta_n \> + \wt\beta_{0, n} \le 0,      & \ \text{if} \ y_i = -1.
        \end{cases}
\end{equation*}
Write $\xx_i = y_i \bmu + \zz_i$, where $\zz_i \iidsim \normal(\bzero, \bI_d)$ and $y_i \indep \zz_i$. Then we obtain
\begin{equation*}
    \begin{cases} 
        \ \wt\beta_{0, n} \ge \mathmakebox[\widthof{$-$}][r]{\tau}
        \kappa - \< \bmu, \wt\vbeta_n \> - \< \zz_i, \wt\vbeta_n \> , & \ \text{if} \ y_i = + 1, \\
        \ \wt\beta_{0, n} \le    - \kappa + \< \bmu, \wt\vbeta_n \> - \< \zz_i, \wt\vbeta_n \> ,      & \ \text{if} \ y_i = -1,
    \end{cases}
\end{equation*}
which implies for all $i, j$ such that $y_i = +1, y_j = -1$,
\begin{equation*}
    \begin{aligned}
        | \wt\beta_{0, n} | & \le 
        \bigl| \tau \kappa - \< \bmu, \wt\vbeta_n \> - \< \zz_i, \wt\vbeta_n \> \bigr|
        + 
        \bigl| \kappa - \< \bmu, \wt\vbeta_n \> + \< \zz_j, \wt\vbeta_n \>  \bigr| \\
        & \le (\tau + 1)\kappa + 2\bigl| \< \bmu, \wt\vbeta_n \> \bigr| + 
        \bigl| \< \zz_i, \wt\vbeta_n \> \bigr| + \bigl| \< \zz_j, \wt\vbeta_n \> \bigr|.
    \end{aligned}
\end{equation*}
Using the inequality $(a+b+c)^2 \le 3(a^2 + b^2 + c^2)$, we have
\begin{equation*}
    \begin{aligned}
        | \wt\beta_{0, n} |^2 
        & \le
        3 \left\{  \bigl( (\tau + 1)\kappa + 2\bigl| \< \bmu, \wt\vbeta_n \> \bigr| \bigr)^2
        + 
        \min_{i: y_i = +1} \bigl| \< \zz_i, \wt\vbeta_n \> \bigr|^2 + 
        \min_{j: y_j = -1} \bigl| \< \zz_j, \wt\vbeta_n \> \bigr|^2  \right\} \\
        & \le 
        3 \, \biggl\{  \bigl( (\tau + 1)\kappa + 2\bigl| \< \bmu, \wt\vbeta_n \> \bigr| \bigr)^2
        + 
        \frac{1}{n_+}\sum_{i: y_i = +1} \bigl| \< \zz_i, \wt\vbeta_n \> \bigr|^2 + 
        \frac{1}{n_-}\sum_{j: y_j = -1} \bigl| \< \zz_j, \wt\vbeta_n \> \bigr|^2  \biggr\} \\
        & \overset{\mathmakebox[0pt][c]{\text{(i)}}}{=} 
        3 \, \biggl\{  \bigl( (\tau + 1)\kappa + 2\bigl| \< \bmu, \wt\vbeta_n \> \bigr| \bigr)^2
        + 
        \frac{1}{n_+} \bigl\| \ZZ_+ \wt\vbeta_n \bigr\|_2^2 + 
        \frac{1}{n_-} \bigl\| \ZZ_- \wt\vbeta_n \bigr\|_2^2  \biggr\} \\
        & \overset{\mathmakebox[0pt][c]{\text{(ii)}}}{\le}
        3 \, \biggl\{  \bigl( (\tau + 1)\kappa + 2 \norm{\bmu}_2 \bigr)^2
        + 
        \frac{1}{n_+} \norm{ \ZZ_+ }_{\mathrm{op}}^2 + 
        \frac{1}{n_-} \norm{ \ZZ_- }_{\mathrm{op}}^2  \biggr\} =: \wt B_{0, n},
    \end{aligned}
\end{equation*}
where in (i) we denote $\ZZ_+ \in \R^{n_+ \times d}$ as a Gaussian random matrix with rows $\zz_i$ such that $y_i = +1$, $\ZZ_- \in \R^{n_- \times d}$ with rows $\zz_j$ such that $y_j = +1$, while in (ii) we use Cauchy--Schwarz inequality, the definition of operator norm, and $\| \wt\vbeta_n \|_2 \le 1$. 

~\\
\noindent
Next, we show that $\wt B_{0, n}$ is asymptotically bounded. Notice $\ZZ_+, \ZZ_-$ have i.i.d. standard Gaussian entries. According to the tail bound of Gaussian matrices \cite[Corollary 7.3.3]{vershynin2018high}, for any $t_n \ge 0$ such that $t_n = o(\sqrt{n})$ and some absolute constants $c, C \in (0, \infty)$, we have
\begin{equation*}
    \begin{aligned}
        \wt B_{0, n} 
        & \overset{\mathmakebox[0pt][c]{\text{(i)}}}{\le} 
        3 \, \biggl\{  \bigl( (\tau + 1)\kappa + 2 \norm{\bmu}_2 \bigr)^2
        + 
        \frac{1}{n_+} \bigl(\sqrt{\smash[b]{n_+}} + \sqrt{d} + t_n \bigr)^2 + 
        \frac{1}{n_-} \bigl(\sqrt{\smash[b]{n_-}} + \sqrt{d} + t_n \bigr)^2  \biggr\} \\
        & \overset{\mathmakebox[0pt][c]{\text{(ii)}}}{\le}
        3 \, \biggl\{  \bigl( (\tau + 1)\kappa + 2 \norm{\bmu}_2 \bigr)^2
        + 
        \biggl(C + \frac{1}{\sqrt{\pi \delta}}       \biggr)^2 + 
        \biggl(C + \frac{1}{\sqrt{\smash[b]{(1 - \pi) \delta}}} \biggr)^2  \biggr\}
        =: B_{0} ,
    \end{aligned}
\end{equation*}
where (i) holds with probability as least $1 - 4 \exp(-c t_n^2)$, and (ii) holds with probability one based on the fact that $n_+/n \to \pi$, $n_-/n \to 1 - \pi$ a.s. (by strong law of large numbers), and $n/d \to \delta$ as $n \to \infty$. Notice the upper bound $B_0$ is a constant which depends on $(\tau, \kappa, \pi, \norm{\bmu}_2, \delta)$.
Let $t_n \to \infty$, then we conclude $\wt B_{0, n} \le  B_{0}$ with high probability.

~\\
\noindent
Combining these results, for any $B > \sqrt{B_0}$,
\begin{equation*}
    \left( \{ \xi_{n,\kappa} = 0 \} \cap \mathcal{D}_n^c \cap \{ \wt B_{0, n} \le  B_{0} \} \right) 
    \subseteq
    \left( \{ \xi_{n,\kappa} = 0 \} \cap \mathcal{D}_n^c \cap \{ | \wt\beta_{0, n} | \le B \} \right) 
    \subseteq
    \{ \xi'_{n,\kappa,B} = 0 \}.
\end{equation*}
Therefore, by union bound we have
\begin{equation*}
    \begin{aligned}
        \P\bigl(\xi_{n,\kappa} = 0\bigr) 
        & = \P\Bigl( \{ \xi_{n,\kappa} = 0 \} \cap 
        \bigl( \mathcal{D}_n^c \cap \{ \wt B_{0, n} \le  B_{0} \} \bigr)
         \Bigr)
         +
         \P\Bigl( \{ \xi_{n,\kappa} = 0 \} \cap 
        \bigl( \mathcal{D}_n \cup \{ \wt B_{0, n} >  B_{0} \} \bigr)
         \Bigr)
         \\
        & \le \P\bigl(\xi'_{n,\kappa,B} = 0\bigr) + \P(\mathcal{D}_n) + \P\bigl(\wt B_{0, n} > B_0\bigr).
    \end{aligned}
\end{equation*}
Finally, by noticing $\xi_{n,\kappa} \le \xi'_{n,\kappa,B}$, we conclude
\begin{equation*}
    0 \le \P\bigl(\xi_{n,\kappa} = 0\bigr) - \P\bigl(\xi'_{n,\kappa,B} = 0\bigr) \le 
    \P(\mathcal{D}_n) + \P\bigl(\wt B_{0, n} > B_0\bigr) \to 0,
    \qquad \text{as $n \to \infty$}.
\end{equation*}
This completes the proof.
\end{proof}

\subsubsection{Step 2 --- Reduction via Gaussian comparison: Proof of \cref{lem:over_CGMT}}
\label{subsubsec:over_CGMT}
\begin{proof}[\textbf{Proof of \cref{lem:over_CGMT}}]
Rewrite $\xx_i = y_i \bmu + \zz_i$, where $\zz_i \iidsim \normal(\bzero, \bI_d)$. Note that $y_i \indep \zz_i$. Denote the projection matrices
\begin{equation*}
    \bP_{\vmu} := \frac{1}{\norm{\vmu}_2^2} \vmu \vmu^\top,
    \qquad 
    \bP_{\vmu}^\perp := \bI_d - \frac{1}{\norm{\vmu}_2^2} \vmu \vmu^\top,
\end{equation*}
where $\bP_{\vmu}$ is the orthogonal projection onto $\spann\{ \vmu \}$ and $\bP_{\vmu}^\perp$ is the orthogonal projection onto the orthogonal complement of $\spann\{ \vmu \}$. Then we have the following decomposition:
\begin{equation*}
    \begin{aligned}
        \< \xx_i, \bbeta \> 
    & = y_i \< \bmu, \bbeta \> + \< \zz_i, \bbeta \> 
    = y_i \< \bmu, \bbeta \> + \< \zz_i, \bP_{\vmu} \bbeta \> + 
    \< \zz_i, \bP_{\vmu}^\perp \bbeta \>  \\
    & = y_i \left\< \bbeta, \frac{\bmu}{\norm{\bmu}_2} \right\> \norm{\bmu}_2 
    +  \left\< \bbeta, \frac{\bmu}{\norm{\bmu}_2} \right\> \left\< \zz_i, \frac{\bmu}{\norm{\bmu}_2} \right\>
    + \< \zz_i , \bP_{\bmu}^{\perp} \bbeta \> \\
    & = y_i \rho \norm{\vmu}_2 + \rho u_i + \< \zz_i , \bP_{\bmu}^{\perp} \bbeta \>,
    \end{aligned}
\end{equation*}
where
\begin{equation*}
    \rho := \left\< \bbeta, \frac{\bmu}{\norm{\bmu}_2} \right\>,
    \qquad
    u_i := \left\< \zz_i, \frac{\bmu}{\norm{\bmu}_2} \right\> \sim \normal(0, 1).
\end{equation*}
Let $\bQ \in \R^{n \times (n - 1)}$ be an orthonormal basis for the subspace $\spann\{ \vmu \}^\perp$ ($\bQ^\top \bQ = \bI_{n-1}$). Note that
\begin{equation*}
    \< \zz_i , \bP_{\bmu}^{\perp} \bbeta \>
    = \< \zz_i, \bQ \bQ^\top \bbeta \> 
    = \< \bQ^\top \zz_i,  \bQ^\top \bbeta \> 
    = \< \vg_i , \vtheta \>,
\end{equation*}
where
\begin{equation*}
    \begin{gathered}
        \vg_i := \bQ^\top \zz_i \sim \normal(\bzero, \bI_{d-1}),
    \qquad
    \vg_i \indep u_i,
    \\
    \vtheta := \bQ^\top \bbeta \in \R^{n-1},
    \qquad
    \norm{\vtheta}_2 
    = \sqrt{\norm{\vbeta}^2_2 - \norm{\bP_{\vmu}\vbeta}^2_2}
    \le \sqrt{1 - \rho^2}.
    \end{gathered}
\end{equation*}
We obtain a one-to-one map $\vbeta \leftrightarrow (\rho, \vtheta)$ in the unit ball. Therefore, we can reparametrize
\begin{equation*}
    \< \xx_i, \bbeta \> + \beta_0 \overset{\mathrm{d}}{=} y_i \rho \norm{\vmu}_2 - \rho u_i - \< \vg_i , \vtheta \> + \beta_0,
\end{equation*}
where $\rho^2 + \norm{\vtheta}_2^2 \le 1$, and $\{(y_i, u_i, \vg_i)\}_{i = 1}^n$ are i.i.d., each has joint distribution:
\begin{equation*}
    y_i \indep u_i \indep \vg_i,
    \qquad
    \P(y_i=+1) = 1 - \P(y_i=-1) = \pi,
    \quad
    u_i \sim \normal(0, 1),
    \quad
    \vg_i \sim \normal(\bzero, \bI_{d-1}).
\end{equation*}
Now denote
\begin{equation*}
    \uu = (u_1, \dots, u_n)^\top \in \R^{n},
    \qquad
    \GG = (\vg_1, \dots, \vg_n)^\top \in \R^{n \times (d-1)}.
\end{equation*}
Therefore, $\xi'^{(0)}_{n,\kappa, B} := \xi'_{n,\kappa, B}$ defined in \cref{eq:xi'_n_kappa_B} can be written as
\begin{align*}
        \xi'^{(0)}_{n,\kappa,B}
        & = 
        \min_{ \substack{ \norm{\vbeta}_2 \le 1 \\ \abs{\beta_0} \le B} } \max_{ \substack{ \norm{\blambda}_2 \le 1 \\ \blambda \odot \yy \ge \bzero} } \frac{1}{\sqrt{d}} \blambda^\top \left( \kappa \bs_\yy \odot \yy - \XX \vbeta  - \beta_0 \bone \right) \\
        & \overset{\mathmakebox[0pt][c]{\mathrm{d}}}{=} \min_{ \substack{ \rho^2 + \norm{\vtheta}_2^2 \le 1 \\  \abs{\beta_0} \le B } } \max_{ \substack{ \norm{\blambda}_2 \le 1 \\ \blambda \odot \yy \ge \bzero} } \frac{1}{\sqrt{d}} \blambda^\top \left( \kappa \bs_\yy \odot \yy - \rho \norm{\bmu}_2 \yy + \rho \uu + \GG \btheta  - \beta_0 \bone \right) \\
        & = \min_{ \substack{ \rho^2 + \norm{\vtheta}_2^2 \le 1 \\  \abs{\beta_0} \le B } } \max_{ \substack{ \norm{\blambda}_2 \le 1 \\ \blambda \odot \yy \ge \bzero} } \frac{1}{\sqrt{d}} \left( \blambda^\top \GG \btheta +
        \blambda^\top ( \kappa \bs_\yy \odot \yy - \rho \norm{\bmu}_2 \yy + \rho \uu   - \beta_0 \bone )
         \right).
\end{align*}
On the other hand, recall $\xi_{n,\kappa}^{(1)}$ defined in \cref{eq:xi1_n_kappa_B}:
\begin{equation*}
    \xi'^{(1)}_{n, \kappa, B}
        = \min_{ \substack{ \rho^2 + \norm{\vtheta}_2^2 \le 1 \\  \abs{\beta_0} \le B } } \max_{ \substack{ \norm{\blambda}_2 \le 1 \\ \blambda \odot \yy \ge \bzero} } \frac{1}{\sqrt{d}}  \left(
    \norm{\blambda}_2 \vg^\top \btheta + \norm{\btheta}_2 \vh^\top \blambda + \blambda^\top \bigl( 
        \kappa \bs_\yy \odot \yy - \rho\norm{\bmu}_2 \yy + \rho \vu - \beta_0 \bone
     \bigr)
     \right).
\end{equation*}
Note that both minimization and maximization above are defined over compact and convex constraint sets, and the objective function in $\xi'^{(0)}_{n, \kappa, B}$ is a bilinear in $(\vtheta, \blambda)$ (not $\beta_0$). In addition, $(\yy, \uu)$ is independent of $\GG, (\vg, \hh)$, so we can apply a variant of CGMT (\cref{lem:CGMT}) by conditioning on $(\yy, \uu)$, which yields for any $v \in \R$ and $t \ge 0$:
\begin{align*}
    & \P\left( \xi'^{(0)}_{n, \kappa, B} \le v+t \,|\, \yy, \uu \right) \le 2\, \P\left( \xi'^{(1)}_{n, \kappa, B} \le v+t \,|\, \yy, \uu \right),
    \\
    & \P\left( \xi'^{(0)}_{n, \kappa, B} \ge v-t \,|\, \yy, \uu \right) \le 2\, \P\left( \xi'^{(1)}_{n, \kappa, B} \ge v-t \,|\, \yy, \uu \right).
\end{align*}
Taking expectation over $(\yy, \uu)$ on both sides of the equation gives for any $v \in \R$ and $t \ge 0$:
\begin{equation*}
    \P\left( \xi'^{(0)}_{n, \kappa, B} \le v+t \right) \le 2\, \P\left( \xi'^{(1)}_{n, \kappa, B} \le v+t \right),
    \quad
    \P\left( \xi'^{(0)}_{n, \kappa, B} \ge v-t \right) \le 2\, \P\left( \xi'^{(1)}_{n, \kappa, B} \ge v-t \right),
\end{equation*}
which proves \Cref{lem:over_CGMT}.
\end{proof}

\subsubsection{Step 3 --- Dimension reduction: Proof of \cref{lem:over_ULLN}}
\label{subsubsec:over_ULLN}
\begin{proof}[\textbf{Proof of \cref{lem:over_ULLN}}]
The expression of $\xi'^{(1)}_{n, \kappa, B}$ can be further simplified to
\begin{align*}
        \xi'^{(1)}_{n, \kappa, B}
        & = \min_{ \substack{ \rho^2 + \norm{\vtheta}_2^2 \le 1 \\ \abs{\beta_0} \le B } } \max_{ \substack{ \norm{\blambda}_2 \le 1 \\ \blambda \odot \yy \ge \bzero} } \frac{1}{\sqrt{d}} \left( \norm{\blambda}_2 \vg^\top \btheta +
        \blambda^\top ( \kappa \bs_\yy \odot \yy - \rho \norm{\bmu}_2 \yy + \rho \uu 
        + \norm{\btheta}_2 \vh - \beta_0 \bone )  \right)  \\
        & \overset{\mathmakebox[0pt][c]{\text{(i)}}}{=} 
        \min_{ \substack{ \rho^2 + \norm{\vtheta}_2^2 \le 1 \\ \abs{\beta_0} \le B } } \frac{1}{\sqrt{d}} \left( \vg^\top \btheta + \bigl\| \left( 
            \kappa \bs_\yy - \rho \norm{\bmu}_2 + \rho \uu \odot \yy
        + \norm{\vtheta}_2 \vh \odot \yy - \beta_0 \yy
         \right)_+ \bigr\|_2  \right)_+ \\
        & \overset{\mathmakebox[0pt][c]{\text{(ii)}}}{=}
         \min_{ \substack{ \rho^2 + r^2 \le 1 \\ r \ge 0, \abs{\beta_0} \le B } } \frac{1}{\sqrt{d}} \left( 
            - r \norm{\vg}_2 + \bigl\| \left( 
            \kappa \bs_\yy - \rho \norm{\bmu}_2 + \rho \uu \odot \yy
        + r \vh \odot \yy - \beta_0 \yy
         \right)_+ \bigr\|_2  \right)_+,
\end{align*}
where in (i) we use the fact
\begin{equation*}
    \begin{aligned}
        \max_{\norm{\blambda}_2 \le 1, \blambda \ge \bzero} 
    \left( a \norm{\blambda}_2 + \blambda^\top \mathrm{\bf b} \right)
    & = \max_{r \in [0, 1]} \max_{\norm{\bv}_2 = 1, \bv \ge \bzero} 
    r \bigl( a + \bv^\top \mathrm{\bf b} \bigr)
    = \left( \max_{\norm{\bv}_2 = 1, \bv \ge \bzero} \bigl( a + \bv^\top \mathrm{\bf b} \bigr) \right)_+
    \\
    & = \Bigl( a + \norm{(\mathrm{\bf b})_+}_2 \Bigr)_+,
    \end{aligned}
\end{equation*}
in (ii) we use Cauchy--Schwarz inequality $\vg^\top \btheta \ge - \norm{\vtheta}_2 \norm{\vg}_2$ and denote $r = \norm{\vtheta}_2$. For convenience, we write the parameter space as $\bar\Theta_{B} := \{ (\rho, r, \beta_0): \rho^2 + r^2 \le 1, r \ge 0, \abs{\beta_0} \le B \}$. Now, define
\begin{align*}
        \bar \xi'^{(1)}_{n, \kappa, B} & := \min_{ (\rho, r, \beta_0) \in \bar\Theta_{B} } \frac{1}{\sqrt{d}} \left( 
            - r \norm{\vg}_2 + \bigl\| \left( 
            \kappa \bs_\yy - \rho \norm{\bmu}_2 + \rho \uu \odot \yy
        + r \vh \odot \yy - \beta_0 \yy
         \right)_+ \bigr\|_2  \right) \\
         & \phantom{:}= 
         \min_{ (\rho, r, \beta_0) \in \bar\Theta_{B} }
         \left\{ 
            -r\frac{\norm{\vg}_2}{\sqrt{d}}
            + \sqrt{\frac{n}{d}} \sqrt{\frac{1}{n} \sum_{i=1}^n \bigl( s(y_i) \kappa - \rho \norm{\bmu}_2 + \rho u_i y_i
            + r h_i y_i - \beta_0 y_i \bigr)_+^2 }
         \right\}
         \\
         & \phantom{:} =\mathmakebox[0pt][c]{:} \min_{ (\rho, r, \beta_0) \in \bar\Theta_{B} } f^{(1)}_{n,\kappa}(\rho,r,\beta_0)  ,
\end{align*}
then $\xi'^{(1)}_{n, \kappa, B} = \bigl( \bar \xi'^{(1)}_{n, \kappa, B} \bigr)_+$. Recall that
\begin{align*}
        \bar\xi'^{(2)}_{\kappa, B} & = \min_{ (\rho, r, \beta_0) \in \bar\Theta_{B} }
    -r + \sqrt{\delta} \left( \E\left[ \bigl(  s(Y) \kappa - \rho \norm{\bmu}_2 + \rho Y G_1 + r Y G_2 - \beta_0 Y \bigr)_+^2 \right] \right)^{1/2} \\
    & = \min_{ (\rho, r, \beta_0) \in \bar\Theta_{B} }
    -r + \sqrt{\delta} \left( \E\left[ \bigl(  s(Y) \kappa - \rho \norm{\bmu}_2 + \rho G_1 + r G_2 - \beta_0 Y \bigr)_+^2 \right] \right)^{1/2} \\
    & =\mathmakebox[0pt][c]{:} \min_{ (\rho, r, \beta_0) \in \bar\Theta_{B} } f^{(2)}_{\kappa}(\rho,r,\beta_0),
\end{align*}
where $Y \indep G_1 \indep G_2$, $\P(Y = +1) = 1 - \P(Y = -1) = \pi$, and $G_1, G_2 \sim \normal(0, 1)$.
We also define
\begin{equation*}
    \begin{aligned}
        \phi^{(1)}_{n, \kappa}(\rho, r, \beta_0) & := \frac1n \sum_{i=1}^n \bigl( s(y_i) \kappa - \rho \norm{\bmu}_2 + \rho u_i y_i
        + r h_i y_i - \beta_0 y_i \bigr)_+^2
        = : \E_n\left[ f(Y, G_1, G_2; \rho, r, \beta_0) \right] , \\
        \phi^{(2)}_{\kappa}(\rho, r, \beta_0) & := \E\left[  \bigl(  s(Y) \kappa - \rho \norm{\bmu}_2 + \rho G_1 Y
        + r G_2 Y - \beta_0 Y \bigr)_+^2 \right]
        = : \mathmakebox[\widthof{$\E_n$}][l]{\E}\left[ f(Y, G_1, G_2 ; \rho, r, \beta_0) \right] ,
    \end{aligned}
\end{equation*}
where $\E_n[\cdot]$ denotes the expectation over the empirical distribution of $\{ (y_i, u_i, h_i) \}_{i = 1}^n$. In order to apply the uniform law of large numbers (ULLN), note that 
\begin{itemize}
    \item $\bar\Theta_{B}$ is compact. $(\rho, r, \beta_0) \mapsto f$ is continuous in $\bar\Theta_{B}$ for each $(Y, G_1, G_2)$, and $(Y, G_1, G_2) \mapsto f$ is measurable for each $(\rho, r, \beta_0)$
    \item $\abs{f(Y, G_1, G_2 ; \rho, r, \beta_0)} \le 3 \left( (\kappa\tau + \norm{\vmu}_2 + B)^2 + G_1^2 + G_2^2 \right)$ for all $(\rho, r, \beta_0) \in \bar\Theta_{B}$ and $\E[G_1^2] = \E[G_2^2] = 1 < \infty$.
\end{itemize}
Therefore, by ULLN \cite[Lemma 2.4]{newey1994large}, we have
\begin{align*}
    & \sup_{ (\rho, r, \beta_0) \in \bar\Theta_{B} }
    \abs{ \bigl( \phi^{(1)}_{n, \kappa}(\rho, r, \beta_0) \bigr)^{1/2} - \bigl( \phi^{(2)}_{\kappa}(\rho, r, \beta_0) \bigr)^{1/2} } \\
    \le {} & \sup_{ (\rho, r, \beta_0) \in \bar\Theta_{B} }
    \abs{ \phi^{(1)}_{n, \kappa}(\rho, r, \beta_0) -  \phi^{(2)}_{\kappa}(\rho, r, \beta_0) }^{1/2} = o_{\P}(1),
\end{align*}
where the inequality comes from the fact that $x \mapsto \sqrt{x}$ is $1/2$-Hölder continuous on $[0, \infty)$. Then
\begin{align*}
        & \sup_{ (\rho, r, \beta_0) \in \bar\Theta_{B} } \abs{ f^{(1)}_{n,\kappa}(\rho,r,\beta_0) - f^{(2)}_{\kappa}(\rho,r,\beta_0)} \\
        \le {} &  \sup_{r \in [-1, 1]} \abs{r - r\frac{\norm{\vg}_2}{\sqrt{d}}} + 
        \sup_{ (\rho, r, \beta_0) \in \bar\Theta_{B} } \abs{ \sqrt{\frac{n}{d}} \bigl( \phi^{(1)}_{n, \kappa}(\rho, r, \beta_0) \bigr)^{1/2}
        -  \sqrt{\delta} \bigl( \phi^{(2)}_{\kappa}(\rho, r, \beta_0) \bigr)^{1/2}  }  \\
        \le {} & \abs{1 - \frac{\norm{\vg}_2}{\sqrt{d}}} +  
        \sqrt{\frac{n}{d}}
        \sup_{ (\rho, r, \beta_0) \in \bar\Theta_{B} }
        \abs{ \bigl( \phi^{(1)}_{n, \kappa}(\rho, r, \beta_0) \bigr)^{1/2} - \bigl( \phi^{(2)}_{\kappa}(\rho, r, \beta_0) \bigr)^{1/2} } \\
        {} &  
        + \abs{ \sqrt{\frac{n}{d}} - \sqrt{\delta} } \sup_{ (\rho, r, \beta_0) \in \bar\Theta_{B} }\bigl( \phi^{(2)}_{\kappa}(\rho, r, \beta_0) \bigr)^{1/2} \\
        = {} & o_{\P}(1),
\end{align*}
by using $n/d \to \delta$ and law of large numbers $\norm{\vg}_2^2/(d - 1) \conp 1$. Finally, since the function $x \mapsto (x)_+$ is $1$-Lipschitz, we conclude
\begin{equation*}
    \Bigl| \xi'^{(1)}_{n, \kappa, B} - \bigl( \bar\xi'^{(2)}_{\kappa, B} \bigr)_+  \Bigr|
    \le
    \abs{ \bar\xi'^{(1)}_{n, \kappa, B} - \bar\xi'^{(2)}_{\kappa, B} }
    \le \sup_{ (\rho, r, \beta_0) \in \bar\Theta_{B} } \abs{ f^{(1)}_{n,\kappa}(\rho,r,\beta_0) - f^{(2)}_{\kappa}(\rho,r,\beta_0)} = o_{\P}(1).
\end{equation*}
This completes the proof.
\end{proof}

\subsubsection{Step 4 --- Investigation of the positivity: Proof of \cref{lem:over_sign}}
\label{subsubsec:over_positive}
\begin{proof}[\textbf{Proof of \cref{lem:over_sign}}]
We claim $\bar\xi'^{(2)}_{\kappa, B} = \bar\xi_{\kappa}^{(2)}$ when $B$ is large enough. Recall that
\begin{equation*}
        \bar\xi_{\kappa}^{(2)} =  \min_{ \substack{ \rho^2 + r^2 \le 1, r \ge 0 \\  \beta_0 \in  \R } }
    -r + \sqrt{\delta} \left( \E\left[ \bigl(  s(Y) \kappa - \rho \norm{\bmu}_2 + \rho G_1 + rG_2 - \beta_0 Y \bigr)_+^2 \right] \right)^{1/2}.
\end{equation*}
Let $(\wt\rho, \wt r, \wt\beta_{0})$ be a minimizer above and notice $\wt\rho G_1 + \wt r G_2 \overset{\mathrm{d}}{=} \wt R G$, where $\wt R = \sqrt{\wt\rho^2 + \wt r^2}$ and $G \sim \normal(0, 1)$. Then
\begin{align*}
        \wt\beta_{0} & \in \phantom{:} \argmin_{\beta_0 \in \R} \E\left[ \bigl(  s(Y) \kappa - \wt\rho \norm{\bmu}_2 + \wt R G - \beta_0 Y \bigr)_+^2 \right] \\
        & = \phantom{:} \argmin_{\beta_0 \in \R} \left\{ 
            \pi \E\left[ \bigl( \tau \kappa - \wt\rho \norm{\bmu}_2 + \wt R G - \beta_0 \bigr)_+^2 \right]
            + (1 - \pi) \E\left[ \bigl( \kappa - \wt\rho \norm{\bmu}_2 + \wt R G + \beta_0 \bigr)_+^2 \right] \right\} \\
        & =: \argmin_{\beta_0 \in \R} g_{\wt\rho, \wt r}(\beta_0).
\end{align*}
Notice that $g_{\wt\rho, \wt r}(\beta_0)$ is convex and continuously differentiable, since
\begin{equation*}
    g'_{\wt\rho, \wt r}(\beta_0)
    = -2\pi \E\left[ \bigl( \tau \kappa - \wt\rho \norm{\bmu}_2 + \wt R G - \beta_0 \bigr)_+ \right]
    + 2(1 - \pi) \E\left[ \bigl( \kappa - \wt\rho \norm{\bmu}_2 + \wt R G + \beta_0 \bigr)_+ \right] 
\end{equation*}
is non-decreasing, which is based on the fact that $x \mapsto \E[(G + x)_+]$ is increasing. Then $\wt\beta_{0}$ must satisfy $g'_{\wt\rho, \wt r}(\wt\beta_{0}) = 0$. Since $g'_{\wt\rho, \wt r}(+\infty) = +\infty$, $g'_{\wt\rho, \wt r}(-\infty) = -\infty$, by our construction in the proof of \cref{lem:over_beta0}, we can choose $B$ large enough such that $\bar\xi'^{(2)}_{\kappa, B} = \bar\xi_{\kappa}^{(2)}$.

~\\
\noindent
We can rewrite $\bar\xi_{\kappa}^{(2)}$ as follows by introducing an auxiliary parameter $c$:
\begin{equation*}
    \bar\xi_{\kappa}^{(2)} =  \min_{ \substack{ \rho^2 + r^2 \le 1, r \ge 0, \beta_0 \in  \R, \\
    \rho^2 + r^2 + \beta_0^2 = c^2, c \ge 0
} }
-r + \sqrt{\delta} \left( \E\left[ \bigl(  s(Y) \kappa - \rho \norm{\bmu}_2 + \rho G_1 + rG_2 - \beta_0 Y \bigr)_+^2 \right] \right)^{1/2},
\end{equation*}
and we also define the following quantity
\begin{equation*}
    \begin{aligned}
        \wt\xi_{\kappa}^{(2)} & :=  \min_{ \substack{ \rho^2 + r^2 \le 1, r \ge 0, \beta_0 \in  \R, \\
        \rho^2 + r^2 + \beta_0^2 = c^2, c \ge 0
    } }
        \frac1c \left\{  -r + \sqrt{\delta} \left( \E\left[ \bigl(  s(Y) \kappa - \rho \norm{\bmu}_2 + \rho G_1 + rG_2 - \beta_0 Y \bigr)_+^2 \right] \right)^{1/2} 
        \right\} \\
        & \phantom{:}=  \min_{ \substack{ \rho^2 + r^2 \le 1, r \ge 0, \beta_0 \in  \R, \\
        \rho^2 + r^2 + \beta_0^2 = c^2, c \ge 0
    } }
        -\frac{r}{c} + \sqrt{\delta} \left( \E\left[ \Bigl( s(Y) \frac{\kappa}{c} - \frac{\rho}{c} \norm{\bmu}_2 + \frac{\rho}{c} G_1 + \frac{r}{c} G_2 - \frac{\beta_0}{c} Y \Bigr)_+^2 \right] \right)^{1/2}.
    \end{aligned}
\end{equation*}
Then for any $\kappa > 0$, we have the following observations:
\begin{itemize}
    \item $\sign(\bar\xi_{\kappa}^{(2)}) = \sign(\wt\xi_{\kappa}^{(2)})$. (Their objective functions differ only by a multiplier $c \ge 0$.\footnote{We allow $c = 0$. If $c = 0$, then $\rho = r = \beta_0 = 0$ and the objective value in $\bar\xi_{\kappa}^{(2)}$ is $\sqrt{ \delta(\pi\tau^2\kappa^2 + (1 - \pi)\kappa^2) } > 0$, and the objective value in $\wt\xi_{\kappa}^{(2)}$ is defined as $+\infty$. Both of them are positive.
    })
    \item The minimizer in $\wt\xi_{\kappa}^{(2)}$ must satisfy $\rho^2 + r^2 = 1$. 
    
    Suppose $(\wt\rho, \wt r, \wt\beta_0, \wt c)$ is a minimizer in $\wt\xi_{\kappa}^{(2)}$ such that $\wt\rho^2 + \wt r^2 < 1$. We can increase $(\wt\rho, \wt r, \wt\beta_0, \wt c)$ proportionally, which results in a better solution. That is, define
    \begin{equation*}
        \check\rho := \frac{1}{\sqrt{\wt\rho^2 + \wt r^2}} \wt\rho, \qquad
        \check r := \frac{1}{\sqrt{\wt\rho^2 + \wt r^2}} \wt r,  \qquad
        \check\beta_0 := \frac{1}{\sqrt{\wt\rho^2 + \wt r^2}} \wt\beta_0,  \qquad
        \check c := \frac{1}{\sqrt{\wt\rho^2 + \wt r^2}} \wt c,
    \end{equation*}
    then $(\wt\rho, \wt r, \wt\beta_0, \wt c)$ has a smaller objective value (because $r/c$, $\rho/c$, $\beta_0/c$ all remain unchanged, but $\kappa/c$ decreases since $\check{c} > \wt c$), which contradicts the optimiality of $(\wt\rho, \wt r, \wt\beta_0, \wt c)$.
\end{itemize}
As a consequence, we can simplify
\begin{align*}
        \wt\xi_{\kappa}^{(2)} & =  \min_{ \substack{ \rho \in [-1, 1], \beta_0 \in  \R, \\
        \beta_0^2 = c^2 - 1, c \ge 1
    } }
        \frac1c \left\{  -\sqrt{1 - \rho^2} + \sqrt{\delta} \left( \E\left[ \bigl(  s(Y) \kappa - \rho \norm{\bmu}_2 + \rho G_1 + \sqrt{1 - \rho^2} G_2 - \beta_0 Y \bigr)_+^2 \right] \right)^{1/2} 
        \right\} \\
        & = \min_{ \substack{ \rho \in [-1, 1], \beta_0 \in  \R, \\
        \beta_0^2 = c^2 - 1, c \ge 1
    } }
        \frac1c \left\{  -\sqrt{1 - \rho^2} + \sqrt{\delta} \left( \E\left[ \bigl(  s(Y) \kappa - \rho \norm{\bmu}_2 + G - \beta_0 Y \bigr)_+^2 \right] \right)^{1/2} 
        \right\},
\end{align*}
where $G := \rho G_1 + \sqrt{1 - \rho^2} G_2 \sim \normal(0, 1)$. By the same argument, $\sign(\bar\xi_{\kappa}^{(3)}) = \sign(\wt\xi_{\kappa}^{(2)})$, where
\begin{equation*}
        \bar\xi_{\kappa}^{(3)}  =  \min_{ \substack{ \rho \in [-1, 1] \\  \beta_0 \in  \R } }
        -\sqrt{1 - \rho^2} + \sqrt{\delta} \left( \E\left[ \bigl(  s(Y) \kappa - \rho \norm{\bmu}_2 + G - \beta_0 Y \bigr)_+^2 \right] \right)^{1/2}.
\end{equation*}
Therefore, $\sign(\bar\xi_{\kappa}^{(2)}) = \sign(\bar\xi_{\kappa}^{(3)})$.

~\\
\noindent
In order to show $\bar\xi_{\kappa}^{(2)} = \bar\xi_{\kappa}^{(3)}$ when $\bar\xi_{\kappa}^{(2)} < 0$, we define the objective function of $\bar\xi_{\kappa}^{(2)}$ as
\begin{equation*}
    T_{\kappa}(\rho, r, \beta_0) := - r + \sqrt{\delta} \left( \E\left[ \bigl(  s(Y) \kappa - \rho \norm{\bmu}_2 + \rho G_1 + rG_2 - \beta_0 Y \bigr)_+^2 \right] \right)^{1/2}.
\end{equation*}
Then it suffices to show the minimizer of $T_{\kappa}$ must satisfy $\rho^2 + r^2 = 1$. Again, suppose $(\wt\rho, \wt r, \wt\beta_0)$ is a minimizer of $T_{\kappa}$ such that $\wt\rho^2 + \wt r^2 < 1$. We can increase $(\wt\rho, \wt r, \wt\beta_0)$ proportionally by defining
\begin{equation*}
        \check\rho := \frac{1}{\sqrt{\wt\rho^2 + \wt r^2}} \wt\rho, \qquad
        \check r := \frac{1}{\sqrt{\wt\rho^2 + \wt r^2}} \wt r,  \qquad
        \check\beta_0 := \frac{1}{\sqrt{\wt\rho^2 + \wt r^2}} \wt\beta_0,  \qquad
        \kappa' := \frac{1}{\sqrt{\wt\rho^2 + \wt r^2}} \kappa,
\end{equation*}
then
\begin{equation*}
    0 > \bar\xi_{\kappa}^{(2)} = T_{\kappa}(\wt\rho, \wt r, \wt\beta_0) 
    > 
    \frac{T_{\kappa}(\wt\rho, \wt r, \wt\beta_0) }{\sqrt{\wt\rho^2 + \wt r^2}} 
    =
    T_{\kappa'}(\check\rho, \check r, \check\beta_0) 
    >
    T_{\kappa}(\check\rho, \check r, \check\beta_0),
\end{equation*}
where the last inequality is because $x \mapsto \E[(G + c_1 x + c_2)_+^2]$ strictly increasing for any $c_1 > 0$ and $c_2 \in \R$, and the fact that $\kappa' > \kappa$. Therefore, a contradiction occurs and we complete the proof.
\end{proof}

\subsubsection{Step 5 --- Phase transition and margin convergence: Proofs of \cref{lem:over_phase_trans}, \ref{lem:over_mar_conp}}
\label{subsubsec:over_phase}
\begin{proof}[\textbf{Proof of \cref{lem:over_phase_trans}}]
We define the following two functions:
\begin{equation}
    \label{eq:T_F_}
    \begin{aligned}
        T_\kappa(\rho, \beta_0) & := - \sqrt{1 - \rho^2} + \sqrt{\delta} \left( \E\left[ \bigl(  s(Y) \kappa - \rho \norm{\bmu}_2 + G - \beta_0 Y \bigr)_+^2 \right] \right)^{1/2}, \\
        F_\kappa(\rho, \beta_0) & := -(1 - \rho^2) + \delta  \E\left[ \bigl(  s(Y) \kappa - \rho \norm{\bmu}_2 + G - \beta_0 Y \bigr)_+^2 \right]
        \\
        & \phantom{:} = \pi \delta \E \left[ \bigl( G - \rho \norm{\bmu}_2 - \beta_0 + \kappa \tau \bigr)_+^2 \right]  + (1-\pi) \delta \E \left[ \bigl( G - \rho \norm{\bmu}_2 + \beta_0 + \kappa \bigr)_+^2 \right] + \rho^2 - 1,
    \end{aligned}
\end{equation}
and then
\begin{equation*}
    \bar\xi_{\kappa}^{(3)}  =  \min_{\rho \in [-1, 1] , \beta_0 \in  \R } T_\kappa(\rho, \beta_0),
    \qquad
    \wt\xi_{\kappa}^{(3)}  :=  \min_{\rho \in [-1, 1] , \beta_0 \in  \R } F_\kappa(\rho, \beta_0).
\end{equation*}
Clearly, $\sign(T_\kappa(\rho, \beta_0)) = \sign(F_\kappa(\rho, \beta_0))$ for any $\rho, \beta_0$ and $\sign(\bar\xi_{\kappa}^{(3)}) = \sign(\wt\xi_{\kappa}^{(3)})$. Also recall that
\begin{equation*}
    \delta^*(\kappa) = \max_{ \substack{\rho \in [-1, 1] \\ \beta_0 \in \R } }  H_\kappa(\rho, \beta_0),
    \qquad 
    H_\kappa(\rho, \beta_0) = \frac{1 - \rho^2}{\E\left[ \bigl(  s(Y) \kappa - \rho \norm{\bmu}_2 + G - \beta_0 Y \bigr)_+^2 \right]}.
\end{equation*}
We can see that $H_\kappa(\rho, \beta_0)$ is well-defined since $\E[ (  s(Y) \kappa - \rho \norm{\bmu}_2 + G - \beta_0 Y )_+^2 ]$ is bounded away from zero for any $\rho \in [-1, 1]$ and $\beta_0 \in \R \cup\{ \pm \infty \}$.

\vspace{0.5\baselineskip}
\noindent
Since $x \mapsto \E[(G + c_1 x + c_2)_+^2]$ is continuous and strictly increasing for any $c_1 > 0, c_2 \in \R$, it can be shown that both $\kappa \mapsto \bar\xi_{\kappa}^{(3)}$, $\kappa \mapsto \wt\xi_{\kappa}^{(3)}$ are continuous strictly increasing, and $\delta^*(\kappa)$ is continuous strictly decreasing (by restricting $\beta_0: \abs{\beta_0} \le B$ for some constant $B$ large enough, similar as Step 4, and then use compactness). Therefore, we have the following equivalent definitions of $\kappa^*$:
\begin{equation}\label{eq:kappa_star}
    \begin{aligned}
        \kappa^*  & := \sup\left\{ \kappa \in \R: \delta^*(\kappa) \ge \delta \right\} \\
        & \phantom{:} = \left\{ \kappa \in \R: \delta^*(\kappa) = \delta \right\} 
        = \left\{ \kappa \in \R: \bar\xi_{\kappa}^{(3)} = 0 \right\}
        = \left\{ \kappa \in \R:  \wt\xi_{\kappa}^{(3)} = 0 \right\}.
    \end{aligned}
\end{equation}
Now we can consider the following two regimes, each with a chain of equivalence:
\begin{equation}
    \label{eq:phase_equiv}
    \begin{aligned}
        \delta \le \delta^*(\kappa) 
        \quad   \overset{\mathmakebox[0pt][c]{\text{(i)}}}{\Longleftrightarrow} \quad
        \kappa \le \kappa^*
        \quad & \overset{\mathmakebox[0pt][c]{\text{(i)}}}{\Longleftrightarrow} \quad
        \bar\xi_{\kappa}^{(3)}, \wt\xi_{\kappa}^{(3)} \le 0
        \quad   \overset{\mathmakebox[0pt][c]{\text{(ii)}}}{\Longleftrightarrow} \quad
        \bar\xi_{\kappa}^{(2)} \le 0 
        \\
        & \overset{\mathmakebox[0pt][c]{\text{(iii)}}}{\Longleftrightarrow} \quad
        \xi'_{n, \kappa, B}  \conp \bigl( \bar\xi^{(2)}_{\kappa} \bigr)_+ = 0
        \quad \overset{\mathmakebox[0pt][c]{\text{(iv)}}}{\Longleftrightarrow} \quad
        \P(\xi_{n, \kappa} = 0) \to 1,
        \\
        \delta > \delta^*(\kappa) 
        \quad   \overset{\mathmakebox[0pt][c]{\text{(i)}}}{\Longleftrightarrow} \quad
        \kappa > \kappa^*
        \quad & \overset{\mathmakebox[0pt][c]{\text{(i)}}}{\Longleftrightarrow} \quad
        \bar\xi_{\kappa}^{(3)}, \wt\xi_{\kappa}^{(3)} > 0
        \quad   \overset{\mathmakebox[0pt][c]{\text{(ii)}}}{\Longleftrightarrow} \quad
        \bar\xi_{\kappa}^{(2)} > 0 
        \\
        & \overset{\mathmakebox[0pt][c]{\text{(iii)}}}{\Longleftrightarrow} \quad
        \xi'_{n, \kappa, B}  \conp \bigl( \bar\xi^{(2)}_{\kappa} \bigr)_+ > 0
        \quad \overset{\mathmakebox[0pt][c]{\text{(iv)}}}{\Longleftrightarrow} \quad
        \P(\xi_{n, \kappa} > 0) \to 1,
    \end{aligned}
\end{equation}
where (i) is from \cref{eq:kappa_star}, (ii) is from \cref{lem:over_sign}, (iii) is from \cref{lem:over_CGMT}, \ref{lem:over_ULLN}, and (iv) is from \cref{lem:over_beta0}. Linear separability considers the special case $\kappa = 0$. From definition \cref{eq:xi_n_kappa}, for any $\kappa \le 0$ we have $\xi_{n, \kappa} = 0$ (by taking $\vbeta = \bzero$, $\beta_0 = 0$). Therefore, 
\begin{itemize}
    \item If $\delta < \delta^*(0)$, by \cref{eq:phase_equiv} $\kappa^* > 0$ and $\P(\mathcal{E}_{n, \kappa^*}) = \P(\xi_{n, \kappa^*} = 0)
    \to 1$, which deduces the data is linearly separable with high probability.
    \item If $\delta > \delta^*(0)$, by \cref{eq:phase_equiv} $\kappa^* < 0$ and $\P(\mathcal{E}_{n, \kappa}) = \P(\xi_{n, \kappa} = 0)
    \to 0$ for any $\kappa > 0$ (as $\kappa \mapsto \xi_{n, \kappa}$ is non-decreasing), which implies the data is not linearly separable with high probability.
\end{itemize}
\end{proof}

\begin{proof}[\textbf{Proof of \cref{lem:over_mar_conp}}]
If $\delta < \delta^*(0)$, then $\kappa^* > 0$ and $\bar\xi_{\kappa^*}^{(3)} = 0$. According to \cref{eq:phase_equiv}, for any $\varepsilon > 0$ small enough, we have
\begin{equation*}
    \begin{aligned}
        \bar\xi_{\kappa^* - \varepsilon}^{(3)} < 0
    \quad & \Longrightarrow \quad
    \P(\mathcal{E}_{n, \kappa^* - \varepsilon}) = \P(\xi_{n, \kappa^* - \varepsilon} = 0) \to 1, \\
        \bar\xi_{\kappa^* + \varepsilon}^{(3)} > 0
    \quad & \Longrightarrow \quad
    \P(\mathcal{E}_{n, \kappa^* + \varepsilon}) = \P(\xi_{n, \kappa^* + \varepsilon} = 0) \to 0.
    \end{aligned}
\end{equation*}
Recall that $\hat\kappa_n = \ind_{1 \le n_+ \le n - 1} \sup\{ \kappa \in \R:  \xi_{n,\kappa} = 0 \}$.
By combining these arguments, we can see that $\kappa^* - \varepsilon \le \hat\kappa_n \le \kappa^* + \varepsilon$ holds on the event $\mathcal{D}_n^c$, with high probability. This proves $\hat\kappa_n \conp \kappa^*$.
\end{proof}

\subsubsection{Convergence of ELD and parameters for $\tau = 1$: 
Proofs of \cref{lem:over_logit_conv}, \ref{lem:H_kappa_1}}
\label{subsubsec:over_logit_conv}

In this section, we provide a proof of parameter convergence in \cref{thm:SVM_main}\ref{thm:SVM_main_param} and ELD convergence in \ref{thm:SVM_main_logit} for the special case $\tau = 1$. For convenience of notation, we drop the subscripts and simply write $\hat\rho := \hat\rho_n$, $\hat\beta_{0} := \hat\beta_{0,n}$. 
Recall the ELD (well-defined version, i.e., \cref{eq:over_ELD_well}) and its asymptotics are respectively defined as
\begin{equation*}
    \hat \nu_{n} = \frac1n \sum_{i=1}^n \delta_{(y_i, \< \xx_i, \hat\vbeta \> + \hat\beta_{0} ) \cdot \ind\{\mathcal{D}_n^c\} },
    \qquad 
    \nu_* = \Law\left( Y,
    Y \max\bigl\{ \kappa^*, \rho^*\norm{\vmu}_2 + G + \beta_0^* Y \bigr\}
    \right).
\end{equation*}
Here $(\rho^*, \beta_0^*, \kappa^*)$ is defined as the maximizer of \cref{eq:SVM_asymp_simple}, and obviously $\kappa^*$ also satisfies \cref{eq:kappa_star}. The uniqueness of $(\rho^*, \beta_0^*)$ will be given by \cref{lem:H_kappa_1}. Analogous to the proof of \cite[Theorem 4.6]{montanari2022overparametrizedlineardimensionalityreductions}, by using the theory of projection pursuit therein, we have the following results.
\begin{lem}[ELD and parameter convergence] 
\label{lem:over_logit_conv}
Consider $\tau = 1$. As $n, d \to \infty$, we have
\begin{equation*}
    W_2 \bigl( \hat \nu_{n}, \nu_* \bigr)
    \conp 0.
\end{equation*}
The convergence of $\hat\rho \conp \rho^*$ and $\hat\beta_{0} \conp \beta_0^*$ are followed by continuity and convexity of $H_\kappa$ in \cref{eq:sep_functions}.
\end{lem}
\begin{proof}
Our proof primarily follows the setup in \cite[Section 4.1]{montanari2022overparametrizedlineardimensionalityreductions} and techniques in \cite[Section 4.3]{montanari2022overparametrizedlineardimensionalityreductions}. Recall that we can rewrite $\xx_i = y_i \bmu + \zz_i$, where $\zz_i \iidsim \normal(\bzero, \bI_d)$ and $y_i \indep \zz_i$. Using notation from \cite{montanari2022overparametrizedlineardimensionalityreductions}, $\P(y_i = 1 \,|\, \zz_i) = \varphi(\vmu_0^\top \zz_i)$, where $\vmu_0 = \vmu/\norm{\vmu}_2$ and $\varphi(x) \equiv \pi$ is a constant function. Recall that we reparametrize $\hat\rho = \vmu_0^\top \hat\vbeta$. Now, define random variables with joint distribution
\begin{equation*}
    Y \indep G \indep Z, 
    \quad \P(Y = +1 \,|\, G) = 1 - \P(Y = -1 \,|\, G) = \varphi(G) \equiv \pi, 
    \quad G, Z \sim \normal(0, 1).
\end{equation*}
Let $(Y, G, Z) \indep \hat\vbeta$. According to the definition in \cite[Lemma 4.2]{montanari2022overparametrizedlineardimensionalityreductions}, we have
\begin{equation*}
    \Law\left( Y, \vmu_0^\top \hat\vbeta \cdot G + \sqrt{1 - (\vmu_0^\top \hat\vbeta)^2} \cdot Z \right)
    = \Law\left( Y, \hat\rho G + \sqrt{1 - \hat\rho^2} Z \right)
    = \Law(Y, Z).
\end{equation*}
Therefore, by using \cite[Theorem 4.3]{montanari2022overparametrizedlineardimensionalityreductions}, for any $\varepsilon, \eta > 0$, with high probability we have
\begin{equation*}
    W_2^{(\eta)} \biggl( 
        \frac1n \sum_{i=1}^n \delta_{(y_i, \< \zz_i, \hat\vbeta \>)} , 
        \Law(Y, Z)        
     \biggr) 
     \le \frac{\sqrt{1 - \hat\rho^2}}{\sqrt{\delta}} + \varepsilon,
\end{equation*}
where $W_2^{(\eta)}$ is the $\eta$-constrained $W_2$ distance \cite[Definition 4.1]{montanari2022overparametrizedlineardimensionalityreductions}. Formally, for any $\eta > 0$, the $\eta$-constrained $W_2$ distance between any two probability measures $P$ and $Q$ in $\R^d$ is defined by
\begin{equation*}
    W_2^{(\eta)}(P, Q) := \left( \inf_{\gamma \in \Gamma^{(\eta)}(P, Q) }
    \int_{\R^d \times \R^d} \norm{\xx - \yy}_2^2 \gamma(\d\xx \times \d\yy)
    \right)^{1/2},
\end{equation*}
where $\Gamma^{(\eta)}(P, Q)$ denotes the set of all couplings $\gamma$ of $P$ and $Q$ which satisfy
\begin{equation}
    \label{eq:W2eta}
    \left(
        \int_{\R^d \times \R^d} |\< \be_1, \xx - \yy \>|^2 \gamma(\d\xx \times \d\yy) 
    \right)^{1/2}    
    \le \eta,
\end{equation}
where $\be_1 = (1, 0, \dots, 0)^\top$.

The following proof is analogous to the proof of \cite[Theorem 4.6]{montanari2022overparametrizedlineardimensionalityreductions}. We show the convergence of logit margins $W_2( \hat\cL_n, \cL_* ) \conp 0$ first, where
\begin{equation}
    \label{eq:margin_logit_dist}
    \hat \cL_{n} := \frac1n \sum_{i=1}^n \delta_{y_i ( \< \xx_i, \hat\vbeta \> + \hat\beta_{0} ) },
    \qquad 
    \cL_* := \Law\left( 
    \max\bigl\{ \kappa^*,  \rho^*\norm{\vmu}_2 + G + \beta_0^* Y    \bigr\}
    \right).
\end{equation}
Throughout this subsection, all the expectations (including the one in $H_\kappa$) are conditional on $\{ (y_i, \zz_i) \}_{i=1}^n$, which will be denoted as $\E_{\cdot | n}[\cdot]$. Now, let
\begin{equation*}
    \frac1n \sum_{i=1}^n \delta_{(y_i, \< \zz_i, \hat\vbeta \>)} =: \Law(Y', Z'),
\end{equation*}
then by definition in \cref{eq:W2eta} and the same arguments in the proof of \cite[Theorem 4.6]{montanari2022overparametrizedlineardimensionalityreductions}, there exists a coupling $(Y, Z, Y', Z')$ and a sufficiently small $\eta$ ($\eta < \varepsilon^2/4$), such that
\begin{equation}
    \label{eq:pursuit_pre}
    \left( \E_{\cdot | n}\bigl[(Y - Y')^2\bigr] \right)^{1/2} \le \eta,
    \qquad
    \left( \E_{\cdot | n}\bigl[(YZ - Y'Z')^2\bigr] \right)^{1/2} \le \frac{\sqrt{1 - \hat\rho^2}}{\sqrt{\delta}} + 2\varepsilon
\end{equation}
holds with high probability. 
We can express the empirical distribution of logit margins \cref{eq:margin_logit_dist} as
\begin{equation}
    \label{eq:emp_ell}
    \hat \cL_{n}
    = \frac1n \sum_{i=1}^n \delta_{y_i \< \zz_i, \hat\vbeta \> + \hat\rho \norm{\vmu}_2 + y_i \hat\beta_0 ) }
    = \Law \Bigl( \underbrace{Y'Z' + \hat\rho \norm{\vmu}_2 + \hat\beta_0 Y' }_{=: V} \Bigr).
\end{equation}
For convenience, denote $\hat V := YZ + \hat\rho \norm{\vmu}_2 + \hat\beta_0 Y$, then with high probability we have
\begin{align}
        \bigl( \E_{\cdot | n}\bigl[(V - \hat V)^2\bigr] \bigr)^{1/2}
        & \overset{\mathmakebox[0pt][c]{\text{(i)}}}{\le} \left( \E_{\cdot | n}\bigl[(YZ - Y'Z')^2\bigr] \right)^{1/2} 
        + \left( \E_{\cdot | n}\bigl[(Y - Y')^2\bigr] \right)^{1/2} |\hat\beta_0|  \notag \\
        & \overset{\mathmakebox[0pt][c]{\text{(ii)}}}{\le} \frac{\sqrt{1 - \hat\rho^2}}{\sqrt{\delta}} + 2\varepsilon + \eta B \notag \\
        & \overset{\mathmakebox[0pt][c]{\text{(iii)}}}{\le} \frac{\sqrt{1 - \hat\rho^2}}{\sqrt{\delta}} + 3\varepsilon,
        \label{eq:V_diff0}
\end{align}
where (i) follows from Minkowski inequality, (ii) uses \cref{eq:pursuit_pre} and $|\hat\beta_0| \le B$ from \cref{lem:over_beta0}, by recalling that $\delta < \delta^*(0)$ and the data is linearly separable with high probability, while in (iii) we choose $\eta < \min\{ \varepsilon^2/4, \varepsilon/B \}$. According to $\hat\kappa_n \conp \kappa^*$ from \cref{lem:over_mar_conp}, we know that
\begin{equation*}
    \lim_{n \to \infty} \P\left( y_i (\< \hat\vbeta , \xx_i \> + \hat\beta_0 ) \ge \kappa^* - \varepsilon,
    \forall\, i \in [n] \right) = 1.
\end{equation*}
Then by definition of $V$ in \cref{eq:emp_ell}, with high probability we have 
\begin{equation}
    \label{eq:V_kappa_as}
    V \ge \kappa^* - \varepsilon,
    \qquad \text{almost surely}.
\end{equation}
Now, recall $\delta = \delta^*(\kappa^*) = H_{\kappa^*}(\rho^*, \beta_0^*)$ by \cref{eq:kappa_star}, where $(\rho^*, \beta_0^*) = \argmin_{\rho \in [-1, 1], \beta_0 \in \R} H_{\kappa^*}(\rho, \beta_0)$. Therefore,
\begin{equation}
    \label{eq:V_diff1}
        \bigl( \E_{\cdot | n}\bigl[(V - \hat V)^2\bigr] \bigr)^{1/2}
        \le \frac{\sqrt{1 - \hat\rho^2}}{\sqrt{\delta}} + 3\varepsilon
        = \frac{\sqrt{1 - \hat\rho^2}}{\sqrt{H_{\kappa^*}(\rho^*, \beta_0^*)}} + 3\varepsilon
\end{equation}
holds with high probability. For $\rho \in [-1, 1], \beta_0 \in \R$, let us define
\begin{equation*}
    h_{\kappa^*}^*(\rho, \beta_0) := \frac{1}{\sqrt{H_{\kappa^*}(\rho, \beta_0)}}
    - \frac{1}{\sqrt{H_{\kappa^*}(\rho^*, \beta_0^*)}}.
\end{equation*}
Note that $h_{\kappa^*}^*(\rho, \beta_0) \ge 0$. Hence, \cref{eq:V_diff1} implies that (reminding $\tau = 1$) with high probability
\begin{align*}
        \bigl( \E_{\cdot | n}\bigl[(V - \hat V)^2\bigr] \bigr)^{1/2}
        & \le \sqrt{1 - \hat\rho^2} \biggl( \sqrt{ \frac{1}{H_{\kappa^*}(\hat\rho, \hat\beta_0)} } - h_{\kappa^*}^*(\hat\rho, \hat\beta_0) \biggr) + 3\varepsilon \\
        & = \! \left( \E_{\cdot | n}\! \left[ \bigl(  \kappa^* - \hat\rho \norm{\bmu}_2 + G - \hat\beta_0 Y \bigr)_+^2 \right] \right)^{1/2}
        - \sqrt{1 - \hat\rho^2} \cdot h_{\kappa^*}^*(\hat\rho, \hat\beta_0) 
        + 3\varepsilon \\
        & \overset{\mathmakebox[0pt][c]{\text{(i)}}}{=} \bigl( \E_{\cdot | n}\bigl[(\kappa^* - \hat V)_+^2 \bigr] \bigr)^{1/2}
        - \sqrt{1 - \hat\rho^2} \cdot h_{\kappa^*}^*(\hat\rho, \hat\beta_0) 
        + 3\varepsilon,
\end{align*}
which can be further written as (with high probability)
\begin{align}
        \bigl( \E_{\cdot | n}\bigl[(V - \hat V)^2\bigr] \bigr)^{1/2} + \sqrt{1 - \hat\rho^2} \cdot h_{\kappa^*}^*(\hat\rho, \hat\beta_0) 
        & \le
        \bigl( \E_{\cdot | n}\bigl[(\kappa^* - \hat V)_+^2\bigr] \bigr)^{1/2}
        + 3\varepsilon \notag  \\
        & \overset{\mathmakebox[0pt][c]{\text{(ii)}}}{\le} 
        \bigl( \E_{\cdot | n}\bigl[(\kappa^* - \varepsilon - \hat V)_+^2\bigr] \bigr)^{1/2}
        + 4\varepsilon.
        \label{eq:V_diff2}
\end{align}
In the derivation above, equation (i) follows from $\hat V = YZ + \hat\rho \norm{\vmu}_2 + \hat\beta_0 Y 
\overset{\mathmakebox[0pt][c]{\mathrm{d}}}{=} -G + \hat\rho \norm{\vmu}_2 + \hat\beta_0 Y$ when conditioning on $\{(y_i, \zz_i)\}_{i = 1}^n$, and (ii) follows from the fact that
\begin{equation*}
    \begin{aligned}
        \frac{\d}{\d \kappa}  \bigl( \E_{\cdot | n}\bigl[(\kappa - \hat V)_+^2\bigr] \bigr)^{1/2}
        & = \frac{
            \E_{\cdot | n}\bigl[(\kappa - \hat V)_+^2\bigr]
        }{
            \bigl( \E_{\cdot | n}\bigl[(\kappa - \hat V)_+^2\bigr] \bigr)^{1/2}
        }
        \le 1.
    \end{aligned}
\end{equation*}
Besides, by using \cref{eq:V_kappa_as} and exactly the same arguments in the proof of \cite[Theorem 4.6]{montanari2022overparametrizedlineardimensionalityreductions}, we can show that with high probability,
\begin{equation}
    \label{eq:V_max}
    \E_{\cdot | n} \!\left[ \bigl(V - \max\{ \kappa^* - \varepsilon, \hat V \} \bigr)^2 \right]
    \le \E_{\cdot | n}\bigl[(V - \hat V)^2\bigr] - 
    \E_{\cdot | n}\bigl[(\kappa^* - \varepsilon - \hat V)_+^2\bigr].
\end{equation}
Combining \cref{eq:V_max} with \eqref{eq:V_diff2} gives the following implications:
\begin{itemize}
    \item \cref{eq:V_max} implies 
    \begin{equation*}
        \E_{\cdot | n}\bigl[(\kappa^* - \varepsilon - \hat V)_+^2\bigr]
        \le
        \E_{\cdot | n} \bigl[(V - \hat V)^2\bigr].
    \end{equation*}
    Plugging this into \cref{eq:V_diff2} yields that with high probability,
    \begin{equation*}
        \sqrt{1 - \hat\rho^2} \cdot h_{\kappa^*}^*(\hat\rho, \hat\beta_0) \le 4 \varepsilon,
    \end{equation*}
    i.e., $\sqrt{1 - \hat\rho^2} \cdot h_{\kappa^*}^*(\hat\rho, \hat\beta_0) \conp 0$. Note that if $\abs{\rho} \to 1$ (i.e., $\sqrt{1 - \rho^2} = o_\varepsilon(1)$), the quantity
    \begin{equation*}
            \sqrt{1 - \rho^2} \cdot h_{\kappa^*}^*(\rho, \beta_0)
        = \left( \E \left[ \bigl(  \kappa^* - \rho \norm{\bmu}_2 + G - \beta_0 Y \bigr)_+^2 \right] \right)^{1/2}
        - \frac{\sqrt{1 - \rho^2}}{\sqrt{H_{\kappa^*}(\rho^*, \beta_0^*)}}
    \end{equation*}
    is bounded away from 0, for any $\beta_0 \in \R \cup\{ \pm \infty\}$. Therefore, we must have $h_{\kappa^*}^*(\hat\rho, \hat\beta_0) \conp 0$. By \cref{lem:H_kappa_1} (proof is deferred to the end of this subsection), we know $h_{\kappa^*}^*(\rho, \beta_0) \ge 0$ for all $\rho \in [-1, 1], \beta_0 \in \R$, and $(\rho, \beta_0) \to (\rho^*, \beta_0^*)$ if and only if $h_{\kappa^*}^*(\rho, \beta_0) \to 0$. Hence, we conclude
    \begin{equation*}
        (\hat\rho, \hat\beta_0) \conp (\rho^*, \beta_0^*),
    \end{equation*}
    which gives parameter convergence.

    \item Let
    \begin{equation*}
        I := \bigl( \E_{\cdot | n}\bigl[(V - \hat V)^2\bigr] \bigr)^{1/2} 
        , \qquad
        I\!I := \bigl( \E_{\cdot | n}\bigl[(\kappa^* - \varepsilon - \hat V)_+^2\bigr] \bigr)^{1/2}.
    \end{equation*}
    Then \cref{eq:V_diff2} implies $I - I\!I \le 4\varepsilon$, and we also have (for $\varepsilon > 0$ small enough)
    \begin{equation*}
        \begin{aligned}
            I\!I & \le \abs{\kappa^* - \varepsilon} + \bigl( \E_{\cdot | n}\bigl[\hat V^2\bigr] \bigr)^{1/2} 
            \le \kappa^* + \bigl( \E_{\cdot | n}\bigl[ (
                G + \hat\rho \norm{\vmu}_2 + \hat\beta_0 Y
            )^2\bigr] \bigr)^{1/2}  \\
            & \le \kappa^* + \left(\E[G^2]\right)^{1/2} + |\hat\rho| \norm{\vmu}_2 + |\hat\beta_0| \\
            & \le \kappa^* + 1 + \norm{\vmu}_2 + B,
        \end{aligned}
    \end{equation*}
    by using Minkowski inequality and $|\hat\beta_0| \le B$ (with high probability) from \cref{lem:over_beta0}. Based on these results and \cref{eq:V_max}, with high probability, we have
    \begin{equation*}
        \begin{aligned}
            \E_{\cdot | n} \!\left[ \bigl(V - \max\{ \kappa^* - \varepsilon, \hat V \} \bigr)^2 \right] 
            & \le I^2 - I\!I^2 = (I - I\!I)(I - I\!I + 2I\!I) \\
            & \le 4\varepsilon \bigl(4\varepsilon + 2 (\kappa^* + 1 + \norm{\vmu}_2 + B) \bigr) \\
            & \le C \varepsilon,
        \end{aligned}
    \end{equation*}
    where $C \in (0, \infty)$ is some constant depending on $(\pi, \norm{\bmu}_2, \delta)$ (through $\kappa^*, B$). Therefore, by recalling $\hat V \overset{\mathmakebox[0pt][c]{\mathrm{d}}}{=} G + \hat\rho \norm{\vmu}_2 + \hat\beta_0 Y$, we obtain that with high probability,
    \begin{equation}
        \label{eq:W2_conv}
        W_2\left( \hat \cL_{n}, \Law \bigl(\max\{ \kappa^* - \varepsilon, G + \hat\rho \norm{\vmu}_2 + \hat\beta_0 Y \} \bigr) \right) \le \sqrt{C \varepsilon}.
    \end{equation}
\end{itemize}    
As a consequence, 
the following holds with high probability:
    \begin{align*}
            W_2\bigl( \hat \cL_{n}, \cL_* \bigr)
            & = W_2 \left( 
                \hat \cL_{n},
                \Law\bigl( \max\{ \kappa^*, G + \rho^*\norm{\vmu}_2 + \beta_0^* Y \} \bigr)
             \right) \\
            & \le
            W_2\left( \hat \cL_{n}, \Law \bigl(\max\{ \kappa^* - \varepsilon, G + \hat\rho \norm{\vmu}_2 + \hat\beta_0 Y \} \bigr) \right) \\
            & \phantom{\le} \  + W_2\left( \Law \bigl(\max\{ \kappa^* - \varepsilon, G + \hat\rho \norm{\vmu}_2 + \hat\beta_0 Y \} \bigr),
            \Law \bigl(\max\{ \kappa^* , G + \hat\rho \norm{\vmu}_2 + \hat\beta_0 Y \} \bigr) \right)
            \\
            & \phantom{\le} \  + W_2\left( \Law \bigl(\max\{ \kappa^* , G + \hat\rho \norm{\vmu}_2 + \hat\beta_0 Y 
            \} \bigr) 
            ,
            \Law\bigl( \max\{ \kappa^*, G + \rho^*\norm{\vmu}_2 + \beta_0^* Y \} \bigr)
            \right) \\
            & \le \sqrt{C \varepsilon} + \varepsilon + o_\varepsilon(1)
            = o_\varepsilon(1),
    \end{align*}
    where in the last inequality, we use that (i) the result from \cref{eq:W2_conv}, (ii) the fact that the mapping $\kappa \mapsto \max\{ \kappa, G + \hat\rho \norm{\vmu}_2 + \hat\beta_0 Y \}$ is 1-Lipschitz, and (iii) the consequence of $(\hat\rho, \hat\beta_0) \conp (\rho^*, \beta_0^*)$ and $|\hat\rho| \le 1$, $|\hat\beta_0| \le B$ (with high probability). 
    
    \vspace{0.5\baselineskip}
	\noindent
    Now we prove the convergence of ELD. Denote $\hat \cL_{n} =: \Law(L')$, $\cL_* =: \Law(L)$, where $(L, L')$ is a coupling such that
    \begin{equation}\label{eq:L_coupling}
        \left( \E_{\cdot | n}\bigl[(L - L')^2\bigr] \right)^{1/2} = o_\varepsilon(1).
    \end{equation}
    Therefore, for some constants $C_1, C_2 > 0$, with high probability, we have
    \begin{align*}
        W_2\bigl( \hat \nu_{n}, \nu_* \bigr)
        & \le
        \left( \E_{\cdot | n}\bigl[(Y - Y')^2\bigr] \right)^{1/2} 
        +
        \left( \E_{\cdot | n}\bigl[(YL - Y'L')^2\bigr] \right)^{1/2} \\
        & \overset{\mathmakebox[0pt][c]{\text{(i)}}}{\le} 
        \eta + \left( \E_{\cdot | n}\bigl[(YL - Y'L)^2\bigr] \right)^{1/2}
        + \left( \E_{\cdot | n}\bigl[(Y'L - Y'L')^2\bigr] \right)^{1/2} \\
        & \overset{\mathmakebox[0pt][c]{\text{(ii)}}}{\le}
        \eta + C_1 \left( \E_{\cdot | n}\bigl[(Y - Y')^2\bigr] \right)^{1/4} \left(\E[L^4]\right)^{1/4}
        + \left( \E_{\cdot | n}\bigl[(L - L')^2\bigr] \right)^{1/2} \\
        & \overset{\mathmakebox[0pt][c]{\text{(iii)}}}{\le}
        \eta + C_2 \sqrt{\eta} + o_\varepsilon(1)
        \overset{\mathmakebox[0pt][c]{\text{(iv)}}}{\le}  o_\varepsilon(1),
    \end{align*}
    where in (i) we use \cref{eq:pursuit_pre} and Minkowski inequality, in (ii) use Cauchy--Schwarz inequality and $Y, Y' \in \{ \pm 1 \}$, in (iii) use \cref{eq:pursuit_pre} and \eqref{eq:L_coupling}, while in (iv) recall that $\eta < \min\{ \varepsilon^2/4, \varepsilon/B \} = o_{\varepsilon}(1)$. By taking $\varepsilon \to 0$, we can show that $W_2\bigl( \hat \nu_{n}, \nu_* \bigr) \conp 0$ for $\tau = 1$. This completes the proof.
\end{proof}

Finally, we prove the following technical lemma.
\begin{lem}
    \label{lem:H_kappa_1}
    For any fixed $\kappa \in \R$ and $\tau > 0$, the function $H_\kappa(\rho, \beta_0)$ in \cref{eq:sep_functions} admits a unique maximizer $(\rho^*(\kappa), \beta_0^*(\kappa)) \in [0, 1) \times \R$.
\end{lem}
\begin{proof}
    For simplicity, write $\rho^* := \rho^*(\kappa)$, $\beta_0^* := \beta_0^*(\kappa)$. First, note that
    \begin{equation*}
        H_\kappa(\rho, \beta_0) = \frac{1 - \rho^2}{\E\left[ \bigl(  s(Y) \kappa - \rho \norm{\bmu}_2 + G - \beta_0 Y \bigr)_+^2 \right]} \le \, \frac{1}{\E\left[ \bigl(  s(Y) \kappa - \norm{\bmu}_2 + G - \beta_0 Y \bigr)_+^2 \right]},
    \end{equation*}
    which converges to $0$ as $\beta_0 \to \pm\infty$. Moreover, $H_\kappa(-\rho, \beta_0) < H_\kappa(\rho, \beta_0)$ for any $\rho \in (0, 1]$. Therefore, $H_\kappa(\rho, \beta_0)$ must have a maximizer $(\rho^*, \beta_0^*) \in [0, 1] \times \R$. Further, $\rho^* \in [0, 1)$ since $H_\kappa(1, \beta_0) \equiv 0$. We prove the uniqueness of $(\rho^*, \beta_0^*)$ by contradiction. For future convenience, we denote $H_{\max} := H_{\kappa} (\rho^*, \beta_0^*)$. Assume that there exist $(\rho_1, \beta_{0, 1})$ and $(\rho_2, \beta_{0, 2})$ such that $(\rho_1, \beta_{0, 1}) \neq (\rho_2, \beta_{0, 2})$, and
    \begin{equation*}
        H_{\kappa} (\rho_1, \beta_{0, 1}) = H_{\kappa} (\rho_2, \beta_{0, 2}) = H_{\max},
    \end{equation*}
    which implies
    \begin{equation*}
        G_{\kappa} (\rho_1, \beta_{0, 1}) = \frac{\sqrt{1 - \rho_1^2}}{\sqrt{H_{\max}}}, 
        \qquad 
        G_{\kappa} (\rho_2, \beta_{0, 2}) = \frac{\sqrt{1 - \rho_2^2}}{\sqrt{H_{\max}}},
    \end{equation*}
    where we define
    \begin{equation*}
        G_\kappa(\rho, \beta_0) := \left( \E\left[ \bigl(  s(Y) \kappa - \rho \norm{\bmu}_2 + G - \beta_0 Y \bigr)_+^2 \right] \right)^{1/2}.
    \end{equation*}
    Similar to \cite[Lemma 6.3]{montanari2023generalizationerrormaxmarginlinear}, we can show that $G_{\kappa}$ is strictly convex. Hence,
    \begin{align*}
        G_{\kappa} \left( \frac{\rho_1 + \rho_2}{2}, \frac{\beta_{0, 1} + \beta_{0, 2}}{2} \right) 
        & <  \frac{1}{2} \bigl( G_{\kappa} (\rho_1, \beta_{0, 1}) + G_{\kappa} (\rho_2, \beta_{0, 2}) \bigr) \\
        & =  \frac{1}{\sqrt{H_{\max}}} \frac{1}{2} \Bigl( 
        \sqrt{1 - \smash[b]{\rho_1^2}} + \sqrt{1 - \smash[b]{\rho_2^2}} \Bigr) \\
        & \le \frac{1}{\sqrt{H_{\max}}} \sqrt{1 - \left( \frac{\rho_1 + \rho_2}{2} \right)^2},
    \end{align*}
    where in the last line we use the concavity of the mapping $x \mapsto \sqrt{1 - x^2}$. It finally follows that
    \begin{equation*}
        H_{\kappa} \left( \frac{\rho_1 + \rho_2}{2}, \frac{\beta_{0, 1} + \beta_{0, 2}}{2} \right) > H_{\max},
    \end{equation*}
    a contradiction. This concludes the proof.
\end{proof}

\subsubsection{Completing the proof of \cref{thm:SVM_main}}
\begin{proof}[\textbf{Proof of \cref{thm:SVM_main}}]
\noindent
\textbf{\ref{thm:SVM_main_trans}} is established by \cref{lem:over_phase_trans}.

\begin{proof}[\textbf{\emph{\ref{thm:SVM_main_var}:}}]
Notice the definition of $(\rho^*, \beta_0^*, \kappa^*)$ we used in our proof (\cref{subsubsec:over_phase}, \ref{subsubsec:over_logit_conv}) is based on \cref{eq:SVM_asymp_simple}. It suffices to show the equivalence of two optimization problems \cref{eq:SVM_variation} and \eqref{eq:SVM_asymp_simple}. Now we fix $\rho$, $\beta_0$ in \cref{eq:SVM_variation} and $X := \rho \| \bmu\|_2 + G + Y \beta_0$. Then \cref{eq:SVM_variation} can be written as
\begin{equation}\label{eq:SVM_var2}
    \begin{aligned}
        \begin{array}{cl}
            \underset{ \kappa > 0 , \, \xi \in \cL^2  }{ \mathrm{maximize} } & \kappa, \\
            \underset{ \phantom{\smash{\bm\beta \in \R^d, \beta_0 \in \R, \kappa \in \R} } }{\text{subject to}} &  
            X + \sqrt{1 - \rho^2} \xi \ge s(Y) \kappa,  
            \qquad \E[\xi^2]  \le  1/\delta .
        \end{array}
    \end{aligned}
\end{equation}
Note that it can be written as a convex optimization problem, and it is infeasible if $\rho = \pm 1$ (since $X$ has support $\R$). Take $\rho \in (-1, 1)$. According to the Karush--Kuhn--Tucker (KKT) and Slater's conditions for variational problems \cite[Theorem 2.9.2]{zalinescu2002convex}, $(\kappa, \xi)$ is the solution to \cref{eq:SVM_var2} if and only if it satisfies the following for some $\Lambda \in \cL^1, \Lambda \ge 0$ (a.s.) and $\nu \ge 0$:
\begin{equation*}
    \begin{aligned}
        -1 + \E[s(Y)\Lambda] = 0, \qquad
        -\sqrt{1 - \rho^2} \Lambda + 2 \nu \xi = 0 \ \  \text{(a.s.)},  \\
        \nu\left( \E[\xi^2] - \delta^{-1} \right) = 0,
        \qquad
        \Lambda \bigl( s(Y)\kappa - X - \sqrt{1 - \rho^2} \xi \bigr) = 0 \ \  \text{(a.s.)}.
    \end{aligned}
\end{equation*}
Clearly $\nu > 0$ (otherwise, $\Lambda = 0$ a.s., a contradiction). Consider the following two cases:
\begin{itemize}
    \item On the event $\{ s(Y(\omega))\kappa - X(\omega) < 0\}$, we obtain $s(Y(\omega))\kappa - X(\omega) - \sqrt{1 - \rho^2} \xi(\omega) < 0$, which implies $\Lambda(\omega) = 0$. Therefore, $\xi(\omega) = 0$.
    \item On the event $\{ s(Y(\omega))\kappa - X(\omega) > 0\}$, we obtain $\sqrt{1 - \rho^2} \xi(\omega) \ge s(Y(\omega))\kappa - X(\omega) > 0$, which implies $\xi(\omega) > 0$. Therefore, $\Lambda(\omega) > 0$, and thus $s(Y(\omega))\kappa - X(\omega) - \sqrt{1 - \rho^2} \xi(\omega) = 0$.
\end{itemize}
(Note $\P(s(Y)\kappa - X = 0) = 0$.) By combining these, we get $\sqrt{1 - \rho^2}\xi = (s(Y)\kappa - X)_+$. This proves \cref{eq:SVM_main_xi_star}. Plug in it into \cref{eq:SVM_variation} gives \cref{eq:SVM_asymp_simple}. The proof of $\rho^* \in (0, 1)$ and its independence of $\tau$ is given by \cref{lem:gordon_eq} in \cref{subsec:over_asymp}. 

This concludes the proof of part \ref{thm:SVM_main_var}.
\end{proof}

\begin{proof}[\textbf{\emph{\ref{thm:SVM_main_mar}, $\delta < \delta^*(0)$:}}]
We show that $\hat\kappa_n \conp \kappa^*$ in \cref{lem:over_mar_conp} can be strengthened to $\hat\kappa_n \conL{2} \kappa^*$. To this end, we show that $\hat\kappa_n^2$ is uniformly integrable (u.i.). Recall that $\kappa(\hat\vbeta_n, \hat\beta_{0,n}) \ge 0$ and
\begin{align*}
        \kappa(\hat\vbeta_n, \hat\beta_{0,n}) 
        & = \min_{i \in [n]} \wt y_i \bigl( \< \xx_i, \hat\vbeta_n \> + \hat\beta_{0,n} \bigr)
        = \min_{i \in [n]} \wt y_i \bigl(  y_i \< \vmu, \hat\vbeta_n \> + \< \zz_i, \hat\vbeta_n \> + \hat\beta_{0,n} \bigr) \\
        & = \min \left\{
            \min_{i: y_i = +1}  \tau^{-1}\bigl( \< \vmu, \hat\vbeta_n \> + \< \zz_i, \hat\vbeta_n \> + \hat\beta_{0,n} \bigr) ,
            \min_{i: y_i = -1}  \bigl( \< \vmu, \hat\vbeta_n \> - \< \zz_i, \hat\vbeta_n \> - \hat\beta_{0,n}\bigr)
        \right\}.
\end{align*}
Hence, on the event $\mathcal{D}_n^c$ (non-degenerate case), we have $\hat\kappa_n = \kappa(\hat\vbeta_n, \hat\beta_{0,n})$ and it can be bounded by the average from each class:
\begin{equation*}
    \begin{aligned}
        \kappa(\hat\vbeta_n, \hat\beta_{0,n}) & \le \tau^{-1}\bigl( \< \vmu, \hat\vbeta_n \> + \< \bar\zz^+_n, \hat\vbeta_n \> + \hat\beta_{0,n} \bigr) := \bar \kappa_n^+ , \\
        \kappa(\hat\vbeta_n, \hat\beta_{0,n}) & \le \phantom{\tau^{-1}\bigl(} \< \vmu, \hat\vbeta_n \> - \< \bar\zz^-_n, \hat\vbeta_n \> - \hat\beta_{0,n} \phantom{\bigr)} := \bar \kappa_n^-  ,
    \end{aligned}
\end{equation*}
where
\begin{equation*}
    \bar\zz^+_n := \frac{1}{n_+} \sum_{i: y_i = +1} \zz_i,
    \qquad
    \bar\zz^-_n := \frac{1}{n_-} \sum_{i: y_i = -1} \zz_i.
\end{equation*}
Combine these two bounds and apply Cauchy--Schwarz inequality, we obtain
\begin{equation*}
        \kappa(\hat\vbeta_n, \hat\beta_{0,n})
        \le \frac{\tau \bar \kappa_n^+  +  \bar \kappa_n^-}{\tau + 1}
        = \frac{2}{\tau + 1} 
        \!
        \left(
        \< \vmu, \hat\vbeta_n \> 
        + \left\< \frac{\bar\zz^+_n - \bar\zz^-_n}{2}, \hat\vbeta_n \right\>
        \right)
        \le 
        \frac{2}{\tau + 1} \left( \norm{\vmu}_2 + \norm{ \wt \zz_n }_2 \right),
\end{equation*}
where
\begin{equation*}
    \wt \zz_n := \frac{\bar\zz^+_n - \bar\zz^-_n}{2},
    \quad
    \wt \zz_n \,|\, \yy \sim \normal\left( \bzero, \frac14\Bigl( \frac{1}{n_+} + \frac{1}{n_-} \Bigr) \bI_d \right).
\end{equation*}
Therefore,
\begin{equation*}
    0   \le   \hat\kappa_n  =  \kappa(\hat\vbeta_n, \hat\beta_{0,n}) \ind_{1 \le n_+ \le n-1} 
    \le \frac{2}{\tau + 1} \left( 
        \norm{\vmu}_2 + \norm{ \wt \zz_n }_2 \ind_{1 \le n_+ \le n-1}
     \right).
\end{equation*}
In order to prove $\hat\kappa_n^2$ is u.i., it suffices to show that $\norm{ \wt \zz_n }_2^2 \ind_{1 \le n_+ \le n-1}$ is u.i.. Next, we prove this by establishing that $\norm{ \wt \zz_n }_2^2 \ind_{1 \le n_+ \le n-1}$ converges in $\cL^1$. It requires two steps (by Scheffé's Lemma):
\begin{subequations}
\begin{align}
    \E \left[ \norm{ \wt \zz_n }_2^2 \ind_{1 \le n_+ \le n-1} \right] \to \frac{1}{4\delta} \left( \frac{1}{\pi} + \frac{1}{1-\pi} \right),
    \label{eq:svm_ui_a}
    \\
    \norm{ \wt \zz_n }_2^2 \ind_{1 \le n_+ \le n-1} \conp \frac{1}{4\delta} \left( \frac{1}{\pi} + \frac{1}{1-\pi} \right).
    \label{eq:svm_ui_b}
\end{align}
\end{subequations}
For \cref{eq:svm_ui_a}, Observe
\begin{equation*}
    \begin{aligned}
      & \E \left[ \norm{ \wt \zz_n }_2^2 \ind_{1 \le n_+ \le n-1} \right]
    = \E \left[ \E\bigl[ \norm{ \wt \zz_n }_2^2 \ind_{1 \le n_+ \le n-1} \,|\, \yy \bigr] \right] \\
    = {} & \E \left[ \frac{d}{4} \Bigl( \frac{1}{n_+} + \frac{1}{n_-} \Bigr) \ind_{1 \le n_+ \le n-1} \right]
    = \frac{d}{4n} \E \left[ \Bigl( \frac{n}{n_+} + \frac{n}{n_-} \Bigr) \ind_{1 \le n_+ \le n-1} \right].
    \end{aligned}
\end{equation*}
To evaluate the expected value, note that (by law of large numbers)
\begin{equation*}
    \frac{n}{n_+} \ind_{1 \le n_+ \le n-1} \le \frac{2n}{n_+ + 1},
    \qquad
    \frac{n}{n_+}\ind_{1 \le n_+ \le n-1} \conp \frac{1}{\pi},
    \qquad
    \frac{2n}{n_+ + 1} \conp \frac{2}{\pi}.
\end{equation*}
A classical result \cite{chao1972negative} gives
\begin{equation*}
    \lim_{n \to \infty} \E\left[ \frac{2n}{n_+ + 1} \right]
    = \lim_{n \to \infty} \frac{2n\left( 1 - (1-\pi)^{n+1} \right)}{(n+1)\pi} = \frac{2}{\pi}.
\end{equation*}
So $\frac{2n}{n_+ + 1} \conL{1} \frac{2}{\pi}$, which implies $\frac{2n}{n_+ + 1}$ is u.i., and so is $\frac{n}{n_+} \ind_{1 \le n_+ \le n-1}$. Therefore, by Vitali convergence theorem, we have $\frac{n}{n_+} \ind_{1 \le n_+ \le n-1} \conL{1} \frac{1}{\pi}$. Similar arguments give $\frac{n}{n_-} \ind_{1 \le n_+ \le n-1} \conL{1} \frac{1}{1-\pi}$. Hence
\begin{equation*}
    \lim_{n \to \infty} \E \left[ \norm{ \wt \zz_n }_2^2 \ind_{1 \le n_+ \le n-1} \right]
    = \lim_{n \to \infty} \frac{d}{4n} \cdot \lim_{n \to \infty} \E \left[ \Bigl( \frac{n}{n_+} + \frac{n}{n_-} \Bigr) \ind_{1 \le n_+ \le n-1} \right]
    = \frac{1}{4\delta} \left( \frac{1}{\pi} + \frac{1}{1-\pi} \right).
\end{equation*}
For \cref{eq:svm_ui_b}, notice that $\norm{ \wt \zz_n }_2^2 \,|\, \yy \sim a_n \chi_d^2$, where $a_n = \frac{1}{4} ( \frac{1}{n_+} + \frac{1}{n_-})$. By concentration inequality (e.g., \cref{lem:subG_concentrate}\ref{lem:subG-Hanson-Wright-I}), we have
\begin{equation*}
    \P\left( \abs{ \norm{ \wt \zz_n }_2^2 - d a_n } \ge \varepsilon \,\big|\, \yy \right) 
    \le 2\exp\left( -c \min\left\{ \frac{\varepsilon^2}{d a_n^2}, \frac{\varepsilon}{a_n} \right\} \right)
    = o_\P(1),
\end{equation*}
where $c > 0$ is a constant, $a_n = o_\P(1)$, $d a_n \conp \frac{1}{4\delta} ( \frac{1}{\pi} + \frac{1}{1-\pi} )$. By taking expectation on both sides and using bounded convergence theorem, we have $\norm{ \wt \zz_n }_2^2 - d a_n = o_\P(1)$. Then we get \cref{eq:svm_ui_b}.

Finally, \cref{eq:svm_ui_a} and \eqref{eq:svm_ui_b} imply that $\norm{ \wt \zz_n }_2^2 \ind_{1 \le n_+ \le n-1}$ converges in $\cL^1$, and thus is u.i.. So $\hat\kappa_n^2$ is also u.i.. By Vitali convergence theorem, convergence in probability of $\hat\kappa_n$ can be strengthen to $\cL^2$ convergence. 

This concludes the proof of part \ref{thm:SVM_main_mar} for $\delta < \delta^*(0)$.
\end{proof}

\begin{proof}[\textbf{\emph{\ref{thm:SVM_main_mar}, $\delta > \delta^*(0)$:}}]
For non-separable regime, we cannot work with $\xi_{n, \kappa}$ in \cref{eq:xi_n_kappa} to show a negative margin, since $\hat\kappa_n \ge 0$ always holds (by taking $\vbeta = 0$, $\beta_0 = 0$). To this end, we define
\begin{equation*}
    \Xi_{n, \kappa} := \min_{ \substack{ \norm{\vbeta}_2 = 1 \\ \beta_0 \in \R} } \frac{1}{\sqrt{d}} \norm{ \left( \kappa \bs_\yy - \yy \odot \XX \vbeta  - \beta_0 \yy \right)_+ }_2,
\end{equation*}
which replace the constraint $\norm{\vbeta}_2 \le 1$ in $\xi_{n, \kappa}$ by $\norm{\vbeta}_2 = 1$. Here we define the margin as
\begin{equation}\label{eq:wt_kappa_def}
    \wt\kappa_n := \sup\{ \kappa \in \R:  \Xi_{n,\kappa} = 0 \}.
\end{equation}
Note that $\wt\kappa_n = \hat\kappa_n$ on separable data, but $\wt\kappa_n$ is allowed to be negative. Then our goal is to show
\begin{equation}\label{eq:neg_kappa}
    \wt\kappa_n \le -\overline{\kappa}
\end{equation}
holds for some $\overline{\kappa} > 0$ with high probability. Then followed by the proof outline at the beginning of \cref{append_subsec:sep}, we can also define a series of random variables in a similar way:
\begin{align*}
    \Xi'_{n, \kappa, B} & := \min_{ \substack{ \norm{\vbeta}_2 = 1 \\ \abs{\beta_0} \le B } } \max_{ \substack{ \norm{\blambda}_2 \le 1 \\ \blambda \odot \yy \ge 0} } \frac{1}{\sqrt{d}} \blambda^\top \left( \kappa \bs_\yy \odot \yy - \XX \vbeta  - \beta_0 \bone \right),
    \\
    \Xi_{n, \kappa, B}'^{(1)} & := \min_{ \substack{ \rho^2 + \norm{\vtheta}_2^2 = 1 \\ \abs{\beta_0} \le B } } \max_{ \substack{ \norm{\blambda}_2 \le 1 \\ \blambda \odot \yy \ge 0} } \frac{1}{\sqrt{d}}  \left(
    \norm{\blambda}_2 \vg^\top \btheta + \norm{\btheta}_2 \vh^\top \blambda + \blambda^\top \bigl( 
        \kappa \bs_\yy \odot \yy - \rho\norm{\bmu}_2 \yy + \rho \vu - \beta_0 \bone
     \bigr)
     \right),
    \\
    \bar\Xi'^{(2)}_{\kappa, B} & :=  \min_{ \substack{ \rho^2 + r^2 = 1, r \ge 0 \\  \abs{\beta_0} \le B } }
    -r + \sqrt{\delta} \left( \E\left[ \bigl(  s(Y) \kappa - \rho \norm{\bmu}_2 + \rho G_1 + rG_2 - \beta_0 Y \bigr)_+^2 \right] \right)^{1/2},
\end{align*}
where the constraints $\norm{\vbeta}_2 \le 1$, $\rho^2 + \norm{\vtheta}_2^2 \le 1$, and $\rho^2 + r^2 \le 1$ in $\xi'_{n, \kappa, B}$, $\xi_{n, \kappa, B}'^{(1)}$, and $\bar\xi'^{(2)}_{\kappa, B}$ all become equality constraints. Then we follow the same arguments in Step 1---5 (\cref{lem:over_beta0}---\ref{lem:over_mar_conp}).
\begin{itemize}
    \item Analogous to the proof of \cref{lem:over_beta0}, we have $| \P\bigl( \Xi_{n, \kappa} = 0 \bigr) - \P\bigl(\Xi'_{n, \kappa, B} = 0 \bigr) | \to 0$.
    \item Analogous to the proof of \cref{lem:over_CGMT}, we can apply CGMT \cref{lem:CGMT} to connect $\Xi'_{n, \kappa, B}$ with $\Xi_{n, \kappa, B}'^{(1)}$. 
\begin{equation*}
    \P\left( \Xi'_{n, \kappa, B} \le t \vphantom{\Xi'^{(1)}_{n, \kappa, B}} \right) 
    \le 2\, \P\left( \Xi'^{(1)}_{n, \kappa, B} \le t \right).
\end{equation*}
    Here we only get a one-sided inequality since $\{(\rho, \vtheta): \rho^2 + \norm{\vtheta}_2^2 = 1 \}$ is non-convex.
    \item Analogous to the proof of \cref{lem:over_ULLN}, we have $\Xi'^{(1)}_{n, \kappa, B}  \conp ( \bar\Xi'^{(2)}_{\kappa, B} )_+$.
    \item Notice that the optimal $r$ in $\bar\Xi'^{(2)}_{\kappa, B}$ must be nonnegative. Hence, by substituting $r = \sqrt{1 - \rho^2}$, we have $\bar\Xi'^{(2)}_{\kappa, B} = \bar\xi_{\kappa}^{(3)}$ (\cref{eq:xi3_kappa}) for some $B > 0$ large enough.
\end{itemize}
Recall that in the proof of \cref{lem:over_phase_trans} and \ref{lem:over_mar_conp}, if $\delta > \delta^*(0)$, then there exists a $\kappa_0 < 0$, such that $\bar\xi_{\kappa_0}^{(3)} = 0$. According to \cref{eq:phase_equiv}, for any $\varepsilon > 0$ small enough, by using above relations, we have 
\begin{align*}
\bar\xi_{\kappa_0 + \varepsilon}^{(3)} = \bar\Xi'^{(2)}_{\kappa_0 + \varepsilon, B} > 0
\quad & \Longrightarrow \quad
\Xi'^{(1)}_{n, \kappa_0 + \varepsilon, B}  \conp \bigl( \bar\xi^{(3)}_{\kappa_0 + \varepsilon} \bigr)_+ > 0
\\
\quad \Longrightarrow \quad
\Xi'_{n, \kappa_0 + \varepsilon, B} > 0 ~ \text{w.h.p.}
\quad & \Longrightarrow \quad
\Xi_{n, \kappa_0 + \varepsilon} > 0 ~ \text{w.h.p.}.
\end{align*}
By \cref{eq:wt_kappa_def}, $\wt\kappa_n < \kappa_0 + \varepsilon < 0$ holds with high probability (by taking $\varepsilon$ to be sufficiently small), which proves \cref{eq:neg_kappa}.

This concludes the proof of part \ref{thm:SVM_main_mar} for $\delta > \delta^*(0)$.
\end{proof}

\begin{proof}[\textbf{\emph{\ref{thm:SVM_main_param}, \ref{thm:SVM_main_logit}:}}]
We have shown parameter and ELD convergence for the case $\tau = 1$ in \cref{lem:over_logit_conv}. Now for any $\tau \ge 1$, denote $\hat\vbeta_n(\tau), \hat\beta_{0, n}(\tau), \hat\kappa_n(\tau)$ as the max-margin solution to \cref{eq:over_max-margin}, and define
\begin{equation*}
     \hat\rho_n(\tau) :=  \left\< \hat\vbeta_n(\tau), \frac{\vmu}{\norm{\vmu}_2 } \right\>.
\end{equation*}
Similarly, denote $\rho^*(\tau), \beta_0^*(\tau), \kappa^*(\tau)$ as the optimal solution to \cref{eq:SVM_asymp_simple}. By \cref{prop:SVM_tau_relation},
\begin{itemize}
    \item We have
    \begin{equation}\label{eq:param_hat_tau}
        \hat\rho_n(\tau) = \hat\rho_n(1),
        \qquad
        \hat\beta_{0, n}(\tau) = \hat\beta_{0, n}(1) + \frac{\tau - 1}{\tau + 1} \hat\kappa_n(1). 
    \end{equation}

    \item We can write
    \begin{equation}\label{eq:Ln_hat_tau}
        \begin{aligned}
            \hat \nu_{n} & = \frac1n \sum_{i=1}^n \delta_{ \left(y_i, \< \xx_i, \hat\vbeta_n \> + \hat\beta_{0,n}(\tau) \right)  \ind\{\cD_n^c\}  }
        =: \Law\left( Y' \ind_{\cD_n^c}, \bigl( \< \xx', \hat\vbeta_n \> + \hat\beta_{0,n}(\tau) \bigr) \ind_{\cD_n^c} \right) \\
        & = \Law \left( Y' \ind_{\cD_n^c},  \bigl( \< \xx', \hat\vbeta_n \> + \hat\beta_{0,n}(1) \bigr) \ind_{\cD_n^c}
        + \frac{\tau - 1}{\tau + 1} \hat\kappa_n(1)
        \right).
        \end{aligned}
    \end{equation}
\end{itemize}
Besides, according to \cref{cor:asymp_tau_relation},
\begin{itemize}
    \item We have
    \begin{equation}\label{eq:param_star_tau}
        \rho^*(\tau) = \rho^*(1),
        \qquad
        \beta_0^*(\tau) = \beta_0^*(1) + \frac{\tau - 1}{\tau + 1} \kappa^*(1).
    \end{equation}

    \item We can also write
    \begin{equation}\label{eq:Ln_star_tau}
    \begin{aligned}
        \nu_*
        & = \Law \, \Bigl( Y, Y \max \bigl\{ s(Y)\kappa^*(\tau) , G + \rho^*\norm{\vmu}_2 + \beta_0^*(\tau) Y \bigr\} \Bigr) \\
        & = \Law \left( Y, Y \max \bigl\{ \kappa^*(1) , G + \rho^*\norm{\vmu}_2 + \beta_0^*(1) Y \bigr\} 
        + \frac{\tau - 1}{\tau + 1} \kappa^*(1)
        \right).
    \end{aligned}
    \end{equation}
\end{itemize}
We have shown $\hat\kappa_n(1) \conL{2} \kappa^*(1)$ and $\hat\beta_{0,n}(1) \conp \beta_0^*(1)$ in \cref{lem:over_logit_conv}. Then by continuous mapping theorem, comparing \cref{eq:param_hat_tau} and \eqref{eq:param_star_tau}, it follows that $\hat\beta_{0, n}(\tau) \conp \beta_0^*(\tau)$ for any $\tau > 0$. 

In \cref{lem:over_logit_conv}, we have shown that $W_2\bigl( \hat \nu_{n}, \nu_* \bigr) = o_\varepsilon(1)$ for $\tau = 1$ with high probability, i.e.,
\begin{equation}
\label{eq:logit_conv_1}
    W_2 \Bigl(
        \Law \Bigl(
        Y'\ind_{\cD_n^c} ,
        \underbrace{    
        \bigl( \< \xx', \hat\vbeta_n \> + \hat\beta_{0,n}(1) \bigr) \ind_{\cD_n^c}
        }_{ =: U_n} \Bigr)
        ,
        \Law \Bigl( 
        Y,
            \underbrace{    
        Y\max \bigl\{ \kappa^*(1) , G + \rho^*\norm{\vmu}_2 + \beta_0^*(1) Y \bigr\} 
        }_{=: U^*} \Bigr)
        \Bigr)
    \! = \!
    o_\varepsilon(1).
\end{equation}
Then there exists a coupling $(Y', Y, U_n, U^*)$ such that, with high probability, 
\begin{align*}
    W_2( \hat \nu_{n},  \nu_* )
    & \le   
        \left( \E_{\cdot | n} \bigl[ (
           Y'\ind_{\cD_n^c} - Y 
        )^2 \bigr] \right)^{\frac12} 
        + \left( \E_{\cdot | n} \biggl[ \Bigl(
           U_n - U^*  + \frac{\tau - 1}{\tau + 1} \hat\kappa_n(1) - \frac{\tau - 1}{\tau + 1} \kappa^*(1)
        \Bigr)^2 \biggr] \right)^{\frac12} 
        \\
    & \overset{\mathmakebox[0pt][c]{\text{(i)}}}{\le}
        \left( \E_{\cdot | n} \bigl[ (
           Y'\ind_{\cD_n^c} - Y 
        )^2 \bigr] \right)^{\frac12} 
        +
        \left( \E_{\cdot | n} \bigl[ (
            U_n - U^* 
        )^2 \bigr] \right)^{\frac12} 
        +
        \frac{\tau - 1}{\tau + 1}
         \left( \E_{\cdot | n} \bigl[ (
            \hat\kappa_n(1) - \kappa^*(1) 
        )^2 \bigr] \right)^{\frac12}
        \\
    & \overset{\mathmakebox[0pt][c]{\text{(ii)}}}{=}  o_\varepsilon(1),
\end{align*}
where in (i) we use Minkowski inequality, while in (ii) we use $\hat\kappa_n(1) \conL{2} \kappa^*(1)$ in \ref{thm:SVM_main_mar} and \cref{eq:logit_conv_1}. By taking $\varepsilon \to 0$, we can show that $W_2\bigl( \hat \nu_{n}, \nu_* \bigr) \conp 0$ holds for any $\tau > 0$.

For TLD convergence, we give a proof of $\hat\nu_{n}^\mathrm{test} \conw \nu_*^\mathrm{test}$. Write $\xx_\mathrm{new} = y_\mathrm{new} \bmu + \zz_\mathrm{new}$, $\zz_\mathrm{new} \sim \normal(\bzero, \bI_d)$, and recall $\hat\nu_{n}^\mathrm{test} = \Law \bigl(y^\mathrm{new}, \< \xx^\mathrm{new}, \hat\vbeta_n \> + \hat\beta_{0,n} \bigr)$. Let $G \sim \normal(0, 1)$ and $G \indep y^\mathrm{new}$, then
\begin{align*}
   \< \xx^\mathrm{new}, \hat\vbeta_n \> + \hat\beta_{0,n}
    & = \< y^\mathrm{new}\bmu + \zz^\mathrm{new}, \hat\vbeta_n \> + \hat\beta_{0,n} \\
    & = y^\mathrm{new} \hat\rho_n \norm{\bmu}_2 + \< \zz^\mathrm{new}, \hat\vbeta_n \>
    + \hat\beta_{0,n} \\
    & \cond y^\mathrm{new} ( \rho^* \norm{\bmu}_2 + G + y^\mathrm{new} \beta_{0}^*),
\end{align*}
where in the last line we use Slutsky's theorem and $y^\mathrm{new} \indep (y^\mathrm{new}\zz^\mathrm{new}, \hat\vbeta_n, \hat\beta_{0,n})$.

This concludes the proof of part \ref{thm:SVM_main_param} and \ref{thm:SVM_main_logit}.
\end{proof}

\begin{proof}[\textbf{\emph{\ref{thm:SVM_main_err}:}}]
In \ref{thm:SVM_main_logit} above, we showed that
\begin{equation*}
        \hat f(\xx_\mathrm{new}) = \< \xx_\mathrm{new}, \hat\vbeta_n \> + \hat\beta_{0,n}
        \cond y_\mathrm{new} \rho^* \norm{\bmu}_2 + G + \beta_{0}^*.
\end{equation*}
Therefore, by bounded convergence theorem, the errors have their limits
\begin{align*}
        \lim_{n \to \infty} \Err_{+,n} & = \P\left( + \rho^* \norm{\bmu}_2 + G + \beta_{0}^* \le 0 \right)
        = \Phi \left(- \rho^* \norm{\bmu}_2  - \beta_0^* \right), \\
        \lim_{n \to \infty} \Err_{-,n} & = \P\left( - \rho^* \norm{\bmu}_2 + G + \beta_{0}^* >  0 \right)
        = \Phi \left(- \rho^* \norm{\bmu}_2  + \beta_0^* \right).
\end{align*}
This concludes the proof of part \ref{thm:SVM_main_err}.
\end{proof}
Finally, we complete the proof of \cref{thm:SVM_main}.
\end{proof}

\subsection{Analysis of the asymptotic optimization problem: Proof of \cref{lem:gordon_eq}}
\label{subsec:over_asymp}

We provide an analysis of the low dimensional asymptotic optimization problem \cref{eq:SVM_asymp_simple} in this subsection. The conclusion below has been used in the proofs of \cref{thm:SVM_main}\ref{thm:SVM_main_var}, \ref{thm:SVM_main_param} and \ref{thm:SVM_main_logit}. It will be also used in \cref{append_sec:mar_reb} to obtain monotonicity results.

For $G \sim \normal(0, 1)$ and $t \in \R$, we define two auxiliary functions
\begin{equation}
\label{eq:fun_g1_g2}
	g_1 (t) := \E \left[ (G + t)_+ \right], \qquad g_2 (t) := \E \left[ (G + t)_+^2 \right].
\end{equation}
Clearly both $g_1$ and $g_2$ are strictly increasing mappings from $\R$ to $\R_{>0}$. Then $g := g_2 \circ g_1^{-1}$ is also strictly increasing. The following lemma shows that the limiting parameters $(\rho^*, \beta_0^*, \kappa^*)$ defined in \cref{thm:SVM_main} can be characterized by the following system of equations, involving $g$ and $g_1^{-1}$.

\begin{lem}[Analysis of the asymptotic problem]
    \label{lem:gordon_eq}
    In the separable regime $\delta < \delta^*(0)$, $(\rho^*, \beta_0^*, \kappa^*)$ is the unique solution to the system of equations
    \begin{subequations}
    \begin{align}
    \label{eq:SVM_sys_eq_rho}
        \pi \delta \cdot g \left( \frac{\rho}{2 \pi \norm{\bmu}_2 \delta} \right) & + (1 - \pi) \delta \cdot g \left( \frac{\rho}{2(1 - \pi) \norm{\bmu}_2 \delta} \right) = 1 - \rho^2, \\
    \label{eq:SVM_sys_eq_bk1}
    - \beta_0 + \kappa \tau & = \rho \norm{\bmu}_2 + g_1^{-1} \left( \frac{\rho}{2 \pi \norm{\bmu}_2 \delta} \right), \\
    \label{eq:SVM_sys_eq_bk2}
	\beta_0 + \kappa & = \rho \norm{\bmu}_2 + g_1^{-1} \left( \frac{\rho}{2 (1 - \pi) \norm{\bmu}_2 \delta} \right),
    \end{align}
    \end{subequations}
    where $\rho^* \in (0, 1)$ does not depend on $\tau$ and $\kappa^* > 0$.
\end{lem}
\begin{proof}
    Recall that in the proof of \cref{lem:over_phase_trans} and \ref{lem:over_logit_conv}, we established that $(\rho^*, \beta_0^*, \kappa^*)$ is the unqiue solution to
    \begin{equation*}
        \begin{array}{rl}
        \maximize\limits_{\rho \in [0, 1], \beta_0 \in \R, \kappa \in \R} & \kappa, \\
        \text{\emph{subject to}} & H_\kappa(\rho, \beta_0) \ge \delta.
        \end{array}
    \end{equation*}
    Let $F(\rho, \beta_0, \kappa) := F_\kappa(\rho, \beta_0)$, where $F_\kappa$ is defined in \cref{eq:T_F_}. Then the above optimization problem is equivalent to
    \begin{equation*}
        \begin{array}{rl}
        \maximize\limits_{\rho \in [0, 1], \beta_0 \in \R, \kappa \in \R} & \kappa, \\
        \text{\emph{subject to}} & F(\rho, \beta_0, \kappa) \le 0,
        \end{array}
    \end{equation*}
    Note $F$ is convex (since $x \mapsto (x)_+^2$ is a convex map, and expectation preserves convexity). Setting $\partial_{\rho} F = 0$ and $\partial_{\beta_0} F = 0$, we obtain the first-order conditions satisfied by $(\rho^*, \beta_0^*)$:
    \begin{equation}
        \label{eq:SVM_foc1}
    \begin{aligned}
    & \E \left[ \big( G - \rho \norm{\bmu}_2 - \beta_0 + \kappa \tau \big)_+ \right] = \frac{\rho}{2 \pi \norm{\bmu}_2 \delta}, \\
    & \E \left[ \big( G - \rho \norm{\bmu}_2 + \beta_0 + \kappa \big)_+ \right] = \frac{\rho}{2 (1 - \pi) \norm{\bmu}_2 \delta}.
    \end{aligned}
    \end{equation}
Moreover, we have $\delta^*(\kappa^*) = \delta$ and thus $F(\rho, \kappa, \beta_0) = 0$ at $(\rho^*, \kappa^*, \beta_0^*)$, which leads to
\begin{equation}\label{eq:SVM_foc2}
    \pi \delta \E \left[ \big( G - \rho \norm{\bmu}_2 - \beta_0 + \kappa \tau \big)_+^2 \right]  + (1-\pi) \delta \E \left[ \big( G - \rho \norm{\bmu}_2 + \beta_0 + \kappa \big)_+^2 \right] = 1 - \rho^2.
\end{equation}
Using $g_1$, $g_2$ defined in \cref{eq:fun_g1_g2}, the first-order conditions \cref{eq:SVM_foc1} can be rewritten as
\begin{equation}
\label{eq:SVM_foc1_ref}
\begin{aligned}
	& g_1 \left( - \rho \norm{\bmu}_2 - \beta_0 + \kappa \tau \right) = \frac{\rho}{2 \pi \norm{\bmu}_2 \delta}, \\
	& g_1 \left( - \rho \norm{\bmu}_2 + \beta_0 + \kappa \right) = \frac{\rho}{2 (1 - \pi) \norm{\bmu}_2 \delta}.
\end{aligned}
\end{equation}
Similarly, we recast \cref{eq:SVM_foc2} into
\begin{equation}
\label{eq:SVM_foc2_ref}
    \pi \delta g_2 \left( - \rho \norm{\bmu}_2 - \beta_0 + \kappa \tau \right) 
    + (1-\pi) \delta g_2 \left( - \rho \norm{\bmu}_2 + \beta_0 + \kappa \right) = 1 - \rho^2.
\end{equation}
By combining \cref{eq:SVM_foc1_ref} and \eqref{eq:SVM_foc2_ref}, we get \cref{eq:SVM_sys_eq_rho}. \cref{eq:SVM_sys_eq_bk1} and \eqref{eq:SVM_sys_eq_bk2} directly come from \cref{eq:SVM_foc1_ref}.

Note that function $g: \R_{> 0} \to \R_{ > 0}$ satisfies $g(0^+) = 0$. As $\rho$ varies from $0$ to $1$, the L.H.S. of \cref{eq:SVM_sys_eq_rho} increases from $0$ to a positive number while the R.H.S. decays to $0$, which guarantees the existence and uniqueness of $\rho^* > 0$. Since \cref{eq:SVM_sys_eq_rho} does not depend on $\tau$ and $\kappa^*$, we know that $\rho^*$ does not depend on $\tau$ and $\kappa^*$. This concludes the proof.
\end{proof}

In parallel to \cref{prop:SVM_tau_relation} for the original non-asymptotic problem, we provide the following similar result on the asymptotic problem \cref{eq:SVM_asymp_simple}.

\begin{cor}\label{cor:asymp_tau_relation}
        In the separable regime $\delta < \delta^*(0)$, let $(\rho^*(\tau), \beta_0^*(\tau), \kappa^*(\tau))$ be the optimal solution to \cref{eq:SVM_asymp_simple} under hyperparameter $\tau$. Then
        \begin{equation}\label{eq:asymp_tau_relation}
		    \rho^*(\tau) = \rho^*(1),
        \qquad
        \beta_0^*(\tau) = \beta_0^*(1) + \frac{\tau - 1}{\tau + 1} \kappa^*(1),
        \qquad
        \kappa^*(\tau) = \frac{2}{\tau + 1} \kappa^*(1).
	\end{equation}
\end{cor}
\begin{proof}
    Conclusion for $\rho^*$ is already shown in \cref{lem:gordon_eq}. For $\beta_0^*$ and $\kappa^*$, note that the R.H.S. of \cref{eq:SVM_sys_eq_bk1} and \eqref{eq:SVM_sys_eq_bk2} are constants under $\rho = \rho^*$ (depending on $\pi$, $\norm{\bmu}_2$ and $\delta$). Then we have
    \begin{align*}
        - \beta_0^*(\tau) + \kappa^*(\tau) \tau & = - \beta_0^*(1) + \kappa^*(1), \\
	\beta_0^*(\tau) + \kappa^*(\tau) & = \beta_0^*(1) + \kappa^*(1).
    \end{align*}
    Combining these two equations gives the expression of $\beta_0^*(\tau), \kappa^*(\tau)$ in terms of $\beta_0^*(1), \kappa^*(1)$ as in \cref{eq:asymp_tau_relation}, completing the proof.
\end{proof}

\subsection{Proof of \cref{prop:opt_transport}}

\begin{proof}[\textbf{Proof of \cref{prop:opt_transport}}]
We can prove a more general result by replacing $\cL_*^\mathrm{test}$ with $\mu$ and $\cL_*$ with $\nu := \Law(\max\{ \kappa^*, X \})$, where $X \sim \mu$ and $\mu$ is any probability measure with atomless (continuous) CDF $F_\mu$. As a special case, in \cref{prop:opt_transport} we consider $\mu$ as a mixture of two Gaussian distributions, and the cost function $c(x, y) = (x - y)^2$. 

We now prove the general statement. Note the CDF of $\nu$ has the form
\begin{equation*}
    F_\nu(t) :=  
    \begin{cases} 
    F_\mu(t) , & \ \text{if} \ t < \kappa^*, \\
    1, & \ \text{if} \ t \ge \kappa^*.
    \end{cases}
\end{equation*}
According to the optimal transport theory \cite[Theorem 2.5]{santambrogio2015optimal}, the unique (also monotone) optimal transport map from $\mu$ to $\nu$ is given by $\mathtt{T}^* := F_{\nu}^{-} \circ F_{\mu}$, where $F_{\nu}^{-}$ is the quantile function of $\nu$:
\begin{equation*}
    F_{\nu}^{-}(x) = \inf\left\{ t \in \R: F_{\nu}(t) \ge x \right\}
    = 
     \begin{cases} 
    F_\mu^{-1}(x) , & \ \text{if} \ x < F_\mu(\kappa^*), \\
    \kappa^*, & \ \text{if} \ x \ge F_\mu(\kappa^*).
    \end{cases}
\end{equation*}
Then we have $\mathtt{T}^*(x) := F_{\nu}^{-} \bigl( F_{\mu}(x) \bigr) = \max\{ \kappa^*, x \}$, which concludes the proof.
\end{proof}

%% file: src/append_underparam_prf.tex
\section{Logit distribution for non-separable data: Proofs for \cref{sec:logit_logistic}}
\label{append_sec:nonsep}

\subsection{Proof of \cref{thm:logistic_main}}

Throughout this section, we assume the loss function $\ell: \R \to \R_{\ge 0}$ is non-increasing, strictly convex, and twice differentiable. Based on these assumptions, we establish the following properties of $\ell$.
\begin{lem}\label{lem:ell}
    Let $\ell \in C^1(\R)$ be a nonnegative, non-increasing, and strictly convex function. Then
    \begin{enumerate}[label=(\alph*)]
        \item $\ell$ is strictly decreasing.
        \item $\ell(-\infty) = +\infty$ and $\ell(+\infty) = \underline{\ell}$ for some $\underline{\ell} \in [0, +\infty)$. 
    \end{enumerate}
\end{lem}
\begin{proof}
    Notice that $\ell'(u) \le 0$ (by non-increasing) and $\ell'(u)$ is strictly increasing (by strict convexity), which implies that $\ell'(u) < 0$ for all $u \in \R$ and hence deduces part (a). For part (b), the limits $\lim_{u \to \pm \infty} \ell(u)$ are well-defined, and $\ell(+\infty) = \underline{\ell}$ for some $\underline{\ell} \in [0, +\infty)$ since $\ell$ is monotone and bounded from below. It remains to show $\ell(-\infty) = +\infty$.

    Assume $\ell(-\infty) = \overline{\ell} < \infty$ by contradiction. By convexity, we have $\ell(u) \le \frac12( \ell(2u) + \ell(0) )$ for any $u \in \R$. Taking $u \to -\infty$ on both sides yields $\overline{\ell} \le \frac{1}{2}(\overline{\ell} + \ell(0))$, hence $\overline{\ell} \le \ell(0)$, which contradicts the fact that $\ell$ is strictly decreasing. Therefore, we must have $\ell(-\infty) = +\infty$.
\end{proof}
Without loss of generality, assume $\underline{\ell} := \ell(+\infty) = 0$. Otherwise, we can just consider $\ell - \underline{\ell}$ instead of $\ell$. In addition, we also assume $\ell$ is pseudo-Lipschitz, i.e., there exists a constant $L > 0$ such that, for all $x, y \in \R$,
\begin{equation*}
    \abs{ \ell(x) - \ell(y) } \le L \left( 1 + \abs{ x } + \abs{ y } \right) \abs{ x - y }.
\end{equation*}
For ease of exposition, we assume $\tau = 1$, as it is not fundamentally different from the case of arbitrary $\tau > 0$. In \cref{subsubsec:under_final}, we will discuss how to extend our proof to general $\tau > 0$.

Recall the original unconstrained empirical risk minimization (ERM) problem \cref{eq:logistic}:
\begin{equation}
\label{eq:ERM-0}
    M_n := 
    \min_{\bbeta \in \R^d, \, \beta_0 \in \R} \hat R_n(\bbeta, \beta_0)
    :=
    \min_{\bbeta \in \R^d, \, \beta_0 \in \R}  \frac1n \sum_{i=1}^n \ell\bigl( 
        y_i(\< \xx_i, \bbeta \> +  \beta_0 )
     \bigr).
\end{equation}

\noindent
We first provide an outline for the proof of \cref{thm:logistic_main}, which involves several intermediate steps of simplifying the random optimization problem $M_n$.
\begin{equation*}
\begin{aligned}
    M_{n}
    \, \xRightarrow[\text{\cref{lem:ERM_bound_beta}}]{\textbf{Step 1}} \,
    M_n(\bTheta_{\vbeta}, \bXi_{\bu})
    \, \xRightarrow[\text{\cref{lem:ERM_CGMT}}]{\textbf{Step 2}} \, 
    M_n^{(1)}(\bTheta_{\vbeta}, \bXi_{\bu})
    \, \Rightarrow \,
    M_n^{(2)}(\bTheta_{c}, \bXi_{\bu}) 
    \\
    \, \xRightarrow[\text{\cref{lem:M2-3}}]{\textbf{Step 3}} \, 
    M_n^{(3)}(\bTheta_{c}, \bXi_{\bu})
    \, \Rightarrow \,
    M_n^{(3)}(\bTheta_{c}) 
    \, \xRightarrow[\text{\cref{lem:M3-star}}]{\textbf{Step 4}} \, 
    M^*(\bTheta_{c})
    \, \Rightarrow \,
    M^*.
\end{aligned}
\left.
\vphantom{\begin{matrix} \dfrac12 \\ \dfrac12 \end{matrix}}
\right\} \text{\scriptsize{\cref{thm:ERM_conv}}}
\end{equation*}

\paragraph{Step 1: Boundedness of $\vbeta$ and $\beta_0$ (from $M_{n}$ to $M_n(\bTheta_{\vbeta}, \bXi_{\bu})$)}
Notice that by introducing the auxiliary variable $\bu = (u_1, \ldots, u_n)^\top \in \R^n$ and Lagrangian multiplier $\bv = (v_1, \ldots, v_n)^\top \in \R^n$, we can rewrite \cref{eq:ERM-0} as a minimax problem
\begin{align*}
        M_n & = \min_{ \substack{ \bbeta \in \R^d, \, \beta_0 \in \R \\  \bu \in \R^n } }
        \max_{ \bv \in \R^n }
        \left\{
        \frac1n \sum_{i=1}^n \ell( u_i )
         + \frac{1}{n} \sum_{i=1}^n v_i \bigl(  y_i(\< \xx_i, \bbeta \> +  \beta_0 ) - u_i\bigr)
         \right\} 
         \\
         & = \min_{ \substack{ \bbeta \in \R^d, \, \beta_0 \in \R \\  \bu \in \R^n } }
         \max_{ \bv \in \R^n }
         \left\{
         \frac1n \sum_{i=1}^n \ell( u_i )
          + \frac{1}{n} \sum_{i=1}^n v_i (  \< \bmu, \vbeta \> +  \< \zz_i, \vbeta \> + y_i \beta_0 - u_i )
          \right\},
\end{align*}
where in the second line, we reformulate $\xx_i = y_i(\vmu + \zz_i)$, $\zz_i \sim \normal(\bzero, \bI_d)$, $y_i \indep \zz_i$. For any closed subsets $\bTheta_{\vbeta} \subset \R^{d} \times \R$, $\bXi_{\bu} \subset \R^{n}$, we also define the quantity $M_n(\bTheta_{\vbeta}, \bXi_{\bu})$, which can be viewed as the constrained version of ERM problem $M_n$.
\begin{equation}
    \label{eq:Mn}
    \begin{aligned}
        M_n(\bTheta_{\vbeta}, \bXi_{\bu})
        :\! & = \min_{ \substack{ (\bbeta , \beta_0) \in \bTheta_{\vbeta} \\  \bu \in \bXi_{\bu} } }
    \max_{ \bv \in \R^n }
    \left\{
    \frac1n \sum_{i=1}^n \ell( u_i )
     + \frac{1}{n} \sum_{i=1}^n v_i (  \< \bmu, \vbeta \> +  \< \zz_i, \vbeta \> + y_i \beta_0 - u_i )
     \right\} \\
     & = \min_{ \substack{ (\bbeta , \beta_0) \in \bTheta_{\vbeta} \\  \bu \in \bXi_{\bu} } }
     \max_{ \bv \in \R^n }
     \left\{
     \frac1n \sum_{i=1}^n \ell( u_i )
      + \frac1n \bv^\top \bone \< \bmu, \vbeta \>
      + \frac1n \bv^\top \ZZ \vbeta + \frac1n \beta_0 \bv^\top \yy - \frac1n \bv^\top \bu
      \right\},
    \end{aligned}
\end{equation}
where $\ZZ = (\zz_1, \ldots, \zz_n)^\top \in \R^{n \times d}$. Let $(\hat\vbeta_n, \hat\beta_{0, n})$ be the unique minimizer of \cref{eq:ERM-0}. The following lemma implies that $\hat\vbeta_n$ and $\hat\beta_{0, n}$ are bounded with high probability, which enables us to work with $M_n(\bTheta_{\vbeta}, \bXi_{\bu})$ instead of $M_n$ for some compact sets $\bTheta_{\vbeta}$ and $\bXi_{\bu}$.

\begin{lem}[Boundedness of $\bbeta$ and $\beta_0$] \label{lem:ERM_bound_beta}
    In the non-separable regime $\delta > \delta^*(0)$, there exists some constants $C_{\bbeta}, C_{\beta_0}, C_{\bu} \in (0, \infty)$, such that $M_n = M_n(\bTheta_{\vbeta}, \bXi_{\bu})$ with high probability, where
    \begin{equation*}
    \bTheta_{\vbeta} = \{ (\vbeta, \beta_0) \in \R^d \times \R:  \norm{\vbeta}_2 \le C_{\vbeta},  \abs{\beta_0} \le C_{\beta_0}  \},
    \qquad
    \bXi_{\bu} = \{ \bu \in \R^n : \norm{\bu}_2 \le C_{\bu} \sqrt{n} \}.
    \end{equation*}
\end{lem}
\noindent
See \cref{subsubsec:under_step1} for the proof.

\paragraph{Step 2: Reduction via Gaussian comparison (from $M_n(\bTheta_{\vbeta}, \bXi_{\bu})$ to $M_n^{(1)}(\bTheta_{\vbeta}, \bXi_{\bu})$)}
The objective function of $M_n(\bTheta_{\vbeta}, \bXi_{\bu})$ in \cref{eq:Mn} is a bilinear form of the Gaussian random matrix $\ZZ$. To simplify the bilinear term, we will use the convex Gaussian minimax theorem (CGMT), i.e., Gordon's comparison inequality \cite{gordon1985some, thrampoulidis2015regularized}. To do so, we introduce another quantity:
\begin{equation*}
    \begin{aligned}
        M_n^{(1)}(\bTheta_{\vbeta}, \bXi_{\bu})
        : = \smash {\min_{ \substack{ (\bbeta , \beta_0) \in \bTheta_{\vbeta} \\  \bu \in \bXi_{\bu} } } }
        \max_{ \bv \in \R^n }
        \, \Biggl\{
        \frac1n \sum_{i=1}^n \ell( u_i )
         + \frac1n \bv^\top \bone \< \bmu, \vbeta \>
         & + \frac1n \norm{\bv}_2 \hh^\top \vbeta + \frac1n \norm{\vbeta}_2 \vg^\top \bv 
         \\
         & + \frac1n \beta_0 \bv^\top \yy - \frac1n \bv^\top \bu
         \Biggr\},
    \end{aligned}
\end{equation*}
where $\hh \sim \normal(\bzero, \bI_{d})$, $\vg \sim \normal(\bzero, \bI_{n})$ are independent Gaussian vectors. However, the classical CGMT cannot be directly applied to $M_n^{(1)}(\bTheta_{\vbeta}, \bXi_{\bu})$ since $\bv$ is maximized over an unbounded set. To this end, we proved the following version of CGMT, which connects $M_n^{(1)}(\bTheta_{\vbeta}, \bXi_{\bu})$ with $M_n(\bTheta_{\vbeta}, \bXi_{\bu})$.

\begin{lem}[CGMT, unbounded for maximum] \label{lem:ERM_CGMT}
For any compact sets $\bTheta_{\vbeta}$ and $\bXi_{\bu}$ (not necessarily convex) and $t \in \R$, we have
\begin{equation}\label{eq:unbounded_CGMT_1}
    \P \, \Bigl( M_n(\bTheta_{\vbeta}, \bXi_{\bu}) \le t \Bigr) \le 2 \, \P \,\Bigl( M_n^{(1)}(\bTheta_{\vbeta}, \bXi_{\bu}) \le t \Bigr).
\end{equation}
Additionally, if $\bTheta_{\vbeta}$ and $\bXi_{\bu}$ are convex, then
\begin{equation}\label{eq:unbounded_CGMT_2}
    \P \, \Bigl( M_n(\bTheta_{\vbeta}, \bXi_{\bu}) \ge t \Bigr) \le 2 \, \P \,\Bigl( M_n^{(1)}(\bTheta_{\vbeta}, \bXi_{\bu}) \ge t \Bigr).
\end{equation}
\end{lem}
\noindent
See \cref{subsubsec:under_step2} for the proof.

\paragraph{Reparametrization in low dimensions (from $M_n^{(1)}(\bTheta_{\vbeta}, \bXi_{\bu})$ to $M_n^{(2)}(\bTheta_{c}, \bXi_{\bu})$)}
To simplify $M^{(1)}_n(\bTheta_{\vbeta}, \bXi_{\bu})$, we consider the following change of variables
\begin{equation}\label{eq:change_of_var}
    \rho := \cos(\vmu, \vbeta) := 
    \begin{cases} 
    \, \displaystyle
    \left\< \frac{\vmu}{\norm{\vmu}_2}, \frac{\vbeta}{\norm{\vbeta}_2} \right\>, 
                  & \ \text{if} \ \vbeta \not= \bzero, \\
    \  0,         & \ \text{if} \ \vbeta = \bzero, 
    \end{cases}
    \qquad
    R := \norm{\vbeta}_2.
\end{equation}
Now, for any closed subset $\bTheta_{c} \subset [-1, 1] \times \R_{\ge 0} \times \R$, we define the quantity $M_n^{(2)}(\bTheta_{c}, \bXi_{\bu})$ by 
\begin{equation*}
    \begin{aligned}
        M_n^{(2)}(\bTheta_{\vbeta}, \bXi_{\bu})
        : = \smash{ \min_{ \substack{ (\bbeta , \beta_0) \in \R^d \times \R: 
        \\ (\cos(\vmu, \vbeta), \norm{\vbeta}_2, \beta_0) \in \bTheta_{c} \\  \bu \in \bXi_{\bu} } }
        }
        \max_{ \bv \in \R^n }
        \, \Biggl\{
        \frac1n \sum_{i=1}^n \ell( u_i )
         & + \frac1n \bv^\top \bone \< \bmu, \vbeta \>
         + \frac1n \norm{\bv}_2 \hh^\top \vbeta  
         \\
         & + \frac1n \norm{\vbeta}_2 \vg^\top \bv + \frac1n \beta_0 \bv^\top \yy - \frac1n \bv^\top \bu
         \Biggr\}.
    \end{aligned}
\end{equation*}
Therefore, $M_n^{(2)}(\bTheta_{c}, \bXi_{\bu})$ can be viewed as reparametrization of $M_n^{(1)}(\bTheta_{\vbeta}, \bXi_{\bu})$ when $\bTheta_{\vbeta} \subset \R^{d} \times \R$ takes the form
\begin{equation*}
        \bTheta_{\vbeta} = \left\{ 
        (\vbeta, \beta_0) \in \R^{d} \times \R:
        \bigl( \cos(\vmu, \vbeta), \norm{\vbeta}_2, \beta_0 \bigr) \in \bTheta_{c}
        \right\}.
\end{equation*}
Then we can simplify $M_n^{(2)}(\bTheta_{\vbeta}, \bXi_{\bu})$ as follows:
\begin{align}
        & M_n^{(2)}(\bTheta_{c}, \bXi_{\bu}) 
        \notag
        \\
        \overset{\mathmakebox[0pt][c]{\text{(i)}}}{=} {} & 
        \min_{ \substack{ (\rho, R, \beta_0) \in \bTheta_{c} \\ \bu \in \bXi_{\bu} } } \min_{ \substack{ \norm{\vbeta}_2 = R \\ \cos(\vmu, \vbeta) = \rho } }
        \max_{\gamma \ge 0} \max_{\norm{\bv_0}_2 = 1}
        \left\{ \frac1n \sum_{i=1}^n \ell( u_i ) +
        \frac{\gamma}{n} \bv_0^\top (\rho\norm{\vmu}_2 R \bone + R\vg + \beta_0 \yy  - \bu )
        + \frac{\gamma}{n} \hh^\top\vbeta 
        \right\} 
        \notag
        \\
        \overset{\mathmakebox[0pt][c]{\text{(ii)}}}{=} {} &  
        \min_{ \substack{ (\rho, R, \beta_0) \in \bTheta_{c} \\ \bu \in \bXi_{\bu} } } \min_{ \substack{ \norm{\vbeta}_2 = R \\ \cos(\vmu, \vbeta) = \rho } }
        \max_{\gamma \ge 0} 
        \left\{ \frac1n \sum_{i=1}^n \ell( u_i ) +
        \frac{\gamma}{n} \bigl\| \rho\norm{\vmu}_2 R \bone + R\vg + \beta_0 \yy  - \bu \bigr\|_2
        + \frac{\gamma}{n} \hh^\top\vbeta 
        \right\} 
        \notag
        \\
        \overset{\mathmakebox[0pt][c]{\text{(iii)}}}{=} {} & 
        \min_{ \substack{ (\rho, R, \beta_0) \in \bTheta_{c} \\ \bu \in \bXi_{\bu} } } 
        \max_{\gamma \ge 0} 
        \, \Biggl\{ 
            \frac1n \sum_{i=1}^n \ell( u_i )
        + \frac{\gamma}{n} \bigl\| \rho\norm{\vmu}_2 R \bone + R\vg + \beta_0 \yy  - \bu \bigr\|_2
        + \frac{\gamma}{n} \min_{ \substack{ \norm{\vbeta}_2 = R \\ \cos(\vmu, \vbeta) = \rho } } \hh^\top\vbeta  
        \Biggr\} 
        \notag
        \\
        \overset{\mathmakebox[0pt][c]{\text{(iv)}}}{=} {} & 
        \min_{ \substack{ (\rho, R, \beta_0) \in \bTheta_{c} \\ \bu \in \bXi_{\bu} } } 
        \max_{\gamma \ge 0} 
        \, \Biggl\{ 
            \frac1n \sum_{i=1}^n \ell( u_i )
        + \frac{\gamma}{n} \bigl\| \rho\norm{\vmu}_2 R \bone + R\vg + \beta_0 \yy  - \bu \bigr\|_2
        \notag
        \\
        & 
        \phantom{
        \min_{ \substack{ (\rho, R, \beta_0) \in \bTheta_{c} \\ \bu \in \bXi_{\bu} } } 
        \max_{\gamma \ge 0} 
        \, \Biggl\{ 
            \frac1n \sum_{i=1}^n \ell( u_i )
        }
        \, + \frac{\gamma}{n} R\biggl( \rho\frac{\hh^\top \vmu}{\norm{\vmu}_2} - \sqrt{1 - \rho^2} \| \bP_{\bmu}^\perp \hh \|_2 \biggr)  
        \Biggr\},
        \label{eq:Mn(2)}
\end{align}
where in (i) we apply the change of variables \cref{eq:change_of_var} and optimize $\bv$ by its length $\gamma$ and direction $\bv_0$ separately, (ii) follows from Cauchy--Schwarz inequality, (iii) is from the linearity of objective function in $\gamma$, and (iv) is based on direct calculation by decomposing $\vbeta$:
\begin{equation*}
    \min_{ \substack{ \vbeta \in \R^d: \norm{\vbeta}_2 = 1 \\ \cos(\vmu, \vbeta) = \rho } } \hh^\top\vbeta
    = \min_{ \vtheta \in \R^d: \substack{ \norm{\vtheta}_2 = 1 \\ \< \bmu, \vtheta \> = 0 } } \hh^\top 
    \biggl( \rho \frac{\bmu}{\norm{\vmu}_2} + \sqrt{1 - \rho^2} \vtheta \biggr)
    =  
    \rho \frac{\hh^\top \bmu}{\norm{\vmu}_2} - \sqrt{1 - \rho^2} \| \bP_{\bmu}^\perp \hh \|_2,
\end{equation*}
where $\bP_{\vmu}^\perp := \bI_d - \vmu \vmu^\top / \norm{\vmu}_2^2$.

\paragraph{Step 3: Convergence in variational forms (from $M_n^{(2)}(\bTheta_{c}, \bXi_{\bu})$ to $M_n^{(3)}(\bTheta_{c}, \bXi_{\bu})$)}
To proceed from \cref{eq:Mn(2)}, we adopt the following trick from \cite{montanari2023generalizationerrormaxmarginlinear}, where $\bu$ could be viewed as a functional of the empirical measure given by $\vg = (g_1, \ldots, g_n)^\top$ and $\yy = (y_1, \ldots, y_n)^\top$. Formally, let $\Q_n$ be the empirical distribution of the coordinates of $(\vg, \yy)$, i.e., the probability measure on $\R^2$ defined by
\begin{equation*}
    \Q_n := \frac{1}{n}\sum_{i=1}^n \delta_{(g_i, y_i)}.
\end{equation*}
Let $\cL^2(\Q_n) := \cL^2(\Q_n, \R^2)$ be the space of functions $U: \R^2 \to \R$, $(g, y) \mapsto U(g, y)$ that are square integrable with respect to $\Q_n$. Notice that the $n$ points that form $\Q_n$ are almost surely distinct, and therefore we can identify this space with the space of vectors $\bu \in \R^n$. We also define the two random variables in the same space by $G(g, y) = g$, $Y(g, y) = y$. Denote $\E_{\Q_n}$, $\norm{\,\cdot\,}_{\Q_n}$ the integral and norm with respect to $\Q_n$ in $\cL^2(\Q_n)$, i.e.,
\begin{equation*}
    \E_{\Q_n}[U] := \int_{\R^2} U(g,y) \, \d \Q_n(g, y)
    = \frac{1}{n} \sum_{i=1}^n U(g_i, y_i)
    ,
    \qquad
    \norm{U}_{\Q_n} := (\E_{\Q_n}[U^2])^{1/2}.
\end{equation*}
Let $\Xi_{u} \subseteq \cL^2(\Q_n)$ be the corresponding subset identified by $\bXi_{\bu} \subseteq \R^n$, that is,
\begin{equation*}
    \Xi_{u} := \left\{ U \in \cL^2(\Q_n) : \bu := \left( U(g_1, y_1), \ldots, U(g_n, y_n) \right)^\top \in \bXi_{\bu}  \right\}. 
\end{equation*}
Then with these definitions, we can rewrite the expression of $M_n^{(2)}(\bTheta_{c}, \bXi_{\bu})$ as
\begin{align*}
        M_n^{(2)}(\bTheta_{c}, \bXi_{\bu}) 
        & = \min_{ \substack{ (\rho, R, \beta_0) \in \bTheta_{c} \\  U \in \Xi_{u} } } 
        \max_{\gamma \ge 0}
        \, \Biggl\{ 
                \E_{\Q_n}[\ell(U)]
        + \frac{\gamma}{\sqrt{n}} \bigl\| \rho\norm{\vmu}_2 R + R G + \beta_0 Y - U \bigr\|_{\Q_n}
            \\
        & \phantom{.} \phantom{ 
            = \min_{ \substack{ (\rho, R, \beta_0) \in \bTheta_{c} \\ U \in \Xi_{u} } } 
        \max_{\gamma \ge 0}
        \, \Biggl\{ 
        }
        + \frac{\gamma}{n} R\biggl( \rho\frac{\hh^\top \vmu}{\norm{\vmu}_2} - \sqrt{1 - \rho^2} \| \bP_{\bmu}^\perp \hh \|_2 \biggr)  
            \Biggr\}
        \\
        &   
        = \min_{ (\rho, R, \beta_0) \in \bTheta_{c} }
        \min_{ U \in \Xi_{u}  \cap   \mathcal{N}_n }
        \E_{\Q_n}[\ell(U)],
\end{align*}
where we define the (stochastic) subset $\mathcal{N}_n = \mathcal{N}_n(\rho, R, \beta_0)$ by
\begin{equation}
\label{eq:set_N_n}
    \mathcal{N}_n := \left\{
        U \in \cL^2(\Q_n): 
        \bigl\| \rho\norm{\vmu}_2 R + R G + \beta_0 Y - U \bigr\|_{\Q_n}
        \le \frac{R}{\sqrt{n}} \biggl( \sqrt{1 - \rho^2}  \| \bP_{\bmu}^\perp \hh \|_2  -  \rho\frac{\hh^\top \vmu}{\norm{\vmu}_2}  \biggr)
     \right\}.
\end{equation}
It can be shown that as $n, d \to \infty$,
\begin{equation*}
    \frac{R}{\sqrt{n}} \biggl( \sqrt{1 - \rho^2}  \| \bP_{\bmu}^\perp \hh \|_2  -  \rho\frac{\hh^\top \vmu}{\norm{\vmu}_2}  \biggr)
    \conp \frac{R \sqrt{1 - \rho^2}}{\sqrt{\delta}}.
\end{equation*}
This convergence then motivates us to define another quantity
\begin{equation}
    \label{eq:Mn(3)}
        M_n^{(3)}(\bTheta_{c}, \bXi_{\bu}) 
        :=  
        \min_{ (\rho, R, \beta_0) \in \bTheta_{c} } 
        \min_{ U \in \Xi_{u}  \cap   \mathcal{N}^\delta_n }
        \E_{\Q_n}[\ell(U)],
\end{equation}
where the subset $\mathcal{N}^\delta_n = \mathcal{N}^\delta_n(\rho, R, \beta_0)$ is given by
\begin{equation}
\label{eq:set_N_n_delta}
    \mathcal{N}_n^\delta := \left\{
        U \in \cL^2(\Q_n): 
        \bigl\| \rho\norm{\vmu}_2 R + R G + \beta_0 Y - U \bigr\|_{\Q_n}
        \le \frac{R \sqrt{1 - \rho^2}}{\sqrt{\delta}}
     \right\}.
\end{equation}
The following lemma shows that $M_n^{(2)}$ and $M_n^{(3)}$ are close to each other:
\begin{lem}\label{lem:M2-3}
    For any compact sets $\bTheta_{c} \subset [-1, 1] \times \R_{\ge 0} \times \R$ and $\bXi_{\bu} \subset \R^{n}$ (not necessarily convex), as $n \to \infty$, we have
    \begin{equation*}
        \abs{ M_n^{(2)}(\bTheta_{c}, \bXi_{\bu}) - M_n^{(3)}(\bTheta_{c}, \bXi_{\bu}) } \conp 0.
    \end{equation*}
\end{lem}
\noindent
See \cref{subsubsec:under_step3} for the proof.

\paragraph{Step 4: Asymptotic characterization (from $M_n^{(3)}(\bTheta_{c}, \bXi_{\bu})$, $M_n^{(3)}(\bTheta_{c})$ to $M^*(\bTheta_{c})$, $M^*$)}
For any closed subsets $\bTheta_{c} \subset [-1, 1] \times \R_{\ge 0} \times \R$, we define the quantity $M_n^{(3)}(\bTheta_{c})$ by 
\begin{equation*}
    M_n^{(3)}(\bTheta_{c}) 
        :=  
        \min_{ (\rho, R, \beta_0) \in \bTheta_{c} } 
        \min_{ U \in \mathcal{N}^\delta_n }
        \E_{\Q_n}[\ell(U)].
\end{equation*}
Compared with \cref{eq:Mn(3)}, clearly $M_n^{(3)}(\bTheta_{c}, \bXi_{\bu}) = M_n^{(3)}(\bTheta_{c})$ when $\bXi_{\bu}$ is large enough. To analyze $M_n^{(3)}(\bTheta_{c})$, we consider the change of variable\footnote{
We will show in \cref{lem:M_star_var} later that the minimizer of $M_n^{(3)}(\bTheta_{c})$ must satisfy $R \sqrt{1 - \rho^2} > 0$, hence the change of variable $\xi$ can be well-defined.
}
\begin{equation*}
    \xi := - \frac{\rho\norm{\vmu}_2 R + RG + \beta_0 Y - U}{R\sqrt{1 - \rho^2}},
\end{equation*}
Then we have
\begin{equation*}
        M_n^{(3)}(\bTheta_{c})
    = \min_{ \substack{ (\rho, R, \beta_0) \in \bTheta_{c} \\ \xi \in \cL^2(\Q_n), \norm{\xi}_{\Q_n} \le 1/\sqrt{\delta} } } 
    \E_{\Q_n} \left[ \ell \bigl( \rho\norm{\vmu}_2 R + RG + \beta_0 Y + R\sqrt{1 - \rho^2} \xi \bigr) \right].
\end{equation*}
Denote $\Q_\infty := \P$ the population measure of $(G, Y)$ (so that $(G, Y) \sim \normal(0, 1) \times P_y$ under $\Q = \Q_\infty$, and we have $\E_{\Q_\infty} := \E$, $\norm{U}_{\Q_\infty} := (\E[U^2])^{1/2}$). Then we also define the asymptotic counterpart of $M_n^{(3)}(\bTheta_{c})$ by replacing $\Q_n$ with $\Q_\infty$:
\begin{equation*}
    M^{*}(\bTheta_{c})
    := \min_{ \substack{ (\rho, R, \beta_0) \in \bTheta_{c} \\ \xi \in \cL^2(\Q_\infty), \norm{\xi}_{\Q_\infty} \le 1/\sqrt{\delta} } } 
    \E \left[ \ell \bigl( \rho\norm{\vmu}_2 R + RG + \beta_0 Y + R\sqrt{1 - \rho^2} \xi \bigr) \right].
\end{equation*}
The following lemma shows that $M_n^{(3)}(\bTheta_{c})$ converges to the deterministic quantity $M^{*}(\bTheta_{c})$:

\begin{lem}\label{lem:M3-star}
    For any compact subset $\bTheta_{c} \subset [-1, 1] \times \R_{\ge 0} \times \R$, as $n \to \infty$, we have
    \begin{equation*}
        M_n^{(3)}(\bTheta_{c}) \conp M^{*}(\bTheta_{c}).
    \end{equation*}
\end{lem}
\noindent
See \cref{subsubsec:under_step4} for the proof.

\vspace{0.5\baselineskip}
Finally, combining \cref{lem:ERM_CGMT}---\ref{lem:M3-star}, we obtain the following theorem.
\begin{thm}
\label{thm:ERM_conv}
    Consider any compact sets $\bTheta_{\vbeta}$ and $\bXi_{\bu}$ such that $\bTheta_{\vbeta}$ has the form of
    \begin{equation}\label{eq:Theta_link}
        \bTheta_{\vbeta} = \left\{ 
        (\vbeta, \beta_0) \in \R^{d} \times \R:
        \bigl( \cos(\vmu, \vbeta), \norm{\vbeta}_2, \beta_0 \bigr) \in \bTheta_{c}
        \right\}
    \end{equation}
    for some compact domain $\bTheta_{c} \subset [-1, 1] \times \R_{\ge 0} \times \R$ of $(\rho, R, \beta_0)$. Assume $\bXi_{\bu}$ is large enough. Then, for any $\veps > 0$, as $n \to \infty$ , we have
    \begin{equation*}
        \P \, \bigl( M_n(\bTheta_{\vbeta}, \bXi_{\bu}) \le M^*(\bTheta_{c}) - \veps \bigr) \to 0.
    \end{equation*}
    Further, if both $\bTheta_{\vbeta}$ and $\bXi_{\bu}$ are convex, then
    \begin{equation*}
        M_n(\bTheta_{\vbeta}, \bXi_{\bu}) \conp M^*(\bTheta_{c}).
    \end{equation*}
\end{thm}
\begin{proof}
According to \cref{lem:M2-3} and \ref{lem:M3-star}, we have $M_n^{(2)}(\bTheta_{c}, \bXi_{\bu}) \conp M^*(\bTheta_{c})$ for any compact sets $\bTheta_{c} \subset [-1, 1] \times \R_{\ge 0} \times \R$ and $\bXi_{\bu} \subset \R^{n}$ large enough such that $\Xi_{u} \subset \mathcal{N}^\delta_n$. When $\bTheta_{\vbeta}$ takes the form \cref{eq:Theta_link}, by CGMT \cref{lem:ERM_CGMT}, for any $\veps > 0$ we have
\begin{align*}
    \P \, \Bigl( M_n(\bTheta_{\vbeta}, \bXi_{\bu}) \le M^*(\bTheta_{c}) - \veps \Bigr) 
    & \le 2 \, \P \,\Bigl( M_n^{(1)}(\bTheta_{\vbeta}, \bXi_{\bu}) \le M^*(\bTheta_{c}) - \veps \Bigr) \\
    & = 2 \, \P \,\Bigl( M_n^{(2)}( \mathmakebox[\widthof{$\bTheta_{\vbeta}$}][l]{ \bTheta_{c} }, \bXi_{\bu}) \le M^*(\bTheta_{c}) - \veps \Bigr) \xrightarrow{n \to \infty} 0.
\end{align*}
If both $\bTheta_{\vbeta}$ and $\bXi_{\bu}$ are also convex, then we can similarly show that
\begin{equation*}
    \P \, \Bigl( M_n(\bTheta_{\vbeta}, \bXi_{\bu}) \ge M^*(\bTheta_{c}) + \veps \Bigr) 
    \le 2 \, \P \,\Bigl( M_n^{(2)}( \mathmakebox[\widthof{$\bTheta_{\vbeta}$}][l]{ \bTheta_{c} }, \bXi_{\bu}) \ge M^*(\bTheta_{c}) + \veps \Bigr) \xrightarrow{n \to \infty} 0.
\end{equation*}
Combining these implies $M_n(\bTheta_{\vbeta}, \bXi_{\bu}) \conp M^*(\bTheta_{c})$, which concludes the proof.
\end{proof}

\paragraph{Parameter convergence}

Next, we define $M^* := M^*( [-1, 1] \times \R_{\ge 0} \times \R )$ to be the unconstrained optimization problem \cref{eq:logistic_variation}, i.e.,
\begin{equation*}
    M^{*}
    = \min_{ \substack{ \rho \in [-1, 1], R \ge 0, \beta_0 \in \R \\ \xi \in \cL^2(\Q_\infty), \norm{\xi}_{\Q_\infty} \le 1/\sqrt{\delta} } } 
    \E \left[ \ell \bigl( \rho\norm{\vmu}_2 R + RG + \beta_0 Y + R\sqrt{1 - \rho^2} \xi \bigr) \right].
\end{equation*}
An analysis of the Karush--Kuhn--Tucker (KKT) conditions shows that $M^*$ has the unique solution $(\rho^*, R^*, \beta_0^*, \xi^*)$, with $\rho^* \in (0, 1)$, $R^* \in (0, \infty)$, and $\beta_0^* \in (-\infty, \infty)$. Combined with \cref{thm:ERM_conv}, it implies $M_n \conp M^*$, which leads to the convergence of parameters:
\begin{lem}[Parameter convergence] 
\label{lem:ERM_param_conv}
As $n, d \to \infty$, we have $M_n \conp M^*$, which implies
    \begin{equation*}
            \| \hat\vbeta_n \|_2 \conp R^*,
            \qquad
            \hat\rho_n = \biggl\< \frac{\hat \vbeta_n}{\| \hat \vbeta_n \|_2}, \frac{\vmu}{\| \vmu \|_2} \biggr\> \conp \rho^*,
            \qquad
            \hat\beta_{0,n} \conp \beta_0^*.
    \end{equation*}
\end{lem}
\noindent
See \cref{append_subsubsec:ERM_param} for the proof.

\paragraph{ELD convergence} Finally, to establish the ELD convergence, we use a proof strategy similar to that in \cref{lem:over_logit_conv} by first defining the following measures
\begin{align*}
    \hat \cL_{n} := \frac1n \sum_{i=1}^n \delta_{y_i ( \< \xx_i, \hat\vbeta \> + \hat\beta_{0} ) }, 
    \quad
    \cL_* := \Law \, (U^*)
    = \Law \, \bigl( \rho^*\norm{\vmu}_2 R^* + R^* G + \beta_0^* Y + R^* \sqrt{1 - \rho^*{}^2} \xi^* \bigr).
\end{align*}
Let $\mathsf{B}_{W_2}(\varepsilon)$ ($\varepsilon > 0$) be the $\varepsilon$-$W_2$ ball at $\cL_*$, i.e.,
\begin{equation*}
    \mathsf{B}_{W_2}(\varepsilon) := \left\{ \bu \in \R^n:   W_2 \biggl( 
        \frac{1}{n}\sum_{i=1}^n \delta_{u_i}, \cL_*
     \biggr)  < \varepsilon \right\}.
\end{equation*}
Then by showing that
\begin{equation*}
    \lim_{n \to \infty} \P \left( M_n(\R^{d + 1}, \mathsf{B}_{W_2}^c(\varepsilon) ) > M_n \right) = 1,
\end{equation*}
we can prove the convergence of logit margins $W_2( \hat \cL_{n}, \cL_* ) \conp 0$, and hence the ELD convergence. The result in summarized in the following lemma.

\begin{lem}[ELD convergence]
\label{lem:ERM_logit_conv}
    As $n, d \to \infty$, we have $W_2( \hat \cL_{n}, \cL_* ) \conp 0$ and $W_2 ( \hat \nu_{n}, \nu_* ) \conp 0$.
\end{lem}
\noindent
See \cref{append_subsubsec:ERM_logit} for the proof.

\subsubsection{Step 1 --- Boundedness of $\vbeta$ and $\beta_0$: Proof of \cref{lem:ERM_bound_beta}}
\label{subsubsec:under_step1}

\begin{proof}[\textbf{Proof of \cref{lem:ERM_bound_beta}}]
    We first assume $\hat\vbeta \not= \bzero$. By \cref{thm:SVM_main}\ref{thm:SVM_main_mar}, if $\delta > \delta^*(0)$, there exists $k \in [n]$ and constant $\overline{\kappa} > 0$, such that
    \begin{equation}
        \label{eq:neg_k_logit}
        y_k \biggl(  \biggl\< \xx_k, \frac{\hat\bbeta}{\|\hat\bbeta\|_2} \biggr\> + \frac{\hat\beta_0}{\|\hat\bbeta\|_2} \biggr) \le - \overline{\kappa}
    \end{equation}
    holds with high probability. Therefore, we have
    \begin{equation*}
            \ell(0) 
               \overset{\mathmakebox[0pt][c]{\smash{\text{(i)}}}}{\ge}
                \frac1n \sum_{i=1}^n \ell\bigl( 
                y_i(\< \xx_i, \hat\bbeta \> +  \hat\beta_0 ) \bigr)
               \overset{\mathmakebox[0pt][c]{\smash{\text{(ii)}}}}{\ge}
                \frac1n \ell\bigl( 
                y_k(\< \xx_k, \hat\bbeta \> +  \hat\beta_0 ) \bigr)
               \overset{\mathmakebox[0pt][c]{\smash{\text{(iii)}}}}{\ge}
                \frac1n \ell( - \overline{\kappa} \|\hat\bbeta\|_2 ),
    \end{equation*}
    where in (i) we note that $\hat R_n(\bzero, 0) \ge \hat R_n(\hat\bbeta, \hat\beta_0) = M_n$, in (ii) we use $\ell \ge 0$, and in (iii) we use \eqref{eq:neg_k_logit}. Clearly the above inequalities also hold for $\hat\vbeta = \bzero$. Notice that $\frac1n \ell( - \overline{\kappa} \|\hat\bbeta\|_2 ) \to +\infty$ as $\|\hat\bbeta\|_2 \to \infty$, which contradicts $\ell(0) < +\infty$. Hence, it implies $\|\hat\bbeta\|_2$ is bounded with high probability.

    Meanwhile, let $j, k \in [n]$ be any two indices $y_j = +1$, $y_k = -1$. Then as $\hat\beta_0 \to \pm\infty$, we have
    \begin{equation*}
        \ell(0) \ge \frac1n \sum_{i=1}^n \ell\bigl( 
            y_i(\< \xx_i, \hat\bbeta \> +  \hat\beta_0 ) \bigr)
        \ge
        \frac1n \ell\bigl( \< \xx_j, \hat\bbeta \> +  \hat\beta_0 \bigr)
        +
        \frac1n \ell\bigl( - \< \xx_k, \hat\bbeta \> - \hat\beta_0 \bigr)
        \to +\infty,
    \end{equation*}
    which leads to a contradiction. So $|\hat\beta_0|$ is also bounded with high probability.

    Finally, in the minimax representation of $M_n$, the optimal $\bu$ must satisfy $u_i = y_i(\< \xx_i, \hat\bbeta \> +  \hat\beta_0 )$ for all $i \in [n]$. Therefore, according to the tail bound of Gaussian matrices \cite[Corollary 7.3.3]{vershynin2018high},
    \begin{align*}
            \norm{\bu}_2 & = \| \yy \odot (\XX \hat\bbeta + \hat\beta_0 \bone_n ) \|_2
            = \| \< \bmu, \hat\bbeta \> \bone_n + \ZZ \hat\bbeta + \hat\beta_0 \yy \|_2 \\
            & \le \sqrt{n} \| \bmu \|_2 \| \hat\bbeta \|_2 +
            \| \ZZ \|_{\mathrm{op}} \| \hat\bbeta \|_2 + \sqrt{n} | \hat\beta_0 | \\
            & \le \sqrt{n} \| \bmu \|_2 \| C_{\bbeta} + \bigl( \sqrt{n}(1 + o(1)) + \sqrt{d} \bigr) C_{\bbeta} 
            + \sqrt{n} C_{\beta_0} \\
            & \le \sqrt{n} C_{\bu}
    \end{align*}
    with high probability, where $C_{\bu} > 0$ is some constant. This completes the proof.
\end{proof}

\subsubsection{Step 2 --- Reduction via Gaussian comparison: Proof of \cref{lem:ERM_CGMT}}
\label{subsubsec:under_step2}

\begin{proof}[\textbf{Proof of \cref{lem:ERM_CGMT}}]
    For $m \in \mathbb{N}_+$, denote $K_m = \{ \bv \in \R^n: \norm{\bv}_2 \le m \}$, and define
    \begin{align*}
        M_n (\bTheta_{\vbeta}, \bXi_{\bu}; K_m)
        & :=  \!\min_{ \substack{ (\bbeta , \beta_0) \in \bTheta_{\vbeta} \\  \bu \in \bXi_{\bu} } } \max_{ \bv \in K_m } \left\{
        \frac1n \sum_{i=1}^n \ell( u_i )
      + \frac1n \bv^\top \bone \< \bmu, \vbeta \>
      + \frac1n \bv^\top \ZZ \vbeta + \frac1n \beta_0 \bv^\top \yy - \frac1n \bv^\top \bu \right\} \! , \!\! 
      \\
        M_n^{(1)} (\bTheta_{\vbeta}, \bXi_{\bu}; K_m) 
        & :=  \!\min_{ \substack{ (\bbeta , \beta_0) \in \bTheta_{\vbeta} \\  \bu \in \bXi_{\bu} } }
        \max_{ \bv \in K_m }
        \, \Biggl\{
        \frac1n \sum_{i=1}^n \ell( u_i )
        + \frac1n \bv^\top \bone \< \bmu, \vbeta \>
         + \frac1n \norm{\bv}_2 \hh^\top \vbeta + \frac1n \norm{\vbeta}_2 \vg^\top \bv 
         \\
         &
         \phantom{:=  \!\min_{ \substack{ (\bbeta , \beta_0) \in \bTheta_{\vbeta} \\  \bu \in \bXi_{\bu} } }
        \max_{ \bv \in K_m }
        \, \Biggl\{}
         + \frac1n \beta_0 \bv^\top \yy - \frac1n \bv^\top \bu
         \Biggr\}.
    \end{align*}
    We first show that
    \begin{equation*}
        \lim_{m \to \infty} M_n (\bTheta_{\vbeta}, \bXi_{\bu}; K_m) = M_n (\bTheta_{\vbeta}, \bXi_{\bu}).
    \end{equation*}
    To this end, note that for any fixed $(\bbeta , \beta_0, \bu)$, by Cauchy--Schwarz inequality we have
    \begin{align}
        & \max_{ \bv \in K_m } \left\{
        \frac1n \sum_{i=1}^n \ell( u_i )
        + \frac1n \bv^\top \bone \< \bmu, \vbeta \>
        + \frac1n \bv^\top \ZZ \vbeta + \frac1n \beta_0 \bv^\top \yy - \frac1n \bv^\top \bu \right\} 
        \notag \\
        = {} & \frac1n \sum_{i=1}^n \ell( u_i ) + \frac{m}{n} \bigl\|\bu - \< \bmu, \vbeta \> \bone - \ZZ \vbeta - \beta_0 \yy \bigr\|_2.
        \label{eq:ell_Mn_ineq}
    \end{align}
    Let $(\bbeta_*^{(m)}, \beta_{0, *}^{(m)}, \bu_*^{(m)})$ be the minimizer of $M_n (\bTheta_{\vbeta}, \bXi_{\bu}; K_m)$. Since $\ell \ge 0$, we know that
    \begin{align*}
        & \frac{m}{n} \norm{\bu_*^{(m)} - \< \bmu, \vbeta_*^{(m)} \> \bone - \ZZ \vbeta_*^{(m)} - \beta_{0, *}^{(m)} \yy}_2 \le \, M_n (\bTheta_{\vbeta}, \bXi_{\bu}; K_m) \le M_n (\bTheta_{\vbeta}, \bXi_{\bu}) \\
        \implies \  & \frac{ \mathmakebox[\widthof{$m$}][c]{1} }{n} \norm{\bu_*^{(m)} - \< \bmu, \vbeta_*^{(m)} \> \bone - \ZZ \vbeta_*^{(m)} - \beta_{0, *}^{(m)} \yy}_2 \le \, \frac{1}{m} M_n (\bTheta_{\vbeta}, \bXi_{\bu}).
    \end{align*}
    Let $\bu' := \< \bmu, \vbeta_*^{(m)} \> \bone + \ZZ \vbeta_*^{(m)} + \beta_{0, *}^{(m)} \yy$, then we have
    \begin{equation}\label{eq:u_diff_Mn}
        \frac{1}{n} \, \bigl\| \bu_*^{(m)} - \bu' \bigr\|_2 \le \, \frac{1}{m} M_n (\bTheta_{\vbeta}, \bXi_{\bu}),
    \end{equation}
    which implies that ($\bu_*^{(m)} = (u_{*, 1}^{(m)}, \ldots, u_{*, n}^{(m)})^\top$, $\bu' = (u'_1, \ldots, u'_n)^\top$)
    \begin{align*}
        M_n (\bTheta_{\vbeta}, \bXi_{\bu}) 
        & = \min_{ \substack{ (\bbeta , \beta_0) \in \bTheta_{\vbeta} \\  \bu \in \bXi_{\bu} } }
        \left\{
        \frac1n \sum_{i=1}^n \ell( u_i )  \,\bigg|\,
         \< \bmu, \vbeta \> +  \< \zz_i, \vbeta \> + y_i \beta_0 - u_i = 0, \forall\, i \in [n]
         \right\}
        \\
        & \le
        \frac1n \sum_{i=1}^n \ell( u'_i ) \le \frac1n \sum_{i=1}^n \ell( u_{*, i}^{(m)} ) + \frac{1}{n} \sum_{i=1}^{n} \left\vert \ell( u_{*, i}^{(m)} ) - \ell( u'_i ) \right\vert 
        \\
        & \overset{\mathmakebox[0pt][c]{\text{(i)}}}{\le} 
        \frac1n \sum_{i=1}^n \ell( u_{*, i}^{(m)} ) + \frac{C_L}{n} \bigl\| \bu_*^{(m)} - \bu' \bigr\|_1
        \\
        & \overset{\mathmakebox[0pt][c]{\text{(ii)}}}{\le} 
        \frac1n \sum_{i=1}^n \ell( u_{*, i}^{(m)} ) + O_m \left( \frac{1}{m} \right) 
        \overset{\mathmakebox[0pt][c]{\text{(iii)}}}{\le}
        M_n (\bTheta_{\vbeta}, \bXi_{\bu}; K_m) + O_m \left( \frac{1}{m} \right),
    \end{align*}
    where (i) follows from the pseudo-Lipschitzness of $\ell$, the compactness of $\bXi_{\bu}$, and $C_L > 0$ is some constant, (ii) follows from \cref{eq:u_diff_Mn}, while (iii) follows from \cref{eq:ell_Mn_ineq}.
    This proves that
    \begin{equation*}
        \lim_{m \to \infty} M_n (\bTheta_{\vbeta}, \bXi_{\bu}; K_m) = M_n (\bTheta_{\vbeta}, \bXi_{\bu}).
    \end{equation*}
    Similarly, one can show that
    \begin{equation*}
        \lim_{m \to \infty} M_n^{(1)} (\bTheta_{\vbeta}, \bXi_{\bu}; K_m) = M_n^{(1)} (\bTheta_{\vbeta}, \bXi_{\bu}).
    \end{equation*}
    Now for any fixed $m$, applying \cref{lem:CGMT}\ref{lem:CGMT(a)} yields that $\forall\, t \in \R$:
    \begin{equation*}
        \P \, \Bigl( M_n(\bTheta_{\vbeta}, \bXi_{\bu}; K_m) \le t \Bigr) 
        \le 2 \, \P \, \Bigl( M_n^{(1)}(\bTheta_{\vbeta}, \bXi_{\bu}; K_m) \le t \Bigr),
    \end{equation*}
    thus leading to \cref{eq:unbounded_CGMT_1} (by continuity and using the two limits above)
    \begin{align*}
        & \P \left( M_n(\bTheta_{\vbeta}, \bXi_{\bu}) \le t \right) =  
        \lim_{m \to \infty} \P \left( M_n(\bTheta_{\vbeta}, \bXi_{\bu}; K_m) \le t \right) \\
        \le {} & 2 \lim_{m \to \infty} \P \left( M_n^{(1)}(\bTheta_{\vbeta}, \bXi_{\bu}; K_m) \le t \right)
        =  2 \, \P \left( M_n^{(1)}(\bTheta_{\vbeta}, \bXi_{\bu}) \le t \right).
    \end{align*}
    Further, if $\bTheta_{\vbeta}$ and $\bXi_{\bu}$ are convex, \Cref{lem:CGMT}(b) implies that
    \begin{equation*}
        \P \, \Bigl( M_n(\bTheta_{\vbeta}, \bXi_{\bu}; K_m) \ge t \Bigr) 
        \le 2 \, \P \, \Bigl( M_n^{(1)}(\bTheta_{\vbeta}, \bXi_{\bu}; K_m) \ge t \Bigr).
    \end{equation*}
    Sending $m \to \infty$ similarly proves the other inequality \cref{eq:unbounded_CGMT_2}.
\end{proof}

\subsubsection{Step 3 --- Convergence in variational forms: Proof of \cref{lem:M2-3}}
\label{subsubsec:under_step3}

\begin{proof}[\textbf{Proof of \cref{lem:M2-3}}]
    First, by definition of $M_n^{(2)}$ and $M_n^{(3)}$:
    \begin{equation*}
        \abs{ M_n^{(2)}(\bTheta_{c}, \bXi_{\bu}) - M_n^{(3)}(\bTheta_{c}, \bXi_{\bu}) } \le \, \sup_{(\rho, R, \beta_0) \in \bTheta_{c}}
            \abs{ 
            \min_{  
                U \in \Xi_{u}  \cap   \mathcal{N}_n  }
                \E_{\Q_n}[\ell(U)]    
        -
        \min_{  
        U \in \Xi_{u}  \cap   \mathcal{N}^\delta_n  }
        \E_{\Q_n}[\ell(U)]
        }.
    \end{equation*}
    For any fixed $(\rho, R, \beta_0) \in \bTheta_{c}$, by definition of $\mathcal{N}_n$ in \cref{eq:set_N_n} and $\mathcal{N}^\delta_n$ in \cref{eq:set_N_n_delta}, we have
    \begin{equation*}
        \left\vert \min_{U \in \Xi_{u} \cap \mathcal{N}_n} \E_{\Q_n}[\ell(U)] - \min_{U \in \Xi_{u} \cap \mathcal{N}^\delta_n}
        \E_{\Q_n}[\ell(U)] \right\vert
        \le  
        \max_{\substack{ U, U' \in \Xi_u \cap \mathcal{N}_n \cap \mathcal{N}^\delta_n \\ 
        \| U - U' \|_{\Q_n} \le \veps_n (\rho, R, \beta_0)}} \left\vert \E_{\Q_n} [\ell(U)] - \E_{\Q_n} [\ell(U')] \right\vert,
    \end{equation*}
    where
    \begin{equation*}
        \veps_n (\rho, R, \beta_0) := \left\vert \frac{R}{\sqrt{n}} \biggl( \sqrt{1 - \rho^2}  \| \bP_{\bmu}^\perp \hh \|_2  -  \rho\frac{\hh^\top \vmu}{\norm{\vmu}_2}  \biggr) - R \frac{\sqrt{1 - \rho^2}}{\sqrt{\delta}} \right\vert.
    \end{equation*}
    By our assumption that $\ell$ is pseudo-Lipschitz, the following estimate holds:
    \begin{align*}
        \left\vert \E_{\Q_n} [\ell(U)] - \E_{\Q_n} [\ell(U')] \right\vert 
        & \le 
        \frac{1}{n} \sum_{i=1}^{n} \vert \ell(u_i) - \ell(u'_i) \vert 
        \le \frac{L}{n} \sum_{i=1}^{n} (1 + \vert u_i \vert + \vert u'_i \vert ) \vert u_i - u'_i \vert \\
        & \stackrel{\text{(i)}}{\le}  L \left( 1 + \|U\|_{\Q_n} + \|U'\|_{\Q_n} \right) \|U - U'\|_{\Q_n} 
        \stackrel{\text{(ii)}}{\le}
        C ( 1 + o_{\P}(1) ) \, \veps_n (\rho, R, \beta_0) ,
    \end{align*}
    where (i) follows from Cauchy--Schwarz inequality, (ii) follows from the compactness of $\mathcal{N}_n^\delta$ and $\bTheta_{c}$, and the upper bound below:
    \begin{align*}
        \norm{U}_{\Q_n} & \le  \sup_{(\rho, R, \beta_0) \in \bTheta_{c}} \bigl\| \rho\norm{\vmu}_2 R + R G + \beta_0 Y \bigr\|_{\Q_n}
        + \frac{R \sqrt{1 - \rho^2}}{\sqrt{\delta}} 
        \\
        & \le \rho\norm{\vmu}_2 R_{\max} + R_{\max} \norm{G}_{\Q_n} + B_{0, \max} \norm{Y}_{\Q_n} + \frac{R_{\max}}{\sqrt{\delta}} \\
        & \overset{\mathmakebox[0pt][c]{\smash{\text{($*$)}}}}{=} 
        \rho\norm{\vmu}_2 R_{\max} + R_{\max} \bigl( 1 + o_{\P}(1) \bigr) + B_{0, \max} + \frac{R_{\max}}{\sqrt{\delta}},
    \end{align*}
    by denoting $R_{\max} := \max_{(\rho, R, \beta_0) \in \bTheta_{c}} R$, $B_{0, \max} := \max_{(\rho, R, \beta_0) \in \bTheta_{c}} \abs{\beta_0}$, and $C > 0$ is some constant. Here, ($*$) is from the law of large numbers: $\norm{G}_{\Q_n} \conp \norm{G}_{\Q_\infty} = (\E[G^2])^{1/2} = 1$. Combining these estimates, we finally deduce that
    \begin{align*}
        & \abs{M_n^{(2)}(\bTheta_{c}, \bXi_{\bu}) - M_n^{(3)}(\bTheta_{c}, \bXi_{\bu})} 
        \le  C ( 1 + o_{\P}(1) ) \max_{(\rho, R, \beta_0) \in \bTheta_{c}} \veps_n (\rho, R, \beta_0) \\
        = {} & 
        C ( 1 + o_{\P}(1) ) \max_{(\rho, R, \beta_0) \in \bTheta_{c}} \left\vert \frac{R}{\sqrt{n}} \biggl( \sqrt{1 - \rho^2}  \| \bP_{\bmu}^\perp \hh \|_2  -  \rho\frac{\hh^\top \vmu}{\norm{\vmu}_2}  \biggr) - R \frac{\sqrt{1 - \rho^2}}{\sqrt{\delta}} \right\vert \\
        \le {} & 
        C ( 1 + o_{\P}(1) ) \cdot R_{\max} \left( \left\vert \frac{1}{\sqrt{n}} \| \bP_{\bmu}^\perp \hh \|_2 - \frac{1}{\sqrt{\delta}} \right\vert + \frac{1}{\sqrt{n}} \frac{\vert \hh^\top \vmu \vert}{\norm{\vmu}_2} \right) \conp 0.
    \end{align*}
    The convergence in the last line follows from
\begin{equation*}
    \frac{\| \bP_{\bmu}^\perp \hh \|_2}{\sqrt{n}} = \frac{\| \bP_{\bmu}^\perp \hh \|_2}{\| \bP_{\bmu}^\perp \|_\mathrm{F}}
    \cdot \frac{\sqrt{d-1}}{\sqrt{n}}
    \conp \frac{1}{\sqrt{\delta}},
    \qquad
    \frac{ \hh^\top \vmu }{ \sqrt{n} \norm{\vmu}_2 } \conp 0,
\end{equation*}
according to \cref{lem:subG_concentrate}\ref{lem:subG-Hanson-Wright-II}\ref{lem:subG-Hoeffding} and $\|\bP_{\bmu}^\perp \|_\mathrm{op} = 1$, $\| \bP_{\bmu}^\perp \|_\mathrm{F} = \sqrt{d - 1}$. This completes the proof.
\end{proof}

\subsubsection{Step 4 --- Asymptotic characterization: Proofs of \cref{lem:M3-star}, \ref{lem:var_fixed}}
\label{subsubsec:under_step4}

We need the following auxiliary result, which studies a general variational problem for both $\Q = \Q_n$ and $\Q = \Q_\infty$ with parameters $(\rho, R, \beta_0)$ fixed. In particular, we are able to express the random variable $\xi$ by $(\rho, R, \beta_0)$, $(G, Y)$, and an additional scalar (Lagrange multiplier). Then, we can rewrite $M_n^{(3)}(\bTheta_{c})$, $M^{*}(\bTheta_{c})$ as low-dimensional convex-concave minimax problems.

\begin{lem}\label{lem:var_fixed}
    For any fixed parameters $\rho \in (-1, 1)$, $R > 0$, $\beta_0 \in \R$, and the probability measure $\Q = \Q_n$ or $\Q = \Q_\infty$, consider the following variational problem
    \begin{equation}
    \label{eq:ERM_var_fix}
        \zeta_{\rho, R, \beta_0}(\Q) :=
        \min_{\xi \in \cL^2(\Q), \norm{\xi}_{\Q}^2 \le 1/\delta} \mathscr{R}_{\Q}(\xi),
        \quad
        \mathscr{R}_{\Q}(\xi) := \E_{\Q} \left[ \ell \bigl( \rho\norm{\vmu}_2 R + RG + \beta_0 Y + R\sqrt{1 - \rho^2} \xi \bigr) \right] \! .
    \end{equation}
    \begin{enumerate}[label=(\alph*)]
        \item \label{lem:var_fixed(a)}
        $\mathscr{R}_{\Q}(\xi)$ has a unique minimizer $\xi^* := \xi^*_{\Q}(\rho, R, \beta_0)$, which must satisfy
        \begin{equation}\label{eq:xi_star}
            \xi^*_{\Q}(\rho, R, \beta_0) = - \frac{\lambda^*}{R \sqrt{1 - \rho^2}}\ell'\bigl(\prox_{ \lambda^* \ell}( \rho\norm{\vmu}_2 R + RG + \beta_0 Y )\bigr),
        \end{equation}
        where $\lambda^*$ is the unique solution such that $\norm{\xi^*}_\Q^2 = 1/\delta$.
        As a consequence, we have
        \begin{equation*}
            \zeta_{\rho, R, \beta_0}(\Q)
                = 
            \E_{\Q} \left[ \ell \bigl( \prox_{ \lambda^* \ell}( \rho \norm{\vmu}_2 R + R G + \beta_0 Y )
                \bigr) \right],
        \end{equation*}
        where $\prox_{\lambda^*\ell}$ and $\envelope_{\lambda^*\ell}$ are the proximal operator and Moreau envelope of $\ell$ defined in \cref{append_subsec_Moreau}. Moreover, $\lambda^*$ is a decreasing function of $\delta$.
        \item \label{lem:var_fixed(b)}
        With change of variables $A := R \rho$, $B := R \sqrt{1 - \rho^2}$, the variational problem \cref{eq:ERM_var_fix} can be recast as $\zeta_{\rho, R, \beta_0}(\Q) = \sup_{\nu > 0} \mathscr{R}_{\nu, \Q}(A, B, \beta_0)$, where
        \begin{equation*}
            \mathscr{R}_{\nu, \Q}(A, B, \beta_0) 
            :=   - \frac{B \nu}{2 \delta }
            +
            \E_{\Q} \left[ \envelope_{\ell} \Bigl( A \norm{\vmu}_2 + A G_1 + B G_2 + \beta_0 Y ; \frac{B}{\nu} \Bigr) \right], 
        \end{equation*}
        and $(Y, G_1, G_2) \sim P_y \times \normal(0, 1) \times \normal(0, 1)$ under $\Q = \Q_\infty$.\footnote{According to the change of variables, we have relation $A G_1 + B G_2 \overset{\mathrm{d}}{=} R G$ under $\Q = \Q_\infty$. We can also construct the realizations $\{G_1(g_i, y_i), G_2(g_i, y_i)\}_{i=1}^n$ such that $A G_1 + B G_2 = R G$, $\Q_n$-a.s., that is, $A G_1(g_i, y_i) + B G_2(g_i, y_i) = R G(g_i, y_i)$, for all $i \in [n]$.} Moreover, $\mathscr{R}_{\nu, \Q}(A, B, \beta_0)$ is convex in $(A, B, \beta_0)$ over $\R_{>0} \times \R_{>0} \times \R$ and concave in $\nu$.
        
    \end{enumerate}
\end{lem}
\begin{proof}
    For \textbf{\ref{lem:var_fixed(a)}}, we first show the existence of a minimizer. The proof is a standard application of direct method in calculus of variations. Since $\ell$ is lower bounded, we know that
        \begin{equation*}
            \inf_{\xi \in \cL^2(\Q), \norm{\xi}_{\Q}^2 \le 1/\delta} \mathscr{R}_{\Q}(\xi) > - \infty.
        \end{equation*}
        Let $\{ \xi_m \}_{m \in \mathbb{N}} \in \cL^2 (\Q)$ be a minimizing sequence such that $\norm{\xi_m}_{\Q}^2 \le 1/\delta$, and
        \begin{equation*}
            \lim_{m \to \infty} \mathscr{R}_{\Q}(\xi_m) = \inf_{\xi \in \cL^2(\Q), \norm{\xi}_{\Q}^2 \le 1/\delta} \mathscr{R}_{\Q}(\xi).
        \end{equation*}
        Since $\cL^2 (\Q)$ is a Hilbert space (and hence self-reflexive), Banach-Alaoglu theorem implies that $\{ \xi_m \}$ has a weak-* convergent (and hence weak convergent) subsequence, which we still denote as $\{ \xi_m \}$. Let $\xi^*$ denote the weak limit of $\{ \xi_m \}$. By using Mazur's lemma, we know that there exists another sequence $\{ \xi'_m \}_{m \in \mathbb{N}}$, such that each $\xi'_m$ is a finite convex combination of $\{ \xi_k \}_{m \le k \le m + N(m)}$ ($N(m) \ge 0$ depends on $m$), and that $\xi'_m$ strongly converges to $\xi^*$. Now since $\mathscr{R}_{\Q}$ is convex (this follows from convexity of $\ell$ and the fact that integration $\E_{\Q}$ preserves convexity), we have
        \begin{equation*}
            \liminf_{m \to \infty} \mathscr{R}_{\Q} (\xi'_m) \le \liminf_{m \to \infty} \mathscr{R}_{\Q} (\xi_m) = \inf_{\xi \in \cL^2(\Q), \norm{\xi}_{\Q}^2 \le 1/\delta} \mathscr{R}_{\Q}(\xi).
        \end{equation*}
        On the other hand, Fatou's lemma implies that
        \begin{equation*}
            \mathscr{R}_{\Q}(\xi^*) \le \liminf_{m \to \infty} \mathscr{R}_{\Q}(\xi'_m).
        \end{equation*}
        This immediately leads to
        \begin{equation*}
            \mathscr{R}_{\Q}(\xi^*) = \inf_{\xi \in \cL^2(\Q), \norm{\xi}_{\Q}^2 \le 1/\delta} \mathscr{R}_{\Q}(\xi),
        \end{equation*}
        i.e., $\xi^*$ is a minimizer of $\mathscr{R}_{\Q}$.
        In order to prove uniqueness of the minimizer, we will show that $\mathscr{R}_{\Q} : \cL^2(\Q) \to \R_{> 0}$ is strictly convex. For any $\alpha \in (0, 1)$ and $\xi_1, \xi_2 \in \cL^2(\Q)$, with a shorthand $V := \rho\norm{\vmu}_2 R + RG + \beta_0 Y$, we notice that
    \begin{align*}
            & \mathscr{R}_{\Q}( \alpha \xi_1 + (1 - \alpha) \xi_2 )
            \\
            = {} & \E_{\Q} \left[  \ell \Bigl( \alpha \bigl(  V + R\sqrt{1 - \rho^2} \xi_1 \bigr) +  
            (1 - \alpha)\bigl( V + R\sqrt{1 - \rho^2} \xi_2 \bigr)
            \Bigr) \right] \\
            \le {} & \E_{\Q} \left[ \alpha \ell \bigl(  V + R\sqrt{1 - \rho^2} \xi_1 \bigr) +  
            (1 - \alpha) \ell\bigl( V + R\sqrt{1 - \rho^2} \xi_2 \bigr)
             \right] 
            = 
            \alpha \mathscr{R}_{\Q}( \xi_1 ) + (1 - \alpha) \mathscr{R}_{\Q}( \xi_2 ),
    \end{align*}
    where the inequality follows from strong convexity of $\ell$, and it becomes equality if and only if $\Q(\xi_1 \not= \xi_2) = 0$. Hence we conclude $\mathscr{R}_{\Q}$ is strictly convex. Since $\{ \xi: \norm{\xi}_{\Q}^2 \le 1/\delta \}$ is a convex set, it implies the uniqueness ($\Q$-a.s.) of the minimizer $\xi^*$.

    As a consequence, the unique minimizer is determined by the Karush--Kuhn--Tucker (KKT) and Slater's conditions for variational problems \cite[Theorem 2.9.2]{zalinescu2002convex}. $\xi$ is the minimizer if and only if, for some scalar $\nu$ (dual variable), the followings hold:
    \begin{equation}\label{eq:KKT_xi}
        \begin{aligned}
            U = \rho\norm{\vmu}_2 R + RG + \beta_0 Y + R\sqrt{1 - \rho^2} \xi,
            \qquad
            \ell' ( U ) + \nu \xi = 0, \\
            \norm{\xi}_{\Q}^2 - \delta^{-1} \le 0,
            \qquad
            \nu \ge 0,
            \qquad
            \nu(\norm{\xi}_{\Q}^2 - \delta^{-1}) = 0.
        \end{aligned}
    \end{equation}
    We claim that the KKT conditions imply that any minimizer $\xi$ and its associated dual variable $\nu$ must satisfy
    \begin{equation*}
        0 < \nu < \infty,  \qquad  \xi > 0 \ \  (\text{$\Q$-a.s.}), \qquad \norm{\xi}_{\Q}^2 = \delta^{-1}.
    \end{equation*}
    To show this, we notice that $R\sqrt{1 - \rho^2} > 0$ and $\ell$ is decreasing. Therefore, for any $\xi \in \cL^2(\Q)$, $\mathscr{R}_{\Q}(\xi) \ge \mathscr{R}_{\Q}(\abs{\xi})$. It implies that $\xi \ge 0$ if $\xi$ is the minimizer. Hence, by stationarity in \cref{eq:KKT_xi}: $\nu\xi = - \ell' ( U ) > 0$, which implies the positivity of $\nu$, $\xi$. Then $\norm{\xi}_{\Q}^2 = \delta^{-1}$ comes from complementary slackness in \cref{eq:KKT_xi}. To show $\nu$ must be finite, notice that $\nu \to +\infty$ implies $\ell'(U) \to -\infty$. Then $U \to -\infty$ since $\ell'$ is strictly increasing, while it contradicts $\xi > 0$ and $\norm{\xi}_{\Q}^2 = \delta^{-1}$.

    By change of variable $\lambda := R \sqrt{1 - \rho^2}/\nu$, now we can rewrite KKT conditions \cref{eq:KKT_xi} as
    \begin{equation}\label{eq:KKT_xi2}
            U + \lambda \ell' ( U ) = \rho\norm{\vmu}_2 R + RG + \beta_0 Y,
            \qquad
            0 < \lambda < \infty,
            \qquad
            \norm{\xi}_{\Q}^2 = \delta^{-1},
    \end{equation}
    where $\xi$ and $U$ are related by
    \begin{equation}\label{eq:xi-U}
        \xi = -\frac{\lambda}{R\sqrt{1 - \rho^2}} \ell'(U).
    \end{equation}
    Notice that \cref{eq:KKT_xi2} has a unique solution for $U$, since $x \mapsto x + \lambda \ell'(x)$ is a strictly increasing continuous function from $\R$ to $\R$, for any $\lambda \in (0, \infty)$. Then, according to \cref{lem:prox}, $U$ can be expressed by the proximal operator of $\ell$,
    \begin{equation}
    \label{eq:U_prox}
        U = \prox_{ \lambda \ell}( \rho\norm{\vmu}_2 R + RG + \beta_0 Y ).
    \end{equation}
    Combine it with \cref{eq:xi-U} gives the expression of $\xi^*$ in \cref{eq:xi_star}. To establish the uniqueness of $\lambda$, we show that $\nu$ satisfying \cref{eq:KKT_xi} must be unique. Note that $\xi = \xi (\nu)$ is determined by
    \begin{equation*}
        \nu \xi(\nu) + \ell' \left( \rho\norm{\vmu}_2 R + RG + \beta_0 Y + R\sqrt{1 - \rho^2} \xi (\nu) \right) = \, 0.
    \end{equation*}
    Since $\nu, \xi(\nu) > 0$ and $\ell'$ is strictly increasing (by strong convexity), we know that $\xi(\nu)$ is strictly decreasing in $\nu$. The uniqueness of $\nu$ immediately follows from the condition $\norm{\xi(\nu)}_{\Q}^2 = \delta^{-1}$. This also implies that $\xi(\nu) > 0$ is decreasing in $\delta$. Then we conclude $\nu$ is increasing in $\delta$, or equivalently $\lambda$ is decreasing in $\delta$. This completes the proof of part \ref{lem:var_fixed(a)}.

    \vspace{0.5\baselineskip}
    \noindent
    For \textbf{\ref{lem:var_fixed(b)}}, as a consequence we have
    \begin{align*}
    \zeta_{\rho, R, \beta_0}(\Q) & = \min_{\xi \in \cL^2(\Q), \norm{\xi}^2_{\Q} \le 1/\delta} \mathscr{R}_{\Q}(\xi) 
    \\
    & = \min_{\xi \in \cL^2(\Q), \norm{\xi}^2_{\Q} \le 1/\delta} \E_{\Q} \left[ \ell \bigl( \rho\norm{\vmu}_2 R + RG + \beta_0 Y + R\sqrt{1 - \rho^2} \xi \bigr) \right] \\
    & = \min_{\xi \in \cL^2(\Q)} \sup_{\nu \ge 0} \E_{\Q} \left[ \ell \bigl( \rho\norm{\vmu}_2 R + RG + \beta_0 Y + 
    R\sqrt{1 - \rho^2}\xi \bigr) + R\sqrt{1 - \rho^2} \cdot \frac{\nu}{2} \left( \xi^2 - \frac{1}{\delta} \right) \right] \\
    & \stackrel{\mathmakebox[0pt][c]{\smash{\text{(i)}}}}{=} 
    \sup_{\nu \ge 0} \min_{\xi \in \cL^2(\Q)} \E_{\Q} \left[ \ell \bigl( \rho\norm{\vmu}_2 R + RG + \beta_0 Y + 
    R\sqrt{1 - \rho^2}\xi \bigr) + R\sqrt{1 - \rho^2} \cdot \frac{\nu}{2} \left( \xi^2 - \frac{1}{\delta} \right) \right] \\
    & \stackrel{\mathmakebox[0pt][c]{\smash{\text{(ii)}}}}{=} 
    \sup_{\lambda > 0} \min_{U \in \cL^2(\Q)} \E_{\Q} \left[ \ell ( U ) + \frac{1}{2 \lambda} \left( U - \rho\norm{\vmu}_2 R - RG - \beta_0 Y \right)^2 - \frac{R^2 (1 - \rho^2)}{2 \lambda \delta} \right] \\
    & \stackrel{\mathmakebox[0pt][c]{\smash{\text{(iii)}}}}{=}
    \sup_{\lambda > 0} \left\{ \E_{\Q} \left[ \envelope_{\ell} \left( \rho\norm{\vmu}_2 R + RG + \beta_0 Y; \lambda \right) \right] - \frac{R^2 (1 - \rho^2)}{2 \lambda \delta} \right\},
\end{align*}
where (i) comes from strong duality in part \ref{lem:var_fixed(a)}, (ii) is by change of variable $U := \rho\norm{\vmu}_2 R + RG + \beta_0 Y + R\sqrt{1 - \rho^2} \xi$ and $\lambda = R\sqrt{1 - \rho^2}/\nu$, (iii) is from the definition of Moreau envelope \cref{eq:envelope}. Now, consider change of variable
\begin{equation*}
    A = R \rho, \qquad
    B = R \sqrt{1 - \rho^2}, \qquad
    \nu = R \sqrt{1 - \rho^2}/\lambda.
\end{equation*}
Note that $0 < \nu < \infty$ by part \ref{lem:var_fixed(a)}, then $\zeta_{\rho, R, \beta_0}(\Q)$ can be expressed as
\begin{equation*}
    \zeta_{\rho, R, \beta_0}(\Q) = 
    \min_{\xi \in \cL^2(\Q), \norm{\xi}^2_{\Q} \le 1/\delta} \mathscr{R}_{\Q}(\xi)  = 
        \sup_{\nu > 0} \,
        \mathscr{R}_{\nu, \Q}(A, B, \beta_0), 
\end{equation*}
where
\begin{equation*} 
    \mathscr{R}_{\nu, \Q}(A, B, \beta_0) 
    =   - \frac{B \nu}{2 \delta }
        +
        \E_{\Q} \left[ \envelope_{\ell} \Bigl( A \norm{\vmu}_2 + A G_1 + B G_2 + \beta_0 Y ; \frac{B}{\nu} \Bigr) \right].
\end{equation*}
Finally, we complete the proof by the following arguments:
\begin{itemize}
    \item $\mathscr{R}_{\nu, \Q}(A, B, \beta_0)$ is convex in $(A, B, \beta_0)$. It comes from \cref{lem:prox}\ref{lem:prox(a)} that $(x, \lambda) \mapsto \envelope_{\ell}(x; \lambda)$ is convex, and the fact that integration $\E_{\Q}$ preserves convexity.
    \item $\mathscr{R}_{\nu, \Q}(A, B, \beta_0)$ is concave in $\nu$. This comes from \cref{eq:envelope} that
    \begin{equation*}
        \envelope_{\ell} \Bigl( A \norm{\vmu}_2 + A G_1 + B G_2 + \beta_0 Y ; \frac{B}{\nu} \Bigr)
        = \min_{t \in \R} \left\{ 
        \ell(t) + \frac{\nu}{2B} (A \norm{\vmu}_2 + A G_1 + B G_2 - t)^2
        \right\},
    \end{equation*}
    with the fact that pointwise minimum and integration $\E_{\Q}$ preserves concavity. 
\end{itemize}
This concludes the proof of part \ref{lem:var_fixed(b)}.
\end{proof}

Then we can use \cref{lem:var_fixed} to show convergence $M_n^{(3)}(\bTheta_{c}) \conp M^{*}(\bTheta_{c})$ in \cref{lem:M3-star}.
\begin{proof}[\textbf{Proof of \cref{lem:M3-star}}]
    Recall the change of variables $A = R \rho$ and $B = R \sqrt{1 - \rho^2}$ defined in \cref{lem:var_fixed}\ref{lem:var_fixed(b)}.
    Note that $f: (\rho, R, \beta_0) \mapsto (R \rho, R \sqrt{1 - \rho^2}, \beta_0)$ is a continuous map. Then $f(\bTheta_{c}) \subset \R_{\ge 0} \times \R_{\ge 0} \times \R$ is still compact. Hence, by \cref{lem:var_fixed} we have
    \begin{equation*}
        M_n^{(3)}(\bTheta_{c}) =
        \min_{ (A, B, \beta_0) \in f(\bTheta_{c}) } \sup_{\nu > 0} \,
        \mathscr{R}_{\nu, \Q_n}(A, B, \beta_0),
        \quad
        M^*(\bTheta_{c}) =
        \min_{ (A, B, \beta_0) \in f(\bTheta_{c}) } \sup_{\nu > 0} \,
        \mathscr{R}_{\nu, \Q_\infty}(A, B, \beta_0).
    \end{equation*}
    For any fixed $A, B \ge 0$, $\beta_0 \in \R$, $\nu > 0$, by law of large numbers,
    \begin{align*}
        \mathscr{R}_{\nu, \Q_n}(A, B, \beta_0) 
        & = - \frac{B \nu}{2 \delta }
        +
        \E_{\Q_n} \left[ \envelope_{\ell} \Bigl( A \norm{\vmu}_2 + A G_1 + B G_2 + \beta_0 Y ; \frac{B}{\nu} \Bigr) \right] \\
        \conp \ \mathscr{R}_{\nu, \Q_\infty}(A, B, \beta_0)
        & = - \frac{B \nu}{2 \delta }
        +
        \E_{\phantom{\Q_n}} \left[ \envelope_{\ell} \Bigl( A \norm{\vmu}_2 + A G_1 + B G_2 + \beta_0 Y ; \frac{B}{\nu} \Bigr) \right]. 
    \end{align*}
    Recall $\mathscr{R}_{\nu, \Q_n}(A, B, \beta_0)$ is concave in $\nu$. Also, note that $\mathscr{R}_{\nu, \Q_\infty}(A, B, \beta_0) \to -\infty$ as $\nu \to \infty$, since by \cref{lem:prox}\ref{lem:prox(a)}, we have
    \begin{equation*}
        \lim_{\nu \to \infty} \E \left[ \envelope_{\ell} \Bigl( A \norm{\vmu}_2 + A G_1 + B G_2 + \beta_0 Y ; \frac{B}{\nu} \Bigr) \right]
        = \E \left[ \ell ( A \norm{\vmu}_2 + A G_1 + B G_2 + \beta_0 Y ) \right] < \infty.
    \end{equation*}
    This implies there exits $\overline{\nu} \in \R_{> 0}$, such that $\sup_{\nu \ge \overline{\nu}} \mathscr{R}_{\nu, \Q_\infty}(A, B, \beta_0) < \sup_{\nu > 0} \mathscr{R}_{\nu, \Q_\infty}(A, B, \beta_0)$. So, we can apply \cite[Lemma 10]{thrampoulidis2018precise} and conclude the uniform convergence
    \begin{equation*}
        \sup_{\nu > 0} \,
        \mathscr{R}_{\nu, \Q_n}(A, B, \beta_0)
        \ \conp \
        \sup_{\nu > 0} \,
        \mathscr{R}_{\nu, \Q_\infty}(A, B, \beta_0).
    \end{equation*}
    Recall that both $\sup_{\nu > 0} \mathscr{R}_{\nu, \Q_n}(A, B, \beta_0)$ and $\sup_{\nu > 0} \mathscr{R}_{\nu, \Q_\infty}(A, B, \beta_0)$ are convex in $(A, B, \beta_0)$ (since pointwise supremum preserves convexity). Then we could obtain uniform convergence on compact set $f(\bTheta_{c})$ by convexity \cite[Lemma 7.75]{liese2008statistical}:
    \begin{equation*}
    \abs{M_n^{(3)}(\bTheta_{c}) - M^*(\bTheta_{c})}
    \le
        \sup_{ (A, B, \beta_0) \in f(\bTheta_{c}) } \abs{
        \, \sup_{\nu > 0} \, \mathscr{R}_{\nu, \Q_n}(A, B, \beta_0)
        - \sup_{\nu > 0} \, \mathscr{R}_{\nu, \Q_\infty}(A, B, \beta_0)
        } \conp 0.
    \end{equation*}
    This completes the proof.
\end{proof}

\subsubsection{Parameter convergence and optimality analysis: Proofs of \cref{lem:boundedness_parameter}---\ref{lem:ERM_param_conv}}
\label{append_subsubsec:ERM_param}

Recall that
\begin{equation}
    \label{eq:M_star}
    M^{*}
    = \min_{ \substack{ \rho \in [-1, 1], R \ge 0, \beta_0 \in \R \\ \xi \in \cL^2(\Q_\infty), \norm{\xi}_{\Q_\infty} \le 1/\sqrt{\delta} } } 
    \E \left[ \ell \bigl( \rho\norm{\vmu}_2 R + RG + \beta_0 Y + R\sqrt{1 - \rho^2} \xi \bigr) \right],
\end{equation}
where $R$, $\beta_0$ are optimized over unbounded sets. The following lemma shows that any minimizer $R^*$, $\beta_0^*$ of \cref{eq:M_star} must be bounded.

\begin{lem}[Boundedness of $R^*$ and $\beta_0^*$]\label{lem:boundedness_parameter}
    Let $(\rho^*, R^*, \beta_0^*, \xi^*)$ be any minimizer of \cref{eq:M_star}. Then in the non-separable regime ($\delta > \delta^*(0)$), we have $R^* < \infty$ and $\abs{\beta_0^*} < \infty$.
\end{lem}
\begin{proof}
    We first prove the following claim: There exists an $\veps > 0$, such that for any $(a, b) \in \R_{> 0} \times \R$ satisfying $a^2 + b^2 = 1$, any $\rho \in [-1, 1]$, and any $\xi \in \cL^2 (\Q_\infty)$, $\norm{\xi}_{\Q_\infty} \le 1 / \sqrt{\delta}$:
    \begin{equation}
    \label{eq:V_claim}
        \P \left( a \rho \norm{\bmu}_2 + a G + b Y + a \sqrt{1 - \rho^2} \xi \le - \veps \right) \ge \veps.
    \end{equation}
    We prove this claim by contradiction. Assume it is not true, then for any $m \in \mathbb{N}$, there exists the corresponding $(a_m, b_m, \rho_m, \xi_m)$ such that $(a_m, b_m) \in \R_{> 0} \times \R$, with $a_m^2 + b_m^2 = 1$, $\rho_m \in [-1, 1]$, and $\xi_m \in \cL^2 (\Q_\infty)$, $\norm{\xi_m}_{\Q_\infty} \le 1 / \sqrt{\delta}$, which satisfy
    \begin{equation}
    \label{eq:Vm_bound}
        \P \left( a_m \rho_m \norm{\bmu}_2 + a_m G + b_m Y + a_m \sqrt{1 - \rho_m^2} \xi_m \le - \frac{1}{m} \right) < \frac{1}{m}.
    \end{equation}
    We can always assume that $(a_m, b_m, \rho_m) \to (a, b, \rho)$ and $\xi_m \to \xi$ weakly in $\cL^2 (\Q_\infty)$ when $m \to \infty$. Otherwise, such a convergent subsequence always exists according to Heine--Borel Theorem and Banach--Alaoglu Theorem. Therefore, $a_m \rho_m \norm{\bmu}_2 + a_m G + b_m Y + a_m \sqrt{1 - \rho_m^2} \xi_m$ weakly converges to $a \rho \norm{\bmu}_2 + a G + b Y + a \sqrt{1 - \rho^2} \xi$ in $\cL^2 (\Q_\infty)$. For any nonnegative $Z \in \cL^2 (\Q_\infty)$, one has
    \begin{align*}
        & \E \left[ \bigl( a \rho \norm{\bmu}_2 + a G + b Y + a \sqrt{1 - \rho^2} \xi \bigr) Z \right] \\
        = {} & \lim_{m \to \infty} \E \left[ \bigl( a_m \rho_m \norm{\bmu}_2 + a_m G + b_m Y + a_m \sqrt{1 - \rho_m^2} \xi_m \bigr) Z \right].
    \end{align*}
    Denote $U_m := a_m \rho_m \norm{\bmu}_2 + a_m G + b_m Y + a_m \sqrt{1 - \rho_m^2} \xi_m$, then we obtain the following estimate:
    \begin{align*}
        \E [U_m Z] & =  \E \left[ U_m \ind_{U_m > - 1/m} Z \right] + \E \left[ U_m \ind_{U_m \le - 1/m} Z \right] 
        \\ 
        & \ge - \frac{1}{m} \E [Z] - \left( \E[ U_m^2 ] \right)^{1/2}  \left(\E[ Z^2 \ind_{U_m \le - 1/m}] \right)^{1/2},
    \end{align*}
    where the last line follows from Cauchy--Schwarz inequality. By definition of $U_m$, we know that $\E [U_m^2]$ is uniformly bounded for any $m \in \mathbb{N}$. Further, since $Z \in \cL^2 (\Q_\infty)$ and $\P (U_m \le -1/m) \le 1/m \to 0$ as $m \to \infty$ by \cref{eq:Vm_bound}, we know that $\E [Z^2 \ind_{U_m \le - 1/m}] \to 0$. It finally follows that
    \begin{equation*}
        \E \left[ \bigl( a \rho \norm{\bmu}_2 + a G + b Y + a \sqrt{1 - \rho^2} \xi \bigr) Z \right] 
        = \lim_{m \to \infty} \E [U_m Z] \ge 0.
    \end{equation*}
    Since this is true for any nonnegative $Z \in \cL^2 (\Q_\infty)$, we know that
    \begin{equation*}
        a \rho \norm{\bmu}_2 + a G + b Y + a \sqrt{1 - \rho^2} \xi \ge 0, \quad \text{almost surely},
    \end{equation*}
    or equivalently, there exists $(\rho, R, \beta_0) \in [-1, 1] \times \R_{>0} \times \R$ and $\xi \in \cL^2 (\Q_\infty)$, $\E[\xi^2] \le 1/\delta$ satisfying
    \begin{equation*}
        R \rho \norm{\bmu}_2 + R G + \beta_0 Y + R \sqrt{1 - \rho^2} \xi \ge 0, \quad \text{almost surely}.
    \end{equation*}
    It implies the constraint of the variational problem for the separable regime (SVM) \cref{eq:SVM_variation}, i.e., $\rho \norm{\bmu}_2 + G + \beta_0' Y + \sqrt{1 - \rho^2} \xi \ge \kappa$ holds for some $\kappa \ge 0$ (with change of variable $\beta_0' := \beta_0 / R$). According to \cref{thm:SVM_main}\ref{thm:SVM_main_var}, we obtain $\kappa^* \ge 0$, or equivalently $\delta \le \delta^*(0)$, which contradicts the non-separable regime $\delta > \delta^* (0)$. Our claim \cref{eq:V_claim} is thus proved. 
    
    Now for any $(\rho, R, \beta_0, \xi)$ such that $R > 0$, denote
    \begin{equation*}
        V(\rho, R, \beta_0, \xi) := \frac{1}{\sqrt{R^2 + \beta_0^2}} \bigl( \rho \norm{\vmu}_2 R + R G + \beta_0 Y + R\sqrt{1 - \rho^2} \xi \bigr). 
    \end{equation*}
    We know that $\P (V(\rho, R, \beta_0, \xi) \le - \veps) \ge \veps$ by \cref{eq:V_claim}. Therefore,
    \begin{align*}
        & \E \left[ \ell \bigl( \rho \norm{\vmu}_2 R + R G + \beta_0 Y + R\sqrt{1 - \rho^2} \xi \bigr) \right]
        \\ 
          = {} & \E \left[ \ell \Bigl( \sqrt{ R^2 + \smash{\beta_0^2} } \, V(\rho, R, \beta_0, \xi) \Bigr) \right] \\
        \ge {} & \E \left[ \ell \Bigl( \sqrt{ R^2 + \smash{\beta_0^2} } \, V(\rho, R, \beta_0, \xi) \Bigr) \ind_{V(\rho, R, \beta_0, \xi) \le - \veps} \right] \\
        \ge {} & \veps \ell \Bigl( - \veps \sqrt{ R^2 + \smash{\beta_0^2} } \Bigr),
    \end{align*}
    which diverges to infinity as $R^2 + \beta_0^2 \to \infty$. This completes the proof.
\end{proof}

A direct consequence of \cref{lem:boundedness_parameter} is that $M^* = M^*(\bTheta_{c})$ for $\bTheta_{c}$ large enough. The following result shows that $M^*$ in \cref{eq:M_star} has a unique minimizer.

\begin{lem}
\label{lem:M_star_var}
    Consider the variational problem $M^*$ defined in \cref{eq:M_star}.
    \begin{enumerate}[label=(\alph*)]
        \item \label{lem:M_star_var(a)}
        $M^*$ has a unique minimizer $(\rho^*, R^*, \beta_0^*, \xi^*)$, which must satisfy
        \begin{equation*}
            \xi^* = - \frac{\lambda^*}{R^* \sqrt{1 - \rho^*{}^2}} \ell'\bigl(\prox_{ \lambda^* \ell}( \rho^*\norm{\vmu}_2 R^* + R^* G + \beta_0^* Y )\bigr),
        \end{equation*}
        where $\lambda^*$ is the unique solution such that $\E[\xi^{* 2}] = 1/\delta$. As a consequence, we have
        \begin{equation*}
           M^* = \E \left[ \ell \bigl( \prox_{ \lambda^* \ell}( \rho^*\norm{\vmu}_2 R^* + R^* G + \beta_0^* Y )
                \bigr) \right].
        \end{equation*}

        \item \label{lem:M_star_var(b)}
        $(\rho^*, R^*, \beta_0^*, \lambda^*)$ is also the unique solution to the system of equations
        \begin{equation}\label{eq:sys_eq_Q}
            \begin{aligned}
                - \frac{R \rho}{\lambda \delta \norm{\vmu}_2}
                & = 
                \E \left[ \ell'\bigl( \prox_{ \lambda \ell}( \rho\norm{\vmu}_2 R + RG + \beta_0 Y ) \bigr) \right],
                \\
                \frac{R}{\lambda \delta }
                & = 
                \E \left[ \ell'\bigl( \prox_{ \lambda \ell}( \rho\norm{\vmu}_2 R + RG + \beta_0 Y ) \bigr) G \right],
                \\
                0
                & = 
                \E \left[\ell'\bigl( \prox_{ \lambda \ell}( \rho\norm{\vmu}_2 R + RG + \beta_0 Y ) \bigr) Y \right],
                \\
                \frac{R^2 (1 - \rho^2)}{\lambda^2 \delta}
                & = 
                \E \left[ \left(\ell'\bigl( \prox_{ \lambda \ell}( \rho\norm{\vmu}_2 R + RG + \beta_0 Y ) \bigr) \right)^2 \right],
            \end{aligned}
        \end{equation}
        where $(\rho^*, R^*, \beta_0^*, \lambda^*) \in (0, 1) \times \R_{>0} \times \R \times \R_{> 0}$.

        \item \label{lem:M_star_var(c)}

        With change of variables $A := R \rho$, $B := R \sqrt{1 - \rho^2}$, the original variational problem \cref{eq:M_star} can be reduced to the following minimax problem
        \[ 
        M^* = 
        \min_{\substack{ A \ge 0, B \ge 0 \\ \beta_0 \in \R} }
        \sup_{\nu > 0} 
        \,
         \biggl\{ 
        - \frac{B \nu}{2 \delta }
        +
        \E \left[ \envelope_{\ell} \Bigl( A \norm{\vmu}_2 + A G_1 + B G_2 + \beta_0 Y ; \frac{B}{\nu} \Bigr) \right]
        \biggr\},
        \]
        where $(Y, G_1, G_2) \sim P_y \times \normal(0, 1) \times \normal(0, 1)$, and the objective function is convex-concave.
    \end{enumerate}
\end{lem}

\begin{proof}
    We first show the optimization problem \cref{eq:M_star} has a unique minimizer. Since its original formulation is non-convex, we make the following change of variables:
    \begin{equation}\label{eq:change_var_AB}
        A := R \rho, \qquad B := R \sqrt{1 - \rho^2}, \qquad \xi_B := B \xi.
    \end{equation}
    Then, the optimization problem is recast as
    \begin{equation}\label{eq:var_optim_Q_recast}
        \min_{
        \substack{ A , B \ge 0, \beta_0 \in \R \\ \xi_B \in \cL^2(\Q_\infty)}
        } \ \E \left[ \ell \bigl( A \norm{\vmu}_2 + A G_1 + B G_2 + \beta_0 Y + \xi_B \bigr) \right], \quad \text{subject to} \ \norm{\xi_B}_{\Q} \le \frac{B}{\sqrt{\delta}},
    \end{equation}
    which is convex, where $(Y, G_1, G_2) \sim P_y \times \normal(0, 1) \times \normal(0, 1)$ (recall that $A G_1 + B G_2 \overset{\smash{\mathrm{d}}}{=} R G$). Now we show that the above optimization problem has a unique minimizer. Note that \cref{lem:boundedness_parameter} also implies that any minimizer of this optimization problem is finite. Therefore, a similar argument as in the proof of \cref{lem:var_fixed}\ref{lem:var_fixed(a)} shows that \cref{eq:var_optim_Q_recast} has a unique minimizer. Since the mapping $(\rho, R, \xi) \mapsto (A, B, \xi_B)$ is one-to-one, this also proves the original optimization problem \cref{eq:M_star} has a unique minimizer.
    
    As a consequence, the unique minimizer is determined by the KKT and Slater's conditions for variational problems \cite[Theorem 2.9.2]{zalinescu2002convex}. $(A, B, \beta_0, \xi_B)$ is the minimizer of \cref{eq:var_optim_Q_recast} if and only if, for some scalar $\nu_B$ (Lagrange multiplier), the followings hold:
    \begin{equation}
    \label{eq:M_star_KKT}
    \begin{aligned}
        A \norm{\vmu}_2 + A G_1 + B G_2 + \beta_0 Y + \xi_B & = U, \\
        \E \left[ \ell'(U) ( \norm{\vmu}_2 + G_1 ) \right] & = 0, \\
        \E \left[ \ell'(U) G_2 \right] - \nu_B \frac{B}{\delta} & = 0, \\
        \E \left[ \ell'(U) Y   \right] & = 0, \\
        \ell'(U) + \nu_B \xi_B & = 0, \\
        \delta  \, \E[\xi_B^2] \le B^2, \quad
        \nu_B \ge 0, \quad
        \nu_B \bigl( \delta \, \E[\xi_B^2] - B^2 \bigr) & = 0.
    \end{aligned}
    \end{equation}
    Using a similar argument as in the proof of \cref{lem:var_fixed}\ref{lem:var_fixed(a)}, we can also show that
    \begin{equation*}
        0 < \nu_B < \infty,  \qquad  \xi_B > 0 \ \  (\text{a.s.}), \qquad \E[\xi_B^2] = B^2/\delta,
    \end{equation*}
    which implies $B > 0$. Plugging this into \cref{eq:M_star_KKT} solves two conditions
    \begin{equation}
    \label{eq:M_star_KKT-1}
        \E \, \bigl[ \bigl( \ell'(U) \bigr)^2 \bigr] = \nu_B^2 \frac{B^2}{\delta},
        \qquad
        \E\left[ \ell'(U) Y \right] = 0.
    \end{equation}
    By Stein's identity, we also have relation
    \begin{equation*}
        \E\left[ \ell'(U) G_1 \right] = A \, \E\left[ \ell''(U) \right],
        \qquad
        \E\left[ \ell'(U) G_2 \right] = B \, \E\left[ \ell''(U) \right].
    \end{equation*}
    Combine the above with \cref{eq:M_star_KKT}, we obtain
    \begin{align*}
        \E \left[ \ell'(U) \right] = -\nu_B \frac{A}{\delta\norm{\bmu}_2},
    \qquad
        \E \left[ \ell'(U) G_1 \right] = \nu_B \frac{A}{\delta},
    \qquad
        \E \left[ \ell'(U) G_2 \right] = \nu_B \frac{B}{\delta},
    \end{align*}
    which is equivalent to (recall that $A G_1 + B G_2 \overset{\smash{\mathrm{d}}}{=} R G$)
    \begin{equation}
    \label{eq:M_star_KKT-2}
        \E \left[ \ell'(U) \right] = -\nu_B \frac{A}{\delta\norm{\bmu}_2},
        \qquad
        \E \left[ \ell'(U) G \right] = \nu_B \frac{R}{\delta}.
    \end{equation}
    The above implies $A > 0$ since $\ell' < 0$ by \cref{lem:ell}. Since both $A, B > 0$, by \cref{eq:change_var_AB} we have $\rho \in (-1, 1) \setminus \{ 0 \}$ and $R > 0$. Moreover, notice that for any $\rho > 0$,
    \begin{equation*}
        \E \left[ \ell \bigl( -\rho\norm{\vmu}_2 R + RG + \beta_0 Y + R\sqrt{1 - \rho^2} \xi \bigr) \right]
        >
        \E \left[ \ell \bigl( \rho\norm{\vmu}_2 R + RG + \beta_0 Y + R\sqrt{1 - \rho^2} \xi \bigr) \right].
    \end{equation*}
    Therefore, we must have $\rho \in (0, 1)$. Then we prove $(\rho^*, R^*, \beta_0^*, \lambda^*) \in (0, 1) \times \R_{>0} \times \R \times \R_{> 0}$. Lastly, by combining \cref{eq:M_star_KKT-1} and \eqref{eq:M_star_KKT-2} with change of variable $\lambda := 1/\nu_B$, and recalling \cref{eq:U_prox} in the proof of \cref{lem:var_fixed}, we obtain the KKT conditions \cref{eq:sys_eq_Q} expressed in $(\rho, R, \beta_0, \lambda)$. Then we complete the proof of part \ref{lem:M_star_var(b)}. Finally, part \ref{lem:M_star_var(c)} directly follows from \cref{lem:var_fixed}.
\end{proof}

We are now in position to establish the convergence of parameters.
\begin{proof}[\textbf{Proof of \cref{lem:ERM_param_conv}}]
Consider any $\varepsilon \ge 0$ and $C_R, C_{\beta_0} \in (0, \infty)$, let
\begin{equation*}
    \bTheta_{c}^*(\varepsilon) := \left\{ (\rho, R, \beta_0) \in [-1, 1] \times [0, C_R] \times [-C_{\beta_0}, C_{\beta_0}] : 
    \norm{ (\rho, R, \beta_0) - (\rho^*, R^*, \beta_0^*) }_2 \ge \varepsilon
    \right\}
\end{equation*}
and let $\bTheta^*_{\vbeta}(\varepsilon)$ defined as \cref{eq:Theta_link}. By \cref{lem:ERM_bound_beta} and \ref{lem:M_star_var}, we can choose some $C_R, C_{\beta_0} > 0$ and compact convex set $\bXi_{\bu} \subset \R^{n}$ large enough, such that as $n, d \to \infty$,
\begin{equation*}
    M_n =
    M_n(\bTheta^*_{\vbeta}(0), \bXi_{\bu}) \ \ (\text{w.h.p.}) ,
    \qquad
    M^* =
    M^*(\bTheta^*_{c}(0)).
\end{equation*}
Then according to \cref{thm:ERM_conv}, we have global convergence
\begin{equation*}
    M_n \conp M^*.
\end{equation*}
However, for any $\varepsilon > 0$ and $\zeta > 0$, by \cref{thm:ERM_conv} we have
\begin{equation*}
    M_n(\bTheta^*_{\vbeta}(\varepsilon), \R^n) =
    M_n(\bTheta^*_{\vbeta}(\varepsilon), \bXi_{\bu}) \ \ (\text{w.h.p.}) ,
    \quad
    \P \left( M_n(\bTheta^*_{\vbeta}(\varepsilon), \bXi_{\bu}) \le M^*(\bTheta_{c}^*(\varepsilon)) - \zeta \right) \to 0.
\end{equation*}
This implies
\begin{equation*}
    \pliminf_{n \to \infty} M_n(\bTheta^*_{\vbeta}(\varepsilon), \R^n) \ge
    M^*(\bTheta_{c}^*(\varepsilon))
    >
    M^*,
\end{equation*}
where the strict inequality comes from the uniqueness of minimizer $(\rho^*, R^*, \beta_0^*, \xi^*)$, established in \cref{lem:M_star_var}\ref{lem:M_star_var(a)}. Since $\varepsilon > 0$ can be arbitrarily small, this proves $(\hat\rho_n, \| \hat\vbeta_n \|_2, \hat\beta_{0,n}) \conp (\rho^*, R^*, \beta_0^*)$. Moreover, we know that $R^* > 0$ by \cref{lem:M_star_var}\ref{lem:M_star_var(b)}. So $\hat\vbeta_n \not= \bzero$ and therefore $\hat\rho_n$ is well-defined with high probability. This concludes the proof of \cref{lem:ERM_param_conv}.
\end{proof}

\subsubsection{ELD convergence: Proof of \cref{lem:ERM_logit_conv}}
\label{append_subsubsec:ERM_logit}

\begin{proof}[\textbf{Proof of \cref{lem:ERM_logit_conv}}]
We first establish the convergence of logit margins. Recall that
\begin{align*}
    \hat \cL_{n} = \  & \frac1n \sum_{i=1}^n \delta_{y_i ( \< \xx_i, \hat\vbeta \> + \hat\beta_{0} ) },
    \\ 
    \cL_* = \  & \Law \, (U^*)
    := \Law \, \bigl( \rho^*\norm{\vmu}_2 R^* + R^* G + \beta_0^* Y + R^* \sqrt{1 - \rho^*{}^2} \xi^* \bigr) 
    \\
    = \ & \Law \, \bigl( \prox_{ \lambda^* \ell}( \rho^* \norm{\vmu}_2 R^* + R^* G + \beta_0^* Y ) \bigr).
\end{align*}
For any $\varepsilon > 0$ small enough, we have defined the $\varepsilon$-$W_2$ open ball by
\begin{equation*}
    \mathsf{B}_{W_2}(\varepsilon) = \left\{ \bu \in \R^n:   W_2 \biggl( 
        \frac{1}{n}\sum_{i=1}^n \delta_{u_i}, \cL_*
     \biggr)  < \varepsilon \right\}.
\end{equation*}
For $C_R, C_{\beta_0} \in (0, \infty)$, let $\bTheta_{c} = [-1, 1] \times [0, C_R] \times [-C_{\beta_0}, C_{\beta_0}]$ and let $\bTheta_{\vbeta}$ be defined as \cref{eq:Theta_link}. When $C_R, C_{\beta_0} > 0$ and compact set $\bXi_{\bu} \subset \R^{n}$ are large enough, by \cref{lem:ERM_bound_beta} we have
\begin{align*}
    \wt M_n^\varepsilon & := M_n(\R^{d + 1}, \mathsf{B}_{W_2}^c(\varepsilon) ) 
    = M_n(\bTheta_{\vbeta}, \bXi_{\bu} \setminus \mathsf{B}_{W_2}(\varepsilon) ) \ \ (\text{w.h.p.}) ,
    \\
    \wt M_n^{\varepsilon(3)} 
    & :=
    M_n^{(3)}(\bTheta_{c}, \mathsf{B}_{W_2}^c(\varepsilon) )
    = M_n^{(3)}(\bTheta_{c}, \bXi_{\bu} \setminus \mathsf{B}_{W_2}(\varepsilon) ).
\end{align*}
Combining these with \cref{lem:ERM_CGMT} and \ref{lem:M2-3} obtains that for any $\zeta > 0$,
\begin{equation}
    \label{eq:Mn-3-eps}
    \lim_{n \to \infty} \P \, \Bigl( \wt M_n^\varepsilon \le \wt M_n^{\varepsilon(3)} - \zeta \Bigr) = 0.
\end{equation}
In order to show $W_2( \hat \cL_{n}, \cL_* ) \conp 0$, our goal is to show that
\begin{equation*}
    \lim_{n \to \infty} \P \, \Bigl( \wt M_n^{\varepsilon} > M_n \Bigr) = 1.
\end{equation*}
Then according to \cref{eq:Mn-3-eps} and \cref{lem:ERM_param_conv}, it suffices to show that
\begin{equation}
    \label{eq:ERM_logit_goal}
    \pliminf_{n \to \infty} \wt M_n^{\varepsilon(3)}  > \plim_{n \to \infty} M_n = M^*.
\end{equation}
By \cref{eq:Mn(3)} and \eqref{eq:set_N_n_delta}, recall that
\begin{equation*}
    \wt M_n^{\varepsilon(3)} = \min_{ (\rho, R, \beta_0) \in \bTheta_{c} }
    \min_{\bu \in 
    \mathsf{N}^\delta_n(\rho, R, \beta_0) \setminus \mathsf{B}_{W_2}(\varepsilon) } 
    \frac1n \sum_{i=1}^n \ell(u_i),
\end{equation*}
where we temporarily define
\begin{equation*}
    \mathsf{N}^\delta_n(\rho, R, \beta_0) := \left\{
        \bu \in \R^n: 
        \frac{1}{\sqrt{n}} \bigl\| \rho\norm{\vmu}_2 R \bone_n + R \vg + \beta_0 \yy - \bu \bigr\|_2
        \le \frac{R \sqrt{1 - \rho^2}}{\sqrt{\delta}}
     \right\}.
\end{equation*}
Now we split $\wt M_n^{\varepsilon(3)}$ into two parts by
\begin{equation*}
    \wt M_n^{\varepsilon(3)} 
    = \min\left\{ I, I\!I \right\}
    := \min\left\{ 
        \min_{ \substack{ (\rho, R, \beta_0) \in \bTheta_{c} \setminus \mathsf{B}_{2,c^*}(\eta)
     \\ \bu \in 
    \mathsf{N}^\delta_n(\rho, R, \beta_0) \setminus \mathsf{B}_{W_2}(\varepsilon)  } } 
    \frac1n \sum_{i=1}^n \ell(u_i)
    ,
        \min_{ \substack{ (\rho, R, \beta_0) \in \bTheta_{c} \cap \mathsf{B}_{2,c^*}(\eta)
     \\ \bu \in 
    \mathsf{N}^\delta_n(\rho, R, \beta_0) \setminus \mathsf{B}_{W_2}(\varepsilon)  } } 
    \frac1n \sum_{i=1}^n \ell(u_i)
     \right\},
\end{equation*}
where $\eta > 0$ and
\begin{equation*}
\mathsf{B}_{2,c^*}(\eta) = \left\{ (\rho, R, \beta_0) \in \R^3 : \norm{(\rho, R, \beta_0) - (\rho^*, R^*, \beta^*_0)}_2 < \eta \right\}
\end{equation*}
is a $\eta$-$\cL^2$ open ball around the global minimizer $(\rho^*, R^*, \beta^*_0)$.

For the first term, with $\bTheta_{c}$ large enough such that $(\rho^*, R^*, \beta^*_0) \in \bTheta_{c}$, by \cref{lem:M3-star} we have
\begin{align*}
        I 
        & \ge    
        \min_{(\rho, R, \beta_0) \in \bTheta_{c} \setminus \mathsf{B}_{2,c^*}(\eta) }
        \min_{\bu \in 
    \mathsf{N}^\delta_n(\rho, R, \beta_0)  }
    \frac1n \sum_{i=1}^n \ell(u_i) 
       \\
       & =  \min_{ \substack{ (\rho, R, \beta_0) \in \bTheta_{c} \setminus \mathsf{B}_{2,c^*}(\eta) 
       \\ \xi \in \cL^2(\Q_n), \norm{\xi}^2_{\Q_n} \le 1/\delta } } 
       \E_{\Q_n} \left[ \ell \bigl( \rho\norm{\vmu}_2 R + RG + \beta_0 Y + R\sqrt{1 - \rho^2} \xi \bigr) \right] \\
       & =  M_n^{(3)}\bigl( \bTheta_{c} \setminus \mathsf{B}_{2,c^*}(\eta) \bigr)
       \conp  
       M_n^*\bigl( \bTheta_{c} \setminus \mathsf{B}_{2,c^*}(\eta) \bigr) >  M^*( \bTheta_{c} ) = M^*,
\end{align*}
where the strict inequality follows from the uniqueness of $(\rho^*, R^*, \beta^*_0)$ according to \cref{lem:M_star_var}\ref{lem:M_star_var(a)}.

For the second term, we can take $\eta > 0$ small enough, such that $(\rho, R, \beta_0) \in \mathsf{B}_{2,c^*}(\eta)$ implies
\begin{align*}
    &
    W_2 \, \Bigl( 
        \Law \left( U^*_{\rho, R, \beta_0} \right), \cL_*
     \Bigr) \\
    = {} &  
    W_2 \, \Bigl(
        \Law \left( U^*_{\rho, R, \beta_0} \right), \Law \, \bigl( U^*_{\rho^*, R^*, \beta_0^*} \bigr)
     \Bigr)
     \le \frac{\varepsilon}{2},
     \qquad
     \forall\, (\rho, R, \beta_0) \in \bTheta_{c} \cap \mathsf{B}_{2,c^*}(\eta),
\end{align*}
where $U^*_{\rho, R, \beta_0} := \rho\norm{\vmu}_2 R + RG + \beta_0 Y + R\sqrt{1 - \rho^2} \xi^*_{\Q_\infty}(\rho, R, \beta_0)$, and $\xi^*_{\Q_\infty}(\rho, R, \beta_0)$ is the unique minimizer of $\mathscr{R}_{\Q}(\xi)$ defined in \cref{eq:ERM_var_fix}, with an expression given by \cref{eq:xi_star}. The existence of such $\eta > 0$ is guaranteed by continuity of $W_2$ distance and $(\rho, R, \beta_0) \mapsto U^*_{\rho, R, \beta_0}$ by \cref{lem:var_fixed}. Then $\bu \notin \mathsf{B}_{W_2}(\varepsilon)$ implies (by triangle inequality)
\begin{equation*}
    W_2 \biggl( 
        \frac{1}{n}\sum_{i=1}^n \delta_{u_i}, \Law \left( U^*_{\rho, R, \beta_0} \right)
     \biggr)
     \ge \frac{\varepsilon}{2} 
     ,
     \qquad
     \forall\, (\rho, R, \beta_0) \in \bTheta_{c} \cap \mathsf{B}_{2,c^*}(\eta). 
\end{equation*} 
Thus we have
\begin{equation}
\label{eq:ERM_II0}
    I\!I = 
     \min_{ \substack{ (\rho, R, \beta_0) \in \bTheta_{c} \cap \mathsf{B}_{2,c^*}(\eta)
     \\ \bu \in 
    \mathsf{N}^\delta_n(\rho, R, \beta_0) \setminus \mathsf{B}_{W_2}(\varepsilon)  } } 
    \frac1n \sum_{i=1}^n \ell(u_i)
    \ge 
     \min_{ \substack{ (\rho, R, \beta_0) \in \bTheta_{c}
     \\ U \in 
    \mathcal{N}^\delta_n(\rho, R, \beta_0) \cap \mathcal{C}_n^\varepsilon(\rho, R, \beta_0)  } } 
    \E_{\Q_n}[\ell(U)],
\end{equation}
where denote
\begin{equation}
\label{eq:ERM_logit1}
    \mathcal{C}_n^\varepsilon(\rho, R, \beta_0) := \left\{ U \in \cL^2(\Q_\infty): 
    \norm{ U - U^*_{\rho, R, \beta_0} }_{\Q_n} \ge \frac{\varepsilon}{2}
    \right\}
\end{equation}
and recall \cref{eq:set_N_n_delta} that 
\begin{equation}
\label{eq:ERM_logit2}
    \mathcal{N}_n^\delta(\rho, R, \beta_0) = \left\{ 
        U \in \cL^2(\Q_n):  \bigl\| \rho\norm{\vmu}_2 R + RG + \beta_0 Y - U \bigr\|_{\Q_n}
        \le  \frac{ R\sqrt{1 - \rho^2} }{ \sqrt{\delta} }
     \right\}.
\end{equation}
Now, denote $\hat U_{\rho, R, \beta_0} := \rho\norm{\vmu}_2 R + RG + \beta_0 Y + R\sqrt{1 - \rho^2} \xi^*_{\Q_n}(\rho, R, \beta_0)$. According to \cref{lem:var_fixed}, we know that $\| \xi^*_{\Q_n}(\rho, R, \beta_0) \|_{\Q_n}^2 = 1/\delta$, that is,
\begin{equation}
\label{eq:ERM_logit3}
    \bigl\| \rho\norm{\vmu}_2 R + RG + \beta_0 Y - \hat U_{\rho, R, \beta_0} \bigr\|_{\Q_n}
        = \frac{ R\sqrt{1 - \rho^2} }{ \sqrt{\delta} }.
\end{equation}
We claim $\bigl\| U^*_{\rho, R, \beta_0} - \hat U_{\rho, R, \beta_0} \bigr\|_{\Q_n} \conp 0$. Otherwise, there exits a convergent sequence $\{ \hat\lambda_m \}_{m \in \mathbb{N}}$ such that $\plim_{m \to \infty} \hat\lambda_m \neq \lambda^*$, where $\hat\lambda_m$ satisfies the conditions in \cref{lem:var_fixed}\ref{lem:var_fixed(a)} under $\Q = \Q_m$, and $\lambda^*$ satisfies the conditions in \cref{lem:var_fixed}\ref{lem:var_fixed(a)} under $\Q = \Q_\infty$. This contradicts the convergence $\argmax_{\nu > 0} \mathscr{R}_{\nu, \Q_m}(A, B, \beta_0) \conp \argmax_{\nu > 0} \mathscr{R}_{\nu, \Q_\infty}(A, B, \beta_0)$ by an argmax theorem for the concave process \cite[Theorem 7.77]{liese2008statistical} according to \cref{lem:var_fixed}\ref{lem:var_fixed(b)}, and change of variable $\nu = R\sqrt{1 - \rho^2}/\lambda$. Hence, for all $n$ large enough, we have
\begin{equation*}
    \bigl\| U^*_{\rho, R, \beta_0} - \hat U_{\rho, R, \beta_0} \bigr\|_{\Q_n} \le \frac{\varepsilon}{2}.
\end{equation*}
Combining this with \cref{eq:ERM_logit1}---\eqref{eq:ERM_logit3} together, by triangle inequality, we obtain
\begin{equation}
\label{eq:ERM_logit_sets}
    \mathcal{N}^\delta_n(\rho, R, \beta_0) \cap \mathcal{C}_n^\varepsilon(\rho, R, \beta_0)
    \subseteq \wt{\mathcal{N}}^{\delta,\varepsilon}_n(\rho, R, \beta_0)
\end{equation}
where
\begin{equation*}
    \wt{\mathcal{N}}^{\delta,\varepsilon}_n(\rho, R, \beta_0)
    := \left\{ 
        U \in \cL^2(\Q_n):  \bigl\| \rho\norm{\vmu}_2 R + RG + \beta_0 Y - U \bigr\|_{\Q_n}
        \le \frac{ R\sqrt{1 - \rho^2} }{ \sqrt{\delta} } - \varepsilon
     \right\}.
\end{equation*}
Recall that $C_R = \max_{(\rho, R, \beta_0) \in \bTheta_{c}} R$. Denote $\delta'_\varepsilon > \delta$ as a constant such that
\begin{equation}
\label{eq:delta_eps}
    \frac{1}{\sqrt{\delta'_\varepsilon}} := \frac{1}{\sqrt{\delta}} - \frac{\varepsilon}{C_R}.
\end{equation}
Then following \cref{eq:ERM_II0}, we have
\begin{align*}
    I\!I 
        \ge & \ \min_{ (\rho, R, \beta_0) \in \bTheta_{c} }
    \min_{ U \in 
    \mathcal{N}^\delta_n(\rho, R, \beta_0) \cap \mathcal{C}_n^\varepsilon(\rho, R, \beta_0)  }
    \E_{\Q_n}[\ell(U)]
    \\
    \stackrel{\mathmakebox[0pt][c]{\smash{\text{(i)}}}}{\ge} & \ \min_{ (\rho, R, \beta_0) \in \bTheta_{c} }
    \min_{ U \in 
    \wt{\mathcal{N}}^{\delta,\varepsilon}_n(\rho, R, \beta_0)  }
    \E_{\Q_n}[\ell(U)]
    \\
    = & \ \min_{ (\rho, R, \beta_0) \in \bTheta_{c} }
       \min_{\xi \in \cL^2(\Q_n), \norm{\xi}_{\Q_n} \le \frac{1}{\sqrt{\delta}} - \frac{\varepsilon}{R\sqrt{1 - \rho^2}}} \E_{\Q_n} \left[ \ell \bigl( \rho\norm{\vmu}_2 R + RG + \beta_0 Y + R\sqrt{1 - \rho^2} \xi \bigr) \right]
    \\
    \stackrel{\mathmakebox[0pt][c]{\smash{\text{(ii)}}}}{\ge}  & \
    \min_{ (\rho, R, \beta_0) \in \bTheta_{c} }
       \min_{\xi \in \cL^2(\Q_n), \norm{\xi}^2_{\Q_n} \le 1/\delta'_\varepsilon } \E_{\Q_n} \left[ \ell \bigl( \rho\norm{\vmu}_2 R + RG + \beta_0 Y + R\sqrt{1 - \rho^2} \xi \bigr) \right]
    \\
    \conp & \ 
    \min_{ (\rho, R, \beta_0) \in \bTheta_{c} }
       \min_{\xi \in \cL^2(\Q_\infty), \norm{\xi}^2_{\Q_\infty} \le 1/\delta'_\varepsilon } \E \left[ \ell \bigl( \rho\norm{\vmu}_2 R + RG + \beta_0 Y + R\sqrt{1 - \rho^2} \xi \bigr) \right]
    \\
    \stackrel{\mathmakebox[0pt][c]{\smash{\text{(iii)}}}}{>} & \ \min_{ (\rho, R, \beta_0) \in \bTheta_{c} }
       \min_{\xi \in \cL^2(\Q_\infty), \norm{\xi}^2_{\Q_\infty} \le 1/\delta } \E \left[ \ell \bigl( \rho\norm{\vmu}_2 R + RG + \beta_0 Y + R\sqrt{1 - \rho^2} \xi \bigr) \right]
    \\
    = & \ M^*( \bTheta_{c} ) = M^*,
\end{align*}
where (i) follows from \cref{eq:ERM_logit_sets}, (ii) follows from \cref{eq:delta_eps} and the fact that 
\begin{equation*}
    \frac{1}{\sqrt{\delta}} - \frac{\varepsilon}{R\sqrt{1 - \rho^2}} \le \frac{1}{\sqrt{\delta'_\varepsilon}},
    \qquad \forall\, (\rho, R, \beta_0) \in \bTheta_{c},
\end{equation*}
the convergence follows from \cref{lem:M3-star}, and (iii) follows from the uniqueness of $(\rho^*, R^*, \beta_0^*)$ and KKT conditions $\norm{\xi^*}^2_{\Q_\infty} = 1/\delta$ in \cref{lem:M_star_var}. 

Finally, combining everthing together, we have
\begin{equation*}
    \pliminf_{n \to \infty} \wt M_n^{\varepsilon(3)}
    \ge \min\left\{ \pliminf_{n \to \infty} I, \ \pliminf_{n \to \infty} I\!I \right\}
    > M^*.
\end{equation*}
This shows \cref{eq:ERM_logit_goal}, and hence completes the proof.
\end{proof}

Using an argument similar to the one at the end of the proof of \cref{lem:over_logit_conv}, we can show the convergence of empirical logit distribution $W_2 ( \hat \nu_{n}, \nu_* ) \conp 0$ from $W_2( \hat \cL_{n}, \cL_* ) \conp 0$ given by \cref{lem:ERM_logit_conv}.

\subsubsection{Completing the proof of \cref{thm:logistic_main}}
\label{subsubsec:under_final}
\begin{proof}[\textbf{Proof of \cref{thm:logistic_main}}]
    Consider the ERM problem \cref{eq:logistic_reg} with arbitrary $\tau > 0$. Recall that $\wt y_i = y_i/s(y_i)$ where $s: \{ \pm 1 \} \to \{ 1 \} \cup \{ \tau \}$ is defined as per \cref{eq:s_fun}. $M_n$ is redefined as \cref{eq:logistic_reg}
    \begin{equation*}
        M_n := \min_{\bbeta \in \R^d, \, \beta_0 \in \R}  \frac1n \sum_{i=1}^n \ell\bigl( 
        \wt y_i(\< \xx_i, \bbeta \> +  \beta_0 )
        \bigr).
    \end{equation*}
    Under this modification, $M_n(\bTheta_{\vbeta}, \bXi_{\bu})$ can be redefined and expressed as
    \begin{align*}
        M_n(\bTheta_{\vbeta}, \bXi_{\bu})
        :\! & = \min_{ \substack{ (\bbeta , \beta_0) \in \bTheta_{\vbeta} \\  \bu \in \bXi_{\bu} } }
        \max_{ \bv \in \R^n } \left\{
        \frac1n \sum_{i=1}^n \ell \biggl( \frac{u_i}{s(y_i)} \biggr)
         + \frac{1}{n} \sum_{i=1}^n v_i \left( y_i(\< \xx_i, \bbeta \> +  \beta_0 ) - u_i \right)
         \right\} \\
        & = \min_{ \substack{ (\bbeta , \beta_0) \in \bTheta_{\vbeta} \\  \bu \in \bXi_{\bu} } }
     \max_{ \bv \in \R^n }
     \left\{
     \frac1n \sum_{i=1}^n \ell \biggl( \frac{u_i}{s(y_i)} \biggr)
      + \frac1n \bv^\top \bone \< \bmu, \vbeta \>
      + \frac1n \bv^\top \ZZ \vbeta + \frac1n \beta_0 \bv^\top \yy - \frac1n \bv^\top \bu
      \right\}.
    \end{align*}
Consequently, quantities $M_n^{(k)}$, $k= 1,2,3$ and $M^*$ used in the proof can be similarly redefined as
    \begin{align*}
        M_n^{(1)}(\bTheta_{\vbeta}, \bXi_{\bu})
        & : = \smash {\min_{ \substack{ (\bbeta , \beta_0) \in \bTheta_{\vbeta} \\  \bu \in \bXi_{\bu} } } }
        \max_{ \bv \in \R^n }
        \, \Biggl\{
        \frac1n \sum_{i=1}^n \ell \biggl( \frac{u_i}{s(y_i)} \biggr)
         + \frac1n \bv^\top \bone \< \bmu, \vbeta \>
         + \frac1n \norm{\bv}_2 \hh^\top \vbeta 
         \\
        & \phantom{.} \phantom{ 
            : = \smash {\min_{ \substack{ (\bbeta , \beta_0) \in \bTheta_{\vbeta} \\  \bu \in \bXi_{\bu} } } }
        \max_{ \bv \in \R^n }
        \, \Biggl\{
        }
         + \frac1n \norm{\vbeta}_2 \vg^\top \bv 
         + \frac1n \beta_0 \bv^\top \yy - \frac1n \bv^\top \bu
         \Biggr\},
        \\
        M_n^{(2)}(\bTheta_{c}, \bXi_{\bu})
        & := \min_{ (\rho, R, \beta_0) \in \bTheta_{c} }
        \min_{ U \in \Xi_{u}  \cap   \mathcal{N}_n }
        \E_{\Q_n}\left[\ell \bigl( U/s(Y) \bigr)\right],
        \\
        M_n^{(3)}(\bTheta_{c}, \bXi_{\bu}) 
        & :=  
        \min_{ (\rho, R, \beta_0) \in \bTheta_{c} } 
        \min_{ U \in \Xi_{u}  \cap   \mathcal{N}^\delta_n }
        \E_{\Q_n}\left[\ell \bigl( U/s(Y) \bigr)\right],
        \\
        M_n^{(3)}(\bTheta_{c}) 
        & :=  
        \min_{ (\rho, R, \beta_0) \in \bTheta_{c} } 
        \min_{ U \in \mathcal{N}^\delta_n }
        \E_{\Q_n}\left[\ell \bigl( U/s(Y) \bigr)\right],
        \\
        & \phantom{:}= \min_{ \substack{ (\rho, R, \beta_0) \in \bTheta_{c} \\ \xi \in \cL^2(\Q_n), \norm{\xi}_{\Q_n} \le 1/\sqrt{\delta} } } 
    \E_{\Q_n} \left[ \ell \biggl( \frac{ \rho\norm{\vmu}_2 R + RG + \beta_0 Y + R\sqrt{1 - \rho^2} \xi }{s(Y)} \biggr) \right],
        \\
         M^*(\bTheta_{c})
     & := \min_{ \substack{ (\rho, R, \beta_0) \in \bTheta_{c} \\ \xi \in \cL^2(\Q_\infty), \norm{\xi}_{\Q_\infty} \le 1/\sqrt{\delta} } } 
    \E \left[ \ell \biggl( \frac{ \rho\norm{\vmu}_2 R + RG + \beta_0 Y + R\sqrt{1 - \rho^2} \xi }{s(Y)} \biggr) \right],
        \\
        M^* & := M^*([-1, 1] \times \R_{\ge 0} \times \R),
    \end{align*}
    where $\mathcal{N}_n$, $\mathcal{N}^\delta_n$ are still defined as \cref{eq:set_N_n}, \eqref{eq:set_N_n_delta}, and we still apply the change of variable
    \begin{equation*}
    U = \rho\norm{\vmu}_2 R + RG + \beta_0 Y + R\sqrt{1 - \rho^2} \xi.
    \end{equation*}
    One can use exactly similar arguments to conclude \cref{lem:ERM_bound_beta}---\ref{lem:M3-star} and \cref{thm:ERM_conv} with definitions above. For \cref{lem:var_fixed}, we can also get similar results, but the KKT condition in \cref{eq:KKT_xi2} now becomes
    \begin{equation*}
        U + \lambda \ell' ( U/s(Y) ) = \rho\norm{\vmu}_2 R + RG + \beta_0 Y,
    \end{equation*}
    which implies
    \begin{equation}
    \label{eq:ERM_U_new}
        \frac{U}{s(Y)} = \prox_{\ell}\left( \frac{\rho\norm{\vmu}_2 R + RG + \beta_0 Y}{s(Y)} ; \frac{\lambda}{s(Y)} \right),
    \end{equation}
    as a substitute of \cref{eq:U_prox},
    and
    \begin{equation*}
        \xi^*_{\Q}(\rho, R, \beta_0) = -\frac{\lambda}{R\sqrt{1 - \rho^2}} \ell'\left( \prox_{\ell}\left( \frac{\rho\norm{\vmu}_2 R + RG + \beta_0 Y}{s(Y)} ; \frac{\lambda}{s(Y)} \right) \right),
    \end{equation*}
    as a substitute of \cref{eq:xi_star}.

\vspace{0.5\baselineskip}
\noindent
\textbf{\ref{thm:logistic_main(a)}:} According to the definition above, the KKT conditions \cref{lem:M_star_var} will become
\begin{align}
                - \frac{R \rho}{\lambda \delta \norm{\vmu}_2}
                & = 
                \E \left[ \wt\ell'_Y(U) \right],
                \label{eq:KKT_new1}
                \\
                \frac{R}{\lambda \delta }
                & = 
                \E \left[ \wt\ell'_Y(U) G \right],
                \label{eq:KKT_new2}
                \\
                0
                & = 
                \E \left[ \wt\ell'_Y(U) Y \right],
                \label{eq:KKT_new3}
                \\
                \frac{R^2 (1 - \rho^2)}{\lambda^2 \delta}
                & = 
                \E \left[ \bigl( \wt\ell'_Y(U) \bigr)^2 \right],
                \notag
\end{align}
where
\begin{equation*}
    \wt\ell'_Y(U) := \frac{1}{s(Y)} \ell'\left( \frac{U}{s(Y)} \right)
    = \frac{1}{s(Y)} \ell'\left( \prox_{\ell}\left( \frac{\rho\norm{\vmu}_2 R + RG + \beta_0 Y}{s(Y)} ; \frac{\lambda}{s(Y)} \right) \right)
    ,
\end{equation*}
and $U$ follows the relation \cref{eq:ERM_U_new}. By Stein's identity, \cref{eq:KKT_new2} can be expressed as
\begin{align*}
    \frac{R}{\lambda \delta } = \E \left[ \wt\ell'_Y(U) G \right]
    & = \E \left[ \frac{1}{s(Y)} \ell''\left( \frac{U}{s(Y)} \right)
    \cdot \frac{\d (U/s(Y))}{\d G} \right] \\
    & = \E \left[ \frac{1}{s(Y)} \ell''\left( \frac{U}{s(Y)} \right)
    \cdot \frac{1}{1 + \dfrac{\lambda}{s(Y)} \ell''\left( \dfrac{U}{s(Y)} \right) } \cdot \dfrac{R}{s(Y)}  \right],
\end{align*}
which gives the third KKT condition in \ref{thm:logistic_main(a)}. Besides, \cref{eq:KKT_new1} and \ref{eq:KKT_new3} can be rewritten as
\begin{align*}
    - \frac{R \rho}{\lambda \delta \norm{\vmu}_2}
    & = \pi \E\left[ \frac{1}{\tau} \ell'\left( \prox_{\ell}\left( \frac{\rho\norm{\vmu}_2 R + RG + \beta_0}{\tau} ; \frac{\lambda}{\tau} \right) \right) \right] \\
    & \phantom{=.} 
    + (1-\pi) \E\left[ \ell'\, \bigl( \prox_{\ell}\left( \rho\norm{\vmu}_2 R + RG - \beta_0 ; \lambda \right) \bigr)  \right],
    \\
    0 & = \pi \E\left[ \frac{1}{\tau} \ell'\left( \prox_{\ell}\left( \frac{\rho\norm{\vmu}_2 R + RG + \beta_0}{\tau} ; \frac{\lambda}{\tau} \right) \right) \right] \\
    & \phantom{=.}
    - (1-\pi) \E\left[ \ell'\, \bigl( \prox_{\ell}\left( \rho\norm{\vmu}_2 R + RG - \beta_0 ; \lambda \right) \bigr)  \right],
\end{align*}
which solves the first two KKT conditions in \ref{thm:logistic_main(a)}. This concludes the proof of part \ref{thm:logistic_main(a)}.

\vspace{0.5\baselineskip}
\noindent
\textbf{\ref{thm:logistic_main(b)}:} \cref{lem:ERM_param_conv} still remains valid under arbitrary $\tau > 0$, which concludes the proof.

\vspace{0.5\baselineskip}
\noindent
\textbf{\ref{thm:logistic_main(c)}:} Similar to the proof of \cref{thm:SVM_main}\ref{thm:SVM_main_err}, we can show that for any test point $(\xx_\mathrm{new}, y_\mathrm{new})$,
\begin{equation*}
        \hat f(\xx_\mathrm{new}) = \< \xx_\mathrm{new}, \hat\vbeta_n \> + \hat\beta_{0,n}
        \cond y_\mathrm{new} R^* \rho^* \norm{\bmu}_2 + R^* G + \beta_{0}^*,
\end{equation*}
where $(y^\mathrm{new}, G) \sim P_y \times \normal(0, 1)$. Therefore, by bounded convergence theorem, the errors have limits
\begin{align*}
        \lim_{n \to \infty} \Err_{+,n} & = \P\left( + R^*\rho^* \norm{\bmu}_2 + R^*G + \beta_{0}^* \le 0 \right)
        = \Phi \left(- \rho^* \norm{\bmu}_2  - \frac{\beta_0^*}{R^*} \right), \\
        \lim_{n \to \infty} \Err_{-,n} & = \P\left( - R^*\rho^* \norm{\bmu}_2 + R^*G + \beta_{0}^* >  0 \right)
        = \Phi \left(- \rho^* \norm{\bmu}_2  + \frac{\beta_0^*}{R^*} \right).
\end{align*}
This concludes the proof of part \ref{thm:logistic_main(c)}.

\vspace{0.5\baselineskip}
\noindent
\textbf{\ref{thm:logistic_main(d)}:} Based on \cref{eq:ERM_U_new}, we redefine $\cL_*$ in \cref{append_subsubsec:ERM_logit} by
\begin{equation*}
    \cL_* := \Law \, (U^*)
    =
    \Law \left( 
    s(Y) \, \prox_{\ell}\left( \frac{\rho^*\norm{\vmu}_2 R^* + R^*G + \beta_0^* Y}{s(Y)}; \frac{\lambda^*}{s(Y)} \right)
    \right).
\end{equation*}
Then \cref{lem:ERM_logit_conv} and the corresponding convergence of ELD still hold. The convergence of TLD directly comes from the proof of part \ref{thm:logistic_main(c)}. This concludes the proof of part \ref{thm:logistic_main(d)}.

Finally, we complete the proof of \cref{thm:logistic_main}.
\end{proof}

%% file: src/append_margin_reb.tex
\section{Margin rebalancing in proportional regime: Proofs for \cref{subsec:rebal_prop}}
\label{append_sec:mar_reb}

\subsection{Proofs of \cref{prop:Err-_mono} and \ref{prop:tau_mono}}


We show the monotonicity of $\Err_+^*$ for $\tau = 1$ in this subsection by first analyzing the monotonicity of asymptotic parameters $\rho^*, \beta_0^*$, which are the solution to the system of equations in \cref{lem:gordon_eq}. We restate these equations here.
\begin{subequations}
\begin{align}
    \label{eq:reb_sys_eq_rho}
        \pi \delta \cdot g \left( \frac{\rho}{2 \pi \norm{\bmu}_2 \delta} \right) & + (1 - \pi) \delta \cdot g \left( \frac{\rho}{2(1 - \pi) \norm{\bmu}_2 \delta} \right) = 1 - \rho^2, \\
    \label{eq:reb_sys_eq_bk1}
    - \beta_0 + \kappa \tau & = \rho \norm{\bmu}_2 + g_1^{-1} \left( \frac{\rho}{2 \pi \norm{\bmu}_2 \delta} \right), \\
    \label{eq:reb_sys_eq_bk2}
	\beta_0 + \kappa & = \rho \norm{\bmu}_2 + g_1^{-1} \left( \frac{\rho}{2 (1 - \pi) \norm{\bmu}_2 \delta} \right).
\end{align}
\end{subequations}
The properties of functions $g_1, g_2, g$ therein are summarized below.
\begin{lem}\label{lem:g1_g2_g} 
Recall $g_1 (x) = \E \left[ (G + x)_+ \right]$, $g_2 (x) = \E \left[ (G + x)_+^2 \right]$, and $g = g_2 \circ g_1^{-1}$.
\begin{enumerate}[label=(\alph*)]
    \item $g_1$, $g_2$ are increasing maps from $\R$ to $\R_{> 0}$, and $g: \R_{> 0} \to \R_{> 0}$ is increasing with $g(0^+) = 0$.  
    \item $g_1$, $g_2$ have explicit expressions
    \begin{equation*}
        g_1(x) = x \Phi(x) + \phi(x), \qquad g_2(x) = (x^2+1)\Phi(x) + x\phi(x). 
    \end{equation*}
    \item \label{lem:g1_g2_g_asymp} 
    $g_1(x) \sim x$, $g_2(x) \sim x^2$, and $g(x) \sim x^2$, as $x \to +\infty$.
\end{enumerate}
\end{lem}

The following preliminary result gives the monotonicity of $\rho^*$. By \cref{thm:SVM_main}\ref{thm:SVM_main_var}, $\rho^* \in (0, 1)$ is invariant with respect to $\tau$. Hence $\rho^*$ can be viewed as a function of model parameters $(\pi, \norm{\bmu}_2, \delta)$ determined by \cref{eq:reb_sys_eq_rho}.
\begin{lem}[Monotonicity of $\rho^*$]\label{lem:rho_mono}
    $\rho^*$ is an increasing function of $\pi \in (0, \frac12)$, $\norm{\bmu}_2$, and $\delta$.
\end{lem}
\begin{proof}
Recall that $\rho^* \in (0, 1)$ as stated in \cref{thm:SVM_main}\ref{thm:SVM_main_var}.

\vspace{0.5\baselineskip}
\noindent
\textbf{(a) $\boldsymbol{\uparrow}$ in $\norm{\vmu}_2$:}
This point is obvious from \cref{eq:reb_sys_eq_rho} and Lemma~\ref{lem:g1_g2_g}(a). 

\vspace{0.5\baselineskip}
\noindent
\textbf{(b) $\boldsymbol{\uparrow}$ in $\delta$:}
Notice that \cref{lem:g_monotone} implies $x \mapsto x \cdot g(1 / x)$ is decreasing in $x$. As a consequence, if we fix $\rho$ and increase $\delta$ on the L.H.S. of \cref{eq:reb_sys_eq_rho}, then the L.H.S. will decrease, and $\rho^*$ have to increase to match the R.H.S.. Therefore, $\rho^*$ is an increasing function of $\delta$. 

\vspace{0.5\baselineskip}
\noindent
\textbf{(c) $\boldsymbol{\uparrow}$ in $\pi \in (0, \frac12)$:}
We prove this using a similar strategy. Define
	\begin{equation*}
		x_1 = x_1 (\pi) := g_1^{-1} \left( \frac{\rho}{2 \pi \norm{\bmu}_2 \delta} \right), 
            \quad
            x_2 = x_2 (\pi) := g_1^{-1} \left( \frac{\rho}{2(1 - \pi) \norm{\bmu}_2 \delta} \right),
	\end{equation*}
	then we know that the L.H.S. of \cref{eq:reb_sys_eq_rho} (for fixed $\delta$ and $\norm{\bmu}_2$) is proportional to
	\begin{equation}\label{eq:rho_LHS_prop}
        \rho \cdot \left(
		\frac{g_2 (x_1(\pi))}{g_1 (x_1(\pi))} + \frac{g_2 (x_2(\pi))}{g_1 (x_2(\pi))}
        \right),
	\end{equation}
	with the only constraint on $x_1$ and $x_2$ being
	\begin{equation*}
		\frac{1}{g_1 (x_1(\pi))} + \frac{1}{g_1 (x_2(\pi))} = C := \frac{2 \norm{\bmu}_2 \delta}{\rho}.
	\end{equation*}
	Taking derivative with respect to $\pi$, it follows that
	\begin{equation*}
		- \frac{g_1' (x_1)}{g_1^2 (x_1)} \cdot x_1'(\pi) - \frac{g_1' (x_2)}{g_1^2 (x_2)} \cdot x_2' (\pi) = 0, 
        \quad
        \Longrightarrow 
        \quad x_1' (\pi) = - \frac{g_1^2 (x_1)}{g_1' (x_1)} \cdot \frac{g_1' (x_2)}{g_1^2 (x_2)}  \cdot x_2'(\pi),
	\end{equation*}
	thus leading to
	\begin{align*}
		& \frac{\d}{\d \pi} \left( \frac{g_2 (x_1(\pi))}{g_1 (x_1(\pi))} + \frac{g_2 (x_2(\pi))}{g_1 (x_2(\pi))} \right) \\
		= {} & \frac{g_2'(x_1) g_1(x_1) - g_2(x_1) g_1'(x_1)}{g_1^2(x_1)} \cdot x_1'(\pi) + \frac{g_2'(x_2) g_1(x_2) - g_2(x_2) g_1'(x_2)}{g_1^2(x_2)} \cdot x_2'(\pi) \\
		= {} & - \frac{g_1' (x_2)}{g_1^2(x_2)} \cdot x_2'(\pi) \cdot \left( \frac{g_2'(x_1) g_1(x_1) - g_2(x_1) g_1'(x_1)}{g_1' (x_1)} - \frac{g_2'(x_2) g_1(x_2) - g_2(x_2) g_1'(x_2)}{g_1' (x_2)} \right) \\
            = {} & - \frac{g_1' (x_2)}{g_1^2(x_2)} \cdot x_2'(\pi) \cdot \left( h(x_1) - h(x_2) \right),
	\end{align*}
        where 
        \begin{equation*}
		h(x) := \frac{g_2'(x) g_1(x) - g_2(x) g_1'(x)}{g_1'(x)}, \quad \forall\, x \in \R
        \end{equation*}
        is a monotone increasing function according to the proof of \cref{lem:g_monotone}. Therefore, $h(x_1) > h(x_2)$ (since $\pi < 1/2 \implies x_1 > x_2$). By definitions of $x_2$ and $g_1$, we know that $x_2'(\pi) > 0$ and $g_1'(x_2) > 0$. As a consequence,
	\begin{equation*}
		\frac{\d}{\d \pi} \left( \frac{g_2 (x_1(\pi))}{g_1 (x_1(\pi))} + \frac{g_2 (x_2(\pi))}{g_1 (x_2(\pi))} \right) < 0.
	\end{equation*}
	Similar to points (a) and (b), by combining \cref{eq:reb_sys_eq_rho} and \eqref{eq:rho_LHS_prop}, we conclude that $\rho^*$ is an increasing function of $\pi \in (0, \frac12)$. This completes the proof.
\end{proof}

As long as $\tau \not= 0$, the linear system \cref{eq:reb_sys_eq_bk1} and \eqref{eq:reb_sys_eq_bk2} for $(\beta_0, \tau)$ is non-singular, so one can solve for $\beta_0$ and $\kappa$:
\begin{subequations}
\begin{align}
	\label{eq:beta0_tau}
	\beta_0 = \, & \frac{1}{1 + \tau} \left( \tau g_1^{-1} \left( \frac{\rho}{2 (1 - \pi) \norm{\bmu}_2 \delta} \right) - g_1^{-1} \left( \frac{\rho}{2 \pi \norm{\bmu}_2 \delta} \right) + (\tau - 1) \rho \norm{\bmu}_2 \right), \\
	\label{eq:kappa_tau}
	\kappa = \, & \frac{1}{1 + \tau} \left( g_1^{-1} \left( \frac{\rho}{2 (1 - \pi) \norm{\bmu}_2 \delta} \right) + g_1^{-1} \left( \frac{\rho}{2 \pi \norm{\bmu}_2 \delta} \right) + 2 \rho \norm{\bmu}_2 \right).
\end{align}
\end{subequations}
The following lemma establishes the monotonicity of $\beta_0^*$ when $\tau = 1$. 
\begin{lem}[Monotonicity of $\beta_0^*$]\label{lem:beta0_mono}
    $\beta_0^*$ is an increasing function of $\pi \in (0, \frac12)$, $\norm{\bmu}_2$, and $\delta$, when $\tau = 1$ (without margin rebalancing). Moreover, $\beta_0^* < 0$.
\end{lem}
\begin{proof}
When $\tau = 1$, the above equations reduce to
\begin{align}
	\beta_0 = \, & \frac{1}{2} \left( g_1^{-1} \left( \frac{\rho}{2 (1 - \pi) \norm{\bmu}_2 \delta} \right) - g_1^{-1} \left( \frac{\rho}{2 \pi \norm{\bmu}_2 \delta} \right) \right), 
    \label{eq:beta0_tau=1}
    \\
	\kappa = \, & \frac{1}{2} \left( g_1^{-1} \left( \frac{\rho}{2 (1 - \pi) \norm{\bmu}_2 \delta} \right) + g_1^{-1} \left( \frac{\rho}{2 \pi \norm{\bmu}_2 \delta} \right) + 2 \rho \norm{\bmu}_2 \right).
    \notag
\end{align}
Clearly $\beta_0^* < 0$, since $g_1^{-1}$ is an increasing function and $\pi < \frac12$.

\vspace{0.5\baselineskip}
\noindent
\textbf{(a) $\boldsymbol{\uparrow}$ in $\norm{\mu}_2$:}
Fixing $\pi$ and $\delta$, taking derivative with respect to $\norm{\bmu}_2$ in \cref{eq:beta0_tau=1}, we have
	\begin{equation*}
		\frac{\d \beta_0}{\d \norm{\bmu}_2} = \frac{1}{2} \left( \frac{1}{2 (1 - \pi) \delta} \cdot (g_1^{-1})' \left( \frac{\rho}{2 (1 - \pi) \norm{\bmu}_2 \delta} \right) - \frac{1}{2 \pi \delta} \cdot (g_1^{-1})' \left( \frac{\rho}{2 \pi \norm{\bmu}_2 \delta} \right) \right) \cdot \frac{\d}{\d \norm{\bmu}_2} \left( \frac{\rho}{\norm{\bmu}_2} \right).
	\end{equation*}
	Since $\pi < \frac12$, from \cref{lem:g_prime_monotone} we know that
	\begin{equation*}
		\frac{1}{2 (1 - \pi) \delta} \cdot (g_1^{-1})' \left( \frac{\rho}{2 (1 - \pi) \norm{\bmu}_2 \delta} \right) - \frac{1}{2 \pi \delta} \cdot (g_1^{-1})' \left( \frac{\rho}{2 \pi \norm{\bmu}_2 \delta} \right) < 0.
	\end{equation*}
	According to \cref{lem:rho_mono}, if we increase $\norm{\bmu}_2$, then $\rho$ will increase, and \cref{eq:reb_sys_eq_rho} implies that $\rho / \norm{\bmu}_2$ will decrease. Hence,
	\begin{equation*}
		\frac{\d}{\d \norm{\bmu}_2} \left( \frac{\rho}{\norm{\bmu}_2} \right) < 0.
	\end{equation*}
	Combining the above inequalities, we know that $\d \beta_0 / \d \norm{\bmu}_2 > 0$.

\vspace{0.5\baselineskip}
\noindent
\textbf{(b) $\boldsymbol{\uparrow}$ in $\delta$:}
Similarly, according to \cref{eq:reb_sys_eq_rho} and \cref{lem:rho_mono}, for fixed $\pi$ and $\norm{\bmu}_2$, we can show that $\rho / \delta$ will decrease if $\delta$ increases. By same approach as (a), we can conclude $\d \beta_0 / \d \delta > 0$.

\vspace{0.5\baselineskip}
\noindent
\textbf{(c) $\boldsymbol{\uparrow}$ in $\pi \in (0, \frac12)$:}
Lastly, we note that if $\pi \in (0, \frac12)$ increases, then $1 - \pi$ will decrease and $\rho$ will increase. According to \cref{lem:g_monotone}, we know that
\begin{equation*}
    \frac{(1 - \pi) \delta}{\rho} \cdot g \left( \frac{\rho}{2(1 - \pi) \norm{\bmu}_2 \delta} \right)
\end{equation*}
will increase. Since $(1 - \rho^2)/\rho$ will decrease, combining with \cref{eq:reb_sys_eq_rho}, we can show that
\begin{equation*}
    \frac{\pi \delta}{\rho} \cdot g \left( \frac{\rho}{2 \pi \norm{\bmu}_2 \delta} \right)
\end{equation*}
will decrease. By \cref{lem:g_monotone} again, we conclude that $\rho / (1 - \pi)$ will increase and $\rho / \pi$ will decrease, which implies that $\beta_0$ \cref{eq:beta0_tau=1} will increase. This completes the proof.
\end{proof}

The monotonicity of minority error is a direct consequence of the two lemmas above.
\begin{proof}[\textbf{Proof of \cref{prop:Err-_mono}}]
When $\tau = 1$, according to \cref{lem:rho_mono} and \ref{lem:beta0_mono}, both $\rho^*$ and $\beta_0^*$ are increasing in $\pi \in (0, \frac12)$, $\norm{\bmu}_2$, and $\delta$. We complete the proof by $\Err_{+}^* = \Phi \left(- \rho^* \norm{\bmu}_2 - \beta_0^* \right)$.
\end{proof}

Now we fix model parameters $\pi \in (0, \frac12)$, $\delta$, $\norm{\bmu}_2$, and consider test errors as functions of $\tau$. In order to prove \cref{prop:tau_mono}, we need the following result on the monotonicity of $\rho^*, \beta_0^*$ on $\tau$.
\begin{lem}[Dependence of $\tau$]\label{lem:tau_mono}
    Fix $\pi \in (0, \frac12)$, $\norm{\bmu}_2$, and $\delta$, then we have
    \begin{enumerate}[label=(\alph*)]
        \item \label{lem:tau_mono_rho} $\rho_0^*$ does not depend on $\tau$.
        \item \label{lem:tau_mono_beta0} $\beta_0^*$ is an increasing function of $\tau \in (0, \infty)$.
	\item \label{lem:tau_mono_kappa} $\kappa^*$ is a decreasing function of $\tau \in (0, \infty)$.
    \end{enumerate}
    As a consequence, $\Err^*_+$ is decreasing in $\tau \in (1, \infty)$, and $\Err^*_-$ is increasing in $\tau \in (1, \infty)$.
\end{lem}
\begin{proof}
    \ref{lem:tau_mono_rho} is already proved in \cref{thm:SVM_main}\ref{thm:SVM_main_var}. For \ref{lem:tau_mono_kappa}, the conclusion is followed by \cref{eq:kappa_tau}, since $\kappa^* \propto (1 + \tau)^{-1}$. For \ref{lem:tau_mono_beta0}, note that $\beta_0^* + \kappa^*$ is a fixed value according to \cref{eq:reb_sys_eq_bk2}. Then by using \ref{lem:tau_mono_kappa}, we conclude $\beta_0^*$ is increasing in $\tau$. This concludes the proof.
\end{proof}

These are consistent with the non-asymptotic monotonicity between $(\hat\rho, \hat\beta_0, \hat\kappa)$ and $\tau$ in \cref{prop:SVM_tau_relation}. Then the monotonicity of test errors is a direct consequence of \cref{lem:tau_mono}.

\begin{proof}[\textbf{Proof of \cref{prop:tau_mono}}]
According to \cref{lem:tau_mono}\ref{lem:tau_mono_rho}\ref{lem:tau_mono_beta0}, we know that $-\rho^* \norm{\bmu}_2 + \beta_0^*$ is increasing in $\tau$ and $-\rho^* \norm{\bmu}_2 - \beta_0^*$ is decreasing in $\tau$. This completes the proof.
\end{proof}

\subsection{Proofs of \cref{prop:tau_optimal} and \ref{prop:Err_monotone}}
\label{subsec:tau_optimal}

\begin{proof}[\textbf{Proof of \cref{prop:tau_optimal}}]
    Recall that
    \begin{equation*}
        \Err_\mathrm{b}^* = \frac12 \Bigl( 
        \Phi\left(- \rho^* \norm{\bmu}_2 - \beta_0^* \right) + \Phi\left(- \rho^* \norm{\bmu}_2  + \beta_0^* \right)
        \Bigr).
    \end{equation*}
    Notice that $\rho^*$ does not depend on $\tau$, and $\rho^*\norm{\bmu}_2 > 0$. We first show that $\tau = \tau^\mathrm{opt}$ if and only if $\beta_0^* = 0$. Then is suffices to show that for any fixed $a > 0$, function
    \begin{equation*}
        f(x) := \Phi(-a + x) + \Phi(-a - x), \qquad x \in \R
    \end{equation*}
    has unique minimizer $x = 0$. This is true by observing $f'(x) = \phi(-a + x) - \phi(-a - x) < 0$ for all $x < 0$, and $f'(x) > 0$ for all $x > 0$. Hence we conclude $\beta_0^* = 0$ and $\Err_+^* = \Err_-^* = \Err_\mathrm{b}^*$.

    Setting $\beta_0 = 0$ in \cref{eq:beta0_tau} and solving for $\tau$, we get \cref{eq:tau_opt}. This completes the proof.
\end{proof}

As stated in \cref{rem:tau_pi}, when $\norm{\bmu}_2$, $\delta$ are fixed and $\pi$ is small, the numerator of $\tau^\mathrm{opt}$ scales as $\sqrt{1/\pi}$. We formally prove this in the following lemma.

\begin{lem}
    When $\pi = o(1)$, we have
    \begin{equation*}
        g_1^{-1} \left( \dfrac{\rho^*}{2 \pi \norm{\bmu}_2 \delta} \right) + \rho^* \norm{\bmu}_2  \sim  \frac{1}{\sqrt{\pi \delta}}.
    \end{equation*}
\end{lem}
\begin{proof}
    By \cref{lem:rho_mono}, $\rho^*$ is monotone increasing in $\pi \in (0, \frac12)$. It can be easily shown that $\rho^* \to 0$ as $\pi \to 0$. Otherwise, suppose $\rho^* \to \underline{\rho} > 0$ as $\pi \to 0$, then by \cref{lem:g1_g2_g}\ref{lem:g1_g2_g_asymp} 
\[ 
    \pi \delta \cdot g \left( \frac{\rho^*}{2 \pi \norm{\bmu}_2 \delta} \right)
    \sim  \pi \delta \cdot \left( \frac{\underline{\rho}}{2 \pi \norm{\bmu}_2 \delta} \right)^2 \propto \frac{1}{\pi} \to \infty,
\]
while the other terms in \cref{eq:reb_sys_eq_rho} are all finite, which is a contradiction. Substitute $\rho^* \to 0$ into \cref{eq:reb_sys_eq_rho},
\[ 
g \left( \frac{\rho^*}{2 \pi \norm{\bmu}_2 \delta} \right) \sim \frac{1}{\pi\delta} \to \infty
\qquad \Longrightarrow \qquad
\frac{\rho^*}{2 \pi \norm{\bmu}_2 \delta} \sim \frac{1}{\sqrt{\pi\delta}}.
\]
The proof is complete by using \cref{lem:g1_g2_g}\ref{lem:g1_g2_g_asymp} again.
\end{proof}
\begin{rem}
    We notice that when $\pi$ is very small or $\norm{\bmu}_2$, $\delta$ are very large, then $\rho^*$ is close to $0$ and the denominator of $\tau^\mathrm{opt}$ can be zero or negative, leading $\tau^\mathrm{opt}$ infinity of negative. According to \cref{fig:SVM_cartoon}, this happens when the optimal decision boundary (the red solid line) falls on or under the margin of majority class (the black dashed line below with negative support vectors). In such cases, we have $\tau < -1$ and the training error for majority class is nonzero.

    Actually, our theory remains valid when $\tau < -1$. When $\tau < -1$, one can modify the objective of \cref{eq:SVM-m-reb} to minimizing $\kappa$ (since $\kappa < 0$ and $\tau \kappa > 0$), then the relation \cref{eq:margin-balance} in \cref{prop:SVM_tau_relation} still holds. For the asymptotic problem, one can similarly modify the variational problem \cref{eq:SVM_variation}. Then one may extend \cref{thm:SVM_main} to negative $\tau$ by relating \cref{eq:margin-balance} to \eqref{eq:asymp_tau_relation}, where \cref{eq:asymp_tau_relation} is derived from \cref{eq:reb_sys_eq_rho}---\eqref{eq:reb_sys_eq_bk2}, which also admits a unique solution when $\tau < -1$.
\end{rem}

Finally, prove the monotonicity of test errors after margin rebalancing.

\begin{proof}[\textbf{Proof of \cref{prop:Err_monotone}}]
    According to \cref{prop:tau_optimal}, $\Err_+^* = \Err_-^* = \Err_\mathrm{b}^* = \Phi(- \rho^* \norm{\bmu}_2 )$. Since $\rho^*$ is increasing in $\pi \in (0, \frac12)$, $\norm{\bmu}_2$, and $\delta$ by \cref{lem:rho_mono}, the proof is complete.
\end{proof}

\subsection{Technical lemmas}

Some technical results used in the proof are summarized below.

\begin{lem}\label{lem:g_monotone}
The function $g_2(x) / g_1(x)$ is increasing in $x$. This implies $g(x) / x$ is increasing in $x$, and $x \cdot g(1/x)$ is decreasing in $x$.
\end{lem}
\begin{proof}
	By direct calculation, we have
	\begin{align*}
		g_2'(x) g_1(x) - g_2(x) g_1'(x) = 2 \left( \E[(G + x)_+] \right)^2 - \Phi(x) \E[(G+x)_+^2].
	\end{align*}
	It suffices to show that
	\begin{equation*}
		h(x) := \frac{2 \left( \E[(G + x)_+] \right)^2}{\Phi(x)} - \E[(G+x)_+^2] > 0, \qquad  \forall\, x \in \R.
	\end{equation*}
	To this end, note that $\lim_{x \to - \infty} h(x) = 0$, and that
	\begin{equation*}
		h'(x) = 2 \E [(G + x)_+] \left( 1 - \frac{\E [(G + x)_+] \phi(x)}{\Phi(x)^2} \right).
	\end{equation*}
	Hence, one only need to show that $h'(x) > 0$, $\forall \, x \in \R$, namely
	\begin{equation*}
		r(x) := \frac{\Phi(x)^2}{\phi(x)} - \E [(G + x)_+] > 0.
	\end{equation*}
	Notice again that $\lim_{x \to -\infty} r(x) = 0$, and
	\begin{equation*}
		r'(x) = \Phi(x) \left( 1 + \frac{x \Phi(x)}{\phi(x)} \right) > 0
	\end{equation*}
	by Mill's ratio, thus we finally conclude that $r (x) > 0$ for any $x \in \R$. Consequently, $g_2 (x) / g_1 (x)$ is increasing in $x$.

    By change of variable $y = g_1(x)$, we show that $g_2 (x)/g_1 (x) = g(y)/y$ is increasing in $y$.
\end{proof}

\begin{lem}\label{lem:g_prime_monotone}
	The function $x \mapsto x \cdot (g_1^{-1})' (x)$ is monotone increasing.
\end{lem}
\begin{proof}
	Let $x = g_1 (y)$, then we know that
	\begin{equation*}
		x \cdot (g_1^{-1})' (x) = \frac{g_1 (y)}{g_1' (y)}.
	\end{equation*}
	Since $y$ is increasing in $x$, it suffices to show that $g_1 (y) / g_1' (y)$ is increasing in $y$. Note that
	\begin{equation*}
		\frac{\d}{\d y} \left( \frac{g_1 (y)}{g_1' (y)} \right) = \frac{g_1' (y)^2 - g_1 (y) g_1'' (y)}{g_1' (y)^2} = \frac{\phi(y) r(y)}{g_1' (y)^2},
	\end{equation*}
	where the function $r(y)$ is defined in the proof of \cref{lem:g_monotone}, and we know that $r (y) > 0$ for all $y \in \R$. Therefore, $g_1 (y) / g_1' (y)$ is increasing. This completes the proof. 
\end{proof}

%% file: src/append_high_imb_prf.tex
\section{Margin rebalancing in high imbalance regime: Proof of \cref{thm:main_high-imbal}}
\label{append_sec:high_imb}

Without loss of generality, we may consider the following case as a substitute of \cref{setup-high-imbalance}:
\begin{equation*}
\pi = d^{-a}, \qquad \norm{\vmu}_2^2 = d^b, \qquad n = d^{c+1}.
\end{equation*}

\noindent
Consider a linear classifier based on $f(\xx) = \< \xx, \vbeta \> + \beta_0$ with $\norm{\vbeta}_2 = 1$. Denote projection matrices
\begin{equation*}
    \bP_{\vmu} := \frac{1}{\norm{\vmu}_2^2} \vmu \vmu^\top,
    \qquad 
    \bP_{\vmu}^\perp := \bI_d - \frac{1}{\norm{\vmu}_2^2} \vmu \vmu^\top,
\end{equation*}
where $\bP_{\vmu}$ is the orthogonal projection onto $\spann\{ \vmu \}$ and $\bP_{\vmu}^\perp$ is the orthogonal projection onto the orthogonal complement of $\spann\{ \vmu \}$. Then we define auxiliary parameters
\begin{equation}\label{eq:def-rho-theta}
    \rho := \left\< \vbeta, \frac{\vmu}{\norm{\vmu}_2} \right\>,
    \qquad
    \vtheta := 
    \begin{cases} 
        \, \dfrac{\bP_{\vmu}^\perp \vbeta}{\|\bP_{\vmu}^\perp \vbeta\|_2}
        = \dfrac{\bP_{\vmu}^\perp \vbeta}{\sqrt{1 - \rho^2}} , & \ \text{if} \ \abs{\rho} < 1, \\
        \, \vmu_\perp,         & \ \text{if} \ \abs{\rho} = 1, 
    \end{cases}
\end{equation}
where $\vmu_\perp \in \S^{d-1}$ is some deterministic vector such that $\vmu_\perp \perp \vmu$.
Therefore, we have the following decomposition:
\begin{equation*}
    \vbeta = \bP_{\vmu} \vbeta + \bP_{\vmu}^\perp \vbeta 
    = \rho\frac{\vmu}{\norm{\vmu}_2} + \sqrt{1 - \rho^2} \vtheta.
\end{equation*}
Note that $\norm{\vtheta}_2 = 1$, $\vtheta \perp \vmu$, and there exists a one-to-one correspondence\footnote{
    In fact, this one-to-one mapping $\vbeta \mapsto (\rho, \vtheta)$ is restricted to $\S^{d-1} \to \Theta_{\rho,\vtheta}$, where the range is $\Theta_{\rho,\vtheta} :=  \{(\rho, \vtheta): \rho \in (-1, 1), \norm{\vtheta}_2 = 1, \vtheta \perp \vmu \} \cup \{ (\rho, \vtheta): \rho = \pm1, \vtheta = \vmu_\perp \} $. However, for simplicity, we can expand the parameter space of $(\rho, \vtheta)$ into $\{(\rho, \vtheta): \rho \in [-1, 1], \norm{\vtheta}_2 = 1, \vtheta \perp \vmu \}$. This is because if $\rho = \pm 1$, we have $\bP_{\vmu}^\perp \vbeta = \bzero$, and $\sqrt{1 - \rho^2} \vtheta = \bzero$ for any $\vtheta$. We will see that $\vtheta$ always appears in the form of $\sqrt{1 - \rho^2} \vtheta$ (for example, in the decomposition of $\vbeta$, and the expression of $\kappa_i$ and $\kappa$). That also explains why we can take $\vmu_\perp$ arbitrarily in \cref{eq:def-rho-theta}.
} between $\vbeta$ and $(\rho, \vtheta)$. Therefore, the logit margin of $f(\xx)$ for the $i$-th data point $(\xx_i, y_i)$ can be reparametrized as
\begin{align}
        \kappa_i  & =  \kappa_i(\vbeta, \beta_0) 
        : = \wt y_i ( \< \xx_i, \vbeta \> + \beta_0 )  \nonumber \\
        & = s_i y_i \left( \Bigl\< y_i \bmu + \zz_i, \rho\frac{\vmu}{\norm{\vmu}_2} + \sqrt{1 - \rho^2}\vtheta  \Bigr\>  + \beta_0 \right) \nonumber \\
        & = s_i \left( \rho \norm{\vmu}_2 + y_i\beta_0 + \rho y_i g_i + \sqrt{1 - \rho^2} y_i \< \zz_i, \vtheta \> \right) 
        = : \kappa_i(\rho, \vtheta, \beta_0),
        \label{eq:logits}
\end{align}
where $\zz_i \sim \subGind(\bzero, \bI_n; K)$ according to \cref{def:subgauss}, $K > 0$ is some absolute constant, and
\begin{equation*}
    s_i := \begin{cases} 
    \, \tau^{-1}, & \ \text{if} \ y_i = +1, \\
    \, 1,         & \ \text{if} \ y_i = -1, \end{cases}
\qquad
    g_i := \left\langle \zz_i, \frac{\vmu}{\norm{\vmu}_2} \right\rangle,
\end{equation*}
where $\vg := (g_1, \dots, g_n)^\top \sim \subGind(\bzero, \bI_n; K)$ by \cref{lem:subG}\ref{lem:subG-b}. Therefore, the margin (in \cref{eq:margin}) of $f(\xx)$ can be viewed as function $(\vbeta, \beta_0) \mapsto \kappa$ or $(\rho, \vtheta, \beta_0) \mapsto \kappa$ based on different parametrization:
\begin{equation}\label{eq:margin_reparam}
    \begin{aligned}
        \kappa & = \mathmakebox[\widthof{$\kappa(\rho, \vtheta, \beta_0)$}][l]{\kappa(\vbeta, \beta_0)} = \min_{i \in [n]} \kappa_i(\vbeta, \beta_0) \\
        & = \kappa(\rho, \vtheta, \beta_0) = \min_{i \in [n]} \kappa_i(\rho, \vtheta, \beta_0).
    \end{aligned}
\end{equation}
As a consequence, the max-margin optimization problem \cref{eq:SVM} or \eqref{eq:SVM-m-reb} can be expressed as
\begin{equation}
	\label{eq:SVM-rho_theta}
    \begin{array}{rl}
    \maximize\limits_{ \rho, \beta_0 \in \R, \vtheta \in \R^{d} } & 
    \kappa(\rho, \vtheta, \beta_0), \\
    \text{subject to} 
	& \rho \in [-1, 1],  
    \vphantom{\maximize\limits_{ \rho }}  \\ 
        & \norm{\vtheta}_2 = 1,  \ \  
    \vtheta \perp \vmu,
    \end{array}
\end{equation}
where
\begin{equation*}
    \kappa(\rho, \vtheta, \beta_0) = 
    \min\limits_{i \in [n]} s_i \left( \rho \norm{\vmu}_2 + y_i\beta_0 + \rho y_i g_i + \sqrt{1 - \rho^2} y_i \< \zz_i, \vtheta \> \right).
\end{equation*}
Recall that $(\hat\vbeta$, $\hat\beta_0)$ is the max-margin solution to \cref{eq:SVM}, and the maximum margin is given by
\begin{equation}\label{eq:max-margin}
    \hat\kappa = \kappa(\hat\vbeta, \hat\beta_0) = \min_{i \in [n]} \kappa_i(\hat\vbeta, \hat\beta_0).
\end{equation}
Similarly, we can also reparametrize $\hat\vbeta$ as in \cref{eq:def-rho-theta}:
\begin{equation}\label{eq:def-rho-theta_hat}
    \hat\rho := \left\< \hat\vbeta, \frac{\vmu}{\norm{\vmu}_2} \right\>,
    \qquad
    \hat\vtheta := 
    \begin{cases} 
        \, \dfrac{\bP_{\vmu}^\perp \hat\vbeta}{\|\bP_{\vmu}^\perp \hat\vbeta\|_2}
        = \dfrac{\bP_{\vmu}^\perp \hat\vbeta}{\sqrt{1 - \hat\rho^2}} , & \ \text{if} \ |\hat\rho| < 1, \\
        \, \vmu_\perp ,         & \ \text{if} \ |\hat\rho| = 1.
    \end{cases}
\end{equation}
Then, $(\hat\rho, \hat\vtheta, \hat\beta_0)$ is the optimal solution to \cref{eq:SVM-rho_theta}\footnote{
    According to \cref{eq:def-rho-theta_hat} and \cref{prop:SVM_tau_relation}, for linearly separable data, $(\hat\rho, \hat\beta_0)$ is the unique solution to \cref{eq:SVM-rho_theta}. If $|\hat\rho| < 1$, then $\hat\vtheta$ is also the unique solution to \cref{eq:SVM-rho_theta}. Otherwise, if $\hat\rho = \pm 1$, then $\sqrt{1 - \hat\rho^2} y_i \< \zz_i, \vtheta \> \equiv 0$ and thus any feasible $\vtheta$ could solve \cref{eq:SVM-rho_theta}.
}. Combining \cref{eq:margin_reparam} and \eqref{eq:max-margin}, the maximum margin can be rewritten as
\begin{equation}\label{eq:max-margin1}
    \hat\kappa = \kappa(\hat\rho, \hat\vtheta, \hat\beta_0) = \min_{i \in [n]} \kappa_i(\hat\rho, \hat\vtheta, \hat\beta_0),
\end{equation}
which is also the optimal objective value of \cref{eq:SVM-rho_theta}. Finally, we define a few quantities:
\begin{gather*}
    \bar g_{+} :=  \frac{1}{n_+}\sum_{i \in \mathcal{I}_+} g_i, 
    \qquad
    \bar g_{-} :=  \frac{1}{n_-}\sum_{i \in \mathcal{I}_-} g_i, 
    \qquad
    \wt g := \frac{\bar g_{+} - \bar g_{-}}{2},
    \\
    \bar \zz_{+} := \frac{1}{n_+}\sum_{i \in \mathcal{I}_+} \zz_i,     
    \qquad
    \bar \zz_{-} := \frac{1}{n_-}\sum_{i \in \mathcal{I}_-} \zz_i,
    \qquad
    \wt\zz  := \frac{\bar\zz_{+} - \bar\zz_{-}}{2}.
\end{gather*}
The proof structure of \cref{thm:main_high-imbal} is as follows:
\begin{enumerate}
    \item In \cref{subsec:highimb_upper}, we provide a (stochastic) tight upper bound for the maximum margin $\hat\kappa$, and a constructed solution $(\wt\rho, \wt\vtheta, \wt\beta_0)$ which approximates $(\hat\rho, \hat\vtheta, \hat\beta_0)$ well.
    \item In \cref{subsec:highimb_asymp}, we derive the asymptotic orders of $(\hat\rho, \hat\vtheta, \hat\beta_0)$ by using $(\wt\rho, \wt\vtheta, \wt\beta_0)$.
    \item In \cref{subsec:highimb_err}, we use these asymptotics to analyze test errors and conclude \cref{thm:main_high-imbal}.
\end{enumerate}

\subsection{A tight upper bound on maximum margin: Proof of \cref{lem:upper_bound}}
\label{subsec:highimb_upper}

The following Lemma provides a data-dependent upper bound on the margin $\kappa(\vbeta, \beta_0)$ which holds for all linear classifiers with $\norm{\vbeta}_2 = 1$. The bound is tight in sense that it can be (almost) achieved by a constructed solution. Therefore, such tightness ensures the optimal margin $\hat\kappa$ should have the same asymptotics given by its upper bound, which also deduces the data is linearly separable with probability tending to one (as $d \to \infty$). 
Notably, \cref{prop:SVM_tau_relation} implies that $\tau$ has no effect on $\hat\vbeta$, and $\hat\kappa \propto (1 + \tau)^{-1}$ in a fixed dataset. Hence, $\tau$ simply scales the magnitude of $\hat\kappa$, and it suffices to consider $\tau = 1$ in the following lemma.
\begin{lem} \label{lem:upper_bound}
    Fix $\tau = 1$. Denote
\begin{equation} \label{eq:kappa_upper}
    \bar\kappa := \sqrt{ ( \norm{\vmu}_2 + \wt g )^2 + \| \bP_{\vmu}^\perp \wt\zz \|_2^2 }.
\end{equation}
    \begin{enumerate}[label=(\alph*)]
        \item 
        \label{lem:upper_bound_a}
        (Upper bound) $\kappa(\rho, \vtheta, \beta_0) \le \bar\kappa$, for any $\rho \in [-1, 1]$, $\vtheta \in \S^{d-1}$, $\vtheta \perp \vmu$, $\beta_0 \in \R$. Moreover,
        \begin{equation*}
            \bar\kappa = \bigl(1 + o_\P(1)\bigr) \sqrt{  d^{b} +  \frac{1}{4} d^{a-c} },
        \end{equation*}
        as $d \to \infty$.
        \item
        \label{lem:upper_bound_b}
        (Tightness) $\kappa(\wt\rho, \wt\vtheta, \wt\beta_0) \ge \bar\kappa - \wt O_{\P}(1)$, where
        \begin{equation} \label{eq:param_star}
            \begin{gathered}
                \wt\rho := \dfrac{ \norm{\vmu}_2 + \wt g }{\sqrt{ ( \norm{\vmu}_2 + \wt g )^2 + \| \bP_{\vmu}^\perp \wt\zz \|_2^2 } },
                \qquad
                \wt\vtheta := \frac{\bP_{\vmu}^\perp \wt\zz}{\| \bP_{\vmu}^\perp \wt\zz \|_2},
                \\
                \wt\beta_0 := - \wt\rho \cdot \frac{\bar g_+ + \bar g_-}{2} - \sqrt{1 - \wt\rho^2} \cdot \left\< \frac{\bar\zz_{+} + \bar\zz_{-}}{2}, \wt\vtheta \right\>
            \end{gathered}
        \end{equation}
        is a feasible solution to \cref{eq:SVM-rho_theta}.
        \item
        \label{lem:upper_bound_c}
        (Asymptotics of $\hat\kappa$) As a consequence, the data is linearly separable with high probability, and the maximum margin satisfies
        $\bar\kappa - \wt O_{\P}(1) \le \hat\kappa \le \bar\kappa$.
    \end{enumerate}
\end{lem}
\begin{proof}
    \textbf{\ref{lem:upper_bound_a}:}
    We reparametrize $\kappa(\vbeta, \beta_0) = \kappa(\rho, \vtheta, \beta_0)$ by using \cref{eq:def-rho-theta} and \eqref{eq:logits}.
    Then, the upper bound is established by calculating the \emph{average logit margin} for each class. Let
    \begin{equation}
            \label{eq:kappa_pm}
            \begin{aligned}
                \bar\kappa_{+}(\rho, \vtheta, \beta_0) & := \frac{1}{n_+}\sum_{i \in \mathcal{I}_+} \kappa_i(\rho, \vtheta, \beta_0)  =
               \rho \norm{\vmu}_2 + \beta_0 + \rho \bar g_{+} + \sqrt{1 - \rho^2} \< \bar\zz_{+}, \vtheta \>,
                \\
                \bar\kappa_{-}(\rho, \vtheta, \beta_0) & := \frac{1}{n_-}\sum_{i \in \mathcal{I}_-} \kappa_i(\rho, \vtheta, \beta_0)  =
                \rho \norm{\vmu}_2 - \beta_0 - \rho \bar g_{-} - \sqrt{1 - \rho^2} \< \bar\zz_{-}, \vtheta \>.
            \end{aligned}
    \end{equation}
    Clearly, $\kappa(\rho, \vtheta, \beta_0) \le \bar\kappa_+(\rho, \vtheta, \beta_0)$ and $\kappa(\rho, \vtheta, \beta_0) \le \bar\kappa_-(\rho, \vtheta, \beta_0)$. By averaging these two bounds,
    \begin{align}
        \kappa(\rho, \vtheta, \beta_0) & \le \frac{\bar\kappa_{+}(\rho, \vtheta, \beta_0) + \bar\kappa_{-}(\rho, \vtheta, \beta_0)}{2} \nonumber \\
        & = 
        \rho \norm{\vmu}_2 + \rho \cdot \frac{\bar g_+ - \bar g_-}{2} + \sqrt{1 - \rho^2} \cdot \left\< \frac{\bar\zz_{+} - \bar\zz_{-}}{2}, \vtheta \right\>
        \nonumber \\
        & = \rho \left( \norm{\vmu}_2 + \wt g \right) + \sqrt{1 - \rho^2} \left< \wt\zz, \vtheta \right\> 
        \nonumber \\
        & \overset{\mathmakebox[0pt][c]{\text{(i)}}}{\le}  \rho \left( \norm{\vmu}_2 +  \wt g \right) + \sqrt{1 - \rho^2}  \| \bP_{\vmu}^\perp \wt\zz \|_2 
        \label{eq:F_AB} \\
        & \overset{\mathmakebox[0pt][c]{\text{(ii)}}}{\le} \sqrt{ ( \norm{\vmu}_2 + \wt g )^2 + \| \bP_{\vmu}^\perp \wt\zz \|_2^2 }
        = \bar\kappa,
        \nonumber
\end{align}
which leads to $\bar\kappa$ defined in \cref{eq:kappa_upper}. Here, (i) is based on the fact that
\begin{equation}\label{eq:theta_min}
    \argmax_{ \vtheta \in \R^d :  \substack{ \| \vtheta \|_2 = 1 \\ \langle \vmu, \vtheta  \rangle = 0 } } \ \langle \wt\zz, \vtheta \rangle = \frac{\bP_{\vmu}^\perp \wt\zz}{\| \bP_{\vmu}^\perp \wt\zz \|_2},
    \qquad
    \max_{ \vtheta \in \R^d :  \substack{ \| \vtheta \|_2 = 1 \\ \langle \vmu, \vtheta  \rangle = 0 } }   \langle \wt\zz, \vtheta \rangle = \| \bP_{\vmu}^\perp \wt\zz \|_2,
\end{equation}
and recall that the optimal $\vtheta$ equals $\wt\vtheta$ defined in \cref{eq:param_star}. Moreover, (ii) is a consequence of Cauchy--Schwarz inequality ($A \in \R$, $B > 0$)
\begin{equation} \label{eq:F_AB_optim}
    \begin{aligned}
        \max_{\rho \in [-1, 1]} \left\{  \rho A + \sqrt{1 - \rho^2} B  \right\}
    & = \max_{\rho \in [-1, 1]} \left\< \begin{pmatrix}
        \rho \\ \sqrt{1-\rho^2}
    \end{pmatrix}, 
    \begin{pmatrix}
        A \\ B
    \end{pmatrix} \right\> = \sqrt{A^2 + B^2}, \\
    \argmax_{\rho \in [-1, 1]} \left\{  \rho A + \sqrt{1 - \rho^2} B \right\}
    & = \frac{A}{\sqrt{A^2 + B^2}},
    \end{aligned}
\end{equation}
and also note that the optimal $\rho$ in (ii) equals $\wt\rho$ defined in \cref{eq:param_star}.

~\\
\noindent
To study the asymptotics of $\bar\kappa$, recall that $\pi = n_+/n = o(1)$, $n_- = n - n_+ = n(1 - o(1))$. Then
\begin{equation*}
    \frac{1}{n_+} + \frac{1}{n_-} = \frac{1}{\pi n } + \frac{1}{n(1 - o(1))} = \frac{1}{\pi n}\bigl(1 + o(1)\bigr).
\end{equation*}
Denote
\begin{equation}\label{eq:alpha_d}
    \alpha_d := \frac12 \sqrt{ \frac{1}{n_+} + \frac{1}{n_-} }
    = \frac{1}{2\sqrt{\pi n}}\bigl(1 + o(1)\bigr).
\end{equation}
\cref{lem:subG}\ref{lem:subG-b} implies $\wt\zz/\alpha_d \sim \subGind(\bzero, \bI_d; K)$. Then according to \cref{lem:subG_concentrate}\ref{lem:subG-Hanson-Wright-II},
\begin{equation*}
    \P\biggl( \biggl| \frac{\| \bP_{\vmu}^\perp \wt\zz \|_2}{ \alpha_d \| \bP_{\vmu}^\perp \|_\mathrm{F}} - 1 \biggr| > t \biggr)
    \le 2 \exp\biggl( -\frac{ct^2}{K^4}\frac{\| \bP_{\vmu}^\perp \|_\mathrm{F}^2}{\| \bP_{\vmu}^\perp \|_{\mathrm{op}}^2} \biggr)
    = 2 \exp\biggl( -\frac{ct^2(d-1)}{K^4} \biggr) ,
\end{equation*}
where $\| \bP_{\vmu}^\perp \|_\mathrm{F} = \sqrt{d - 1}$, $\| \bP_{\vmu}^\perp \|_{\mathrm{op}} = 1$, and $c$ is an absolute constant. Therefore,
\begin{equation} \label{eq:P_wtz}
    \| \bP_{\vmu}^\perp \wt\zz \|_2 = \alpha_d \| \bP_{\vmu}^\perp \|_\mathrm{F} \bigl(1 + o_\P(1)\bigr)
    = \frac{1}{2\sqrt{\pi n}}\big(1 + o(1) \big) \cdot \sqrt{d - 1}\bigl(1 + o_\P(1)\bigr) 
    = \frac12 \sqrt{\frac{d}{\pi n}}\big(1 + o_\P(1) \big).
\end{equation}
In addition, by \cref{lem:subG_concentrate}\ref{lem:subG-Hoeffding},
\begin{equation*}
    \wt g = O_{\P}(\alpha_d) = O_\P\biggl( \frac{1}{\sqrt{\pi n}} \biggr).
\end{equation*}
Recall that $a - c - 1 < 0$. Finally, we have
\begin{align}
        \bar\kappa 
    & = \sqrt{ ( \norm{\vmu}_2 + \wt g )^2 + \| \bP_{\vmu}^\perp \wt\zz \|_2^2 } \nonumber \\
    & = \sqrt{ \left( \norm{\vmu}_2 + O_\P(1/\sqrt{\pi n}) \right)^2 + \frac{d}{4\pi n}\bigl(1 + o_\P(1)\bigr) }
    \nonumber \\
    & =  \sqrt{ \left( d^{b/2} + O_\P\bigl(d^{(a-c-1)/2}\bigr) \right)^2 + \frac{1}{4} d^{a - c} \bigl(1 + o_\P(1)\bigr) }
    \nonumber \\
    & = \sqrt{ d^{b} +  \frac{1}{4} d^{a-c} } \bigl(1 + o_\P(1)\bigr).
    \label{eq:margin_upper}
\end{align}
This concludes the proof of part \ref{lem:upper_bound_a}.

\vspace{0.5\baselineskip}
\noindent
\textbf{\ref{lem:upper_bound_b}:}
Next we show that the upper bound $\bar\kappa$ is nearly attainable, by a constructed solution $(\wt\rho, \wt\vtheta, \wt\beta_0)$ defined in \cref{eq:param_star}. Clearly, $(\wt\rho, \wt\vtheta, \wt\beta_0)$ satisfies the constraints in \cref{eq:SVM-rho_theta}. This candidate solution is motivated by the optimal $(\rho, \vtheta)$ that makes (i) and (ii) equal in \cref{eq:F_AB}, i.e.,
\begin{equation*}
    \bar\kappa = \wt\rho \left(\norm{\vmu}_2 +  \wt g \right) + \sqrt{1 - \wt\rho^2} \< \wt\zz, \wt\vtheta \>, 
\end{equation*}
and $\beta_0$ that balances the magnitude of average logit margins from the two classes, i.e., we choose $\beta_0$ such that $\bar\kappa_+ = \bar\kappa_-$ in \cref{eq:kappa_pm}. Substituting $(\wt\rho, \wt\vtheta, \wt\beta_0)$ back into \cref{eq:logits}, we obtain
\begin{equation*}
        \begin{aligned}
        \kappa_i(\wt\rho, \wt\vtheta, \wt\beta_0) & = \wt\rho \norm{\vmu}_2 + y_i\wt\beta_0 + \wt\rho y_i g_i + \sqrt{1 - \wt\rho^2} y_i \< \zz_i, \wt\vtheta \> \\
        & = 
        \wt\rho \left( \norm{\vmu}_2 + y_i g_i - y_i\frac{\bar g_{+} + \bar g_{-}}{2} \right) + \sqrt{1 - \wt\rho^2} \left\< y_i \zz_i - y_i \frac{\bar\zz_{+} + \bar\zz_{-}}{2}, \wt\vtheta \right\>.
    \end{aligned}
\end{equation*}
Therefore, the difference between each logit margin and the upper bound can be expressed as
\begin{align}
        \bar\kappa - \kappa_i(\wt\rho, \wt\vtheta, \wt\beta_0)
        & = \wt\rho \left( \wt g + y_i\frac{\bar g_{+} + \bar g_{-}}{2} - y_i g_i \right) + \sqrt{1 - \wt\rho^2} \left\< \wt\zz + y_i \frac{\bar\zz_{+} + \bar\zz_{-}}{2} - y_i \zz_i , \wt\vtheta \right\>  \nonumber  \\
        & = 
        \begin{cases} 
            \,  \wt\rho (\bar g_+ - g_i) +  \sqrt{1 - \wt\rho^2} \< \bar\zz_+ - \zz_i, \wt\vtheta \> , & \ \text{if} \ y_i = +1, \\
            \,  \wt\rho (g_i - \bar g_-) +  \sqrt{1 - \wt\rho^2} \< \zz_i - \bar\zz_-, \wt\vtheta \> , & \ \text{if} \ y_i = -1, \end{cases}        \label{eq:logit_diff}
\end{align}
where the leading terms $\rho\norm{\vmu}_2$, $\< \bar\zz_-, \wt\vtheta\>$ (for $i = +1$), $\< \bar\zz_+, \wt\vtheta\>$ (for $i = -1$) are all cancelled out. 
Our goal is to bound the maximum difference over all data points. Note that
\begin{equation}\label{eq:max_diff}
    \begin{aligned}
        & \max_{i \in [n]} \abs{\bar\kappa - \kappa_i(\wt\rho, \wt\vtheta, \wt\beta_0)}
        = 
        \max_{i \in \mathcal{I}_+} \abs{\bar\kappa - \kappa_i(\wt\rho, \wt\vtheta, \wt\beta_0)}
         \vee
        \max_{i \in \mathcal{I}_-} \abs{\bar\kappa - \kappa_i(\wt\rho, \wt\vtheta, \wt\beta_0)}  \\
        \le {} & 
        \max_{i \in \mathcal{I}_+} \left\{ \bigl|g_i - \bar g_+\bigr| + \bigl|\< \zz_i - \bar\zz_+, \wt\vtheta \>\bigr| \right\}
        \vee
        \max_{i \in \mathcal{I}_-} \left\{ \bigl|g_i - \bar g_-\bigr| + \bigl|\< \zz_i - \bar\zz_-, \wt\vtheta \>\bigr| 
        \right\} 
        \\
        \le {} & \left\{ 
            \max_{i \in \mathcal{I}_+} \bigl|g_i - \bar g_+\bigr|
            + \max_{i \in \mathcal{I}_+} \bigl|\< \zz_i - \bar\zz_+, \wt\vtheta \>\bigr|
        \right\}
        \vee
        \left\{ 
            \max_{i \in \mathcal{I}_-} \bigl|g_i - \bar g_-\bigr|
            + \max_{i \in \mathcal{I}_-} \bigl|\< \zz_i - \bar\zz_-, \wt\vtheta \>\bigr|
        \right\}.
    \end{aligned}
\end{equation}
For the first term involving $g_i$'s, recall that $\max_{i \in [n]}\norm{g_i}_{\psi_2} \lesssim K$. Therefore, as per \cref{lem:subG}\ref{lem:subG-c} and \cref{lem:subG_concentrate}\ref{lem:subG-Hoeffding}, $g_i, \bar g_\pm$ are sub-gaussian, and
\begin{equation}\label{eq:max_g}
    \begin{aligned}
        \max_{i \in \mathcal{I}_+} \bigl|g_i - \bar g_+\bigr|
        & \le \max_{i \in \mathcal{I}_+} \abs{g_i} + \abs{\bar g_+} 
        = O_\P(\sqrt{\log n_+}) + O_\P\biggl( \frac{1}{\sqrt{n_+}} \biggr)
        = O_\P(\sqrt{\log d}), \\
        \max_{i \in \mathcal{I}_-} \bigl|g_i - \bar g_-\bigr|
        & \le \max_{i \in \mathcal{I}_-} \abs{g_i} + \abs{\bar g_-} 
        = O_\P(\sqrt{\log n_-}) + O_\P\biggl( \frac{1}{\sqrt{n_-}} \biggr)
        = O_\P(\sqrt{\log d}).
    \end{aligned}
\end{equation}
For the second term involving $\zz_i$'s, note that
\begin{equation}\label{eq:maxmax_z}
    \begin{aligned}
        \max_{i \in \mathcal{I}_+} \bigl| \< \zz_i - \bar\zz_+, \wt\vtheta \> \bigr|
        & \le \max_{i \in \mathcal{I}_+} \frac{1}{n_+} \sum_{j \in \mathcal{I}_+} \bigl| \< \zz_i - \zz_j, \wt\vtheta \> \bigr|
        \le \frac{1}{\| \bP_{\vmu}^\perp \wt\zz \|_2} \max_{i,j \in \mathcal{I}_+} \bigl| \< \zz_i - \zz_j, \bP_{\vmu}^\perp \wt\zz \>  \bigr|,  \\
        \max_{i \in \mathcal{I}_-} \bigl| \< \zz_i - \bar\zz_-, \wt\vtheta \> \bigr|
        & \le \max_{i \in \mathcal{I}_-} \frac{1}{n_-} \sum_{j \in \mathcal{I}_-} \bigl| \< \zz_i - \zz_j, \wt\vtheta \> \bigr|
        \le \frac{1}{\| \bP_{\vmu}^\perp \wt\zz \|_2} \max_{i, j \in \mathcal{I}_-} \bigl| \< \zz_i - \zz_j, \bP_{\vmu}^\perp \wt\zz \>  \bigr|.
    \end{aligned}
\end{equation}
So it remains to bound $\< \zz_i - \zz_j, \bP_{\vmu}^\perp \wt\zz \>$ uniformly. We decompose it as
\begin{equation*}
    \< \zz_i - \zz_j, \bP_{\vmu}^\perp \wt\zz \>
    = 
    \< \zz_i - \zz_j, \wt\zz \> - \left\< \zz_i - \zz_j, \frac{\vmu}{\norm{\vmu}_2} \right\>
            \left\< \wt\zz, \frac{\vmu}{\norm{\vmu}_2} \right\>
    := I -  I\!I.
\end{equation*}
We will show that both $I$ and $I\!I$ are sub-exponential. To bound $\norm{I}_{\psi_1}$ via \cref{lem:subExp}\ref{lem:subExp-b}, we claim the inner product
\begin{equation*}
    I = \< \zz_i - \zz_j, \wt\zz \> = \sum_{k=1}^d (\zz_i - \zz_j)_k (\wt\zz)_k
\end{equation*}
is the sum of $d$ mean-zero random variables, i.e., $\E[(\zz_i - \zz_j)_k (\wt\zz)_k] = 0$, $\forall\, k \in [d]$, where we write $(\ba)_k$ as the $k$-th entry of vector $\ba$. To see this, we decompose $\wt\zz$ into terms that are independent or dependent of $(\zz_i, \zz_j)$.
\begin{itemize}
    \item If $y_i = y_j = +1$ and $i \not= j$, then
    \begin{align*}
        \E[ (\zz_i - \zz_j) \odot \wt\zz ]
        & =  \E\biggl[(\zz_i - \zz_j) \odot 
        \biggl( \underbrace{  \frac1{2 n_+} \!\!\! \sum_{k: \substack{k \not= i, j \\ y_k = +1}} \zz_k + \frac{\bar\zz_-}2  }_{ =: \wt\zz_{-ij}^+ }
        \biggr)
        \biggr] 
        + \frac{1}{2 n_+} \E[ (\zz_i - \zz_j) \odot (\zz_i + \zz_j) ] \\
        & = \E[\zz_i - \zz_j] \odot \E[ \wt\zz_{-ij}^+ ]
        + \frac{1}{2 n_+} \left( \E[\zz_i \odot \zz_i] - \E[\zz_j \odot \zz_j] \right) 
        \qquad (\wt\zz_{-ij}^+  \indep  \zz_i, \zz_j) \\
        & = \bzero \odot \bzero + \frac{1}{2 n_+} ( \bone - \bone ) = \bzero.
    \end{align*}
    \item If $y_i = y_j = -1$ and $i \not= j$, similarly
    \begin{align*}
        \E[ (\zz_i - \zz_j) \odot \wt\zz ]
        & =  \E\biggl[(\zz_i - \zz_j) \odot 
        \biggl( \underbrace{  \frac1{2 n_-} \!\!\! \sum_{k: \substack{k \not= i, j \\ y_k = -1}} \zz_k + \frac{\bar\zz_+}2  }_{ =: \wt\zz_{-ij}^- }
        \biggr)
        \biggr] 
        + \frac{1}{2 n_-} \E[ (\zz_i - \zz_j) \odot (\zz_i + \zz_j) ] \\
        & = \E[\zz_i - \zz_j] \odot \E[ \wt\zz_{-ij}^- ]
        + \frac{1}{2 n_-} \left( \E[\zz_i \odot \zz_i] - \E[\zz_j \odot \zz_j] \right) 
        \qquad (\wt\zz_{-ij}^-  \indep  \zz_i, \zz_j) 
        \\
        & = \bzero.
    \end{align*}
\end{itemize}
Therefore, when $d$ is large enough, we have
\begin{align*}
    \norm{I}_{\psi_1} & = \norm{\< \zz_i - \zz_j, \wt\zz \>}_{\psi_1}
    = \norm{ \sum_{k=1}^d (\zz_i - \zz_j)_k (\wt\zz)_k }_{\psi_1} \\
    & \overset{\mathmakebox[0pt][c]{\smash{\text{(i)}}}}{\lesssim} \sqrt{d} \max_{1 \le k \le d} \norm{ (\zz_i - \zz_j)_k (\wt\zz)_k }_{\psi_1} \\
    & \overset{\mathmakebox[0pt][c]{\smash{\text{(ii)}}}}{\le} \sqrt{d} \max_{1 \le k \le d} \norm{ (\zz_i - \zz_j)_k }_{\psi_2}
    \max_{1 \le k \le d} \norm{ (\wt\zz)_k }_{\psi_2} \\
    & \overset{\mathmakebox[0pt][c]{\smash{\text{(iii)}}}}{\lesssim} \sqrt{d} K \cdot \alpha_{d} K 
    \lesssim \sqrt{\frac{d}{\pi n}} K^2,
\end{align*}
where (i) results from coordinate independence and \cref{lem:subExp}\ref{lem:subExp-b}, (ii) is from \cref{lem:subExp}\ref{lem:subExp-d}, and (iii) is based on $\wt\zz/\alpha_{d} , (\zz_i - \zz_j)/\sqrt{2} \sim \subGind(\bzero, \bI_d; K)$. For the term $I\!I$, we have
\begin{align*}
    \norm{I\!I}_{\psi_2} & = \norm{\left\< \zz_i - \zz_j, \frac{\vmu}{\norm{\vmu}_2} \right\>
    \left\< \wt\zz, \frac{\vmu}{\norm{\vmu}_2} \right\>}_{\psi_1} \\
    & \overset{\mathmakebox[0pt][c]{\smash{\text{(i)}}}}{\le} \norm{\left\< \zz_i - \zz_j, \frac{\vmu}{\norm{\vmu}_2} \right\>}_{\psi_2}
    \norm{\left\< \wt\zz, \frac{\vmu}{\norm{\vmu}_2} \right\>}_{\psi_2} \\
    & \overset{\mathmakebox[0pt][c]{\smash{\text{(ii)}}}}{\lesssim}  \max_{1 \le k \le d} \norm{ (\zz_i - \zz_j)_k }_{\psi_2}
    \max_{1 \le k \le d} \norm{ (\wt\zz)_k }_{\psi_2} \\
    & \lesssim K \cdot \alpha_{d} K 
    \lesssim \frac{1}{\sqrt{\pi n}} K^2,
\end{align*}
where (i) is from \cref{lem:subExp}\ref{lem:subExp-d}, and (ii) is from \cref{lem:subG}\ref{lem:subG-b}. Hence,
\begin{equation*}
    \norm{\frac{\< \zz_i - \zz_j, \bP_{\vmu}^\perp \wt\zz \>}{\sqrt{d/\pi n}}}_{\psi_1}
    \le \sqrt{\frac{\pi n}{d}} \bigl( \norm{I}_{\psi_1} + \norm{I\!I}_{\psi_1} \bigr)
    \lesssim K^2.
\end{equation*}
Substituting this back into \cref{eq:maxmax_z}, referring to \cref{eq:P_wtz} and \cref{lem:subExp}\ref{lem:subExp-c}, we obtain
\begin{align}
    \!\!\! \max_{i \in \mathcal{I}_+} \bigl| \< \zz_i - \bar\zz_+, \wt\vtheta \> \bigr|
        & \le \frac{1}{\| \bP_{\vmu}^\perp \wt\zz \|_2} \max_{\substack{i \in \mathcal{I}_+ \\ j \in \mathcal{I}_+}} \bigl| \< \zz_i - \zz_j, \bP_{\vmu}^\perp \wt\zz \>  \bigr|
    = \bigl(1 + o_\P(1)\bigr) \max_{\substack{i \in \mathcal{I}_+ \\ j \in \mathcal{I}_+}} \abs{\frac{\< \zz_i - \zz_j, \bP_{\vmu}^\perp \wt\zz \>}{\sqrt{d/\pi n}}}
    \notag \\
    & =  O_{\P}(\log n_+^2) = O_{\P}(\log d), \notag \\
    \!\!\! \max_{i \in \mathcal{I}_-} \bigl| \< \zz_i - \bar\zz_-, \wt\vtheta \> \bigr|
        & \le \frac{1}{\| \bP_{\vmu}^\perp \wt\zz \|_2} \max_{\substack{i \in \mathcal{I}_- \\ j \in \mathcal{I}_-}} \bigl| \< \zz_i - \zz_j, \bP_{\vmu}^\perp \wt\zz \>  \bigr|
    = \bigl(1 + o_\P(1)\bigr) \max_{\substack{i \in \mathcal{I}_- \\ j \in \mathcal{I}_-}} \abs{\frac{\< \zz_i - \zz_j, \bP_{\vmu}^\perp \wt\zz \>}{\sqrt{d/\pi n}}} \notag \\
    & =  O_{\P}(\log n_-^2) = O_{\P}(\log d).
    \label{eq:max_z}
\end{align}
Finally, incorporating \cref{eq:max_g} and \cref{eq:max_z} into \cref{eq:max_diff}, we have
\begin{align*}
        & \max_{i \in [n]} \abs{\bar\kappa - \kappa_i(\wt\rho, \wt\vtheta, \wt\beta_0)}
        \\
        \le {} & \left\{ 
            \max_{i \in \mathcal{I}_+} \bigl|g_i - \bar g_+\bigr|
            + \max_{i \in \mathcal{I}_+} \bigl|\< \zz_i - \bar\zz_+, \wt\vtheta \>\bigr|
        \right\}
        \vee
        \left\{ 
            \max_{i \in \mathcal{I}_-} \bigl|g_i - \bar g_-\bigr|
            + \max_{i \in \mathcal{I}_-} \bigl|\< \zz_i - \bar\zz_-, \wt\vtheta \>\bigr|
        \right\} \\
        \le {} & \left\{ O_\P(\sqrt{\log d}) + O_\P(\log d) \right\} \vee \left\{ O_\P(\sqrt{\log d}) + O_\P(\log d) \right\} 
        = O_\P(\log d).
\end{align*}
Therefore, the difference between the margin of classifier characterized by $(\wt\rho, \wt\vtheta, \wt\beta_0)$ and its upper bound $\bar\kappa$ is bounded by
\begin{equation*}
    \bar\kappa - \kappa(\wt\rho, \wt\vtheta, \wt\beta_0) 
    = 
     \bar\kappa - \min_{i \in [n]} \kappa_i(\wt\rho, \wt\vtheta, \wt\beta_0) 
    = \max_{i \in [n]} \abs{\bar\kappa - \kappa_i(\wt\rho, \wt\vtheta, \wt\beta_0)} = O_\P(\log d) = \wt O_\P(1).
\end{equation*}
This concludes the proof of part \ref{lem:upper_bound_b}.

\vspace{0.5\baselineskip}
\noindent
\textbf{\ref{lem:upper_bound_c}:}
According to max-margin optimization problem \cref{eq:SVM-rho_theta}, note that
\begin{equation*}
    \hat\kappa = 
    \max_{ \substack{ \rho \in [-1, 1],  \beta_0 \in \R
    \\
    \vtheta \in \S^{d-1}, \vtheta \perp \vmu } } \kappa(\rho, \vtheta, \beta_0)
    \ge \kappa(\wt\rho, \wt\vtheta, \wt\beta_0),
\end{equation*}
hence the asymptotics of $\hat\kappa$ is followed by (a) and (b). As $d \to \infty$, note that
\begin{equation*}
    \hat\kappa \ge \bar\kappa - \wt O_\P(1)
    = \bigl(1 + o_\P(1)\bigr) \sqrt{  d^{b} +  \frac{1}{4} d^{a-c} },
    \qquad
    \sqrt{  d^{b} +  \frac{1}{4} d^{a-c} } \ge d^{b/2} \to + \infty,
\end{equation*}
which implies $\hat\kappa$ diverges with high probability, i.e., $\lim_{d \to \infty} \P(\hat\kappa > C) = 1$, $\forall\, C \in \R$. As the result, $\P\{\text{linearly separable}\} = \P(\hat\kappa > 0) \to 1$ as $d \to \infty$, deducing $\|\hat\vbeta\|_2 = 1$ with high probability. This concludes the proof of part \ref{lem:upper_bound_c}.
\end{proof}

\subsection{Asymptotics of optimal parameters: Proofs of \cref{lem:rho_hat}, \ref{lem:theta_hat_z}, \ref{lem:beta0_asymp}}
\label{subsec:highimb_asymp}

Followed by tightness of the upper bound $\bar\kappa$, we show that the optimal parameters $\hat\rho, \hat\vtheta$ should be very ``close'' to the constructed solution $\wt\rho, \wt\vtheta$ defined in \cref{eq:param_star} in some sense. On the event that the data is linearly separable, we have showed that $\hat\vbeta$, and therefore both $\hat\rho$ and $\hat\vtheta$, do not depend on $\tau$ in \cref{prop:SVM_tau_relation}. Hence, it still suffices to consider $\tau = 1$ in our proof.

\subsubsection{Asymptotic order of $\hat\rho$: Proofs of \cref{lem:rho_hat}}

The following technical Lemma is important for deriving the asymptotics of $\hat\rho$, which introduces a function of $\rho$ used implicitly in \cref{eq:F_AB} and \eqref{eq:F_AB_optim} for optimization.

\begin{lem} \label{lem:F_AB}
    Define $F_{A, B}(\rho) = \rho A + \sqrt{1 - \rho^2} B$, $\rho \in [-1, 1]$, with $A \in \R$, $B > 0$. Then
\begin{equation*}
    F_{A, B}'(\rho) = A - \frac{\rho}{\sqrt{1 - \rho^2}} B,
    \qquad
    F_{A, B}''(\rho) = -\frac{1}{(1 - \rho^2)^{3/2}} B,
\end{equation*}
which implies $F_{A,B}$ is $B$-strongly concave, that is, for all $\rho_1, \rho_2 \in [-1, 1]$,
\begin{equation*}
    F_{A,B}(\rho_2) \le F_{A,B}(\rho_1) + F'_{A,B}(\rho_1)(\rho_2 - \rho_1) - \frac12 B (\rho_2 - \rho_1)^2.
\end{equation*} 
Moreover,
\begin{equation*}
    \argmax_{\rho \in [-1, 1]} F_{A,B}(\rho) = \frac{A}{\sqrt{A^2 + B^2}},  \qquad \max_{\rho \in [-1, 1]} F_{A,B}(\rho) = \sqrt{A^2 + B^2}.
\end{equation*}
\end{lem}
\begin{proof}
    Strongly concavity is given by direct calculation and the fact that
    \begin{equation*}
        \sup_{\rho \in [-1, 1]} F''_{A,B}(\rho) = -B.
    \end{equation*}
    The optimality condition is already derived in \cref{eq:F_AB_optim}. This concludes the proof.
\end{proof}

\noindent
In the rest of this section, the (stochastic) parameters $A, B$ are defined as
\begin{equation} \label{eq:AB_def}
    A := \norm{\vmu}_2 + \wt g
    = d^{b/2} \bigl(1 + o_{\P}(1)\bigr),
    \qquad
    B := \| \bP_{\vmu}^\perp \wt\zz \|_2
    = \frac12 d^{(a-c)/2} \bigl(1 + o_{\P}(1)\bigr).
\end{equation}
Then followed by \cref{lem:F_AB}, we have $\wt\rho = \argmax_{\rho \in [-1, 1]} F_{A,B}(\rho)$ and $F'_{A,B}(\wt\rho) = 0$, where $\wt\rho$ is defined in \cref{eq:param_star}. The following Lemma describes the asymptotics of $\hat\rho$ with respect to $\wt\rho$.

\begin{lem}[Asymptotics of $\hat\rho$ and $\wt\rho$] \label{lem:rho_hat}
    Suppose that $a < c + 1$.
    \begin{enumerate}[label=(\alph*)]
        \item \label{lem:rho_hat(a)} If $a < b + c$, then $\wt\rho = 1 - o_\P(1)$, $\hat\rho = 1 - o_\P(1)$, and
        \begin{equation*}
            \sqrt{1 - \wt\rho^2} = \frac12 d^{(a-b-c)/2}\bigl( 1 + o_\P(1) \bigr).
        \end{equation*}
        Moreover, we further assume:
        \begin{itemize}
            \item[i.] If $a > \frac{b}{2} + c$, then $\sqrt{1 - \hat\rho^2} = \sqrt{1 - \wt\rho^2} \bigl(1 + o_\P(1)\bigr)$.
            \item[ii.] If $a \le \frac{b}{2} + c$, then $\sqrt{1 - \hat\rho^2} = \wt O_\P(d^{-b/4})$ and thus $\sqrt{1 - \hat\rho^2}\sqrt{d/\pi n} = \wt O_\P(1)$.
        \end{itemize}
        \item \label{lem:rho_hat(b)} If $a > b + c$, then $\wt\rho = o_\P(1)$, $\hat\rho = o_\P(1)$, and
        \begin{equation*}
            \wt\rho = 2d^{(b-a+c)/2}\bigl(1 + o_\P(1)\bigr).
        \end{equation*}
        Moreover, we further assume:
        \begin{itemize}
            \item[i.] If $a < 2b + c$, then $\hat\rho = \wt\rho \bigl(1 + o_\P(1)\bigr)$.
            \item[ii.] If $a > 2b + c$, then $\hat \rho = \wt O_\P(d^{-(a-c)/4})$ and thus $\hat \rho \norm{\vmu}_2 = o_\P(1)$.
        \end{itemize}
    \end{enumerate}
\end{lem}
\begin{proof}
According to \cref{eq:param_star} and \eqref{eq:AB_def}, an explicit expression of $\wt\rho$ is given by
\begin{equation}\label{eq:rho_star}
    \wt\rho = 
    \frac{A}{\sqrt{A^2 + B^2}}
    =
    \dfrac{ \norm{\vmu}_2 + \wt g }{\sqrt{ ( \norm{\vmu}_2 + \wt g )^2 + \| \bP_{\vmu}^\perp \wt\zz \|_2^2 } }
    = \frac{d^{b/2}}{\sqrt{ d^{b} + \frac14 d^{a-c} }}  \bigl(1 + o_{\P}(1)\bigr).
\end{equation}
In order to connect $\hat\rho$ with $\wt\rho$, recall \cref{eq:F_AB} that
\begin{equation*}
    \kappa(\rho, \vtheta, \beta_0) \le F_{A,B}(\rho) \le F_{A,B}(\wt\rho) = \bar\kappa,
    \qquad \forall\, \rho \in [-1, 1], \  \vtheta \in \S^{d-1}, \vtheta \perp \vmu, \  \beta_0 \in \R.
\end{equation*}
Apply this to $\hat\kappa = \kappa(\hat\rho, \hat\vtheta, \hat\beta_0)$ and use \cref{lem:upper_bound}, we have
\begin{equation}\label{eq:F_AB_diff}
    0 \le F_{A,B}(\wt\rho) - F_{A,B}(\hat\rho) \le \wt O_\P(1).
\end{equation}
Since $\wt O_\P(1)/F_{A,B}(\wt\rho) = \wt O_\P(1)/\sqrt{A^2 + B^2} \le \wt O_\P(d^{-b/2}) = o_\P(1)$, it implies
\begin{equation*}
    1 - o_\P(1)  =  \frac{F_{A,B}(\hat\rho)}{F_{A,B}(\wt\rho)} 
    = \frac{\hat\rho A }{\sqrt{A^2 + B^2}} + \frac{ \sqrt{1 - \hat\rho^2} B}{\sqrt{A^2 + B^2}}
    = \hat\rho \wt\rho + \sqrt{1 - \hat\rho^2}\sqrt{1 - \wt\rho^2}.
\end{equation*}
Therefore,
\begin{equation}\label{eq:rho_star_to_hat}
    \begin{aligned}
        \wt\rho = 1 - o_\P(1) & \ \Longrightarrow & \hat\rho = 1 - o_\P(1)
        \\
        \wt\rho = o_\P(1)     & \ \Longrightarrow & \hat\rho = o_\P(1)
    \end{aligned}
\end{equation}

\vspace{0.5\baselineskip}
\noindent
\textbf{\ref{lem:rho_hat(a)}:}
If $a - c < b$, then $d^{b/2} \gg d^{(a-c)/2}$ and by \cref{eq:rho_star} and \eqref{eq:rho_star_to_hat} we have $\wt\rho, \hat\rho = 1 - o_\P(1)$. Also,
\begin{equation*}
    \sqrt{1 - \wt\rho^2} = \frac{B}{\sqrt{A^2 + B^2}}
    = \frac{\frac12 d^{(a-b-c)/2}}{\sqrt{ 1 + \frac14 d^{a-b-c} }}\bigl(1 + o_{\P}(1)\bigr)
    = \frac12 d^{(a-b-c)/2}\bigl(1 + o_{\P}(1)\bigr).
\end{equation*}
To derive the precise order of $\sqrt{1 - \hat\rho^2}$, we define $r := \sqrt{1 - \rho^2}$ and $F_{B,A}(r) := rB + \sqrt{1 - r^2}A$. Then $F_{A,B}(\rho) = F_{B,A}(r)$ for any $\rho \in [0, 1]$. We similarly define $\hat r := \sqrt{1 - \hat\rho^2}$ and $\wt r := \sqrt{1 - \wt\rho^2}$. On the event $\mathcal{E} = \{ A > 0, \wt\rho >0, \hat\rho > 0 \}$, by \cref{lem:F_AB}, we have
\begin{equation*}
    F_{A,B}(\hat\rho) - F_{A,B}(\wt\rho) 
    = F_{B,A}(\hat r) - F_{B,A}(\wt r) 
    \le -\frac12 A (\hat r - \wt r)^2.
\end{equation*}
Combined with \cref{eq:F_AB_diff}, it implies
\begin{equation*}
    (\hat r - \wt r)^2 \le  \frac{2}{A}\bigl( F_{A,B}(\wt\rho) - F_{A,B}(\hat\rho) \bigr) \le \wt O_\P(d^{-b/2}),
\end{equation*}
so $|\hat r - \wt r| = \wt O_\P(d^{-b/4})$. Now consider different scenarios. Recall that $\wt r = \frac12 d^{(a-b-c)/2}\bigl(1 + o_{\P}(1)\bigr)$.
\begin{itemize}
    \item If $a - c > b/2$, then $|\hat r - \wt r|/\wt r = \wt O_\P(d^{(-2a + b + 2c)/4}) = o_\P(1)$, deduces $\hat r = \wt r \bigl(1 + o_{\P}(1)\bigr)$.
    \item If $a - c \le b/2$, then we only get $\hat r = \wt O_\P(d^{-b/4})$, and $\hat r \sqrt{d/\pi n} = \wt O_\P(d^{(2a - b-2c)/4})  \le \wt O_\P(1)$.
\end{itemize}
Recall that these hold on event $\mathcal{E}$. Since $\P(\mathcal{E}) \to 1$ as $d \to \infty$, these asymptotic results involving $o_\P(\,\cdot\,)$ and $\wt O_\P(\,\cdot\,)$ also hold on the whole sample space $\Omega$. This concludes the proof of part \ref{lem:rho_hat(a)}.

\vspace{0.5\baselineskip}
\noindent
\textbf{\ref{lem:rho_hat(b)}:}
If $a - c > b$, then $d^{b/2} \ll d^{(a-c)/2}$ and by \cref{eq:rho_star} and \eqref{eq:rho_star_to_hat} we have $\wt\rho, \hat\rho = o_\P(1)$. Also,
\begin{equation*}
    \wt\rho = 
    \frac{A}{\sqrt{A^2 + B^2}}
    = \frac{d^{(b-a+c)/2}}{\sqrt{ d^{b-a+c} + \frac14 }}  \bigl(1 + o_{\P}(1)\bigr)
    = 2d^{(b-a+c)/2} \bigl(1 + o_{\P}(1)\bigr).
\end{equation*}
Again, by \cref{lem:F_AB},
\begin{equation*}
    F_{A,B}(\hat\rho) - F_{A,B}(\wt\rho) \le -\frac12 B(\hat\rho - \wt\rho)^2.
\end{equation*}
Combined with \cref{eq:F_AB_diff}, it implies
\begin{equation*}
    (\hat\rho - \wt\rho)^2 \le  \frac{2}{B}\bigl( F_{A,B}(\wt\rho) - F_{A,B}(\hat\rho) \bigr) \le \wt O_\P(d^{-(a-c)/2}),
\end{equation*}
so $|\hat\rho - \wt\rho| = \wt O_\P(d^{-(a-c)/4})$. Now consider different scenarios.
\begin{itemize}
    \item If $a - c < 2b$, then $|\hat \rho - \wt\rho|/\wt\rho = \wt O_\P(d^{(a - 2b -c)/4}) = o_\P(1)$, deduces $\hat \rho = \wt\rho \bigl(1 + o_{\P}(1)\bigr)$.
    \item If $a - c > 2b$, then we only get $\hat \rho = \wt O_\P(d^{-(a-c)/4})$, and $\hat \rho \norm{\vmu}_2 = \wt O_\P(d^{(2b-a+c)/4}) = o_\P(1)$.
\end{itemize}
This concludes the proof of part \ref{lem:rho_hat(b)}.
\end{proof}
\begin{rem}
    In each part i. of \cref{lem:rho_hat}\ref{lem:rho_hat(a)} and \ref{lem:rho_hat(b)}, we can derive the precise asymptotic of $\hat\rho$, which is same as $\wt\rho$. It is difficult to do so in part ii. of \ref{lem:rho_hat(a)} and \ref{lem:rho_hat(b)}. However, as we will show in \cref{lem:theta_hat_z} and \ref{lem:beta0_asymp}, in case ii. the corresponding term ($\sqrt{1 - \hat\rho}$ or $\hat\rho$) is negligible, which won't affect the asymptotics of test errors.
\end{rem}

\subsubsection{Asymptotic order of $\<\zz_i, \hat\vtheta \>$'s on the margin: Proof of \cref{lem:theta_hat_z}}

Next, we discuss the asymptotics of $\hat\vtheta$ and $\wt\vtheta$. In fact, it suffices to consider the magnitude of their projection on some ``important'' $\zz_i$, which is related to the \emph{support vectors}, defined in \cref{eq:SV_def}. As we mentioned, $\mathcal{SV}_+(\vbeta), \mathcal{SV}_-(\vbeta)$ only depend on $\vbeta$ and $(\XX, \yy)$, not $\beta_0$ or $\tau$. If we fix $\rho = \hat\rho$, then the dependency of $\mathcal{SV}_\pm$ on $\vbeta$ only comes from $\vtheta$. So, recalling \cref{eq:logits}, 
\[ \kappa_i(\rho, \vtheta, \beta_0) = s_i \left( \rho \norm{\vmu}_2 + y_i\beta_0 + \rho y_i g_i + \sqrt{1 - \rho^2} y_i \< \zz_i, \vtheta \> \right), \]
we can rewrite \cref{eq:SV_def} in terms of $\vtheta$:
\begin{equation}\label{eq:SV_def_theta}
	\begin{aligned}
		\mathcal{SV}_+ = \mathcal{SV}_+(\vtheta) & :=  \argmin_{i \in \mathcal{I}_+} \kappa_i(\hat\rho, \vtheta, \beta_0)
	= \argmin_{i \in \mathcal{I}_+} \left\{ \phantom{-} \hat\rho g_i + \sqrt{1 - \hat\rho^2} \< \zz_i, \vtheta \> \right\} , \\
		\mathcal{SV}_- = \mathcal{SV}_-(\vtheta) & :=  \argmin_{i \in \mathcal{I}_-} \kappa_i(\hat\rho, \vtheta, \beta_0)
	= \argmin_{i \in \mathcal{I}_-} \left\{  - \hat\rho g_i - \sqrt{1 - \hat\rho^2} \< \zz_i, \vtheta \> \right\} . \\
	\end{aligned}
\end{equation}
As before, let $\mathsf{sv}_+(\vtheta), \mathsf{sv}_-(\vtheta)$ be (the indices of) any positive and negative support vectors, i.e.,
\begin{equation*}
	\mathsf{sv}_+(\vtheta) \in \mathcal{SV}_+(\vtheta),
	\qquad
	\mathsf{sv}_-(\vtheta) \in \mathcal{SV}_-(\vtheta).
\end{equation*} 
Now, recall that whenever a slope parameter $\vbeta$ is given, the optimal intercept $\check\beta_0 := \check\beta_0(\vbeta)$ (defined in \cref{eq:beta0_optim}) must satisfy the \emph{margin-balancing} condition \cref{eq:margin-bal}, according to \cref{lem:indep_tau}. Hence, fixing $\rho = \hat\rho$ and considering arbitrary $\vtheta$, we can rewrite \cref{eq:margin_pm} and \eqref{eq:margin-bal} as
\begin{equation}\label{eq:svm_sv_bal}
    \begin{aligned}
       \kappa(\hat\rho, \vtheta, \check\beta_0) & = \kappa_{\mathsf{sv}_+(\vtheta)}(\hat\rho, \vtheta, \check\beta_0) = 
       \hat\rho \norm{\vmu}_2 + \check\beta_0 + \hat\rho g_{\mathsf{sv}_+(\vtheta)} + \sqrt{1 - \hat\rho^2} \< \zz_{\mathsf{sv}_+(\vtheta)}, \vtheta \> \\
       & = \kappa_{\mathsf{sv}_-(\vtheta)}(\hat\rho, \vtheta, \check\beta_0) = 
       \hat\rho \norm{\vmu}_2 - \check\beta_0 - \hat\rho g_{\mathsf{sv}_-(\vtheta)} - \sqrt{1 - \hat\rho^2} \< \zz_{\mathsf{sv}_-(\vtheta)}, \vtheta \>.
    \end{aligned}
\end{equation}
In particular, if $\vtheta = \hat\vtheta$, we denote $\mathsf{sv}_+(\hat\vtheta) \in \mathcal{SV}_+(\hat\vtheta)$, $\mathsf{sv}_-(\hat\vtheta) \in \mathcal{SV}_-(\hat\vtheta)$ as the support vectors of max-margin classifier. The Lemma below describes the magnitude of $\< \zz_{\mathsf{sv}_+(\hat\vtheta)}, \hat\vtheta \>$ and $\< \zz_{\mathsf{sv}_-(\hat\vtheta)}, \hat\vtheta \>$.

\begin{lem}[Asymptotics of $\< \zz_{i}, \hat\vtheta \>$'s for support vectors] \label{lem:theta_hat_z}
    Suppose that $a < c + 1$.
    \begin{enumerate}[label=(\alph*)]
        \item \label{lem:theta_hat_z(a)}
        If $a < b + c$, then 
        \begin{equation*}
            \sqrt{1 - \hat\rho^2} \< \zz_{\mathsf{sv}_+(\hat\vtheta)}, \hat\vtheta \>  =  \wt O_\P(d^{a-\frac{b}{2}-c} \vee 1),
            \qquad
            \sqrt{1 - \hat\rho^2} \< \zz_{\mathsf{sv}_-(\hat\vtheta)}, \hat\vtheta \>  =  \wt O_\P(1).
        \end{equation*}
        \item \label{lem:theta_hat_z(b)}
        If $a > b + c$, then 
        \begin{equation*}
            \< \zz_{\mathsf{sv}_+(\hat\vtheta)}, \hat\vtheta \> = \sqrt{\frac{d}{\pi n}}\bigl(1 + o_{\P}(1)\bigr),
            \qquad
            \< \zz_{\mathsf{sv}_-(\hat\vtheta)}, \hat\vtheta \> = \wt O_\P(1).
        \end{equation*}
    \end{enumerate}
\end{lem}
\begin{proof}
    $\mathcal{SV}_\pm(\vtheta)$ may not be tractable, since it involves a nuisance term $\hat\rho g_i$ as defined in \cref{eq:SV_def_theta}. Therefore, we introduce a proxy of support vectors, which is easier to work with. Formally, let
    \begin{equation}\label{eq:V_def_theta}
        \begin{aligned}
            \mathcal{V}_+ = \mathcal{V}_+(\vtheta) & :=  
        \argmin_{i \in \mathcal{I}_+}    +  \< \zz_i, \vtheta \>  , \\
            \mathcal{V}_- = \mathcal{V}_-(\vtheta) & :=  
        \argmin_{i \in \mathcal{I}_-}    -  \< \zz_i, \vtheta \>  , \\
        \end{aligned}
    \end{equation}
    where $\mathcal{V}_+, \mathcal{V}_-$ are sets of (the indices of) the smallest $y_i \< \zz_{i}, \vtheta \>$ from each class. Similarly, let
    \begin{equation*}
        \mathsf{v}_+(\vtheta) \in  \mathcal{V}_+(\vtheta),
        \qquad
        \mathsf{v}_-(\vtheta) \in  \mathcal{V}_-(\vtheta),
    \end{equation*}
    which are arbitrary elements in $\mathcal{V}_+(\vtheta)$ and $\mathcal{V}_-(\vtheta)$. Note that $\mathcal{V}_\pm$ is simply $\mathcal{SV}_\pm$ but ignoring term $\hat\rho g_i$. Indeed, as we will show later, the impact of $\hat\rho g_i = O_\P(1)$ is almost negligible.

    We are going to prove \cref{lem:theta_hat_z} by deriving tight upper bounds and lower bounds for both $\pm \sqrt{1 - \hat\rho^2} \< \zz_{\mathsf{sv}_\pm(\hat\vtheta)}, \hat\vtheta \>$. Then we conclude the precise asymptotics by verifying the upper and lower bounds are matched.

    \paragraph{Upper bounds}
    Applying the same idea as \cref{eq:kappa_pm}, we can bound $\pm \< \zz_{\mathsf{v}_\pm(\vtheta)}, \vtheta \>$ via averaging:
    \begin{equation}
        \label{eq:zV_upper}
            \< \zz_{\mathsf{v}_+(\vtheta)}, \vtheta \> \le \< \bar\zz_{+}, \vtheta \> \le  \| \bP_{\vmu}^\perp \bar\zz_+ \|_2,
            \qquad
            - \< \zz_{\mathsf{v}_-(\vtheta)}, \vtheta \> \le - \< \bar\zz_{-}, \vtheta \> \le  \| \bP_{\vmu}^\perp \bar\zz_- \|_2,
    \end{equation}
    where the second inequality for each comes from \cref{eq:theta_min}. 
    To create a connection between $\mathsf{sv}_\pm(\hat\vtheta)$ and $\mathsf{v}_\pm(\hat\vtheta)$, note that by definition \cref{eq:SV_def_theta}
\begin{equation*}
    \begin{aligned}
        \phantom{-} \hat\rho g_{\mathsf{sv}_+(\hat\vtheta)} + \sqrt{1 - \hat\rho^2} \< \zz_{\mathsf{sv}_+(\hat\vtheta)}, \hat\vtheta \> 
        & \le \phantom{-} \hat\rho g_{\mathsf{v}_+(\hat\vtheta)} + \sqrt{1 - \hat\rho^2} \< \zz_{\mathsf{v}_+(\hat\vtheta)}, \hat\vtheta \>,
        \\
        - \hat\rho g_{\mathsf{sv}_-(\hat\vtheta)} - \sqrt{1 - \hat\rho^2} \< \zz_{\mathsf{sv}_-(\hat\vtheta)}, \hat\vtheta \> 
        & \le - \hat\rho g_{\mathsf{v}_-(\hat\vtheta)} - \sqrt{1 - \hat\rho^2} \< \zz_{\mathsf{v}_-(\hat\vtheta)}, \hat\vtheta \>.
    \end{aligned}
\end{equation*}
Using \cref{eq:zV_upper}, therefore we obtain the following non-asymptotic upper bounds on $\< \zz_{\mathsf{sv}_\pm(\hat\vtheta)}, \hat\vtheta \>$:
\begin{equation}
    \label{eq:zSV_upper}
    \begin{aligned}
        \phantom{-} \sqrt{1 - \hat\rho^2} \< \zz_{\mathsf{sv}_+(\hat\vtheta)}, \hat\vtheta \> \ \ 
        & \le   \phantom{-} \sqrt{1 - \hat\rho^2} \< \zz_{\mathsf{v}_+(\hat\vtheta)}, \hat\vtheta \> + \hat\rho \bigl(g_{\mathsf{v}_+(\hat\vtheta)} - g_{\mathsf{sv}_+(\hat\vtheta)} \bigr)  \\
        & \le  \phantom{-} \sqrt{1 - \hat\rho^2} \| \bP_{\vmu}^\perp \bar\zz_+ \|_2 + 
        2 \hat\rho \max_{i \in [n]} \abs{g_i} , \\
        - \sqrt{1 - \hat\rho^2} \< \zz_{\mathsf{sv}_-(\hat\vtheta)}, \hat\vtheta \> \ \ 
        & \le  - \sqrt{1 - \hat\rho^2} \< \zz_{\mathsf{v}_-(\hat\vtheta)}, \hat\vtheta \> - \hat\rho \bigl(g_{\mathsf{v}_-(\hat\vtheta)} - g_{\mathsf{sv}_-(\hat\vtheta)} \bigr) \\
        & \le  \phantom{-} \sqrt{1 - \hat\rho^2} \| \bP_{\vmu}^\perp \bar\zz_- \|_2 + 
        2 \hat\rho \max_{i \in [n]} \abs{g_i} .
    \end{aligned}
\end{equation}
To compute its asymptotics, recall that $\sqrt{n_+} \cdot \bar\zz_+ \sim \subGind(\bzero, \bI_d; K)$, $\sqrt{n_-} \cdot \bar\zz_- \sim \subGind(\bzero, \bI_d; K)$, and $\| \bP_{\vmu}^\perp \|_\mathrm{F} = \sqrt{d - 1}$. Then by \cref{lem:subG_concentrate}\ref{lem:subG-Hanson-Wright-II},
\begin{align}
        \| \bP_{\vmu}^\perp \bar\zz_+ \|_2 & = \frac{1}{\sqrt{n_+}} \| \bP_{\vmu}^\perp \|_\mathrm{F} \bigl(1 + o_\P(1)\bigr)
    = \sqrt{\frac{d}{\pi n}}\big(1 + o_\P(1) \big), \notag \\
        \| \bP_{\vmu}^\perp \bar\zz_- \|_2 & = \frac{1}{\sqrt{n_-}} \| \bP_{\vmu}^\perp \|_\mathrm{F} \bigl(1 + o_\P(1)\bigr)
    = \mathmakebox[\widthof{$\displaystyle \sqrt{\frac{d}{\pi n}}\big(1 + o_\P(1) \big)$}][r]{\sqrt{\frac{d}{n}}\big(1 + o_\P(1) \big)} = o_\P(1).
    \label{eq:Pzpm_asymp}
\end{align}
While, by maximal inequality \cref{lem:subG}\ref{lem:subG-c} or \cref{eq:max_g}, we have
\begin{equation}
    \label{eq:g_asymp}
    \max_{i \in [n]} \abs{g_i} = O_\P(\log n) = \wt O_\P(1),
\end{equation}
Plugging \cref{eq:Pzpm_asymp} and \eqref{eq:g_asymp} into \cref{eq:zSV_upper} gives the asymptotic upper bounds (involving $\hat\rho$):
\begin{equation}
    \label{eq:zSV_upper_asymp}
    \begin{aligned}
        \phantom{-} \sqrt{1 - \hat\rho^2} \< \zz_{\mathsf{sv}_+(\hat\vtheta)}, \hat\vtheta \>  
        & \le  \sqrt{1 - \hat\rho^2} \sqrt{\frac{d}{\pi n}}\big(1 + o_\P(1) \big) + 
        \hat\rho \cdot \wt O_\P(1) , \\
        - \sqrt{1 - \hat\rho^2} \< \zz_{\mathsf{sv}_-(\hat\vtheta)}, \hat\vtheta \> 
        & \le  \sqrt{1 - \hat\rho^2} \cdot o_\P(1) + 
        \hat\rho \cdot \wt O_\P(1) .
    \end{aligned}
\end{equation}

\paragraph{Lower bounds} Similar as the proof of \cref{lem:upper_bound}, a lower bound can be obtained by plugging our constructed solution $\wt\vtheta = \bP_{\vmu}^\perp \wt\zz/\| \bP_{\vmu}^\perp \wt\zz \|_2$, which can be a good ``proxy'' of $\vtheta$. Again, by margin-balancing condition \cref{eq:svm_sv_bal}, we can express the optimal $\vtheta$ as\footnote{
    Notice that if $|\hat\rho| < 1$, then $\argmax_{ \vtheta \in \S^{d-1}, \vtheta \perp \vmu } \kappa(\hat\rho, \vtheta, \hat\beta_0)$ is unique (on the event of $\{ \hat\kappa > 0\}$), and we could write $\hat\vtheta = \argmax_{ \vtheta \in \S^{d-1}, \vtheta \perp \vmu } \kappa(\hat\rho, \vtheta, \hat\beta_0)$. However, if $|\hat\rho| = 1$, then according to our construction \cref{eq:def-rho-theta}, the arguments of the maxima can be any $\vtheta \in \S^{d-1}$ such that $\vtheta \perp \vmu$, while $\hat\vtheta = \vmu_\perp$ as defined in \cref{eq:def-rho-theta_hat}.
}
\begin{equation*}
    \begin{aligned}
        \hat\vtheta & \in \argmax_{ \vtheta \in \S^{d-1}, \vtheta \perp \vmu } \kappa(\hat\rho, \vtheta, \hat\beta_0)
    = \argmax_{ \vtheta \in \S^{d-1}, \vtheta \perp \vmu } \frac{\kappa_{\mathsf{sv}_+(\vtheta)}(\hat\rho, \vtheta, \hat\beta_0) + \kappa_{\mathsf{sv}_-(\vtheta)}(\hat\rho, \vtheta, \hat\beta_0)}{2} \\
    & = \argmax_{ \vtheta \in \S^{d-1}, \vtheta \perp \vmu } 
    \left\{ \hat\rho \norm{\vmu}_2 + \hat\rho \frac{g_{\mathsf{sv}_+(\vtheta)} - g_{\mathsf{sv}_-(\vtheta)} }{2} + \sqrt{1 - \hat\rho^2}  \frac{ \< \zz_{\mathsf{sv}_+(\vtheta)}, \vtheta \>  -  \< \zz_{\mathsf{sv}_-(\vtheta)}, \vtheta \>  }{2} \right\}   \\
    & = \argmax_{ \vtheta \in \S^{d-1}, \vtheta \perp \vmu } \left\{  \hat\rho \left(  g_{\mathsf{sv}_+(\vtheta)} - g_{\mathsf{sv}_-(\vtheta)} \right) +  \sqrt{1 - \hat\rho^2} \left( \< \zz_{\mathsf{sv}_+(\vtheta)}, \vtheta \>  +  \< \zz_{\mathsf{sv}_-(\vtheta)}, \vtheta \> \right)  \right\}.
    \end{aligned}
\end{equation*}
Therefore, recalling \cref{eq:V_def_theta}, we have
\begin{align*}
        & \sqrt{1 - \hat\rho^2}  \bigl( \< \zz_{\mathsf{sv}_+(\hat\vtheta)}, \hat\vtheta \>  -  \< \zz_{\mathsf{sv}_-(\hat\vtheta)}, \hat\vtheta \>  \bigr) \\
        \ge {} & \sqrt{1 - \hat\rho^2}  \bigl( \< \zz_{\mathsf{sv}_+(\wt\vtheta)}, \wt\vtheta \>  -  \< \zz_{\mathsf{sv}_-(\wt\vtheta)}, \wt\vtheta \>  \bigr) + \hat\rho\bigl( g_{\mathsf{sv}_+(\wt\vtheta)} - g_{\mathsf{sv}_-(\wt\vtheta)} -  g_{\mathsf{sv}_+(\hat\vtheta)} + g_{\mathsf{sv}_-(\hat\vtheta)} \bigr) \\
        \ge {} & \sqrt{1 - \hat\rho^2}   \bigl( \< \zz_{\mathsf{v}_+(\wt\vtheta)}, \wt\vtheta \>  -  \< \zz_{\mathsf{v}_-(\wt\vtheta)}, \wt\vtheta \>  \bigr) 
        - 4 \hat\rho \max_{i \in [n]} \abs{g_i}.
\end{align*}
Combining it with \cref{eq:zSV_upper}, we can obtain a lower bound for each term using $\wt\vtheta$:
\begin{equation}\label{eq:zSV_lower}
    \begin{aligned}
        \phantom{-} \sqrt{1 - \hat\rho^2} \< \zz_{\mathsf{sv}_+(\hat\vtheta)}, \hat\vtheta \>
        & \ge \sqrt{1 - \hat\rho^2}   \bigl( \< \zz_{\mathsf{v}_+(\wt\vtheta)}, \wt\vtheta \>  -  \< \zz_{\mathsf{v}_-(\wt\vtheta)}, \wt\vtheta \>  \bigr) 
        - 4 \hat\rho \max_{i \in [n]} \abs{g_i}
        + \sqrt{1 - \hat\rho^2} \< \zz_{\mathsf{sv}_-(\hat\vtheta)}, \hat\vtheta \> \\
        & \ge \sqrt{1 - \hat\rho^2} \bigl( 
        - \< \zz_{\mathsf{v}_-(\wt\vtheta)}, \wt\vtheta \> - \| \bP_{\vmu}^\perp \bar\zz_- \|_2 \bigr)
        - 6 \hat\rho \max_{i \in [n]} \abs{g_i}
        + \sqrt{1 - \hat\rho^2} \< \zz_{\mathsf{v}_+(\wt\vtheta)}, \wt\vtheta \>, 
        \\
        - \sqrt{1 - \hat\rho^2} \< \zz_{\mathsf{sv}_-(\hat\vtheta)}, \hat\vtheta \>
        & \ge \sqrt{1 - \hat\rho^2}   \bigl( \< \zz_{\mathsf{v}_+(\wt\vtheta)}, \wt\vtheta \>  -  \< \zz_{\mathsf{v}_-(\wt\vtheta)}, \wt\vtheta \>  \bigr) 
        - 4 \hat\rho \max_{i \in [n]} \abs{g_i}
        - \sqrt{1 - \hat\rho^2} \< \zz_{\mathsf{sv}_+(\hat\vtheta)}, \hat\vtheta \> \\
        & \ge \sqrt{1 - \hat\rho^2} \bigl( + \< \zz_{\mathsf{v}_+(\wt\vtheta)}, \wt\vtheta \>
        - \| \bP_{\vmu}^\perp \bar\zz_+ \|_2  \bigr)
        - 6 \hat\rho \max_{i \in [n]} \abs{g_i}
        - \sqrt{1 - \hat\rho^2} \< \zz_{\mathsf{v}_-(\wt\vtheta)}, \wt\vtheta \>.
    \end{aligned}
\end{equation}
To derive its asymptotic order, we first define two statistics that are closely related to $\wt\vtheta$: 
\begin{equation}\label{eq:wt_theta_pm}
    \wt\vtheta_+ := \frac{\bP_{\vmu}^\perp \bar\zz_+}{\| \bP_{\vmu}^\perp \bar\zz_+ \|_2},
    \qquad
    \wt\vtheta_- := \frac{-\bP_{\vmu}^\perp \bar\zz_-}{\| \bP_{\vmu}^\perp \bar\zz_- \|_2}.
\end{equation}
Then, the difference terms inside the parentheses in \cref{eq:zSV_lower} can be expressed as
\begin{equation}\label{eq:zV_diff}
    \begin{aligned}
        \phantom{+} \< \zz_{\mathsf{v}_+(\wt\vtheta)}, \wt\vtheta \> - \| \bP_{\vmu}^\perp \bar\zz_+ \|_2
        & = \min_{i \in \mathcal{I}_+} \< + \zz_i, \wt\vtheta \> - \< \bar\zz_+, \wt\vtheta_+ \>
        = \min_{i \in \mathcal{I}_+} \< \zz_i - \bar\zz_+, \wt\vtheta \> + \< \bar\zz_+, \wt\vtheta - \wt\vtheta_+ \>,
        \\
        - \< \zz_{\mathsf{v}_-(\wt\vtheta)}, \wt\vtheta \> - \| \bP_{\vmu}^\perp \bar\zz_- \|_2
        & = \min_{i \in \mathcal{I}_-} \< - \zz_i, \wt\vtheta \> + \< \bar\zz_-, \wt\vtheta_- \>
        = \min_{i \in \mathcal{I}_-} \< \bar\zz_- - \zz_i, \wt\vtheta \> - \< \bar\zz_-, \wt\vtheta - \wt\vtheta_- \>.
    \end{aligned}
\end{equation}
Now we study the two terms on the R.H.S. of \cref{eq:zV_diff}. For the first term, based on \cref{eq:max_z},
\begin{equation}\label{eq:zV_diff_1}
    \begin{aligned}
        \min_{i \in \mathcal{I}_+} \< \zz_i - \bar\zz_+, \wt\vtheta \>
        \ge -\max_{i \in \mathcal{I}_+} \bigl| \< \zz_i - \bar\zz_+, \wt\vtheta \> \bigr|
        & = \wt O_\P(1), \\
        \min_{i \in \mathcal{I}_-} \< \bar\zz_- - \zz_i, \wt\vtheta \>
        \ge -\max_{i \in \mathcal{I}_-} \bigl| \< \zz_i - \bar\zz_-, \wt\vtheta \> \bigr|
        & = \wt O_\P(1).
    \end{aligned}
\end{equation}
For the second term,
\begin{equation}\label{eq:zV_diff_2+}
    \begin{aligned}
        \< \bar\zz_+, \wt\vtheta - \wt\vtheta_+ \>
        & =
        \frac{1}{\| \bP_{\vmu}^\perp \wt\zz \|_2} \< \bar\zz_+, \bP_{\vmu}^\perp \wt\zz \> 
        - \frac{1}{\| \bP_{\vmu}^\perp \bar\zz_+ \|_2} \< \bar\zz_+, \bP_{\vmu}^\perp \bar\zz_+ \> \\
        & = 
        \frac{1}{\| \bP_{\vmu}^\perp \wt\zz \|_2} 
        \left\{  \< \bar\zz_+, \bP_{\vmu}^\perp \wt\zz \>  -  \< \bar\zz_+, \bP_{\vmu}^\perp \bar\zz_+ \> \cdot \frac{\| \bP_{\vmu}^\perp \wt\zz \|_2}{\| \bP_{\vmu}^\perp \bar\zz_+ \|_2} \right\}
        \\
        & \overset{\mathmakebox[0pt][c]{\text{(i)}}}{\ge}
        \frac{1}{\| \bP_{\vmu}^\perp \wt\zz \|_2} 
        \left\{  \< \bar\zz_+, \bP_{\vmu}^\perp \wt\zz \> 
        - \< \bar\zz_+, \bP_{\vmu}^\perp \bar\zz_+ \> \cdot 
        \frac12 \biggl( 1 +  \frac{\| \bP_{\vmu}^\perp \bar\zz_- \|_2}{\| \bP_{\vmu}^\perp \bar\zz_+ \|_2} \biggr) 
        \right\} \\
        & \overset{\mathmakebox[0pt][c]{\text{(ii)}}}{=} - \frac{1}{2\| \bP_{\vmu}^\perp \wt\zz \|_2} \left( 
            \< \bar\zz_+, \bP_{\vmu}^\perp \bar\zz_-  \> + \| \bP_{\vmu}^\perp \bar\zz_+ \|_2 \| \bP_{\vmu}^\perp \bar\zz_- \|_2
         \right) \\
        & \overset{\mathmakebox[0pt][c]{\text{(iii)}}}{=}
         - \sqrt{\frac{\pi n}{d}}\big(1 + o_\P(1) \big) \left\{ 
            O_\P\biggl(  \sqrt{\frac{d}{\pi n^2}} \biggr) + \sqrt{\frac{d}{\pi n}}\sqrt{\frac{d}{n}} \big(1 + o_\P(1) \big)
         \right\} \\
        & = - \sqrt{\frac{d}{n}}\big(1 + o_\P(1) \big)
        = o_\P(1),
    \end{aligned}
\end{equation}
where (i) is from triangular inequality $2\| \bP_{\vmu}^\perp \wt\zz \|_2 \le \| \bP_{\vmu}^\perp \bar\zz_+ \|_2 + \| \bP_{\vmu}^\perp \bar\zz_- \|_2$, (ii) uses $2 \wt\zz - \bar z_+ = -\bar z_-$, and (iii) applies the asymptotic results \cref{eq:P_wtz}, \eqref{eq:Pzpm_asymp}, and the fact that $\bar\zz_+ \indep \bar\zz_-$, 
\begin{equation*}
    \< \bar\zz_{+}, \bP_{\vmu}^\perp \bar\zz_{-} \> = \frac{1}{\sqrt{n_+ n_-}} O_{\P}(\| \bP_{\vmu}^\perp \|_\mathrm{F}) = O_\P\biggl( \sqrt{\frac{d}{\pi n^2}}  \biggr),
\end{equation*}
by \cref{lem:subG_concentrate}\ref{lem:subG-Bernstein}. Similarly, we also have
\begin{equation}\label{eq:zV_diff_2-}
    \begin{aligned}
        - \< \bar\zz_-, \wt\vtheta - \wt\vtheta_- \>
        & =
        - \frac{1}{\| \bP_{\vmu}^\perp \wt\zz \|_2} \< \bar\zz_-, \bP_{\vmu}^\perp \wt\zz \> 
        - \frac{1}{\| \bP_{\vmu}^\perp \bar\zz_- \|_2} \< \bar\zz_-, \bP_{\vmu}^\perp \bar\zz_- \> \\
        & \ge - \frac{1}{2\| \bP_{\vmu}^\perp \wt\zz \|_2} \left( 
            \< \bar\zz_+, \bP_{\vmu}^\perp \bar\zz_-  \> + \| \bP_{\vmu}^\perp \bar\zz_+ \|_2 \| \bP_{\vmu}^\perp \bar\zz_- \|_2
         \right) \\
        & = o_\P(1).
    \end{aligned}
\end{equation}
Substituting \cref{eq:zV_diff_1}, \eqref{eq:zV_diff_2+}, and \eqref{eq:zV_diff_2-} into \cref{eq:zV_diff}, we get
\begin{equation}
    \label{eq:zv_diff_asymp}
    \< \zz_{\mathsf{v}_+(\wt\vtheta)}, \wt\vtheta \> - \| \bP_{\vmu}^\perp \bar\zz_+ \|_2
    \ge \wt O_\P(1),
    \qquad
    - \< \zz_{\mathsf{v}_-(\wt\vtheta)}, \wt\vtheta \> - \| \bP_{\vmu}^\perp \bar\zz_- \|_2
    \ge \wt O_\P(1).
\end{equation}
And combining this with \cref{eq:Pzpm_asymp}, we have
\begin{equation}
    \label{eq:zv_wt_asymp}
    \< \zz_{\mathsf{v}_+(\wt\vtheta)}, \wt\vtheta \> 
    \ge \sqrt{\frac{d}{\pi n}}\big(1 + o_\P(1) \big) + \wt O_\P(1),
    \qquad
    - \< \zz_{\mathsf{v}_-(\wt\vtheta)}, \wt\vtheta \> 
    \ge \wt O_\P(1).
\end{equation}
Plugging \cref{eq:zv_diff_asymp}, \eqref{eq:zv_wt_asymp}, and \eqref{eq:g_asymp} into \cref{eq:zSV_lower} gives the asymptotic lower bounds (involving $\hat\rho$):
\begin{equation}
    \label{eq:zSV_lower_asymp}
    \begin{aligned}
        \phantom{-} \sqrt{1 - \hat\rho^2} \< \zz_{\mathsf{sv}_+(\hat\vtheta)}, \hat\vtheta \>  
        & \ge  \sqrt{1 - \hat\rho^2} \biggl( \sqrt{\frac{d}{\pi n}}\big(1 + o_\P(1) \big) 
        + \wt O_\P(1)
        \biggr)
        + 
        \hat\rho \cdot \wt O_\P(1) , \\
        - \sqrt{1 - \hat\rho^2} \< \zz_{\mathsf{sv}_-(\hat\vtheta)}, \hat\vtheta \> 
        & \ge  \sqrt{1 - \hat\rho^2} \cdot \wt O_\P(1) + 
        \hat\rho \cdot \wt O_\P(1) .
    \end{aligned}
\end{equation}

~\\
\noindent
Finally, combining upper bounds \cref{eq:zSV_upper_asymp} and lower bounds \cref{eq:zSV_lower_asymp}, we obtain the exact order
\begin{equation*}
    \begin{aligned}
        \phantom{-} \sqrt{1 - \hat\rho^2} \< \zz_{\mathsf{sv}_+(\hat\vtheta)}, \hat\vtheta \>  
        & =  \sqrt{1 - \hat\rho^2} \biggl( \sqrt{\frac{d}{\pi n}}\big(1 + o_\P(1) \big) 
        + \wt O_\P(1)
        \biggr)
        + 
        \hat\rho \cdot \wt O_\P(1) \\
        & = \sqrt{1 - \hat\rho^2} \sqrt{\frac{d}{\pi n}}\big(1 + o_\P(1) \big) 
        + \wt O_\P(1),
        \\
        - \sqrt{1 - \hat\rho^2} \< \zz_{\mathsf{sv}_-(\hat\vtheta)}, \hat\vtheta \> 
        & =  \sqrt{1 - \hat\rho^2} \cdot \wt O_\P(1) + 
        \hat\rho \cdot \wt O_\P(1) \\
        & = \wt O_\P(1).
    \end{aligned}
\end{equation*}

\paragraph{\ref{lem:theta_hat_z(a)}:}
If $a < b + c$, according to \cref{lem:rho_hat}\ref{lem:rho_hat(a)}, $\hat\rho = 1 - o_\P(1)$. It is clear that \cref{lem:theta_hat_z} holds for $\hat\rho = \pm 1$. Now, restrict on the event $\{ |\hat\rho| < 1 \}$.
\begin{itemize}
    \item If $a > \frac{b}{2} + c$, then $\sqrt{1 - \hat\rho^2} = \frac12 d^{(a-b-c)/2}\bigl( 1 + o_\P(1) \bigr)$, hence
    \begin{equation*}
        \sqrt{1 - \hat\rho^2} \< \zz_{\mathsf{sv}_+(\hat\vtheta)}, \hat\vtheta \>
        = \frac12 d^{a - \frac{b}{2} - c}\big(1 + o_\P(1) \big).
    \end{equation*}
    \item If $a \le \frac{b}{2} + c$, then $\sqrt{1 - \hat\rho^2}\sqrt{d/\pi n} = \wt O_\P(1)$, hence
    \begin{equation*}
        \sqrt{1 - \hat\rho^2} \< \zz_{\mathsf{sv}_+(\hat\vtheta)}, \hat\vtheta \>
        = \wt O_\P(1).
    \end{equation*}
\end{itemize}
\paragraph{\ref{lem:theta_hat_z(b)}:}
If $a > b + c$, according to \cref{lem:rho_hat}\ref{lem:rho_hat(b)}, $\hat\rho = o_\P(1)$. Hence, on the event $\{ |\hat\rho| < 1 \}$,
\begin{equation*}
    \< \zz_{\mathsf{sv}_+(\hat\vtheta)}, \hat\vtheta \>
    = \sqrt{\frac{d}{\pi n}}\big(1 + o_\P(1) \big).
\end{equation*}
This also holds regardless of $\hat\rho$, since $\P( |\hat\rho| < 1 ) \to 1$ as $d \to \infty$. Then we complete the proof.
\end{proof}

\subsubsection{Asymptotic expression of $\hat\beta_0$: Proof of \cref{lem:beta0_asymp}}

Finally, we consider arbitrary $\tau \ge 1$ and give an explicit expression for $\hat\beta_0$ with its asymptotics. Be aware that $\tau = \tau_d$ may depend on $d$.
\begin{lem}[Asymptotics of $\hat\beta_0$] \label{lem:beta0_asymp}
    Suppose that $a < c + 1$ and $\tau \ge 1$. Then we have
    \begin{equation*}
        \begin{aligned}
            \hat\beta_0 & = \left(1 - \frac{2}{\tau + 1}\right) \hat\rho\norm{\vmu}_2 
            - \hat\rho \frac{\tau g_{\mathsf{sv}_-(\hat\vtheta)} + g_{\mathsf{sv}_+(\hat\vtheta)} }{\tau + 1}
            - \sqrt{1 - \hat\rho^2} \frac{ \tau \< \zz_{\mathsf{sv}_-(\hat\vtheta)}, \hat\vtheta \> + \< \zz_{\mathsf{sv}_+(\hat\vtheta)}, \hat\vtheta \> }{\tau + 1} \\
            & = \left(1 - \frac{2}{\tau + 1}\right) \hat\rho\norm{\vmu}_2 
            - \frac{1}{\tau + 1} \sqrt{1 - \hat\rho^2}\< \zz_{\mathsf{sv}_+(\hat\vtheta)}, \hat\vtheta \> + \wt O_\P(1).
        \end{aligned}
    \end{equation*}
    \begin{enumerate}[label=(\alph*)]
        \item \label{lem:beta0_asymp(a)}
        If $a < b + c$, then
        \begin{equation*}
            \begin{aligned}
                \hat\beta_0 & = \left(1 - \frac{2}{\tau + 1}\right) \hat\rho\norm{\vmu}_2 
            - \frac{1}{\tau + 1} \wt O_\P(d^{a-\frac{b}{2}-c} \vee 1) 
            + \wt O_\P(1) \\
            & = \left(1 - \frac{2}{\tau + 1}\right) d^{b/2}\bigl(1 + o_{\P}(1)\bigr) 
            - \frac{1}{\tau + 1} \wt O_\P(d^{a-\frac{b}{2}-c} \vee 1)
            + \wt O_\P(1).
            \end{aligned}
    \end{equation*}
        \item \label{lem:beta0_asymp(b)}
        If $a > b + c$, then
        \begin{equation*}
            \begin{aligned}
            & \hat\beta_0  = \left(1 - \frac{2}{\tau + 1}\right) \hat\rho\norm{\vmu}_2 
            - \frac{1}{\tau + 1}\sqrt{\frac{d}{\pi n}}\bigl(1 + o_{\P}(1)\bigr) 
            + \wt O_\P(1) \\
            = {} & 
            \begin{cases} 
                \,  \displaystyle \left(1 - \frac{2}{\tau + 1}\right) 2d^{(2b-a+c)/2} \bigl(1 + o_{\P}(1)\bigr) 
                - \frac{1}{\tau + 1} d^{(a-c)/2} \bigl(1 + o_{\P}(1)\bigr) 
                + \wt O_\P(1) , & \ \text{if} \ a < 2b + c, \\
                \,  \displaystyle \phantom{\left(1 - \frac{2}{\tau + 1}\right) 2d^{(2b-a+c)/2} \bigl(1 + o_{\P}(1)\bigr) 
                }
                - \frac{1}{\tau + 1} d^{(a-c)/2} \bigl(1 + o_{\P}(1)\bigr) 
                + \wt O_\P(1) , & \ \text{if} \ a > 2b + c. \end{cases}
        \end{aligned}
    \end{equation*}
    \end{enumerate}
\end{lem}
\begin{proof}
    We rewrite the \emph{margin-balancing} condition \cref{eq:margin_pm}, \eqref{eq:margin-bal} in terms of $\hat\rho, \hat\vtheta, \hat\beta_0$, which generalizes \cref{eq:svm_sv_bal} to arbitrary $\tau \ge 1$:
    \begin{equation*}
        \begin{aligned}
           \kappa(\hat\rho, \hat\vtheta, \hat\beta_0) & = \kappa_{\mathsf{sv}_+(\hat\vtheta)}(\hat\rho, \hat\vtheta, \hat\beta_0) = 
           \tau^{-1} \Bigl(
           \hat\rho \norm{\vmu}_2 + \hat\beta_0 + \hat\rho g_{\mathsf{sv}_+(\hat\vtheta)} + \sqrt{1 - \hat\rho^2} \< \zz_{\mathsf{sv}_+(\hat\vtheta)}, \hat\vtheta \> 
           \Bigr)
           \\
           & = \kappa_{\mathsf{sv}_-(\hat\vtheta)}(\hat\rho, \hat\vtheta, \hat\beta_0) = 
           \phantom{ \tau^{-1} \Bigl( }
           \hat\rho \norm{\vmu}_2 - \hat\beta_0 - \hat\rho g_{\mathsf{sv}_-(\hat\vtheta)} - \sqrt{1 - \hat\rho^2} \< \zz_{\mathsf{sv}_-(\hat\vtheta)}, \hat\vtheta \>
           .
        \end{aligned}
    \end{equation*}
    Then we can solve the expression for $\hat\beta_0$ (this equals \cref{eq:beta0_hat} in \cref{lem:indep_tau} with parametrization \cref{eq:def-rho-theta_hat}). Its asymptotic simplification is followed by \cref{eq:g_asymp}:
    \begin{equation*}
        \abs{\hat\rho \frac{\tau g_{\mathsf{sv}_-(\hat\vtheta)} + g_{\mathsf{sv}_+(\hat\vtheta)} }{\tau + 1} } 
        \le \abs{\hat\rho} \frac{\tau |g_{\mathsf{sv}_-(\hat\vtheta)}| + |g_{\mathsf{sv}_+(\hat\vtheta)}| }{\tau + 1}
        \le \max_{i \in [n]} \abs{g_i} = \wt O_\P(1),
    \end{equation*}
    and \cref{lem:theta_hat_z}:
    \begin{equation*}
        \abs{
        \frac{\tau}{\tau + 1} \sqrt{1 - \hat\rho^2}\< \zz_{\mathsf{sv}_-(\hat\vtheta)}, \hat\vtheta \>
        } = \wt O_\P(1).
    \end{equation*}
    For \textbf{\ref{lem:beta0_asymp(a)}}, plugging $\hat\rho = 1 - o_\P(1)$ by \cref{lem:rho_hat}\ref{lem:rho_hat(a)} and asymptotics of $\< \zz_{\mathsf{sv}_+(\hat\vtheta)}, \hat\vtheta \>$ by \cref{lem:theta_hat_z}\ref{lem:theta_hat_z(a)}. For \textbf{\ref{lem:beta0_asymp(b)}}, plugging $\hat\rho = 2d^{(b-a+c)/2}\bigl(1 + o_\P(1)\bigr)$ by \cref{lem:rho_hat}\ref{lem:rho_hat(b)} from i., while $\hat\rho\norm{\bmu}_2 = o_\P(1)$ from ii., and asymptotics of $\< \zz_{\mathsf{sv}_+(\hat\vtheta)}, \hat\vtheta \>$ by \cref{lem:theta_hat_z}\ref{lem:theta_hat_z(b)}. This completes the proof.
\end{proof}

\subsection{Classification error: Completing the proof of \cref{thm:main_high-imbal}}
\label{subsec:highimb_err}

\begin{proof}[\textbf{Proof of \cref{thm:main_high-imbal}}]
Let $(\xx_\mathrm{new}, y_\mathrm{new})$ be a test data point independent of the training set $\{(\xx_i, y_i)\}_{i=1}^n$, such that $\xx_\mathrm{new} = y_\mathrm{new} \bmu + \zz_\mathrm{new}$, and $\zz_\mathrm{new} \sim \subGind(\bzero, \bI_d; K)$. Recall $\hat f(\xx) = \< \xx, \hat\vbeta \> + \hat\beta_0$. Following the same decomposition as \cref{eq:logits},
\begin{equation*}
    \begin{aligned}
        y_\mathrm{new} \hat f(\xx_\mathrm{new}) & = y_\mathrm{new} (\< \xx_\mathrm{new}, \hat\vbeta \> + \hat\beta_0) \\
        & = \hat\rho \norm{\vmu}_2 + y_\mathrm{new} \hat\beta_0 + 
        y_\mathrm{new} \bigl(  \hat\rho g_\mathrm{new} + \sqrt{1 - \hat\rho^2} \< \zz_\mathrm{new}, \hat\vtheta \> \bigr) \\
        & = \hat\rho \norm{\vmu}_2 + y_\mathrm{new} \hat\beta_0 + y_\mathrm{new} G_d,
    \end{aligned}
\end{equation*}
where 
\begin{equation*}
g_\mathrm{new} := \left\< \zz_\mathrm{new}, \frac{\vmu}{\norm{\vmu}_2} \right\>,
\qquad 
G_d := \hat\rho g_\mathrm{new} + \sqrt{1 - \hat\rho^2} \< \zz_\mathrm{new}, \hat\vtheta \>.
\end{equation*}
Therefore, the minority and majority test errors are
\begin{equation*}
    \begin{aligned}
        \Err_+ & = \P\left( \hat f(\xx_\mathrm{new}) \le 0 \,\big|\, y_\mathrm{new} = +1 \right)
        = \P\left( \hat\rho \norm{\vmu}_2 + \hat\beta_0 + G_d \le 0 \right), \\
        \Err_- & = \P\left( \hat f(\xx_\mathrm{new}) > 0 \,\big|\, y_\mathrm{new} = -1 \right)
        = \P\left( \hat\rho \norm{\vmu}_2 - \hat\beta_0 - G_d < 0 \right). \\
    \end{aligned}
\end{equation*}
By \cref{lem:subG_concentrate}\ref{lem:subG-Hoeffding}, we have $\norm{g_\mathrm{new}}_{\psi_2}, \| \< \zz_\mathrm{new}, \hat\vtheta \> \|_{\psi_2} \lesssim K$, since $\zz_\mathrm{new} \indep (\hat\rho, \hat\vtheta)$ and then $\forall\, t > 0$,
\begin{equation*}
     \P\left( \bigl| \< \zz_\mathrm{new}, \hat\vtheta \> \bigr| > t \right)
     =  \E\left[ \P\left( \bigl| \< \zz_\mathrm{new}, \hat\vtheta \> \bigr| > t \,\big|\, \hat\vtheta \right) \right]
     \le 2 e^{-ct^2/K^2}, \qquad \text{for some} ~ c > 0.
\end{equation*}
Then by \cref{lem:subG}\ref{lem:subG-a},
\begin{equation*}
     \norm{G_d}_{\psi_2} \le 
     \norm{\hat\rho g_\mathrm{new}}_{\psi_2} + \| \sqrt{1 - \hat\rho^2} \< \zz_\mathrm{new}, \hat\vtheta \>\|_{\psi_2}
     \le \norm{g_\mathrm{new}}_{\psi_2} + \| \< \zz_\mathrm{new}, \hat\vtheta \> \|_{\psi_2} \lesssim K,
\end{equation*}
which implies $G_d = O_\P(1)$.

\vspace{0.5\baselineskip}
\noindent
\textbf{\ref{thm:high-imb_high}. High signal:}
If $a < b + c$, then we have $\hat\rho = 1 - o_\P(1)$ by \cref{lem:rho_hat}\ref{lem:rho_hat(a)}. Therefore, according to \cref{lem:beta0_asymp}\ref{lem:beta0_asymp(a)}, for all $\tau_d \ge 1$, we have
\begin{align*}
        \hat\rho \norm{\vmu}_2 + \hat\beta_0
        & = \left(2 - \frac{2}{\tau_d + 1}\right) \hat\rho\norm{\vmu}_2 
        - \frac{1}{\tau_d + 1} \wt O_\P(d^{a-\frac{b}{2}-c} \vee 1) 
        + \wt O_\P(1) \\
        & \ge  d^{b/2}\bigl(1 + o_{\P}(1)\bigr)
        - \wt O_\P(d^{a-\frac{b}{2}-c} \vee 1)  \\
        & \overset{\mathmakebox[0pt][c]{\text{(i)}}}{=} d^{b/2}\bigl(1 + o_{\P}(1)\bigr),
        \qquad  \lim_{d \to \infty} d^{b/2} = +\infty,
\end{align*}
where (i) is because $d^{b/2} \gg d^{a-\frac{b}{2}-c}$, as $d \to \infty$. If $1 \le \tau_d \ll d^{b/2}$, we also have
\begin{align*}
    \hat\rho \norm{\vmu}_2 - \hat\beta_0
    & = \frac{2}{\tau_d + 1} \hat\rho\norm{\vmu}_2 
    + \frac{1}{\tau_d + 1} \wt O_\P(d^{a-\frac{b}{2}-c} \vee 1) + \wt O_\P(1) \\
    & = \frac{2}{\tau_d + 1} d^{b/2} 
    + \frac{1}{\tau_d + 1} \wt O_\P(d^{a-\frac{b}{2}-c} \vee 1) + \wt O_\P(1) \\
    & \overset{\mathmakebox[0pt][c]{\text{(ii)}}}{=} \frac{2}{\tau_d + 1} d^{b/2} \bigl(1 + o_{\P}(1)\bigr) + \wt O_\P(1) \\
    & \ge \tau_d^{-1} d^{b/2} \bigl(1 + o_{\P}(1)\bigr) + \wt O_\P(1),
    \qquad
    \lim_{d \to \infty} \tau_d^{-1} d^{b/2} = +\infty,
\end{align*}
where (ii) is because $(\tau_d + 1)^{-1}d^{b/2} \gg (\tau_d + 1)^{-1}d^{a-\frac{b}{2}-c}$ and $(\tau_d + 1)^{-1}d^{b/2} \gg (\log d)^k$, $\forall\, k \ge 0$, as $d \to \infty$. Under these conditions, both $\hat\rho \norm{\vmu}_2 \pm \hat\beta_0$ diverges to $+\infty$ with high probability, i.e.,
\begin{equation*}
    \lim_{d \to \infty} \P\left(\hat\rho \norm{\vmu}_2 + \hat\beta_0 + G_d > C\right) 
    =
    \lim_{d \to \infty} \P\left(\hat\rho \norm{\vmu}_2 - \hat\beta_0 - G_d > C\right) 
    = 1,
    \qquad
    \forall\, C \in \R.
\end{equation*}
Hence
\begin{equation*}
    \Err_+ = o(1), \qquad \Err_- = o(1).
\end{equation*}
This concludes the proof for high signal regime.

\vspace{0.5\baselineskip}
\noindent
\textbf{\ref{thm:high-imb_moderate}. Moderate signal:}
If $b + c < a < 2b + c$, then $\hat\rho = 2d^{(b-a+c)/2}\bigl(1 + o_\P(1)\bigr)$ by \cref{lem:rho_hat}\ref{lem:rho_hat(b)}. Therefore, according to \cref{lem:beta0_asymp}\ref{lem:beta0_asymp(b)}, if $\tau_d \gg d^{a-b-c}$, then
\begin{align*}
        \hat\rho \norm{\vmu}_2 + \hat\beta_0 
        & = \left(2 - \frac{2}{\tau_d + 1}\right) \hat\rho\norm{\vmu}_2 
        - \frac{1}{\tau_d + 1}\sqrt{\frac{d}{\pi n}}\bigl(1 + o_{\P}(1)\bigr) 
        + \wt O_\P(1)  \\
        & = 4 d^{(2b-a+c)/2}  \bigl(1 + o_{\P}(1)\bigr) - \tau_d^{-1} d^{(a-c)/2} \bigl(1 + o_{\P}(1)\bigr) + \wt O_\P(1) \\
        & \overset{\mathmakebox[0pt][c]{\text{(iii)}}}{=} 4 d^{(2b-a+c)/2}  \bigl(1 + o_{\P}(1)\bigr),
        \qquad
    \lim_{d \to \infty} d^{(2b-a+c)/2} = +\infty,
\end{align*}
where (iii) is because $d^{(2b-a+c)/2} \gg \tau_d^{-1} d^{(a-c)/2}$ and $d^{(2b-a+c)/2} \gg (\log d)^k$, $\forall\, k \ge 0$, as $d \to \infty$. If $1 \le \tau_d \ll d^{(a-c)/2}$, we also have
\begin{align*}
        \hat\rho \norm{\vmu}_2 - \hat\beta_0 
        & = \frac{2}{\tau_d + 1} \hat\rho\norm{\vmu}_2 
        + \frac{1}{\tau_d + 1}\sqrt{\frac{d}{\pi n}}\bigl(1 + o_{\P}(1)\bigr) 
        + \wt O_\P(1)  \\
        & = \frac{4}{\tau_d + 1} d^{(2b-a+c)/2} \bigl(1 + o_{\P}(1)\bigr) + \frac{1}{\tau_d + 1} d^{(a-c)/2} \bigl(1 + o_{\P}(1)\bigr) + \wt O_\P(1) \\
        & \overset{\mathmakebox[0pt][c]{\text{(iv)}}}{=} \frac{1}{\tau_d + 1} d^{(a-c)/2} \bigl(1 + o_{\P}(1)\bigr) + \wt O_\P(1) \\
        & \ge \frac12 \tau_d^{-1} d^{(a-c)/2} \bigl(1 + o_{\P}(1)\bigr) + \wt O_\P(1),
        \qquad 
        \lim_{d \to \infty} \tau_d^{-1} d^{(a-c)/2} = +\infty,
\end{align*}
where (iv) is from $(\tau_d + 1)^{-1} d^{(2b-a+c)/2} \ll (\tau_d + 1)^{-1} d^{(a-c)/2}$. Under these conditions on $\tau_d$, both $\hat\rho \norm{\vmu}_2 \pm \hat\beta_0 $ diverges to $+\infty$ with high probability. Using the same approach, we can show that
\begin{equation*}
    \Err_+ = o(1), \qquad \Err_- = o(1).
\end{equation*}

\vspace{0.5\baselineskip}
\noindent
Now suppose $\tau_d \asymp 1$, then again $\hat\rho \norm{\vmu}_2 - \hat\beta_0 \to + \infty$ and hence $\Err_- = o_\P(1)$ still holds. However,
\begin{align*}
        \hat\rho \norm{\vmu}_2 + \hat\beta_0 
        & = \left(2 - \frac{2}{\tau_d + 1}\right) \hat\rho\norm{\vmu}_2 
        - \frac{1}{\tau_d + 1}\sqrt{\frac{d}{\pi n}}\bigl(1 + o_{\P}(1)\bigr) 
        + \wt O_\P(1)  \\
        & \le 2 d^{(2b-a+c)/2} \bigl(1 + o_{\P}(1)\bigr) - C d^{(a-c)/2} \bigl(1 + o_{\P}(1)\bigr) + \wt O_\P(1), \\
        & \overset{\mathmakebox[0pt][c]{\text{(v)}}}{=} - C d^{(a-c)/2} \bigl(1 + o_{\P}(1)\bigr),
        \qquad 
        \lim_{d \to \infty} -d^{(a-c)/2} = -\infty,
\end{align*}
where (v) is because $d^{(2b-a+c)/2} \ll d^{(a-c)/2}$, and $C \in (0, \infty)$ is an absolute constant. As the result, $- \hat\rho \norm{\vmu}_2 - \hat\beta_0 $ diverges to $+\infty$ with high probability. Using the same approach, we have
\begin{equation*}
    \Err_+ = 1 - o(1).
\end{equation*}
This concludes the proof for moderate signal regime.

\vspace{0.5\baselineskip}
\noindent
\textbf{\ref{thm:high-imb_low}. Low signal:}
If $a > 2b + c$, then $\hat\rho \norm{\vmu}_2 = o_\P(1) > 0$ by \cref{lem:rho_hat}\ref{lem:rho_hat(b)}. Therefore,
\begin{align*}
        \Err_+ + \Err_- & 
        = \P\left( \hat\rho \norm{\vmu}_2 + \hat\beta_0 + G_d \le 0 \right) 
        + \P\left( \hat\rho \norm{\vmu}_2 - \hat\beta_0 - G_d < 0 \right) \\
        & = 
        1 - \P\left( - \hat\rho \norm{\vmu}_2 \le \hat\beta_0 + G_d < \hat\rho \norm{\vmu}_2 \right)
        \\
        & = 1 - o(1).
\end{align*}
Hence, we have $\Err_\mathrm{b} \ge \frac12 - o(1)$. This concludes the proof for low signal regime.

Finally, we complete the proof of \cref{thm:main_high-imbal}.
\end{proof}

%% file: src/append_calibration.tex
\section{Confidence estimation and calibration: Proofs for \cref{sec:calibration}}
\label{append_sec:calib}

\subsection{Proof of \cref{prop:conf}}

The following preliminary result summarizes the precise asymptotics of three quantities: $\hat p(\xx)$ (max-margin confidence), $p^*(\xx)$ (Bayes optimal probability), and $\hat p_0(\xx)$ (true posterior probability).

\begin{lem}\label{lem:conf_limit}
    Consider 2-GMM and proportional settings in \cref{sec:logit_SVM} on separable dataset ($\delta < \delta^*(0)$). Let $(\rho^*, \beta_0^*)$ be defined as per \cref{thm:SVM_main}, and $(Y, G, H) \sim P_y \times \normal(0,1) \times \normal(0,1)$. Let $G' := \rho^* G + \sqrt{1 - \rho^{*2}} H$. Then for any test point $(\xx, y) \sim P_{\xx, y}$ independent of $\hat p$, as $n \to \infty$,
    \begin{equation}
    \label{eq:p(x)_asymp}
        \begin{pmatrix}
        y 
        \vphantom{\left( \log\frac{\pi}{1 - \pi} \right)} 
        \\
        \hat p(\xx)
        \vphantom{\left( \log\frac{\pi}{1 - \pi} \right)}
        \\
        p^*(\xx) 
        \vphantom{\left( \log\frac{\pi}{1 - \pi} \right)}
        \\
        \hat p_0(\xx)
        \vphantom{\left( \log\frac{\pi}{1 - \pi} \right)}
        \end{pmatrix}
        \cond
        \begin{pmatrix}
            Y 
            \vphantom{\left( \log\frac{\pi}{1 - \pi} \right)} 
            \\
            \sigma\left( \rho^*\|\vmu\|_2 Y + G + \beta_0^* \right) 
            \vphantom{\left( \log\frac{\pi}{1 - \pi} \right)} 
            \\
            \sigma \left( 2 \|\bmu\|_2 (\|\bmu\|_2 Y + G') + \log\frac{\pi}{1 - \pi} \right) \\
            \sigma \left( 2 \rho^* \|\bmu\|_2 (\rho^* \|\bmu\|_2 Y + G) + \log\frac{\pi}{1 - \pi} \right)
        \end{pmatrix}.
    \end{equation}
\end{lem}
\begin{proof}
    Rewrite $\xx = y\bmu + \zz$ where $\zz \sim \normal(\bzero, \bI_d)$. By direct calculation, the three quantities $\hat p(\xx)$, $p^*(\xx)$, and $\hat p_0(\xx)$ can be expressed by
    \begin{align}
        \hat p(\xx) = \sigma\bigl( \< \xx, \hat\vbeta \> + \hat\beta_0 \bigr) 
        & = \sigma\left( \hat \rho \norm{\bmu}_2 y + \< \zz, \hat\vbeta \> +\beta_0 \right),
        \label{eq:p_hat_exp}
        \\
        p^*(\xx) = \P( y = 1 \,|\, \xx )
        & = \frac{\pi e^{-\frac12 \| \xx - \bmu \|_2^2} }{\pi e^{-\frac12 \| \xx - \bmu \|_2^2} + (1 - \pi) e^{-\frac12 \| \xx + \bmu \|_2^2}} 
        \label{eq:p_star_bayes}
        \\
        & = \sigma \left( 2\< \xx, \bmu \> + \log\frac{\pi}{1 - \pi} \right)  \notag \\
        & = \sigma \left( 2 \|\bmu\|_2 \Bigl( \norm{\bmu}_2 y + \< \zz, \bmu/\|\bmu\|_2 \> \Bigr)  + \log\frac{\pi}{1 - \pi} \right),
        \label{eq:p_star_exp}
        \\
        \hat p_0(\xx) = \P \bigl( y = 1 \,|\, \hat p(\xx) \bigr)
        & = \frac{\pi e^{-\frac12 (\hat f(\xx) - \hat\rho\|\bmu\|_2 - \hat\beta_0)^2} }{\pi e^{-\frac12 (\hat f(\xx) - \hat\rho\|\bmu\|_2 - \hat\beta_0)^2} + (1 - \pi) e^{-\frac12 (\hat f(\xx) + \hat\rho\|\bmu\|_2 - \hat\beta_0)^2} } 
        \label{eq:p0_bayes}
        \\
        & = \sigma \left( 2 \, \hat\rho \, \|\bmu\|_2 \< \xx, \hat\bbeta \> + \log\frac{\pi}{1 - \pi} \right) \notag \\
        & = \sigma \left( 2 \, \hat\rho \, \|\bmu\|_2 \left(\hat \rho \norm{\bmu}_2 y + \< \zz, \hat\vbeta \> \right) + \log\frac{\pi}{1 - \pi} \right),
        \label{eq:p0_exp}
    \end{align}
    where the Bayes' law is applied in \cref{eq:p_star_bayes} and \eqref{eq:p0_bayes}.

    Next, it suffices to obtain the joint asymptotics of $\< \zz, \hat\vbeta \>$ and $\< \zz, \bmu/\|\bmu\|_2 \>$, which appear in the expressions of \cref{eq:p_hat_exp}, \eqref{eq:p_star_exp}, \eqref{eq:p0_exp}. Note that $\< \zz, \hat\vbeta \> \cond \normal(0, 1)$ (since $\zz \indep \hat\vbeta$, $\P(\| \hat\vbeta \|_2 = 1) \to 1$), $\< \zz, \bmu/\|\bmu\|_2 \> \sim \normal(0, 1)$. Moreover, $\E[\< \zz, \hat\vbeta \> \< \zz, \bmu/\|\bmu\|_2 \> ] = \E[ \hat\rho ] \to \rho^*$ by \cref{thm:SVM_main} and bounded convergence. These implies
    \begin{equation*}
        \begin{pmatrix}
            \< \zz, \hat\vbeta \> \\
            \< \zz, \bmu/\|\bmu\|_2 \>
        \end{pmatrix}
        \cond
        \normal\left(
        \begin{pmatrix}
            0 \\
            0
        \end{pmatrix},
        \begin{pmatrix}
            1 & \rho^* \\
            \rho^* & 1
        \end{pmatrix}
        \right) \overset{\mathrm{d}}{=} 
        \begin{pmatrix}
            G \\
            G'
        \end{pmatrix}.
    \end{equation*}
    Since $y \indep (\zz, \hat\vbeta)$ and $(\hat\rho, \hat\beta_0) \conp (\rho^*, \beta_0^*)$, we conclude \cref{eq:p(x)_asymp} by \cref{eq:p_hat_exp}, \eqref{eq:p_star_exp}, \eqref{eq:p0_exp} and then using the Slutsky's theorem. This completes the proof.
\end{proof}

The proof of \cref{prop:conf} is primarily based on asymptotics in \cref{lem:conf_limit}.

\begin{proof}[\textbf{Proof of \cref{prop:conf}}]
\textbf{\ref{prop:conf_asymp}:} For $\mathrm{MSE}$, by directly using the asymptotics in \cref{lem:conf_limit} and bounded convergence theorem, we have
\begin{align*}
    \lim_{n \to \infty} \mathrm{MSE}(\hat p) & = \lim_{n \to \infty} \E\left[ \bigl( \mathbbm{1}\{ y = 1 \} - \hat p(\xx) \bigr)^2 \right] \\
    & = \E\left[ \bigl( \mathbbm{1}\{ Y = 1 \} -  \sigma\left( \rho^*\|\vmu\|_2 Y + G + \beta_0^* \right)  \bigr)^2 \right] \\
    & = \E \left[ \sigma \bigl( -\rho^* \norm{\bmu}_2 - \beta_0^* Y + G \bigr)^2 \right] = \mathrm{MSE}^*, 
    \\
    \lim_{n \to \infty} \mathrm{mMSE}(\hat p) & = \mathrm{MSE}^*
    - \var\, \bigl[\mathbbm{1}\{ y = 1 \} \bigr] = \mathrm{MSE}^* - \pi(1 - \pi).
\end{align*}
For $\mathrm{CalErr}$, we similarly get
\begin{align*}
    & \lim_{n \to \infty} \mathrm{CalErr}(\hat p)
    = \lim_{n \to \infty} \E\left[ \bigl( \hat p(\xx) - \hat p_0(\xx) \bigr)^2 \right] \\
    = {} & \E\left[ \left(  \sigma\bigl( \rho^* \norm{\bmu}_2 Y + G + \beta_0^* \bigr)
    - \sigma\Bigl( 2\rho^* \norm{\bmu}_2 (\rho^* \norm{\bmu}_2 Y + G) + \log\frac{\pi}{1-\pi} \Bigr)
     \right)^2 \right]
     = \mathrm{CalErr}^*.
\end{align*}
For $\mathrm{ConfErr}$, we can first obtain
\begin{align*}
    \lim_{n \to \infty} \E \left[ p^*(\xx) \bigl( 1 - p^*(\xx) \bigr) \right]
    & = \lim_{n \to \infty} \E \, \Bigl[ \var\, \bigl[\mathbbm{1}\{ y = 1 \} \,|\, \xx \bigr] \Bigr] \\
    & = \lim_{n \to \infty} \E\left[ \bigl( \mathbbm{1}\{ y = 1 \} - p^*(\xx) \bigr)^2 \right] \\
    & = \E\left[ \left( \mathbbm{1}\{ Y = 1 \} - \sigma \Bigl( 2 \|\bmu\|_2 (\|\bmu\|_2 Y + G) + \log\frac{\pi}{1 - \pi} \Bigr) \right)^2 \right] \\
    & = \E\left[ \sigma\Bigl( -2\norm{\bmu}_2( \norm{\bmu}_2 + G ) - \log\frac{\pi}{1-\pi} Y \Bigr)^2 \right]
    = V_{y|\xx}^*,
\end{align*}
and then by relation between $\mathrm{ConfErr}$ and $\mathrm{MSE}$ \cref{eq:MSE_vs_ConfErr}
\begin{equation*}
    \lim_{n \to \infty} \mathrm{ConfErr}(\hat p)
    = \lim_{n \to \infty} \mathrm{MSE}(\hat p) - \lim_{n \to \infty} \E \left[ p^*(\xx) \bigl( 1 - p^*(\xx) \bigr) \right]
    = \mathrm{MSE}^* - V_{y|\xx}^*.
\end{equation*}
This concludes the proof of part \ref{prop:conf_asymp}.

\vspace{0.5\baselineskip}
\noindent
\textbf{\ref{prop:conf_mono}:}
When $\tau = \tau^\mathrm{opt}$, by \cref{prop:tau_opt} $\beta_0^* = 0$. Then we can simplify
    \begin{equation*}
            \mathrm{MSE}^* =  \E\left[ \bigl(  1 + \exp( \rho^* \norm{\bmu}_2 + G) \bigr)^{-2} \right] .
    \end{equation*}
    According to \cref{lem:rho_mono}, we know that $\rho^* \norm{\bmu}_2$ is increasing in $\pi \in (0, \frac12)$, $\norm{\bmu}_2$, and $\delta$. It suffices to show that $\mathrm{MSE}^*$ is decreasing in $\rho^* \norm{\bmu}_2$, which is obvious by noticing $t \mapsto ( 1 + \exp(t) )^{-2}$ is a strictly decreasing function.

    For $\mathrm{mMSE}^*$, note that $\pi(1 - \pi)$ is a increasing function of $\pi \in (0, \frac12)$, and it does not depend on $\norm{\bmu}_2$, $\delta$. These shows the monotonicity of $\mathrm{mMSE}^* = \mathrm{MSE}^* - \pi(1 - \pi)$.

    For $\mathrm{ConfErr}^*$, note that $V_{y|\xx}^*$ does not depend on $\delta$. This implies that $\mathrm{ConfErr}^*$ has the same monotonicity in $\delta$ as $\mathrm{MSE}^*$, which concludes the proof of part \ref{prop:conf_mono}.
\end{proof}

\subsection{Verification of \cref{claim:conf}}

The analytical dependence of $\mathrm{CalErr}^*$ and $\mathrm{ConfErr}^*$ on model parameters is more complicated. We provide a numerical verification of \cref{claim:conf}.

\begin{proof}[\textbf{Verification of \cref{claim:conf}}]
    For $\mathrm{CalErr}^*$, denote
            \begin{align*}
                h_1(t) & := \E\left[ \Bigl( \sigma\bigl( 2 t (G+t) + c \bigr) - \sigma( G+t ) \Bigr)^2 \right]
                \\
                h_2(t) & := \E\left[ \Bigl( \sigma\bigl( 2 t (G-t) + c \bigr) - \sigma( G-t ) \Bigr)^2 \right]
            \end{align*}
            where $c < 0$ is a constant. When $\tau = \tau^\mathrm{opt}$, we have $\beta_0^* = 0$ and
        \begin{equation*}
            \mathrm{CalErr}^*
            = \pi h_1( \rho^* \norm{\bmu}_2 ) + (1 - \pi) h_2( \rho^* \norm{\bmu}_2 ),
            \qquad \text{where} ~ c = \log\frac{\pi}{1 - \pi}.
    \end{equation*}
    According to \cref{fig:mono_fun}, we can numerically show that $h(t) :=  \pi h_1(t) + (1 - \pi) h_2(t)$ is a decreasing function when $\pi \le \overline{\pi} \approx 0.25$ is fixed. Under this condition, $\mathrm{CalErr}^*$ is decreasing in $\rho^* \norm{\bmu}_2$. Then by using \cref{lem:rho_mono} and similar arguments in the proof of \cref{prop:conf}\ref{prop:conf_mono}, we can conclude the monotonicity of $\mathrm{CalErr}^*$ in $\norm{\bmu}_2$ and $\delta$.

    \begin{figure}[h!]
    \centering
    \includegraphics[width=0.32\textwidth]{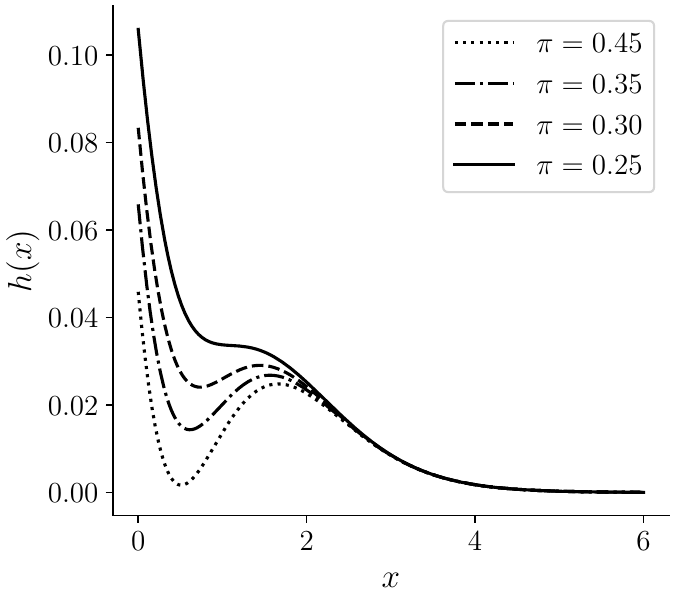}
    \includegraphics[width=0.32\textwidth]{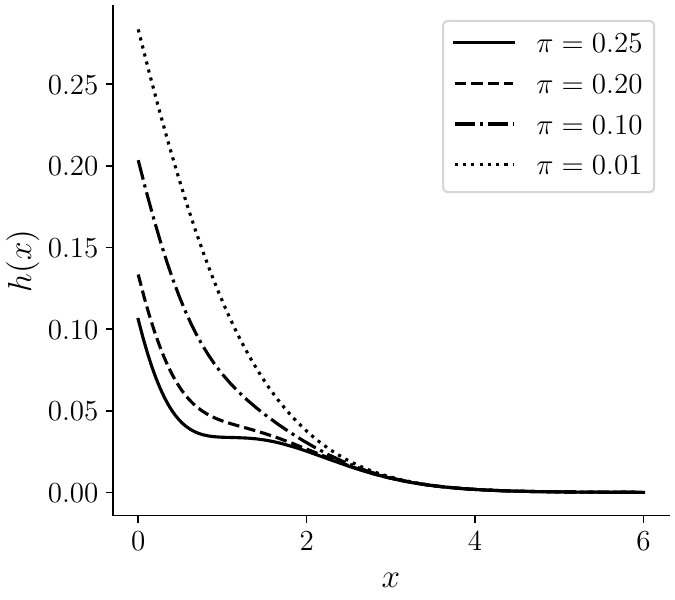}
    \includegraphics[width=0.32\textwidth]{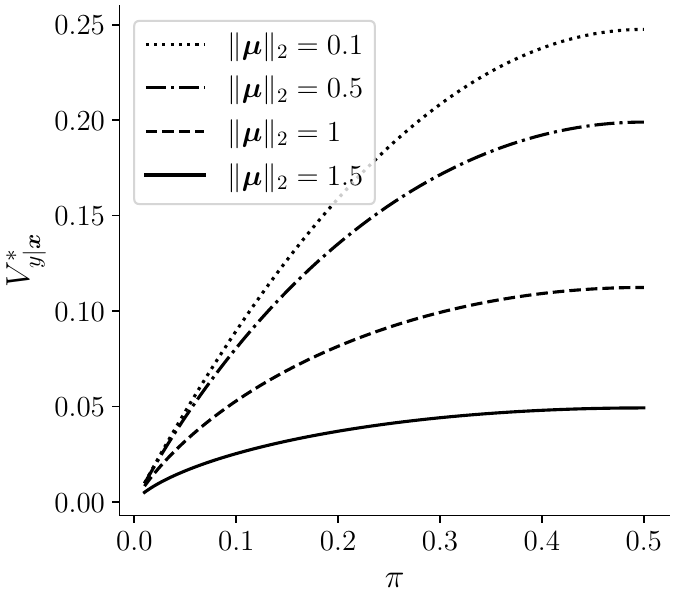}
    \caption{
    \textbf{Monotonicity of $x \mapsto h(x)$ and $\pi \mapsto V_{y|\xx}^*$}. \textbf{Left:} $h$ is not monotone when $\pi > \overline{\pi} \approx 0.25$. \textbf{Middle:} $h$ is monotone decreasing when $\pi \le \overline{\pi} \approx 0.25$. \textbf{Right:} $V_{y|\xx}^*$ is monotone increasing in $\pi$ for different values of $\norm{\bmu}_2$.
    }
    \label{fig:mono_fun}
\end{figure}

For $\mathrm{ConfErr}^*$, in \cref{fig:mono_fun} we numerically show that $V_{y|\xx}^*$ is increasing in $\pi$ when $\norm{\bmu}_2$ is fixed. Since $\mathrm{ConfErr}^* = \mathrm{MSE}^* - V_{y|\xx}^*$ and we have shown in \cref{prop:conf}\ref{prop:conf_mono} that $\mathrm{MSE}^*$ is decreasing in $\pi$, we conclude $\mathrm{ConfErr}^*$ is also decreasing in $\pi$.
\end{proof}

%% file: src/append_tech_lem.tex
\section{Technical Lemmas}
\label{append_sec:tech}

\subsection{Properties of Gaussian random variables}

We need the following variant of Gordon's comparison theorem for Gaussian processes.
\begin{lem}[CGMT]
    \label{lem:CGMT}
    Let $D_{\bu} \subset \R^{n_1 + n_2}$, $D_{\bv} \subset \R^{m_1 + m_2}$ be compact sets and let $Q: D_{\bu} \times D_{\bv} \to \R$ be a continuous function. Let $\GG = (G_{i,j}) \iidsim \normal(0, 1)$, $\vg \sim \normal(\bzero, \bone_{n_1})$, $\hh \sim \normal(\bzero, \bone_{m_1})$ be independent standard Gaussian vectors. For any $\bu \in \R^{n_1 + n_2}$ and $\bv \in \R^{m_1 + m_2}$ we define $\wt\bu = (u_1, \dots, u_{n_1})$ and $\wt\bv = (v_1, \dots, v_{m_1})$. Define
    \begin{equation*}
        \begin{aligned}
            C^*(\GG)      & = \min_{\bu \in D_{\bu}} \max_{\bv \in D_{\bv}}  \wt\bv^\top \GG \wt\bu + Q(\bu, \bv), \\
            L^*(\vg, \hh) & = \min_{\bu \in D_{\bu}} \max_{\bv \in D_{\bv}}  \| \wt\bv \|_2 \vg^\top \wt\bu
            + \| \wt\bu \|_2 \hh^\top \wt\bv + Q(\bu, \bv).
        \end{aligned}
    \end{equation*}
    Then we have:
    \begin{enumerate}[label=(\alph*)]
        \item \label{lem:CGMT(a)}
        For all $t \in \R$,
        \begin{equation*}
            \P\left( C^*(\GG) \le t \right) \le 2 \, \P\left( L^*(\vg, \hh) \le t \right).
        \end{equation*}
        \item \label{lem:CGMT(b)}
        If $D_{\bu}$ and $D_{\bv}$ are convex and if $Q$ is convex concave, then for all $t \in \R$,
        \begin{equation*}
            \P\left( C^*(\GG) \ge t \right) \le 2 \, \P\left( L^*(\vg, \hh) \ge t \right).
        \end{equation*}
    \end{enumerate}
\end{lem}
\begin{proof}
    See \cite[Corollary G.1]{miolane2018distributionlassouniformcontrol}.
\end{proof}

\subsection{Properties of sub-gaussian and sub-exponential random variables}

\begin{defn}[Sub-gaussianity]
    \label{def:subgauss}
    The sub-gaussian norm of random variable $X$ is defined as
    \begin{equation*}
        \norm{X}_{\psi_2} := \inf\left\{ K > 0: \E[\exp(X^2/K^2)] \le 2 \right\}.
    \end{equation*}
    \begin{itemize}
        \item A random variable $X \in \R$ is called sub-gaussian if $\norm{X}_{\psi_2} < \infty$.
        \item A random vector $\xx = (X_1, \dots, X_d)^\top \in \R^d$ is called sub-gaussian if $\sup_{\bv \in \S^{d-1}}\norm{\langle \xx, \bv \rangle}_{\psi_2} < \infty$. Specifically, write $\xx \sim \subGind(\bzero, \bI_d; K)$ if
        $X_1, \dots, X_d$ are independent random variables with $\E[X_i] = 0$, $\var(X_i) = 1$,
        and $\max_{1 \le i \le d} \norm{X_i}_{\psi_2} \lesssim K$. 
    \end{itemize}   
\end{defn}

\noindent
\cref{lem:subG} and \ref{lem:subG_concentrate} summarize some basic facts and concentration inequalities about sub-gaussian random variables and vectors.

\begin{lem}\label{lem:subG}
    Some facts about sub-gaussian 
    random variables.
    \begin{enumerate}[label=(\alph*)]
        \item \label{lem:subG-a} $\norm{\, \cdot \,}_{\psi_2}$ is a norm on the space of sub-gaussian random variables.
        \item \label{lem:subG-b} Let $X_1, \ldots, X_N$ be independent mean-zero sub-gaussian random variables. Then $\sum_{i=1}^N X_i$ is also a sub-gaussian random variable, and
        \[ \norm{\sum_{i=1}^N X_i}_{\psi_2}^2\le C \sum_{i=1}^N \norm{X_i}^2_{\psi_2}, \]
        where $C$ is an absolute constant.
        \item \label{lem:subG-c} (Maximum) Let $X_1, \ldots, X_N$ be sub-gaussian random variables (not necessarily independent) with $K := \max_{1 \le i \le N} \norm{X_i}_{\psi_2}$. Then
        \[ \E\left[ \max_{1 \le i \le N} \abs{X_i} \right] \le  C K \sqrt{\log N},
        \qquad (N \ge 2), \]
        where $C$ is an absolute constant.
    \end{enumerate}
\end{lem}
\begin{proof}
    See \cite[Exercise 2.5.7, Proposition 2.6.1, Exercise 2.5.10]{vershynin2018high}.
\end{proof}

\begin{lem}[Concentration]\label{lem:subG_concentrate}
    Suppose $\xx, \yy \sim \subGind(\bzero, \bI_d; K)$ and $\xx \indep \yy$.
    \begin{enumerate}[label=(\alph*)]
        \item \label{lem:subG-Hanson-Wright-I} (Hanson-Wright inequality I)  \  Let $\bA \in \R^{d \times d}$ be a matrix. Then, for every $t \ge 0$,
        \begin{equation*}
           \P\left( \bigl| \xx^\top \bA \xx - \E[\xx^\top \bA \xx] \bigr| \ge t \right)
           \le 2 \exp\biggl( -c \min \biggl\{ \frac{t^2}{K^4 \norm{\bA}_\mathrm{F}^2} , \frac{ t }{ K^2 \| \bA \|_{\mathrm{op}} } \biggr\} \biggr),
        \end{equation*}
        where $c$ is an absolute constant.
        \item \label{lem:subG-Hanson-Wright-II} (Hanson-Wright inequality II) \  Let $\bB \in \R^{d' \times d}$ be a matrix. Then, for every $t \ge 0$,
        \begin{equation*}
            \P\biggl( \biggl| \frac{\| \bB \xx \|_2}{ \| \bB \|_\mathrm{F}} - 1 \biggr| > t \biggr)
            \le 2 \exp\biggl( -\frac{ct^2 \| \bB \|_\mathrm{F}^2 }{K^4 \| \bB \|_{\mathrm{op}}^2 } \biggr),
        \end{equation*}
        where $c$ is an absolute constant. In particular, when $\bB = \bI_d$,
        \begin{equation*}
            \P\biggl( \biggl| \frac{\| \xx \|_2}{ \sqrt{d} } - 1 \biggr| > t \biggr)
            \le 2 \exp\biggl( -\frac{ct^2 d}{K^4} \biggr).
        \end{equation*}

        \item \label{lem:subG-Hoeffding} (Hoeffding's inequality) \ Let $\ba \in \R^{d}$ be a vector. Then, for every $t \ge 0$,
        \begin{equation*}
            \P\biggl( \frac{ \abs{\< \xx, \ba \>} }{ \norm{\ba}_2 }  > t \biggr)
            \le 2 \exp\biggl( -\frac{ct^2}{K^2} \biggr),
        \end{equation*}
        where $c$ is an absolute constant.

        \item \label{lem:subG-Bernstein} (Bernstein's inequality) \ Let $\bB \in \R^{d \times d}$ be a matrix. Then, for every $t \ge 0$,
        \begin{equation*}
            \P\biggl( \frac{ | \xx^\top \bB \yy | }{ \| \bB \|_\mathrm{F} }  > t \biggr)
            \le 2 \exp\biggl( -c \min \biggl\{ \frac{t^2}{K^4} , \frac{ t \| \bB \|_\mathrm{F}}{ K^2 \| \bB \|_{\mathrm{op}} } \biggr\} \biggr),
        \end{equation*}
        where $c$ is an absolute constant. In particular, when $\bB = \bI_d$,
        \begin{equation*}
            \P\biggl( \frac{ \abs{\< \xx, \yy \>} }{ \sqrt{d} }  > t \biggr)
            \le 2 \exp\biggl( -c \min \biggl\{ \frac{t^2}{K^4} , \frac{t\sqrt{d}}{K^2} \biggr\} \biggr).
        \end{equation*}
    \end{enumerate}
\end{lem}
\begin{proof}
    For \ref{lem:subG-Hanson-Wright-I}, \ref{lem:subG-Hanson-Wright-II} and \ref{lem:subG-Hoeffding}, see \cite[Theorem 6.2.1, Theorem 6.3.2, Theorem 2.6.3]{vershynin2018high}. For \ref{lem:subG-Bernstein},
    let $\bar\xx = \begin{pmatrix}
        \xx \\ \yy
    \end{pmatrix}$, $\bar\bA = \dfrac12\begin{pmatrix}
        \bzero & \bB \\
        \bB & \bzero
    \end{pmatrix}$, then apply $\bar\xx^\top \bar\bA \bar\xx = \xx^\top \bB \yy$ to \ref{lem:subG-Hanson-Wright-I} and simplify.
\end{proof}

\begin{defn}[Sub-exponentiality]
    The sub-exponential norm of random variable $X$ is defined as
    \begin{equation*}
        \norm{X}_{\psi_1} = \inf\left\{ K > 0: \E[\exp(\abs{X}/K)] \le 2 \right\}.
    \end{equation*}
    \begin{itemize}
        \item A random variable $X \in \R$ is called sub-exponential if $\norm{X}_{\psi_1} < \infty$.
    \end{itemize}   
\end{defn}

\noindent
\cref{lem:subExp} summarizes some basic facts about sub-exponential random variables.
\begin{lem}\label{lem:subExp}
    Some facts about sub-exponential random variables.
    \begin{enumerate}[label=(\alph*)]
        \item \label{lem:subExp-a} $\norm{\, \cdot \,}_{\psi_1}$ is a norm on the space of sub-exponential random variables.
        \item \label{lem:subExp-b} Let $X_1, \ldots, X_N$ be independent mean-zero sub-exponential random variables. Then $\sum_{i=1}^N X_i$ is also a sub-exponential random variable. If $K := \max_{1 \le i \le N} \norm{X_i}_{\psi_1}$ and $N \ge C$, then
        \[ 
            \norm{\sum_{i=1}^N X_i}_{\psi_1} \le C' K \sqrt{N},   
        \]
        where $C, C'$ are absolute constants.
        \item \label{lem:subExp-c} (Maximum) Let $X_1, \ldots, X_N$ be sub-exponential random variables (not necessarily independent) with $K := \max_{1 \le i \le N} \norm{X_i}_{\psi_1}$. Then
        \[ \E\left[ \max_{1 \le i \le N} \abs{X_i} \right] \le  C K \log N,
        \qquad (N \ge 2), \]
        where $C$ is an absolute constant.
        \item \label{lem:subExp-d} Let $X$ and $Y$ be sub-gaussian random variables. Then $XY$ is sub-exponential. Moreover,
        \[ \norm{XY}_{\psi_1} \le \norm{X}_{\psi_2} \norm{Y}_{\psi_2}. \]
        In particular, $X^2$ is sub-exponential, and
        \[ \| X^2 \|_{\psi_1} \le \norm{X}_{\psi_2}^2. \] 
    \end{enumerate}
\end{lem}
\begin{proof}
    For \ref{lem:subExp-a} and \ref{lem:subExp-d}, see \cite[Exercise 2.7.11, Lemma 2.7.6, Lemma 2.7.7]{vershynin2018high}. For (b), the proof is analogous to \cite[Proposition 2.6.1]{vershynin2018high}. For any $\abs{\lambda} \le 1/K$, we have
    \begin{equation*}
            \E\biggl[ \exp\biggl(\lambda \sum_{i=1}^N X_i \biggr) \biggr]
             = \prod_{i=1}^{N} \E[\exp(\lambda X_i)]
            \le \prod_{i=1}^{N} \exp\bigl( C \lambda^2 \norm{X_i}_{\psi_1}^2 \bigr)
             \le \exp\bigl( C \lambda^2 N K^2 \bigr),
    \end{equation*}
    where sub-exponential properties \cite[Proposition 2.7.1 (iv)(v)]{vershynin2018high} are used, and $C$ is an absolute constant. If $N \ge 1/C$, then $1/\sqrt{CNK^2} \le 1/K $ and therefore
    \[ \E\biggl[ \exp\biggl(\lambda \sum_{i=1}^N X_i \biggr) \biggr] \le \exp\bigl( \lambda^2 C N K^2 \bigr),
    \quad \text{for all $\lambda$ such that }  \abs{\lambda} \le \frac{1}{\sqrt{CN} K}.  \]
    Then the proof is completed by using \cite[Proposition 2.7.1 (iv)(v)]{vershynin2018high} again.

    For (d), the proof is analogous to \cite[Exercise 2.5.10]{vershynin2018high}. By \cite[Proposition 2.7.1 (i)(iv)]{vershynin2018high}, $\P(\abs{X_i} \ge t) \le 2\exp(-ct/\norm{X}_{\psi_1}) \le 2\exp(-ct/K)$, $\forall\, t \ge 0$, where $c$ is an absolute constant. Denote $t_0 := 2K/c$, then
	\begin{align*}
	& \E\left[ \max_{i \ge 1} \frac{\abs{X_i}}{1 + \log i} \right]
	\le t_0 + \int_{t_0}^\infty \P\biggl( \max_{i \ge 1} \frac{\abs{X_i}}{1 + \log i} > t \biggr) \d t 
	\le \frac{2K}{c} + \int_{t_0}^\infty  \sum_{i=1}^\infty  \P\biggl( \frac{\abs{X_i}}{1 + \log i} > t \biggr) \d t
	\\
	= {} & \frac{2K}{c} + \sum_{i=1}^\infty \int_{t_0}^\infty  \P\bigl( \abs{X_i} > t (1 + \log i) \bigr) \d t
	   \le \frac{2K}{c} + \sum_{i=1}^\infty \int_{t_0}^\infty 2\exp\bigl(-ct(1 + \log i)/K\bigr) \d t\\
	\le {} & 
    \frac{2K}{c} + \sum_{i=1}^\infty \int_{t_0}^\infty  \exp\bigl(- (\log i) ct_0/K \bigr) \cdot  2\exp(-ct/K) \d t
        \le 
    \frac{2K}{c} + \sum_{i=1}^\infty i^{-2} \int_{0}^\infty 2\exp(-ct/K) \d t \\
	= {} & \frac{2K}{c} + C_0 \cdot \frac{2K}{c} 
        \le C K,
	\end{align*}
    where $C_0, C$ are absolute constants. Hence, for any $N \ge 2$,
	\[ \E\left[\max_{1 \le i \le N} \abs{X_i}\right] 
	\le (1+\log N) \cdot \E\left[ \max_{1 \le i \le N} \frac{\abs{X_i}}{1 + \log i} \right] \lesssim K \log N.
	\]
    This concludes the proof.
\end{proof}

\subsection{Properties of the Moreau envelope and proximal operator}
\label{append_subsec_Moreau}

Let $\ell: \R \to \R_{\ge 0}$ be a continuous convex function. For any $x \in \R$ and $\lambda > 0$, the Moreau envelope of $\ell$ is defined as
\begin{equation}\label{eq:envelope}
    \envelope_\ell(x; \lambda) = \envelope_{\lambda\ell}(x)
    := \min_{t \in \R} \left\{  \ell(t) +  \frac1{2\lambda} (t - x)^2 \right\},
\end{equation}
and the proximal operator of $\ell$ is defined as
\begin{equation*}
    \prox_{\ell}(x; \lambda) =
    \prox_{\lambda \ell}(x) := \argmin_{t \in \R} \left\{ \ell(t) +  \frac1{2\lambda} (t - x)^2 \right\}.
\end{equation*}
\begin{lem} \label{lem:prox}
For any $x \in \R, \lambda > 0$, $\prox_{\ell}(x; \lambda)$ is uniquely determined by stationarity condition
    \begin{equation*}
        \prox_{\ell}(x; \lambda) + \lambda \ell'\bigl( \prox_{\ell}(x; \lambda) \bigr) - x = 0.
    \end{equation*}
    \begin{enumerate}[label=(\alph*)]
        \item \label{lem:prox(a)}
        $\envelope_\ell(x; \lambda)$ is continuous and convex in $(x, \lambda)$. If $\ell$ is differentiable, then $\envelope_\ell(x; \lambda)$ is also differentiable in its domain, with partial derivatives
        \begin{equation*}
            \begin{aligned}
                \frac{\partial \envelope_\ell(x; \lambda)}{\partial x}
            & = 
            \mathmakebox[\widthof{$\displaystyle -\frac{1}{2\lambda^2} \big( x - \prox_{\ell}(x; \lambda) \big)^2$}][r]{\frac{1}{\lambda} \big( x - \prox_{\ell}(x; \lambda) \big)\phantom{^2}}
            = 
            \mathmakebox[\widthof{$\displaystyle -\frac12 \bigl(\ell'(z) \bigr)^2 \big|_{z = \prox_{\ell}(x; \lambda) }$}][r]{
            \ell'(z) \phantom{^2} \big|_{z = \prox_{\ell}(x; \lambda) }
            },
            \\
                \frac{\partial \envelope_\ell(x; \lambda)}{\partial \lambda}
            & = -\frac{1}{2\lambda^2} \big( x - \prox_{\ell}(x; \lambda) \big)^2
            = -\frac12 \bigl(\ell'(z) \bigr)^2 \big|_{z = \prox_{\ell}(x; \lambda) }.
            \end{aligned}
        \end{equation*}
        Moreover, $\envelope_\ell(x; \lambda)$ is non-increasing in $\lambda$ and $\envelope_\ell(x; \lambda) \to \ell(x)$ when $\lambda \to 0^+$.
        \item \label{lem:prox(b)}
        $\prox_{\ell}(x; \lambda)$ is continuous in $(x, \lambda)$. If $\ell$ is twice differentiable, then $\prox_\ell(x; \lambda)$ is also differentiable in its domain, with partial derivatives
        \begin{equation*}
            \frac{\partial \prox_{\ell}(x; \lambda)}{\partial x}
            =  \frac{1}{1 + \lambda \ell''(z)} \bigg|_{z = \prox_{\ell}(x; \lambda)}
            \qquad
            \frac{\partial \prox_{\ell}(x; \lambda)}{\partial \lambda}
            =  -\frac{\ell'(z)}{1 + \lambda \ell''(z)} \bigg|_{z = \prox_{\ell}(x; \lambda)} .
        \end{equation*}
        Moreover, $\prox_\ell(x; \lambda) \to x$ when $\lambda \to 0^+$.
    \end{enumerate}
\end{lem}
\begin{proof}
See \cite[Lemma 15]{thrampoulidis2018precise}, \cite[Proposition A.1]{donoho2016high}, \cite[Lemma 2, Lemma 4]{salehi2019impact}, and relevant references therein.
\end{proof}

\section{Miscellaneous}

Let $\hat\kappa$ be the optimal objective value in \cref{eq:SVM}, which is the \emph{maximum margin} for data $(\XX, \yy)$. Moreover, $(\hat\vbeta, \hat\beta_0, \hat\kappa)$ is also the optimal solution to \cref{eq:SVM-m-reb}. Notice $\hat\kappa \ge 0$ always holds (by taking $\vbeta = 0$, $\beta_0 = 0$ in \cref{eq:SVM}), and we can observe the following relation.
\begin{equation*}
	\begin{aligned}
		\text{(linearly separable)} 
		\ \ & \text{$\exists\, \vbeta\not=\bzero$, $\beta_0 \in \R$, such that $y_i ( \< \xx_i, \bbeta \> + \beta_0 ) > 0$, $\forall\, i \in [n]$,} \\
		& \Longleftrightarrow \quad \hat\kappa > 0, \quad \Longrightarrow \quad \|\hat\vbeta\|_2 = 1, \\
		\text{(not linearly separable)} 
		\ \ & \text{$\forall\, \vbeta\not=\bzero$, $\beta_0 \in \R$, such that $y_i ( \< \xx_i, \bbeta \> + \beta_0 ) \overset{\mathmakebox[0pt][c]{\smash{(*)}}}{\le}  0$, $\forall\, i \in [n]$,} \\
		& \Longleftrightarrow \quad \hat\kappa = 0, \quad \Longrightarrow \quad \hat\vbeta = \bzero, \ \hat\beta_0 = 0 \text{ is a solution.\footnotemark}
	\end{aligned}
    \footnotetext{
	If $(*)$ is strict ($<$), then $\hat\vbeta = \bzero$, $\hat\beta_0 = 0$ is the \emph{unique} solution.
}
\end{equation*}
When data is linearly separable, it turns out \cref{eq:SVM-m-reb} also has the following equivalent form:
\begin{equation}
	\label{eq:SVM-1}
    \begin{array}{rl}
    \minimize\limits_{\bw \in \R^d, \, w_0 \in \R} & \norm{\bw}_2^2, \\
    \text{subject to} &  \wt y_i(\< \xx_i, \bw \> + w_0) \ge 1, \quad \forall\, i \in [n].
    \end{array}
\end{equation}
The parameters in \cref{eq:SVM-m-reb} and \eqref{eq:SVM-1} have one-to-one relation $(\kappa, \vbeta, \beta_0) = (1, \bw, w_0)/\|\bw\|_2$. Notably, \cref{eq:SVM-1} is known as the hard-margin SVM \cite{vapnik1998statistical} if $\tau = 1$.